%% file: perc-thin.tex
\documentclass[12pt]{article}
\usepackage{amsthm,amsfonts,amssymb,amsmath,enumerate}

\usepackage{url}
\usepackage{xcolor,tikz}
\usetikzlibrary{matrix}
\usepackage{multicol}
\usepackage{enumitem}
\usepackage{authblk,fullpage}

\setlist[itemize]{noitemsep}

\numberwithin{equation}{section}

\newtheorem{thm}[equation]{Theorem}

\newtheorem{lem}[equation]{Lemma}

\theoremstyle{definition}

\newtheorem{const}[equation]{Construction}

\theoremstyle{remark}

\title{Three-Neighbour Bootstrap Percolation in Thin Three-Dimensional Grids}

\author{Will Dolphin} 
\affil{\normalsize{Pacific School of Innovation and Inquiry, Victoria, B.C., Canada.}}

\author{Peter J. Dukes\thanks{Research supported by NSERC Discovery Grant RGPIN-312595-2023.} }
\affil{\normalsize{Department of Mathematics and Statistics, University of Victoria, Victoria, B.C., Canada.}}
\affil{\texttt{dukes@uvic.ca}}

\begin{document}

\maketitle

\begin{abstract}
We improve the status of the problem of determining minimum-sized percolating sets in $a \times b \times c$ grids under the $3$-neighbour process.  Using several new constructions, we show that optimal percolating sets exist whenever $\min(a,b,c) \ge 7$.  As an important step toward this, we also show that all grids with $\min(a,b,c) \ge 4$ have a percolating set whose size exactly achieves the lower bound $(ab+ac+bc)/3$ whenever this value is an integer.
\end{abstract}

\section{Introduction}

Let $G$ be a graph and $r$ a positive integer.  The \emph{$r$-neighbour bootstrap percolation process} on $G$ begins with an initial set $A_0 \subseteq V(G)$ of `infected' vertices; at each time step, a new vertex becomes infected if it has at least $r$ infected neighbours.  That is, for $t\geq1$, the set of vertices infected at time step $t$ is $$A_t:=A_{t-1}\cup\{v\in V(G): |N_G(v)\cap A_{t-1}|\geq r\}.$$
We say that $A_0$ \emph{percolates} under the $r$-neighbour process on $G$ if $\cup_{t=0}^\infty A_t=V(G)$.  An early application of bootstrap percolation was in modelling the dynamics of ferromagnetism~\cite{ChalupaLeathReich79}.  

For finite graphs, it is natural to ask for the smallest size of an initial set which percolates.  This paper builds on some recent work on the case $r=3$ for a special class of graphs.

For positive integers $a_1,\dots,a_d$, the $d$-\emph{dimensional grid graph} with side lengths $a_1,\dots,a_d$ has vertex set $[a_1] \times \dots \times [a_d]$, with
two vertices $u=(u_1,\dots,u_d)$ and $v=(v_1,\dots,v_d)$ adjacent if and only if, for some  $i\in [d]$ we have $|u_i-v_i|=1$ and $u_j=v_j$ for all $j\neq i$.  Alternatively, this grid is the Cartesian product of paths $P_{a_1} \Box \dots \Box P_{a_d}$.
A focal point of research on bootstrap percolation assumes $G$ is a grid graph; see for instance
\cite{Benevides+21+,DukesNoelRomer23,GravnerHolroydSivakoff21,HedzetHenning25,PrzykuckiShelton20} and the references therein.  The first two of these references make significant progress on the extremal problem with $r=d=3$; in other words, they investigated the smallest size of a percolating set under the $3$-neighbour process on $3$-dimensional grids.  In particular, this question was answered in \cite{DukesNoelRomer23} for $a \times b \times c$ grids whenever $\min(a,b,c) \ge 11$.  One of our main results, Theorem~\ref{thm:opt7}, is to extend this to $\min(a,b,c) \ge 7$.

As we see later in Lemma~\ref{lem:lower-bound}, an initial set $A_0$ which percolates on the $a \times b \times c$ grid must satisfy the lower bound $|A_0| \ge (ab+ac+bc)/3$.  When equality holds, the right side is an integer and the
$3$-neighbour process always infects each new vertex in $A_t$ with exactly three neighbours in $A_{t-1}$.  We say that a grid is \emph{perfect} if it admits a percolating set $A_0$ exactly meeting this condition. 

For convenience, we abbreviate the grid graph $P_a \Box P_b \Box P_c$ with the ordered triple $(a,b,c)$ and say that it has \emph{thickness} equal to $\min(a,b,c)$. 
It was shown in \cite{Benevides+21+} that grids $(1,2^k-1,2^k-1)$ are perfect under the $3$-neighbour process for each positive integer $k$.  Later, it was shown in \cite{DukesNoelRomer23} that these are the only perfect grids of thickness $1$.  Of course, grids of thickness $1$ are equivalent to $2$-dimensional grids, so the problem is fully settled for $d=2$.

Using various nicely structured and computer generated initial sets $A_0$, it was also shown in \cite{DukesNoelRomer23} that all grids $(a,b,c)$ with $3 \mid ab+ac+bc$ and thickness at least $5$ are perfect.  Here, we reduce this minimum thickness to $4$.  This is helpful for our above-mentioned result answering the overall extremal problem for thickness at least $7$.

In the next section, we give the lower bound on $|A_0|$ for $r=d=3$.  Proofs also appear in earlier work on the topic, but we sketch a proof here for completeness.  We also review a useful recursive construction from \cite{DukesNoelRomer23} that assembles a minimum-sized percolating set in terms of those for four smaller grids.  In Section~\ref{sec:perf24}, we give several new constructions of perfect grids having thickness $2$ or $4$.  Some of these are infinite families and analyzed in more detail; illustrations showing the percolation process are also given in the appendix.  In Sections~\ref{sec:perf24} and \ref{sec:opt}, we feed these into the recursive construction to prove our main results.  We conclude in Section~\ref{sec:concl} with a summary of the status of the problem.

\section{Preliminaries}
\label{sec:prelims}

The following lower bound on $3$-neighbour percolating sets in $3$-dimensional grids serves as our starting point and benchmark for our constructions to follow.  A proof can be found in the references, but we provide a sketch here for completeness.

\begin{lem}[Surface area lower bound]
\label{lem:lower-bound}
Suppose $A_0$ is a subset of the grid $(a,b,c)$ that percolates under the $3$-neighbour process. Then 
\begin{equation}
\label{eq:lower-bound}
|A_0| \ge  \frac{ab+ac+bc}{3}.
\end{equation}
\end{lem}

\begin{proof}
Let $G$ denote the grid and consider a set $A \subseteq V(G)$. For $x \in A$, let $n(A,x):=|N_G(x) \cap A|$, and define $n(A):=\sum_{x \in A} n(A,x)$.

Now, suppose $A_0,A_1,\dots,A_t=V(G)$ is a sequence of steps in the $3$-neighbour process on $G$.  The key observation is that the quantity $6|A_i|-n(A_i)$ is nonincreasing over time steps $i=0,1,\dots,t$. This is because each newly infected vertex $x \in A_{i+1}\setminus A_i$ adds six but subtracts two for every neighbour of $x$ in $A_i$.

Suppose $\min(a,b,c) \ge 2$.  There are
\begin{itemize}
\item
$(a-2)(b-2)(c-2)$ vertices of degree $6$, 
\item
$2(a-2)(b-2)+2(a-2)(c-2)+2(b-2)(c-2)$ vertices of degree $5$, 
\item
$4(a-2)+4(b-2)+4(c-2)$ vertices of degree $4$, and
\item
$8$ vertices of degree $3$.
\end{itemize} 
So, after simplification,
\begin{equation}
\label{eq:nG}
n(G)=6abc-2ab-2ac-2bc.    
\end{equation}
Suppose, on the other hand, that $\min(a,b,c)=1$.  If $a=1$ and $b,c \ge 2$, there are $(b-2)(c-2)$ vertices of degree $4$, $2(b-2)+2(c-2)$ vertices of degree $3$, and $4$ vertices of degree $2$.  This gives $n(G)=4bc-2b-2c$, agreeing with \eqref{eq:nG}.  The case $a=b=1$ ($G=P_c$) is trivial, with $n(G)=2c$.  This also matches \eqref{eq:nG}.

It follows that 
$$6|A_0| \ge 6|A_0|-n(A_0) \ge 6|A_t|-n(A_t) = 2ab+2ac+2bc,$$
which implies the bound \eqref{eq:lower-bound}.
\end{proof}

A more general lower bound for higher-dimensional grids can be found in \cite{DukesNoelRomer23}.

As mentioned in the introduction, a grid $(a,b,c)$ is called perfect if it has a percolating set $A_0$ achieving the lower bound in \eqref{eq:lower-bound}.
This can only occur when $ab+ac+bc$ is a multiple of $3$, or equivalently when either two of the integers $a,b,c$ are divisible by three, or all three are in the same congruence class (mod $3$).

A grid $(a,b,c)$ is \emph{optimal} if it has a percolating set of size equal to the ceiling of the fraction on the right of \eqref{eq:lower-bound}.  

Next, we give an important recursive construction used in earlier work, and also useful for our main results to follow.  The key idea is that optimal and perfect grids can be built up from four smaller grids by assembling percolating sets with a $2 \times 2 \times 2$ template configuration.

\begin{lem}[\cite{DukesNoelRomer23}]
\label{lem:recursive}
Let $a_i,b_i,c_i$ be positive integers for $i\in \{1,2\}$.  Suppose that each of the grids $(a_2,b_2,c_1)$, $(a_2,b_1,c_2)$ and $(a_1,b_2,c_2)$ is perfect.
\begin{enumerate}[label={\rm (\alph*)}]
\item If $(a_1,b_1,c_1)$ is perfect, then so is $(a_1+a_2,b_1+b_2,c_1+c_2)$.
\item If $(a_1,b_1,c_1)$ is optimal, then so is $(a_1+a_2,b_1+b_2,c_1+c_2)$.
\end{enumerate}
\end{lem}

\section{Perfect grids of thickness 2 and 4}
\label{sec:perf24}

Recall that the grid $(1,b,c)$ is perfect only if $b=c=2^k-1$ for some positive integer $k$.  Therefore, it is of limited use to apply Lemma~\ref{lem:recursive}(a) with grids of thickness one as ingredients.  In more detail, if (say) $a_1=1$, then the construction would require $b_1=c_1$ and $b_2=c_2$.  This would result in perfect grids $(a,b,c)$ only when $b=c$, and indeed only for a sparse set of such values.

The construction becomes more versatile using grids of thickness $2$ as ingredients.
In \cite{DukesNoelRomer23}, it was shown that grids of the form $(2,b,c)$ are perfect if $b,c \ge 3$ have different parity and $3 \mid 2(b+c)+bc$.  This is summarized below for later reference.

\begin{lem}[{\rm \cite{DukesNoelRomer23}}]
\label{lem:thickness2grids}
The grid $(2,b,c)$ is perfect if $b,c \ge 3$, $b \equiv c \equiv 0$ or $2 \pmod{3}$, and $b\equiv c+3 \pmod{6}$.
\end{lem}

To supplement these constructions, the next lemma gives some perfect constructions for thickness 2 which do not have the parity restriction on $b,c$.

\begin{lem}
\label{lem:thickness2strips}
The following grids are perfect:
\begin{enumerate}[label={\rm (\alph*)}]
\item $(2,3,c)$ and $(2,6,c)$  for all $c \equiv 0 \pmod{3}$, $c \ge 6$.
\item $(2,5,c)$ and $(2,8,c)$  for all $c \equiv 2 \pmod{3}$, $c \ge 5$.
\end{enumerate}
\end{lem}

\begin{proof}
We focus on the cases not already covered by Lemma~\ref{lem:thickness2grids}.  The cases $(2,3,c)$ for $c \equiv 3 \pmod{6}$ were shown to be perfect in Proposition 5.4 of \cite{DukesNoelRomer23}.  For the remaining cases,
the reader is referred to Appendix~\ref{app:perfect-families}, which gives constructions based on periodic initial sets.
\end{proof}

Although Lemma~\ref{lem:thickness2strips} suffices for our purposes, we remark that certain cases admit easy generalizations to two parameter families $(2,b,c)$ with general values $b \equiv c \pmod{6}$.

Shortly, it will be shown that Lemmas~\ref{lem:thickness2grids} and \ref{lem:thickness2strips} combine via the recursive construction to settle most cases with thickness $4$.  However, we need a few special constructions when one or both of the other sides is small.  The next two lemmas cover these cases.  

\begin{lem}
\label{lem:thickness4sporadic}
The following grids are perfect: $(4,6,6)$, $(4,6,9)$, and $(4,9,9)$.
\end{lem}

\begin{proof}
For $(4,6,6)$, we can apply Lemma~\ref{lem:recursive} with $(a_1,a_2)=(1,3)$ and $(b_1,b_2)=(c_1,c_2)=(3,3)$.
For $(4,6,9)$ and $(4,9,9)$, explicit constructions meeting the bound \eqref{eq:lower-bound} were found by computer.  These are given in the appendix as Constructions~\ref{const:469} and \ref{const:499}, respectively.
\end{proof}

\begin{lem}
\label{lem:thickness4strips}
The following grids are perfect:
\begin{enumerate}[label={\rm (\alph*)}]
\item $(4,4,c)$ for all $c \equiv 1 \pmod{3}$, $c \ge 4$.
\item $(4,7,c)$ for all $c \equiv 1 \pmod{3}$, $c \ge 4$.
\end{enumerate}
\end{lem}

\begin{proof}
We divide each case into slightly separate constructions according to the parity of $c$.  Similar to the thickness 2 families, the reader is referred to Appendix~\ref{app:perfect-families} for constructions based on periodic initial sets.
\end{proof}

We are now ready for the first of our main results, which characterizes all perfect grids of the form $(4,b,c)$.

\begin{thm}
\label{thm:perf4}
The grid $(4,b,c)$ is perfect for all $b \equiv c \equiv 0$ or $1 \pmod{3}$, $b,c \ge 4$.
\end{thm}

\begin{proof}
Suppose without loss of generality that $b \le c$.  
For $b \in \{4,7\}$, we make direct use of Lemma~\ref{lem:thickness4strips}.  If both $b,c \in \{6,9\}$, we apply Lemma~\ref{lem:thickness4sporadic}. 
Therefore, in what follows we may assume $c \ge 10$.  We divide the remainder of the proof according to cases on the congruence classes of $b$ and $c$ (mod $6$).  In each of these cases, we use Lemma~\ref{lem:recursive}(a) with $a_1=a_2=2$, and choose $b_1,b_2,c_1,c_2$ such that $b_1+b_2=b$ and $c_1+c_2=c$.

{\sc Case 1}: $b \equiv c \equiv 0 \pmod{3}$.
First, if $b=6$ and $c \ge 12$, we can take
$(b_1,c_1)=(3,6)$.  Perfect grids exist for $(2,3,6)$ and $(2,3,c-6)$ by Lemma~\ref{lem:thickness2strips}(a). Suppose next that $b \equiv c \pmod{2}$ with $b \ge 9$ and $c \ge 15$. We can again take $(b_1,c_1)=(3,6)$.  The grids $(2,3,6)$, $(2,3,c-6)$, $(2,b-3,6)$ are perfect by 
Lemma~\ref{lem:thickness2strips}(a), and $(2,b-3,c-6)$ is perfect by Lemma~\ref{lem:thickness2grids}.  Suppose $b \not\equiv c \pmod{2}$ with $b \ge 9$ and $c \ge 12$.  This time, we take $(b_1,c_1)=(6,6)$.  The grids $(2,6,6)$, $(2,6,c-6)$ and $(2,b-6,6)$ are perfect by Lemma~\ref{lem:thickness2strips} and $(2,b-6,c-6)$ is perfect by Lemma~\ref{lem:thickness2grids}.  In each of these cases, the recursive construction Lemma~\ref{lem:recursive}(a) with the indicated ingredients gives us that $(4,b,c)$ is perfect. 

{\sc Case 2}: $b \equiv c \equiv 1 \pmod{3}$.
If $b=c=10$, we can apply the recursive construction with four copies of $(2,5,5)$, which is perfect by Lemma~\ref{lem:thickness2strips}.  So assume $c \ge 13$ in the remainder of the argument. 
If $b \equiv c \pmod{2}$, use $(b_1,c_1)=(5,8)$.  If $b \not\equiv c \pmod{2}$, use $(b_1,c_1)=(5,5)$.  In either case, $b_2,c_2$ are each at least $5$ and have different parity.  Thus all required perfect grids exist by Lemmas~\ref{lem:thickness2grids} and \ref{lem:thickness2strips}(b).  The recursive construction shows that $(4,b,c)$ is perfect.
\end{proof}

\section{Optimal grids of thickness 7 to 10}
\label{sec:opt}

In this section, we use Theorem~\ref{thm:perf4} to lower the known minimum side length of optimal $3$-dimensional grids under the $3$-neighbour process.  First, we set up some helpful specific constructions, beginning with a series of small optimal grids found by computer search.

\begin{lem}
\label{lem:45small-opt}
The grid $(a,b,c)$ is optimal whenever $a \in \{4,5\}$ and $b,c \in \{2,\dots,7\}$.
\end{lem}

\begin{proof}
Perfect constructions for $\min(a,b,c) \in \{2,3\}$ follow from various small cases found in \cite{DukesNoelRomer23}, and for $\min(a,b,c) \in \{4,5\}$ from Theorem~\ref{thm:perf4}. Otherwise, explicit optimal (non-perfect) constructions for the remaining cases are given in Appendix~\ref{app:optimal}.
\end{proof}

The next lemma summarizes various results in Section 5 of \cite{DukesNoelRomer23}.

\begin{lem}
\label{lem:3small}
The grid $(3,u,v)$ is perfect whenever $u \in \{2,\dots,7\}$ and $v \equiv 3\pmod{6}$, with the exception of $(u,v)=(2,3)$.
\end{lem}

\begin{proof}
These follow directly from the following four results in \cite{DukesNoelRomer23}:
Proposition 5.4 for $u=2$,
Proposition 5.9 for $u=4$,
Proposition 5.11 for $u=6$, and
Proposition 5.13 for $u \in \{3,5,7\}$.
\end{proof}

With these tools in place, we are ready for our main result on optimal 3-dimensional grids.

\begin{thm}
\label{thm:opt7}
The grid $(a,b,c)$ is optimal whenever $\min(a,b,c) \ge 7$.
\end{thm}

\begin{proof}
Suppose without loss of generality that $7 \le a \le b \le c$.  
Theorem 1.6 of \cite{DukesNoelRomer23} handles all cases with $a \ge 11$.  The proof of this uses the recursive construction, stated here as our Lemma~\ref{lem:recursive}(b), combining one optimal and three perfect grids of thickness at least $5$.  With our Theorem~\ref{thm:perf4} and Lemma~\ref{lem:45small-opt}, we have access to optimal and perfect grids of thickness at least $4$.  The same proof as in \cite{DukesNoelRomer23} now works for $a=10$.  We give a minor modification of the argument for each of the cases $a \in \{7,8,9\}$.

Suppose $a=9$. If either $3 \mid b$ or $3 \mid c$, then
$(a,b,c)$ is perfect by Theorem 6.3 of \cite{DukesNoelRomer23}.  Otherwise, let $a_1=6$ and choose $b_1,c_1 \in \{2,4,5,7\}$ such that
$b-b_1\equiv c-c_1\equiv 3 \pmod{6}$.  Put $a_2=3$, $b_2=b-b_1$, and $c_2=c-c_1$. The grids 
$(a_2,b_1,c_2)$ and $(a_2,b_2,c_1)$
are perfect by Lemma~\ref{lem:3small}.  The grid $(a_1,b_2,c_2)$ is perfect, either by Lemma~\ref{lem:3small} if $\min(b_2,c_2)=3$ or by Theorem 6.3 of \cite{DukesNoelRomer23} when $\min(b_2,c_2) \ge 9$.  The grid $(a_1,b_1,c_1)$ is optimal, either by Lemma~\ref{lem:45small-opt} if $\max(b_1,c_1) \le 5$ or in the case $(6,7,7)$ via optimal $(3,4,4)$ and perfect grids $(3,3,3)$ and $(3,3,4)$.  With these choices of $a_i,b_i,c_i$, Lemma~\ref{lem:recursive}(b) implies that $(a,b,c)$ is optimal.

Suppose $a=8$.
If $b \equiv c \equiv 0 \pmod{3}$ or $2 \pmod{3}$, then $(8,b,c)$ is perfect by Theorem 6.3 of \cite{DukesNoelRomer23}.  Otherwise, let $a_1=5$ and choose $b_1,c_1 \in \{2,\dots,7\}$ such that $b-b_1\equiv c-c_1\equiv 3 \pmod{6}$.  Put $a_2=3$, $b_2=b-b_1$, and $c_2=c-c_1$. 
Since $b_2,c_2 \equiv 3 \pmod{6}$, we know that $(a_1,b_2,c_2)=(5,b_2,c_2)$ is perfect, either by Theorem 6.3 of \cite{DukesNoelRomer23} if $b_2,c_2 \ge 9$, or by Lemma~\ref{lem:3small} if $b_2$ or $c_2=3$.  Suppose, for the moment, that neither $(b_1,c_2)$ nor $(c_1,b_2)$ equals $(2,3)$.  Then each of $(a_2,b_1,c_2)=(3,b_1,c_2)$ and $(a_2,b_2,c_1)=(3,b_2,c_1)$ is perfect by Lemma~\ref{lem:3small}.  

The case $(b_1,c_2)=(2,3)$ implies $c = c_1+c_2 \le 7+3 = 10$ while $b=b_1+b_2 \ge 2+9=11$, a contradiction to our assumption that $a \le b \le c$.
The case $(b_2,c_1)=(3,2)$ implies $b = b_1+b_2 \le 7+3 = 10$, and hence $b \in \{8,9,10\}$.  Since perfect grids of thickness 8 are already settled and $c \equiv c_2 = 2 \pmod{3}$, we may assume $b \neq 8$.

For $b=9$ or $10$, we use the minor modification $(b_1,b_2)=(3,6)$ or $(4,6)$, respectively.
We again have that $(a_1,b_2,c_2)=(5,6,c_2)$ is perfect, this time using Lemma~\ref{lem:45small-opt} if $c_2=3$.  The grids $(a_2,b_1,c_2)=(3,b_1,c_2)$ and $(a_2,b_2,c_1)=(3,6,2)$ are perfect by Lemma~\ref{lem:3small}.

The grid $(a_1,b_1,c_1)=(5,b_1,c_1)$ is optimal for each choice of $b_1$ and $c_1$, by Lemma~\ref{lem:45small-opt}.  It follows by Lemma~\ref{lem:recursive}(b) that $(a,b,c)$ is optimal.

Finally, suppose $a=7$.  The proof in this case is similar to that for $a=8$, except for the following differences.  Perfect cases occur when $b\equiv c \equiv 0 \pmod{3}$ or $1 \pmod{3}$.  We use $(a_1,a_2)=(4,3)$. When $b_2 \equiv c_2 \equiv 3 \pmod{6}$, the grid $(a_1,b_2,c_2)$ is perfect by Theorem~\ref{thm:perf4}.  Grids $(a_2,b_1,c_2)$ and $(a_2,b_2,c_1)$ are again perfect by Lemma~\ref{lem:3small} unless $(b_2,c_1)=(3,2)$.  This only occurs for $b \le 10$ and $c\equiv 5 \pmod{6}$.  
The same patch taking $b_2=6$ works except when $b=7$. In that case, we have $(a,b,c)=(7,7,c)$.  Here, we can use $(a_1,a_2)=(b_1,b_2)=(4,3)$.  When $c=11$, the choice $c_1=5$ yields one optimal and three perfect ingredients by Lemma~\ref{lem:45small-opt}.  When $c \ge 17$, the recursive construction with $c_1=8$ works assuming an optimal grid $(4,4,8)$ and perfect grids $(3,3,8)$ and $(3,4,c-8)$.  The first of these is given as Construction~\ref{const:448} in the Appendix and the other two are shown to exist in Propositions 5.8 and 5.9, respectively, of \cite{DukesNoelRomer23}.
\end{proof}

\section{Conclusion}
\label{sec:concl}

In this paper, we have studied the $3$-neighbour bootstrap percolation process on $a \times b \times c$ grids.  We have extended the near-classification of perfect grids via several new constructions.  In particular, it is now known that $(a,b,c)$ is perfect whenever $\min(a,b,c) \ge 4$ and $3 \mid ab+ac+bc$.  As a consequence, we have also shown that $(a,b,c)$ is optimal whenever $\min(a,b,c) \ge 7$.

Perfect two-dimensional grids were fully classified in \cite{DukesNoelRomer23}.  
To complete the classification of perfect three-dimensional grids, all that remains is to find certain still outstanding two-parameter perfect constructions for grids of thickness 2 and 3. Recent work of Jaan Allikvere \cite{Allikvere26} has made significant progress on these cases.

Following this, a full classification of optimal $3$-dimensional grids should be within reach.  The thickness one (two-dimensional) case was recently addressed by Bushaw and Clifton \cite{BushawClifton26}.  For thickness two or more, the non-optimal grids are almost surely limited to $(2,3,3)$ and $(2,2,c)$ for $c \ge 8$.  In the first case, the minimum size of a percolating set is 8.  In the second family of cases, a straightforward argument gives $\lceil \frac{3c+1}{2} \rceil$ as the minimum size.

\subsection*{AI Disclosure}

We did not use any AI tools in the creation of this manuscript.

\bibliographystyle{plain}
\bibliography{perc-thin}

\newpage
\appendix

\section{Appendix: Constructions in Small Grids}

We present several perfect and optimal constructions in specific grids used as building blocks in our proofs of Theorems~\ref{thm:perf4} and \ref{thm:opt7}.  

An $a \times b \times c$ grid is presented as $a$ layers, each with $b$ rows and $c$ columns. An initial set $A_0$ of cardinality $\lceil (ab+ac+bc)/3 \rceil$ is indicated in red, along with time steps for percolation to other cells in the grid.

\subsection{Small perfect constructions}
\label{app:perfect-small}

\begin{const}
\label{const:469}
$(4,6,9)$ is perfect.
\end{const}

\input{4-6-9-tikz.txt}

\newpage
\begin{const}
\label{const:499}
$(4,9,9)$ is perfect.
\end{const}

\input{4-9-9-tikz.txt}

\newpage
\subsection{One-parameter perfect constructions}
\label{app:perfect-families}

The following constructions are periodic with repeating blocks of six columns in the middle of the grid.  Each such block has $2(a+b)$ marked cells so that, when boundary cells are included, the initial set $A_0$ meets the bound of Lemma~\ref{lem:lower-bound}.

Examples are shown for $24 \le c < 30$.  The common behavior is that various walls or sub-grids percolate after a fixed number of time steps, followed by additional milestones reached at roughly multiples of $c$ further time steps.  The progression is outlined in a table following each construction.
A cell labeled $(x,y,z)$ appears in layer $x$, row $y$, column $z$.

\begin{const}
\label{const:25c}
$(2,5,c)$ is perfect for all $c \equiv 5 \pmod{6}$.
\end{const}

\input{2-5-29-tikz.txt}

\begin{tabular}{llll}
 & time step & region & reasons \\
\hline
(i) & 1 & column 1 of layer 1 \\
(ii) & 3 & right three columns of layer 1 \\
(iii) & 3 & upper right $2 \times 2$ grid of layer 2 & (ii) \\
(iv) & 5 & upper-left $3 \times (c-4)$ grid of layer 1 \\
(v) & 5 & lower-left $3 \times (c-4)$ grid of layer 2 \\
(vi) & $c-1$ & layer 1, except for column $c-3$ & (i)--(v) \\
(vii) & $c-1$ & lower-left $4 \times (c-4)$ grid of layer 2 & (v), $(1,2,2)$ \\
(viii) & $c+4$ & rest of layer 1 & (ii), (v), (vi), $(2,5,c-3)$ \\
(ix) & $c+10$ & layer 2, except for row 1 & (vii), (viii)\\
(x) & $2c+3$ & rest of layer 2 & (viii), (ix)\\
\hline
\end{tabular}

\vspace{5mm}
\noindent
{\bf Remark.}
The minimal case $(2,5,5)$ for this construction uses columns $1,c-4,\dots,c$, and a similar choice of boundary columns can be used for minimal instances of those to follow.
\newpage
\begin{const}
\label{const:26c}
$(2,6,c)$ is perfect for all $c \equiv 0 \pmod{6}$.
\end{const}

\input{2-6-24-tikz.txt}

\begin{tabular}{llll}
 & time step & region & reasons \\
\hline
(i) & 2 & rows 5 and 6 of layer 1,\\
&& ~~except for $(1,5,1)$, $(1,6,1)$, $(1,5,2)$ \\
(ii) & 6 & rows 1 to 4 of layer 2 \\
(iii) & 6 & right three columns of layer 2 \\
(iv) & 7 & right column of layer 1, except for $(1,1,c)$ \\
(v) & $c+3$ & rest of layer 2 & (i)--(iii), $(2,6,1)$ \\
(vi) & $c+6$ & rows 2 to 6 of layer 1, except for column 1 & (i), (iv), (v) \\
(vii) & $2c+5$ & rest of layer 1 & (v), (vi), $(1,1,1)$\\
\hline
\end{tabular}

\newpage
\begin{const}
\label{const:28c}
$(2,8,c)$ is perfect for all $c \equiv 2 \pmod{6}$.
\end{const}

\input{2-8-26-tikz.txt}

\begin{tabular}{llll}
 & time step & region & reasons \\
\hline
(i) & 3 & rows 7 and 8 of layer 1,   \\
&& ~~except for $(1,7,1)$, $(1,8,1)$, $(1,7,2)$ \\
(ii) & 6 & rows 1 to 6 of layer 2 \\
(iii) & 6 & right two columns of layer 2 \\
(iv) & 7 & right column of layer 1, except for $(1,1,c)$ \\
(v) & $c+1$ & rest of layer 2 & (i)--(iii), $(2,8,1)$ \\
(vi) & $c+6$ & rows 2 to 8 of layer 1, except for column 1 & (i), (iv), (v) \\
(vi) & $2c+5$ & rest of layer 1 & (v), (vi), $(1,1,1)$ \\
\hline
\end{tabular}

\newpage
\begin{const}
\label{const:44c1}
$(4,4,c)$ is perfect for all $c \equiv 1 \pmod{6}$, $c \ge 7$.
\end{const}

\input{4-4-25-tikz.txt}

\begin{tabular}{llll}
 & time step & region & reasons \\
\hline
(i) & 1 & rows 3 and 4 of layer 1, \\
&&~~except for (1,3,2), (1,3,3), (1,3,4), (1,4,3) \\
(ii) & 7 & rows 1 to 3 of layer 4, 
except (4,3,1), (4,3,2) \\
(iii) & 
$t^*$ & rows 1 to 3 of layer 3, except (3,3,1), (3,3,2) & (ii) \\
(iv) & $t+1$ & $(2,3,c)$ and $(3,4,c)$ & (i), (iii), $(2,4,c)$, $(4,4,c)$\\
(v) & $t+c$ & layer 4, except the bottom left $2 \times 2$ grid & (ii)--(iv) \\
(vi) &  $t+c+2$  & layer 3, except the bottom left $2 \times 2$ grid & (iii)--(v), $(4,4,3)$ \\
(vii) &  $t+c+12$  & rows 3 and 4 of layers 1 and 2 & (i), (vi), $(2,2,3)$ \\
(viii) & $t+c+13$ & rest of rows $3$ and $4$ on all layers & (i), (vi)--(viii), $(2,4,1)$ \\
(ix) & $t+2c+12$ & rest of layers 1 and 2 & (viii), $(1,1,1)$, $(2,2,3)$ \\
\hline
\end{tabular}

\hspace{5mm} $~^*$-time step $t:=\frac{5}{3}(c-1)-8$

\newpage
\begin{const}
\label{const:44c4}
$(4,4,c)$ is perfect for all $c \equiv 4 \pmod{6}$, $c \ge 10$.
\end{const}

\input{4-4-28-tikz.txt}

\begin{tabular}{llll}
 & time step & region & reasons \\
\hline
(i) & 4 & row 3 of layer 3 \\
(ii) & 7 & rows 3 and 4 of layer 1,\\ 
&&~~except for (1,3,4), (1,4,4), (1,4,5) \\
(iii) & 7 & rows 1 to 3 of layer 4 \\
(iv) & 9 & columns 1 to 4 of layers 3 and 4 \\
(v) & 10 & columns 1 to 3 of layers 1 and 2 \\
(vi) & $c-2$ & row 3 of layer 2 & (i), (ii), $(2,3,c)$ \\
(vii) & $c+1$ & (1,3,4), (1,4,4), (1,4,5) & (ii), (v), (vi)\\
(viii) & $c+5$ & rows 1 and 2 of layer 3 & (i), (iii), (iv)\\
(ix) & $2c-3$ & rows 1 and 2 of layers 1 and 2 & (v)--(viii) \\
(x) & $2c-1$ & row 4 of layers 2 to 4 & (iii), (vi) \\
\hline
\end{tabular}


\vspace{1cm}

\begin{const}
\label{const:47c1}
$(4,7,c)$ is perfect for all $c \equiv 1 \pmod{6}$, $c \ge 7$.
\end{const}

\newpage
\input{4-7-25-tikz.txt}

\begin{tabular}{llll}
 & time step & region & reasons \\
\hline
(i) & 8 & row 7 of layer 3 \\
(ii) & 15 & layer 4, except for columns 1 and 2 & (i) \\
(iii) & $c+7$ & rows 5 and 6 of layer 3 & (i), (ii), $(3,5,1)$, $(3,6,2)$ \\
(iv) & $t^*$ & rows 5 to 7 of layers 1 and 2, & (iii), $(1,4,2)$, $(1,6,2)$,\\
&&~~except for columns 1, $c-1$, $c$ & ~~$(2,5,2)$, etc. \\
(v) & $t+3$ & rows 5 to 7 of layer 2, except column $1$ & (iii), (iv), $(2,7,c-1)$, $(2,4,c)$ \\
(vi) & $2c+8$ & layer 3, except for columns 1 and 2 & (ii), (iii), $(3,1,c)$  \\
(vii) & $t+c+2$ & layer 2, except for columns 1 and 2& (vi), $(2,2,c)$, $(2,4,c)$\\
(viii) & $t+c+8$ & rest of layers 2 to 4 & (ii), (vi), (vii), $(3,3,1)$\\
(ix) & $t+2c+14$ & rest of layer 1 & (iv), (viii), (1,1,1), $(1,1,c)$\\
\hline
\end{tabular}

\hspace{5mm} $~^*$-time step $t:=\frac{5}{3}(c-1)+4$

\newpage
\begin{const}
\label{const:47c4}
$(4,7,c)$ is perfect for all $c \equiv 4 \pmod{6}$, $c \ge 10$.
\end{const}

\input{4-7-28-tikz.txt}

\begin{tabular}{llll}
 & time step & region & reasons \\
\hline
(i) & 9 & columns 6 to $c$ of layer 1,
except for $(1,1,c)$ \\
(ii) & 17 & row 7 of each layer \\
(iii) & 22 & columns 1 to 5 of layers 1 to 4 \\
&&~~except for rows 1 and 2 of layer 4 \\
(iv) & $c+16$ & layers 2 and 3, except for $(2,1,c)$ and $(3,1,c)$ & (i)--(iii) \\
(v) & $c+16$ & rows 3 to 7 of layer 4 & (ii)--(iv) \\
(vi) & $c+19$ & the rest of layers 1 to 3 & (i), (iv), $(4,1,c)$\\
(vii) & $2c+17$ & the rest of layer 4 & (ii), (iii), (vi)\\
\hline
\end{tabular}

\newpage
\subsection{Small optimal constructions}
\label{app:optimal}

\begin{const}
\label{const:234}
$(2,3,4)$ is optimal.
\end{const}

\input{2-3-4-tikz.txt}

\begin{const}
\label{const:235}
$(2,3,5)$ is optimal.
\end{const}

\input{2-3-5-tikz.txt}

\begin{const}
\label{const:244}
$(2,4,4)$ is optimal.
\end{const}

\input{2-4-4-tikz.txt}

\begin{const}
\label{const:245}
$(2,4,5)$ is optimal.
\end{const}

\input{2-4-5-tikz.txt}

\begin{const}
\label{const:246}
$(2,4,6)$ is optimal.
\end{const}

\input{2-4-6-tikz.txt}

\begin{const}
\label{const:247}
$(2,4,7)$ is optimal.
\end{const}

\input{2-4-7-tikz.txt}

\begin{const}
\label{const:256}
$(2,5,6)$ is optimal.
\end{const}

\input{2-5-6-tikz.txt}

\begin{const}
\label{const:257}
$(2,5,7)$ is optimal.
\end{const}

\input{2-5-7-tikz.txt}

\begin{const}
\label{const:267}
$(2,6,7)$ is optimal.
\end{const}

\input{2-6-7-tikz.txt}

\begin{const}
\label{const:344}
$(3,4,4)$ is optimal.
\end{const}

\input{3-4-4-tikz.txt}

\begin{const}
\label{const:345}
$(3,4,5)$ is optimal.
\end{const}

\input{3-4-5-tikz.txt}

\begin{const}
\label{const:347}
$(3,4,7)$ is optimal.
\end{const}

\input{3-4-7-tikz.txt}

\begin{const}
\label{const:355}
$(3,5,5)$ is optimal.
\end{const}

\input{3-5-5-tikz.txt}

\begin{const}
\label{const:357}
$(3,5,7)$ is optimal.
\end{const}

\input{3-5-7-tikz.txt}

\begin{const}
\label{const:445}
$(4,4,5)$ is optimal.
\end{const}

\input{4-4-5-tikz.txt}

\begin{const}
\label{const:446}
$(4,4,6)$ is optimal.
\end{const}

\input{4-4-6-tikz.txt}

\begin{const}
\label{const:455}
$(4,5,5)$ is optimal.
\end{const}

\input{4-5-5-tikz.txt}

\begin{const}
\label{const:456}
$(4,5,6)$ is optimal.
\end{const}

\input{4-5-6-tikz.txt}

\begin{const}
\label{const:457}
$(4,5,7)$ is optimal.
\end{const}

\input{4-5-7-tikz.txt}

\begin{const}
\label{const:467}
$(4,6,7)$ is optimal.
\end{const}

\input{4-6-7-tikz.txt}

\begin{const}
\label{const:448}
$(4,4,8)$ is optimal.
\end{const}

\input{4-4-8-tikz.txt}

\end{document}

%% file: 4-6-9-tikz.txt
\definecolor{9x6x4color1}{RGB}{255, 128, 128}
\definecolor{9x6x4color2}{RGB}{255, 131, 131}
\definecolor{9x6x4color3}{RGB}{255, 134, 134}
\definecolor{9x6x4color4}{RGB}{255, 137, 137}
\definecolor{9x6x4color5}{RGB}{255, 141, 141}
\definecolor{9x6x4color6}{RGB}{255, 144, 144}
\definecolor{9x6x4color7}{RGB}{255, 147, 147}
\definecolor{9x6x4color8}{RGB}{255, 150, 150}
\definecolor{9x6x4color9}{RGB}{255, 154, 154}
\definecolor{9x6x4color10}{RGB}{255, 157, 157}
\definecolor{9x6x4color11}{RGB}{255, 160, 160}
\definecolor{9x6x4color12}{RGB}{255, 163, 163}
\definecolor{9x6x4color13}{RGB}{255, 167, 167}
\definecolor{9x6x4color14}{RGB}{255, 170, 170}
\definecolor{9x6x4color15}{RGB}{255, 173, 173}
\definecolor{9x6x4color16}{RGB}{255, 176, 176}
\definecolor{9x6x4color17}{RGB}{255, 180, 180}
\definecolor{9x6x4color18}{RGB}{255, 183, 183}
\definecolor{9x6x4color19}{RGB}{255, 186, 186}
\definecolor{9x6x4color20}{RGB}{255, 189, 189}
\definecolor{9x6x4color21}{RGB}{255, 193, 193}
\definecolor{9x6x4color22}{RGB}{255, 196, 196}
\definecolor{9x6x4color23}{RGB}{255, 199, 199}
\definecolor{9x6x4color24}{RGB}{255, 202, 202}
\definecolor{9x6x4color25}{RGB}{255, 206, 206}
\definecolor{9x6x4color26}{RGB}{255, 209, 209}
\definecolor{9x6x4color27}{RGB}{255, 212, 212}
\definecolor{9x6x4color28}{RGB}{255, 215, 215}
\definecolor{9x6x4color29}{RGB}{255, 219, 219}
\definecolor{9x6x4color30}{RGB}{255, 222, 222}
\definecolor{9x6x4color31}{RGB}{255, 225, 225}
\definecolor{9x6x4color32}{RGB}{255, 228, 228}
\definecolor{9x6x4color33}{RGB}{255, 232, 232}
\definecolor{9x6x4color34}{RGB}{255, 235, 235}
\definecolor{9x6x4color35}{RGB}{255, 238, 238}
\definecolor{9x6x4color36}{RGB}{255, 241, 241}
\definecolor{9x6x4color37}{RGB}{255, 245, 245}
\definecolor{9x6x4color38}{RGB}{255, 248, 248}
\definecolor{9x6x4color39}{RGB}{255, 255, 255}

\begin{center}
\begin{tikzpicture}[scale=0.5]
\begin{scope}
\draw[very thick] (0,-0) rectangle (9,-6);
\foreach \b in {0,...,5}
 \foreach \c in {0,...,8}
  \draw (\c,\b*-1-0) rectangle (\c+1,(\b*-1-1);
\draw[very thick] (0,-7) rectangle (9,-13);
\foreach \b in {0,...,5}
 \foreach \c in {0,...,8}
  \draw (\c,\b*-1-7) rectangle (\c+1,(\b*-1-8);
\draw[very thick] (0,-14) rectangle (9,-20);
\foreach \b in {0,...,5}
 \foreach \c in {0,...,8}
  \draw (\c,\b*-1-14) rectangle (\c+1,(\b*-1-15);
\draw[very thick] (0,-21) rectangle (9,-27);
\foreach \b in {0,...,5}
 \foreach \c in {0,...,8}
  \draw (\c,\b*-1-21) rectangle (\c+1,(\b*-1-22);

\draw[fill=red] (0,-0) rectangle (1,-1);
\draw[fill=9x6x4color17] (1,-0) rectangle (2,-1);
\node at (1.5,-0.5) {17};
\draw[fill=9x6x4color16] (2,-0) rectangle (3,-1);
\node at (2.5,-0.5) {16};
\draw[fill=9x6x4color15] (3,-0) rectangle (4,-1);
\node at (3.5,-0.5) {15};
\draw[fill=red] (4,-0) rectangle (5,-1);
\draw[fill=9x6x4color1] (5,-0) rectangle (6,-1);
\node at (5.5,-0.5) {1};
\draw[fill=red] (6,-0) rectangle (7,-1);
\draw[fill=9x6x4color35] (7,-0) rectangle (8,-1);
\node at (7.5,-0.5) {35};
\draw[fill=9x6x4color36] (8,-0) rectangle (9,-1);
\node at (8.5,-0.5) {36};
\draw[fill=9x6x4color19] (0,-1) rectangle (1,-2);
\node at (0.5,-1.5) {19};
\draw[fill=9x6x4color18] (1,-1) rectangle (2,-2);
\node at (1.5,-1.5) {18};
\draw[fill=9x6x4color15] (2,-1) rectangle (3,-2);
\node at (2.5,-1.5) {15};
\draw[fill=9x6x4color14] (3,-1) rectangle (4,-2);
\node at (3.5,-1.5) {14};
\draw[fill=9x6x4color13] (4,-1) rectangle (5,-2);
\node at (4.5,-1.5) {13};
\draw[fill=9x6x4color14] (5,-1) rectangle (6,-2);
\node at (5.5,-1.5) {14};
\draw[fill=9x6x4color25] (6,-1) rectangle (7,-2);
\node at (6.5,-1.5) {25};
\draw[fill=9x6x4color34] (7,-1) rectangle (8,-2);
\node at (7.5,-1.5) {34};
\draw[fill=9x6x4color35] (8,-1) rectangle (9,-2);
\node at (8.5,-1.5) {35};
\draw[fill=9x6x4color20] (0,-2) rectangle (1,-3);
\node at (0.5,-2.5) {20};
\draw[fill=9x6x4color19] (1,-2) rectangle (2,-3);
\node at (1.5,-2.5) {19};
\draw[fill=red] (2,-2) rectangle (3,-3);
\draw[fill=9x6x4color11] (3,-2) rectangle (4,-3);
\node at (3.5,-2.5) {11};
\draw[fill=red] (4,-2) rectangle (5,-3);
\draw[fill=9x6x4color15] (5,-2) rectangle (6,-3);
\node at (5.5,-2.5) {15};
\draw[fill=9x6x4color24] (6,-2) rectangle (7,-3);
\node at (6.5,-2.5) {24};
\draw[fill=9x6x4color23] (7,-2) rectangle (8,-3);
\node at (7.5,-2.5) {23};
\draw[fill=red] (8,-2) rectangle (9,-3);
\draw[fill=9x6x4color31] (0,-3) rectangle (1,-4);
\node at (0.5,-3.5) {31};
\draw[fill=9x6x4color30] (1,-3) rectangle (2,-4);
\node at (1.5,-3.5) {30};
\draw[fill=9x6x4color29] (2,-3) rectangle (3,-4);
\node at (2.5,-3.5) {29};
\draw[fill=9x6x4color28] (3,-3) rectangle (4,-4);
\node at (3.5,-3.5) {28};
\draw[fill=9x6x4color27] (4,-3) rectangle (5,-4);
\node at (4.5,-3.5) {27};
\draw[fill=9x6x4color26] (5,-3) rectangle (6,-4);
\node at (5.5,-3.5) {26};
\draw[fill=9x6x4color25] (6,-3) rectangle (7,-4);
\node at (6.5,-3.5) {25};
\draw[fill=9x6x4color22] (7,-3) rectangle (8,-4);
\node at (7.5,-3.5) {22};
\draw[fill=9x6x4color1] (8,-3) rectangle (9,-4);
\node at (8.5,-3.5) {1};
\draw[fill=9x6x4color32] (0,-4) rectangle (1,-5);
\node at (0.5,-4.5) {32};
\draw[fill=9x6x4color31] (1,-4) rectangle (2,-5);
\node at (1.5,-4.5) {31};
\draw[fill=9x6x4color30] (2,-4) rectangle (3,-5);
\node at (2.5,-4.5) {30};
\draw[fill=9x6x4color29] (3,-4) rectangle (4,-5);
\node at (3.5,-4.5) {29};
\draw[fill=9x6x4color28] (4,-4) rectangle (5,-5);
\node at (4.5,-4.5) {28};
\draw[fill=9x6x4color27] (5,-4) rectangle (6,-5);
\node at (5.5,-4.5) {27};
\draw[fill=9x6x4color26] (6,-4) rectangle (7,-5);
\node at (6.5,-4.5) {26};
\draw[fill=9x6x4color21] (7,-4) rectangle (8,-5);
\node at (7.5,-4.5) {21};
\draw[fill=red] (8,-4) rectangle (9,-5);
\draw[fill=9x6x4color33] (0,-5) rectangle (1,-6);
\node at (0.5,-5.5) {33};
\draw[fill=9x6x4color32] (1,-5) rectangle (2,-6);
\node at (1.5,-5.5) {32};
\draw[fill=9x6x4color31] (2,-5) rectangle (3,-6);
\node at (2.5,-5.5) {31};
\draw[fill=9x6x4color30] (3,-5) rectangle (4,-6);
\node at (3.5,-5.5) {30};
\draw[fill=9x6x4color29] (4,-5) rectangle (5,-6);
\node at (4.5,-5.5) {29};
\draw[fill=9x6x4color28] (5,-5) rectangle (6,-6);
\node at (5.5,-5.5) {28};
\draw[fill=9x6x4color27] (6,-5) rectangle (7,-6);
\node at (6.5,-5.5) {27};
\draw[fill=red] (7,-5) rectangle (8,-6);
\draw[fill=9x6x4color1] (8,-5) rectangle (9,-6);
\node at (8.5,-5.5) {1};
\draw[fill=9x6x4color1] (0,-7) rectangle (1,-8);
\node at (0.5,-7.5) {1};
\draw[fill=red] (1,-7) rectangle (2,-8);
\draw[fill=9x6x4color13] (2,-7) rectangle (3,-8);
\node at (2.5,-7.5) {13};
\draw[fill=red] (3,-7) rectangle (4,-8);
\draw[fill=9x6x4color1] (4,-7) rectangle (5,-8);
\node at (4.5,-7.5) {1};
\draw[fill=red] (5,-7) rectangle (6,-8);
\draw[fill=9x6x4color27] (6,-7) rectangle (7,-8);
\node at (6.5,-7.5) {27};
\draw[fill=9x6x4color34] (7,-7) rectangle (8,-8);
\node at (7.5,-7.5) {34};
\draw[fill=9x6x4color35] (8,-7) rectangle (9,-8);
\node at (8.5,-7.5) {35};
\draw[fill=red] (0,-8) rectangle (1,-9);
\draw[fill=9x6x4color1] (1,-8) rectangle (2,-9);
\node at (1.5,-8.5) {1};
\draw[fill=9x6x4color12] (2,-8) rectangle (3,-9);
\node at (2.5,-8.5) {12};
\draw[fill=9x6x4color11] (3,-8) rectangle (4,-9);
\node at (3.5,-8.5) {11};
\draw[fill=9x6x4color12] (4,-8) rectangle (5,-9);
\node at (4.5,-8.5) {12};
\draw[fill=9x6x4color13] (5,-8) rectangle (6,-9);
\node at (5.5,-8.5) {13};
\draw[fill=9x6x4color26] (6,-8) rectangle (7,-9);
\node at (6.5,-8.5) {26};
\draw[fill=9x6x4color33] (7,-8) rectangle (8,-9);
\node at (7.5,-8.5) {33};
\draw[fill=9x6x4color34] (8,-8) rectangle (9,-9);
\node at (8.5,-8.5) {34};
\draw[fill=9x6x4color1] (0,-9) rectangle (1,-10);
\node at (0.5,-9.5) {1};
\draw[fill=red] (1,-9) rectangle (2,-10);
\draw[fill=9x6x4color1] (2,-9) rectangle (3,-10);
\node at (2.5,-9.5) {1};
\draw[fill=9x6x4color10] (3,-9) rectangle (4,-10);
\node at (3.5,-9.5) {10};
\draw[fill=9x6x4color11] (4,-9) rectangle (5,-10);
\node at (4.5,-9.5) {11};
\draw[fill=red] (5,-9) rectangle (6,-10);
\draw[fill=9x6x4color23] (6,-9) rectangle (7,-10);
\node at (6.5,-9.5) {23};
\draw[fill=red] (7,-9) rectangle (8,-10);
\draw[fill=9x6x4color1] (8,-9) rectangle (9,-10);
\node at (8.5,-9.5) {1};
\draw[fill=red] (0,-10) rectangle (1,-11);
\draw[fill=9x6x4color1] (1,-10) rectangle (2,-11);
\node at (1.5,-10.5) {1};
\draw[fill=9x6x4color4] (2,-10) rectangle (3,-11);
\node at (2.5,-10.5) {4};
\draw[fill=9x6x4color9] (3,-10) rectangle (4,-11);
\node at (3.5,-10.5) {9};
\draw[fill=9x6x4color14] (4,-10) rectangle (5,-11);
\node at (4.5,-10.5) {14};
\draw[fill=9x6x4color15] (5,-10) rectangle (6,-11);
\node at (5.5,-10.5) {15};
\draw[fill=9x6x4color22] (6,-10) rectangle (7,-11);
\node at (6.5,-10.5) {22};
\draw[fill=9x6x4color21] (7,-10) rectangle (8,-11);
\node at (7.5,-10.5) {21};
\draw[fill=red] (8,-10) rectangle (9,-11);
\draw[fill=9x6x4color1] (0,-11) rectangle (1,-12);
\node at (0.5,-11.5) {1};
\draw[fill=red] (1,-11) rectangle (2,-12);
\draw[fill=9x6x4color3] (2,-11) rectangle (3,-12);
\node at (2.5,-11.5) {3};
\draw[fill=red] (3,-11) rectangle (4,-12);
\draw[fill=9x6x4color15] (4,-11) rectangle (5,-12);
\node at (4.5,-11.5) {15};
\draw[fill=9x6x4color18] (5,-11) rectangle (6,-12);
\node at (5.5,-11.5) {18};
\draw[fill=9x6x4color19] (6,-11) rectangle (7,-12);
\node at (6.5,-11.5) {19};
\draw[fill=9x6x4color20] (7,-11) rectangle (8,-12);
\node at (7.5,-11.5) {20};
\draw[fill=9x6x4color1] (8,-11) rectangle (9,-12);
\node at (8.5,-11.5) {1};
\draw[fill=red] (0,-12) rectangle (1,-13);
\draw[fill=9x6x4color1] (1,-12) rectangle (2,-13);
\node at (1.5,-12.5) {1};
\draw[fill=9x6x4color2] (2,-12) rectangle (3,-13);
\node at (2.5,-12.5) {2};
\draw[fill=9x6x4color1] (3,-12) rectangle (4,-13);
\node at (3.5,-12.5) {1};
\draw[fill=red] (4,-12) rectangle (5,-13);
\draw[fill=9x6x4color19] (5,-12) rectangle (6,-13);
\node at (5.5,-12.5) {19};
\draw[fill=red] (6,-12) rectangle (7,-13);
\draw[fill=9x6x4color1] (7,-12) rectangle (8,-13);
\node at (7.5,-12.5) {1};
\draw[fill=red] (8,-12) rectangle (9,-13);
\draw[fill=9x6x4color16] (0,-14) rectangle (1,-15);
\node at (0.5,-14.5) {16};
\draw[fill=9x6x4color15] (1,-14) rectangle (2,-15);
\node at (1.5,-14.5) {15};
\draw[fill=9x6x4color16] (2,-14) rectangle (3,-15);
\node at (2.5,-14.5) {16};
\draw[fill=9x6x4color17] (3,-14) rectangle (4,-15);
\node at (3.5,-14.5) {17};
\draw[fill=9x6x4color18] (4,-14) rectangle (5,-15);
\node at (4.5,-14.5) {18};
\draw[fill=9x6x4color19] (5,-14) rectangle (6,-15);
\node at (5.5,-14.5) {19};
\draw[fill=9x6x4color28] (6,-14) rectangle (7,-15);
\node at (6.5,-14.5) {28};
\draw[fill=9x6x4color29] (7,-14) rectangle (8,-15);
\node at (7.5,-14.5) {29};
\draw[fill=red] (8,-14) rectangle (9,-15);
\draw[fill=9x6x4color15] (0,-15) rectangle (1,-16);
\node at (0.5,-15.5) {15};
\draw[fill=9x6x4color14] (1,-15) rectangle (2,-16);
\node at (1.5,-15.5) {14};
\draw[fill=9x6x4color13] (2,-15) rectangle (3,-16);
\node at (2.5,-15.5) {13};
\draw[fill=red] (3,-15) rectangle (4,-16);
\draw[fill=9x6x4color13] (4,-15) rectangle (5,-16);
\node at (4.5,-15.5) {13};
\draw[fill=9x6x4color14] (5,-15) rectangle (6,-16);
\node at (5.5,-15.5) {14};
\draw[fill=9x6x4color27] (6,-15) rectangle (7,-16);
\node at (6.5,-15.5) {27};
\draw[fill=9x6x4color32] (7,-15) rectangle (8,-16);
\node at (7.5,-15.5) {32};
\draw[fill=9x6x4color33] (8,-15) rectangle (9,-16);
\node at (8.5,-15.5) {33};
\draw[fill=9x6x4color16] (0,-16) rectangle (1,-17);
\node at (0.5,-16.5) {16};
\draw[fill=9x6x4color9] (1,-16) rectangle (2,-17);
\node at (1.5,-16.5) {9};
\draw[fill=red] (2,-16) rectangle (3,-17);
\draw[fill=9x6x4color1] (3,-16) rectangle (4,-17);
\node at (3.5,-16.5) {1};
\draw[fill=9x6x4color12] (4,-16) rectangle (5,-17);
\node at (4.5,-16.5) {12};
\draw[fill=9x6x4color1] (5,-16) rectangle (6,-17);
\node at (5.5,-16.5) {1};
\draw[fill=9x6x4color24] (6,-16) rectangle (7,-17);
\node at (6.5,-16.5) {24};
\draw[fill=9x6x4color33] (7,-16) rectangle (8,-17);
\node at (7.5,-16.5) {33};
\draw[fill=9x6x4color34] (8,-16) rectangle (9,-17);
\node at (8.5,-16.5) {34};
\draw[fill=9x6x4color17] (0,-17) rectangle (1,-18);
\node at (0.5,-17.5) {17};
\draw[fill=9x6x4color8] (1,-17) rectangle (2,-18);
\node at (1.5,-17.5) {8};
\draw[fill=9x6x4color5] (2,-17) rectangle (3,-18);
\node at (2.5,-17.5) {5};
\draw[fill=9x6x4color8] (3,-17) rectangle (4,-18);
\node at (3.5,-17.5) {8};
\draw[fill=9x6x4color13] (4,-17) rectangle (5,-18);
\node at (4.5,-17.5) {13};
\draw[fill=red] (5,-17) rectangle (6,-18);
\draw[fill=9x6x4color23] (6,-17) rectangle (7,-18);
\node at (6.5,-17.5) {23};
\draw[fill=9x6x4color34] (7,-17) rectangle (8,-18);
\node at (7.5,-17.5) {34};
\draw[fill=9x6x4color35] (8,-17) rectangle (9,-18);
\node at (8.5,-17.5) {35};
\draw[fill=9x6x4color18] (0,-18) rectangle (1,-19);
\node at (0.5,-18.5) {18};
\draw[fill=9x6x4color7] (1,-18) rectangle (2,-19);
\node at (1.5,-18.5) {7};
\draw[fill=9x6x4color6] (2,-18) rectangle (3,-19);
\node at (2.5,-18.5) {6};
\draw[fill=9x6x4color7] (3,-18) rectangle (4,-19);
\node at (3.5,-18.5) {7};
\draw[fill=9x6x4color16] (4,-18) rectangle (5,-19);
\node at (4.5,-18.5) {16};
\draw[fill=9x6x4color17] (5,-18) rectangle (6,-19);
\node at (5.5,-18.5) {17};
\draw[fill=red] (6,-18) rectangle (7,-19);
\draw[fill=9x6x4color35] (7,-18) rectangle (8,-19);
\node at (7.5,-18.5) {35};
\draw[fill=9x6x4color36] (8,-18) rectangle (9,-19);
\node at (8.5,-18.5) {36};
\draw[fill=9x6x4color19] (0,-19) rectangle (1,-20);
\node at (0.5,-19.5) {19};
\draw[fill=red] (1,-19) rectangle (2,-20);
\draw[fill=9x6x4color1] (2,-19) rectangle (3,-20);
\node at (2.5,-19.5) {1};
\draw[fill=red] (3,-19) rectangle (4,-20);
\draw[fill=9x6x4color17] (4,-19) rectangle (5,-20);
\node at (4.5,-19.5) {17};
\draw[fill=9x6x4color20] (5,-19) rectangle (6,-20);
\node at (5.5,-19.5) {20};
\draw[fill=9x6x4color21] (6,-19) rectangle (7,-20);
\node at (6.5,-19.5) {21};
\draw[fill=9x6x4color36] (7,-19) rectangle (8,-20);
\node at (7.5,-19.5) {36};
\draw[fill=9x6x4color37] (8,-19) rectangle (9,-20);
\node at (8.5,-19.5) {37};
\draw[fill=9x6x4color17] (0,-21) rectangle (1,-22);
\node at (0.5,-21.5) {17};
\draw[fill=red] (1,-21) rectangle (2,-22);
\draw[fill=9x6x4color17] (2,-21) rectangle (3,-22);
\node at (2.5,-21.5) {17};
\draw[fill=9x6x4color18] (3,-21) rectangle (4,-22);
\node at (3.5,-21.5) {18};
\draw[fill=9x6x4color19] (4,-21) rectangle (5,-22);
\node at (4.5,-21.5) {19};
\draw[fill=9x6x4color20] (5,-21) rectangle (6,-22);
\node at (5.5,-21.5) {20};
\draw[fill=9x6x4color29] (6,-21) rectangle (7,-22);
\node at (6.5,-21.5) {29};
\draw[fill=red] (7,-21) rectangle (8,-22);
\draw[fill=9x6x4color1] (8,-21) rectangle (9,-22);
\node at (8.5,-21.5) {1};
\draw[fill=red] (0,-22) rectangle (1,-23);
\draw[fill=9x6x4color15] (1,-22) rectangle (2,-23);
\node at (1.5,-22.5) {15};
\draw[fill=9x6x4color16] (2,-22) rectangle (3,-23);
\node at (2.5,-22.5) {16};
\draw[fill=9x6x4color17] (3,-22) rectangle (4,-23);
\node at (3.5,-22.5) {17};
\draw[fill=9x6x4color18] (4,-22) rectangle (5,-23);
\node at (4.5,-22.5) {18};
\draw[fill=9x6x4color19] (5,-22) rectangle (6,-23);
\node at (5.5,-22.5) {19};
\draw[fill=9x6x4color30] (6,-22) rectangle (7,-23);
\node at (6.5,-22.5) {30};
\draw[fill=9x6x4color31] (7,-22) rectangle (8,-23);
\node at (7.5,-22.5) {31};
\draw[fill=red] (8,-22) rectangle (9,-23);
\draw[fill=9x6x4color17] (0,-23) rectangle (1,-24);
\node at (0.5,-23.5) {17};
\draw[fill=9x6x4color16] (1,-23) rectangle (2,-24);
\node at (1.5,-23.5) {16};
\draw[fill=9x6x4color1] (2,-23) rectangle (3,-24);
\node at (2.5,-23.5) {1};
\draw[fill=red] (3,-23) rectangle (4,-24);
\draw[fill=9x6x4color13] (4,-23) rectangle (5,-24);
\node at (4.5,-23.5) {13};
\draw[fill=red] (5,-23) rectangle (6,-24);
\draw[fill=9x6x4color31] (6,-23) rectangle (7,-24);
\node at (6.5,-23.5) {31};
\draw[fill=9x6x4color34] (7,-23) rectangle (8,-24);
\node at (7.5,-23.5) {34};
\draw[fill=9x6x4color35] (8,-23) rectangle (9,-24);
\node at (8.5,-23.5) {35};
\draw[fill=9x6x4color18] (0,-24) rectangle (1,-25);
\node at (0.5,-24.5) {18};
\draw[fill=9x6x4color17] (1,-24) rectangle (2,-25);
\node at (1.5,-24.5) {17};
\draw[fill=red] (2,-24) rectangle (3,-25);
\draw[fill=9x6x4color9] (3,-24) rectangle (4,-25);
\node at (3.5,-24.5) {9};
\draw[fill=9x6x4color14] (4,-24) rectangle (5,-25);
\node at (4.5,-24.5) {14};
\draw[fill=9x6x4color15] (5,-24) rectangle (6,-25);
\node at (5.5,-24.5) {15};
\draw[fill=9x6x4color32] (6,-24) rectangle (7,-25);
\node at (6.5,-24.5) {32};
\draw[fill=9x6x4color35] (7,-24) rectangle (8,-25);
\node at (7.5,-24.5) {35};
\draw[fill=9x6x4color36] (8,-24) rectangle (9,-25);
\node at (8.5,-24.5) {36};
\draw[fill=9x6x4color19] (0,-25) rectangle (1,-26);
\node at (0.5,-25.5) {19};
\draw[fill=9x6x4color18] (1,-25) rectangle (2,-26);
\node at (1.5,-25.5) {18};
\draw[fill=9x6x4color7] (2,-25) rectangle (3,-26);
\node at (2.5,-25.5) {7};
\draw[fill=9x6x4color10] (3,-25) rectangle (4,-26);
\node at (3.5,-25.5) {10};
\draw[fill=9x6x4color17] (4,-25) rectangle (5,-26);
\node at (4.5,-25.5) {17};
\draw[fill=9x6x4color18] (5,-25) rectangle (6,-26);
\node at (5.5,-25.5) {18};
\draw[fill=9x6x4color33] (6,-25) rectangle (7,-26);
\node at (6.5,-25.5) {33};
\draw[fill=9x6x4color36] (7,-25) rectangle (8,-26);
\node at (7.5,-25.5) {36};
\draw[fill=9x6x4color37] (8,-25) rectangle (9,-26);
\node at (8.5,-25.5) {37};
\draw[fill=9x6x4color20] (0,-26) rectangle (1,-27);
\node at (0.5,-26.5) {20};
\draw[fill=9x6x4color19] (1,-26) rectangle (2,-27);
\node at (1.5,-26.5) {19};
\draw[fill=red] (2,-26) rectangle (3,-27);
\draw[fill=9x6x4color11] (3,-26) rectangle (4,-27);
\node at (3.5,-26.5) {11};
\draw[fill=9x6x4color18] (4,-26) rectangle (5,-27);
\node at (4.5,-26.5) {18};
\draw[fill=9x6x4color21] (5,-26) rectangle (6,-27);
\node at (5.5,-26.5) {21};
\draw[fill=9x6x4color34] (6,-26) rectangle (7,-27);
\node at (6.5,-26.5) {34};
\draw[fill=9x6x4color37] (7,-26) rectangle (8,-27);
\node at (7.5,-26.5) {37};
\draw[fill=9x6x4color38] (8,-26) rectangle (9,-27);
\node at (8.5,-26.5) {38};

\end{scope}
\end{tikzpicture}
\end{center}

%% file: 4-9-9-tikz.txt
\definecolor{9x9x4color1}{RGB}{255, 128, 128}
\definecolor{9x9x4color2}{RGB}{255, 132, 132}
\definecolor{9x9x4color3}{RGB}{255, 137, 137}
\definecolor{9x9x4color4}{RGB}{255, 142, 142}
\definecolor{9x9x4color5}{RGB}{255, 146, 146}
\definecolor{9x9x4color6}{RGB}{255, 151, 151}
\definecolor{9x9x4color7}{RGB}{255, 156, 156}
\definecolor{9x9x4color8}{RGB}{255, 160, 160}
\definecolor{9x9x4color9}{RGB}{255, 165, 165}
\definecolor{9x9x4color10}{RGB}{255, 170, 170}
\definecolor{9x9x4color11}{RGB}{255, 175, 175}
\definecolor{9x9x4color12}{RGB}{255, 179, 179}
\definecolor{9x9x4color13}{RGB}{255, 184, 184}
\definecolor{9x9x4color14}{RGB}{255, 189, 189}
\definecolor{9x9x4color15}{RGB}{255, 193, 193}
\definecolor{9x9x4color16}{RGB}{255, 198, 198}
\definecolor{9x9x4color17}{RGB}{255, 203, 203}
\definecolor{9x9x4color18}{RGB}{255, 207, 207}
\definecolor{9x9x4color19}{RGB}{255, 212, 212}
\definecolor{9x9x4color20}{RGB}{255, 217, 217}
\definecolor{9x9x4color21}{RGB}{255, 222, 222}
\definecolor{9x9x4color22}{RGB}{255, 226, 226}
\definecolor{9x9x4color23}{RGB}{255, 231, 231}
\definecolor{9x9x4color24}{RGB}{255, 236, 236}
\definecolor{9x9x4color25}{RGB}{255, 240, 240}
\definecolor{9x9x4color26}{RGB}{255, 245, 245}
\definecolor{9x9x4color27}{RGB}{255, 255, 255}

\begin{center}
\begin{tikzpicture}[scale=0.5]
\begin{scope}
\draw[very thick] (0,-0) rectangle (9,-9);
\foreach \b in {0,...,8}
 \foreach \c in {0,...,8}
  \draw (\c,\b*-1-0) rectangle (\c+1,(\b*-1-1);
\draw[very thick] (0,-10) rectangle (9,-19);
\foreach \b in {0,...,8}
 \foreach \c in {0,...,8}
  \draw (\c,\b*-1-10) rectangle (\c+1,(\b*-1-11);
\draw[very thick] (0,-20) rectangle (9,-29);
\foreach \b in {0,...,8}
 \foreach \c in {0,...,8}
  \draw (\c,\b*-1-20) rectangle (\c+1,(\b*-1-21);
\draw[very thick] (0,-30) rectangle (9,-39);
\foreach \b in {0,...,8}
 \foreach \c in {0,...,8}
  \draw (\c,\b*-1-30) rectangle (\c+1,(\b*-1-31);

\draw[fill=red] (0,-0) rectangle (1,-1);
\draw[fill=9x9x4color19] (1,-0) rectangle (2,-1);
\node at (1.5,-0.5) {19};
\draw[fill=9x9x4color20] (2,-0) rectangle (3,-1);
\node at (2.5,-0.5) {20};
\draw[fill=9x9x4color21] (3,-0) rectangle (4,-1);
\node at (3.5,-0.5) {21};
\draw[fill=9x9x4color22] (4,-0) rectangle (5,-1);
\node at (4.5,-0.5) {22};
\draw[fill=9x9x4color23] (5,-0) rectangle (6,-1);
\node at (5.5,-0.5) {23};
\draw[fill=9x9x4color24] (6,-0) rectangle (7,-1);
\node at (6.5,-0.5) {24};
\draw[fill=9x9x4color25] (7,-0) rectangle (8,-1);
\node at (7.5,-0.5) {25};
\draw[fill=9x9x4color26] (8,-0) rectangle (9,-1);
\node at (8.5,-0.5) {26};
\draw[fill=9x9x4color19] (0,-1) rectangle (1,-2);
\node at (0.5,-1.5) {19};
\draw[fill=9x9x4color18] (1,-1) rectangle (2,-2);
\node at (1.5,-1.5) {18};
\draw[fill=9x9x4color17] (2,-1) rectangle (3,-2);
\node at (2.5,-1.5) {17};
\draw[fill=9x9x4color16] (3,-1) rectangle (4,-2);
\node at (3.5,-1.5) {16};
\draw[fill=9x9x4color15] (4,-1) rectangle (5,-2);
\node at (4.5,-1.5) {15};
\draw[fill=red] (5,-1) rectangle (6,-2);
\draw[fill=9x9x4color15] (6,-1) rectangle (7,-2);
\node at (6.5,-1.5) {15};
\draw[fill=9x9x4color16] (7,-1) rectangle (8,-2);
\node at (7.5,-1.5) {16};
\draw[fill=9x9x4color17] (8,-1) rectangle (9,-2);
\node at (8.5,-1.5) {17};
\draw[fill=red] (0,-2) rectangle (1,-3);
\draw[fill=9x9x4color1] (1,-2) rectangle (2,-3);
\node at (1.5,-2.5) {1};
\draw[fill=red] (2,-2) rectangle (3,-3);
\draw[fill=9x9x4color7] (3,-2) rectangle (4,-3);
\node at (3.5,-2.5) {7};
\draw[fill=red] (4,-2) rectangle (5,-3);
\draw[fill=9x9x4color1] (5,-2) rectangle (6,-3);
\node at (5.5,-2.5) {1};
\draw[fill=9x9x4color14] (6,-2) rectangle (7,-3);
\node at (6.5,-2.5) {14};
\draw[fill=9x9x4color15] (7,-2) rectangle (8,-3);
\node at (7.5,-2.5) {15};
\draw[fill=9x9x4color16] (8,-2) rectangle (9,-3);
\node at (8.5,-2.5) {16};
\draw[fill=9x9x4color1] (0,-3) rectangle (1,-4);
\node at (0.5,-3.5) {1};
\draw[fill=red] (1,-3) rectangle (2,-4);
\draw[fill=9x9x4color5] (2,-3) rectangle (3,-4);
\node at (2.5,-3.5) {5};
\draw[fill=9x9x4color6] (3,-3) rectangle (4,-4);
\node at (3.5,-3.5) {6};
\draw[fill=9x9x4color1] (4,-3) rectangle (5,-4);
\node at (4.5,-3.5) {1};
\draw[fill=red] (5,-3) rectangle (6,-4);
\draw[fill=9x9x4color13] (6,-3) rectangle (7,-4);
\node at (6.5,-3.5) {13};
\draw[fill=9x9x4color14] (7,-3) rectangle (8,-4);
\node at (7.5,-3.5) {14};
\draw[fill=9x9x4color15] (8,-3) rectangle (9,-4);
\node at (8.5,-3.5) {15};
\draw[fill=red] (0,-4) rectangle (1,-5);
\draw[fill=9x9x4color1] (1,-4) rectangle (2,-5);
\node at (1.5,-4.5) {1};
\draw[fill=9x9x4color4] (2,-4) rectangle (3,-5);
\node at (2.5,-4.5) {4};
\draw[fill=9x9x4color3] (3,-4) rectangle (4,-5);
\node at (3.5,-4.5) {3};
\draw[fill=9x9x4color2] (4,-4) rectangle (5,-5);
\node at (4.5,-4.5) {2};
\draw[fill=9x9x4color1] (5,-4) rectangle (6,-5);
\node at (5.5,-4.5) {1};
\draw[fill=9x9x4color12] (6,-4) rectangle (7,-5);
\node at (6.5,-4.5) {12};
\draw[fill=9x9x4color13] (7,-4) rectangle (8,-5);
\node at (7.5,-4.5) {13};
\draw[fill=9x9x4color14] (8,-4) rectangle (9,-5);
\node at (8.5,-4.5) {14};
\draw[fill=9x9x4color1] (0,-5) rectangle (1,-6);
\node at (0.5,-5.5) {1};
\draw[fill=red] (1,-5) rectangle (2,-6);
\draw[fill=9x9x4color1] (2,-5) rectangle (3,-6);
\node at (2.5,-5.5) {1};
\draw[fill=red] (3,-5) rectangle (4,-6);
\draw[fill=9x9x4color3] (4,-5) rectangle (5,-6);
\node at (4.5,-5.5) {3};
\draw[fill=red] (5,-5) rectangle (6,-6);
\draw[fill=9x9x4color11] (6,-5) rectangle (7,-6);
\node at (6.5,-5.5) {11};
\draw[fill=red] (7,-5) rectangle (8,-6);
\draw[fill=9x9x4color3] (8,-5) rectangle (9,-6);
\node at (8.5,-5.5) {3};
\draw[fill=9x9x4color18] (0,-6) rectangle (1,-7);
\node at (0.5,-6.5) {18};
\draw[fill=9x9x4color17] (1,-6) rectangle (2,-7);
\node at (1.5,-6.5) {17};
\draw[fill=9x9x4color16] (2,-6) rectangle (3,-7);
\node at (2.5,-6.5) {16};
\draw[fill=9x9x4color15] (3,-6) rectangle (4,-7);
\node at (3.5,-6.5) {15};
\draw[fill=9x9x4color14] (4,-6) rectangle (5,-7);
\node at (4.5,-6.5) {14};
\draw[fill=9x9x4color13] (5,-6) rectangle (6,-7);
\node at (5.5,-6.5) {13};
\draw[fill=9x9x4color12] (6,-6) rectangle (7,-7);
\node at (6.5,-6.5) {12};
\draw[fill=9x9x4color1] (7,-6) rectangle (8,-7);
\node at (7.5,-6.5) {1};
\draw[fill=red] (8,-6) rectangle (9,-7);
\draw[fill=9x9x4color19] (0,-7) rectangle (1,-8);
\node at (0.5,-7.5) {19};
\draw[fill=9x9x4color18] (1,-7) rectangle (2,-8);
\node at (1.5,-7.5) {18};
\draw[fill=9x9x4color17] (2,-7) rectangle (3,-8);
\node at (2.5,-7.5) {17};
\draw[fill=9x9x4color16] (3,-7) rectangle (4,-8);
\node at (3.5,-7.5) {16};
\draw[fill=9x9x4color15] (4,-7) rectangle (5,-8);
\node at (4.5,-7.5) {15};
\draw[fill=9x9x4color14] (5,-7) rectangle (6,-8);
\node at (5.5,-7.5) {14};
\draw[fill=9x9x4color13] (6,-7) rectangle (7,-8);
\node at (6.5,-7.5) {13};
\draw[fill=9x9x4color2] (7,-7) rectangle (8,-8);
\node at (7.5,-7.5) {2};
\draw[fill=9x9x4color1] (8,-7) rectangle (9,-8);
\node at (8.5,-7.5) {1};
\draw[fill=9x9x4color20] (0,-8) rectangle (1,-9);
\node at (0.5,-8.5) {20};
\draw[fill=9x9x4color19] (1,-8) rectangle (2,-9);
\node at (1.5,-8.5) {19};
\draw[fill=9x9x4color18] (2,-8) rectangle (3,-9);
\node at (2.5,-8.5) {18};
\draw[fill=9x9x4color17] (3,-8) rectangle (4,-9);
\node at (3.5,-8.5) {17};
\draw[fill=9x9x4color16] (4,-8) rectangle (5,-9);
\node at (4.5,-8.5) {16};
\draw[fill=9x9x4color15] (5,-8) rectangle (6,-9);
\node at (5.5,-8.5) {15};
\draw[fill=9x9x4color14] (6,-8) rectangle (7,-9);
\node at (6.5,-8.5) {14};
\draw[fill=9x9x4color3] (7,-8) rectangle (8,-9);
\node at (7.5,-8.5) {3};
\draw[fill=red] (8,-8) rectangle (9,-9);
\draw[fill=9x9x4color21] (0,-10) rectangle (1,-11);
\node at (0.5,-10.5) {21};
\draw[fill=red] (1,-10) rectangle (2,-11);
\draw[fill=9x9x4color1] (2,-10) rectangle (3,-11);
\node at (2.5,-10.5) {1};
\draw[fill=red] (3,-10) rectangle (4,-11);
\draw[fill=9x9x4color15] (4,-10) rectangle (5,-11);
\node at (4.5,-10.5) {15};
\draw[fill=red] (5,-10) rectangle (6,-11);
\draw[fill=9x9x4color15] (6,-10) rectangle (7,-11);
\node at (6.5,-10.5) {15};
\draw[fill=9x9x4color16] (7,-10) rectangle (8,-11);
\node at (7.5,-10.5) {16};
\draw[fill=9x9x4color17] (8,-10) rectangle (9,-11);
\node at (8.5,-10.5) {17};
\draw[fill=9x9x4color20] (0,-11) rectangle (1,-12);
\node at (0.5,-11.5) {20};
\draw[fill=9x9x4color11] (1,-11) rectangle (2,-12);
\node at (1.5,-11.5) {11};
\draw[fill=red] (2,-11) rectangle (3,-12);
\draw[fill=9x9x4color11] (3,-11) rectangle (4,-12);
\node at (3.5,-11.5) {11};
\draw[fill=9x9x4color14] (4,-11) rectangle (5,-12);
\node at (4.5,-11.5) {14};
\draw[fill=9x9x4color13] (5,-11) rectangle (6,-12);
\node at (5.5,-11.5) {13};
\draw[fill=9x9x4color14] (6,-11) rectangle (7,-12);
\node at (6.5,-11.5) {14};
\draw[fill=9x9x4color15] (7,-11) rectangle (8,-12);
\node at (7.5,-11.5) {15};
\draw[fill=9x9x4color16] (8,-11) rectangle (9,-12);
\node at (8.5,-11.5) {16};
\draw[fill=9x9x4color11] (0,-12) rectangle (1,-13);
\node at (0.5,-12.5) {11};
\draw[fill=9x9x4color10] (1,-12) rectangle (2,-13);
\node at (1.5,-12.5) {10};
\draw[fill=9x9x4color9] (2,-12) rectangle (3,-13);
\node at (2.5,-12.5) {9};
\draw[fill=9x9x4color10] (3,-12) rectangle (4,-13);
\node at (3.5,-12.5) {10};
\draw[fill=9x9x4color11] (4,-12) rectangle (5,-13);
\node at (4.5,-12.5) {11};
\draw[fill=9x9x4color12] (5,-12) rectangle (6,-13);
\node at (5.5,-12.5) {12};
\draw[fill=9x9x4color13] (6,-12) rectangle (7,-13);
\node at (6.5,-12.5) {13};
\draw[fill=9x9x4color14] (7,-12) rectangle (8,-13);
\node at (7.5,-12.5) {14};
\draw[fill=9x9x4color15] (8,-12) rectangle (9,-13);
\node at (8.5,-12.5) {15};
\draw[fill=9x9x4color10] (0,-13) rectangle (1,-14);
\node at (0.5,-13.5) {10};
\draw[fill=9x9x4color9] (1,-13) rectangle (2,-14);
\node at (1.5,-13.5) {9};
\draw[fill=9x9x4color8] (2,-13) rectangle (3,-14);
\node at (2.5,-13.5) {8};
\draw[fill=9x9x4color7] (3,-13) rectangle (4,-14);
\node at (3.5,-13.5) {7};
\draw[fill=red] (4,-13) rectangle (5,-14);
\draw[fill=9x9x4color1] (5,-13) rectangle (6,-14);
\node at (5.5,-13.5) {1};
\draw[fill=9x9x4color12] (6,-13) rectangle (7,-14);
\node at (6.5,-13.5) {12};
\draw[fill=9x9x4color13] (7,-13) rectangle (8,-14);
\node at (7.5,-13.5) {13};
\draw[fill=9x9x4color14] (8,-13) rectangle (9,-14);
\node at (8.5,-13.5) {14};
\draw[fill=9x9x4color7] (0,-14) rectangle (1,-15);
\node at (0.5,-14.5) {7};
\draw[fill=9x9x4color6] (1,-14) rectangle (2,-15);
\node at (1.5,-14.5) {6};
\draw[fill=9x9x4color5] (2,-14) rectangle (3,-15);
\node at (2.5,-14.5) {5};
\draw[fill=red] (3,-14) rectangle (4,-15);
\draw[fill=9x9x4color1] (4,-14) rectangle (5,-15);
\node at (4.5,-14.5) {1};
\draw[fill=red] (5,-14) rectangle (6,-15);
\draw[fill=9x9x4color11] (6,-14) rectangle (7,-15);
\node at (6.5,-14.5) {11};
\draw[fill=9x9x4color12] (7,-14) rectangle (8,-15);
\node at (7.5,-14.5) {12};
\draw[fill=9x9x4color13] (8,-14) rectangle (9,-15);
\node at (8.5,-14.5) {13};
\draw[fill=red] (0,-15) rectangle (1,-16);
\draw[fill=9x9x4color1] (1,-15) rectangle (2,-16);
\node at (1.5,-15.5) {1};
\draw[fill=red] (2,-15) rectangle (3,-16);
\draw[fill=9x9x4color1] (3,-15) rectangle (4,-16);
\node at (3.5,-15.5) {1};
\draw[fill=9x9x4color4] (4,-15) rectangle (5,-16);
\node at (4.5,-15.5) {4};
\draw[fill=9x9x4color5] (5,-15) rectangle (6,-16);
\node at (5.5,-15.5) {5};
\draw[fill=9x9x4color10] (6,-15) rectangle (7,-16);
\node at (6.5,-15.5) {10};
\draw[fill=9x9x4color1] (7,-15) rectangle (8,-16);
\node at (7.5,-15.5) {1};
\draw[fill=9x9x4color2] (8,-15) rectangle (9,-16);
\node at (8.5,-15.5) {2};
\draw[fill=9x9x4color5] (0,-16) rectangle (1,-17);
\node at (0.5,-16.5) {5};
\draw[fill=9x9x4color4] (1,-16) rectangle (2,-17);
\node at (1.5,-16.5) {4};
\draw[fill=9x9x4color5] (2,-16) rectangle (3,-17);
\node at (2.5,-16.5) {5};
\draw[fill=9x9x4color6] (3,-16) rectangle (4,-17);
\node at (3.5,-16.5) {6};
\draw[fill=9x9x4color9] (4,-16) rectangle (5,-17);
\node at (4.5,-16.5) {9};
\draw[fill=9x9x4color10] (5,-16) rectangle (6,-17);
\node at (5.5,-16.5) {10};
\draw[fill=9x9x4color11] (6,-16) rectangle (7,-17);
\node at (6.5,-16.5) {11};
\draw[fill=red] (7,-16) rectangle (8,-17);
\draw[fill=9x9x4color1] (8,-16) rectangle (9,-17);
\node at (8.5,-16.5) {1};
\draw[fill=red] (0,-17) rectangle (1,-18);
\draw[fill=9x9x4color3] (1,-17) rectangle (2,-18);
\node at (1.5,-17.5) {3};
\draw[fill=9x9x4color6] (2,-17) rectangle (3,-18);
\node at (2.5,-17.5) {6};
\draw[fill=9x9x4color7] (3,-17) rectangle (4,-18);
\node at (3.5,-17.5) {7};
\draw[fill=9x9x4color10] (4,-17) rectangle (5,-18);
\node at (4.5,-17.5) {10};
\draw[fill=9x9x4color11] (5,-17) rectangle (6,-18);
\node at (5.5,-17.5) {11};
\draw[fill=9x9x4color12] (6,-17) rectangle (7,-18);
\node at (6.5,-17.5) {12};
\draw[fill=9x9x4color1] (7,-17) rectangle (8,-18);
\node at (7.5,-17.5) {1};
\draw[fill=red] (8,-17) rectangle (9,-18);
\draw[fill=9x9x4color1] (0,-18) rectangle (1,-19);
\node at (0.5,-18.5) {1};
\draw[fill=red] (1,-18) rectangle (2,-19);
\draw[fill=9x9x4color7] (2,-18) rectangle (3,-19);
\node at (2.5,-18.5) {7};
\draw[fill=9x9x4color8] (3,-18) rectangle (4,-19);
\node at (3.5,-18.5) {8};
\draw[fill=9x9x4color11] (4,-18) rectangle (5,-19);
\node at (4.5,-18.5) {11};
\draw[fill=9x9x4color12] (5,-18) rectangle (6,-19);
\node at (5.5,-18.5) {12};
\draw[fill=9x9x4color13] (6,-18) rectangle (7,-19);
\node at (6.5,-18.5) {13};
\draw[fill=red] (7,-18) rectangle (8,-19);
\draw[fill=9x9x4color1] (8,-18) rectangle (9,-19);
\node at (8.5,-18.5) {1};
\draw[fill=9x9x4color22] (0,-20) rectangle (1,-21);
\node at (0.5,-20.5) {22};
\draw[fill=9x9x4color19] (1,-20) rectangle (2,-21);
\node at (1.5,-20.5) {19};
\draw[fill=9x9x4color18] (2,-20) rectangle (3,-21);
\node at (2.5,-20.5) {18};
\draw[fill=9x9x4color17] (3,-20) rectangle (4,-21);
\node at (3.5,-20.5) {17};
\draw[fill=9x9x4color16] (4,-20) rectangle (5,-21);
\node at (4.5,-20.5) {16};
\draw[fill=9x9x4color15] (5,-20) rectangle (6,-21);
\node at (5.5,-20.5) {15};
\draw[fill=red] (6,-20) rectangle (7,-21);
\draw[fill=9x9x4color1] (7,-20) rectangle (8,-21);
\node at (7.5,-20.5) {1};
\draw[fill=red] (8,-20) rectangle (9,-21);
\draw[fill=9x9x4color21] (0,-21) rectangle (1,-22);
\node at (0.5,-21.5) {21};
\draw[fill=9x9x4color18] (1,-21) rectangle (2,-22);
\node at (1.5,-21.5) {18};
\draw[fill=9x9x4color17] (2,-21) rectangle (3,-22);
\node at (2.5,-21.5) {17};
\draw[fill=9x9x4color16] (3,-21) rectangle (4,-22);
\node at (3.5,-21.5) {16};
\draw[fill=9x9x4color15] (4,-21) rectangle (5,-22);
\node at (4.5,-21.5) {15};
\draw[fill=9x9x4color14] (5,-21) rectangle (6,-22);
\node at (5.5,-21.5) {14};
\draw[fill=9x9x4color1] (6,-21) rectangle (7,-22);
\node at (6.5,-21.5) {1};
\draw[fill=9x9x4color2] (7,-21) rectangle (8,-22);
\node at (7.5,-21.5) {2};
\draw[fill=9x9x4color3] (8,-21) rectangle (9,-22);
\node at (8.5,-21.5) {3};
\draw[fill=9x9x4color18] (0,-22) rectangle (1,-23);
\node at (0.5,-22.5) {18};
\draw[fill=9x9x4color17] (1,-22) rectangle (2,-23);
\node at (1.5,-22.5) {17};
\draw[fill=9x9x4color16] (2,-22) rectangle (3,-23);
\node at (2.5,-22.5) {16};
\draw[fill=9x9x4color15] (3,-22) rectangle (4,-23);
\node at (3.5,-22.5) {15};
\draw[fill=9x9x4color14] (4,-22) rectangle (5,-23);
\node at (4.5,-22.5) {14};
\draw[fill=9x9x4color13] (5,-22) rectangle (6,-23);
\node at (5.5,-22.5) {13};
\draw[fill=red] (6,-22) rectangle (7,-23);
\draw[fill=9x9x4color1] (7,-22) rectangle (8,-23);
\node at (7.5,-22.5) {1};
\draw[fill=red] (8,-22) rectangle (9,-23);
\draw[fill=9x9x4color17] (0,-23) rectangle (1,-24);
\node at (0.5,-23.5) {17};
\draw[fill=9x9x4color16] (1,-23) rectangle (2,-24);
\node at (1.5,-23.5) {16};
\draw[fill=9x9x4color15] (2,-23) rectangle (3,-24);
\node at (2.5,-23.5) {15};
\draw[fill=9x9x4color14] (3,-23) rectangle (4,-24);
\node at (3.5,-23.5) {14};
\draw[fill=9x9x4color13] (4,-23) rectangle (5,-24);
\node at (4.5,-23.5) {13};
\draw[fill=9x9x4color12] (5,-23) rectangle (6,-24);
\node at (5.5,-23.5) {12};
\draw[fill=9x9x4color1] (6,-23) rectangle (7,-24);
\node at (6.5,-23.5) {1};
\draw[fill=red] (7,-23) rectangle (8,-24);
\draw[fill=9x9x4color1] (8,-23) rectangle (9,-24);
\node at (8.5,-23.5) {1};
\draw[fill=9x9x4color16] (0,-24) rectangle (1,-25);
\node at (0.5,-24.5) {16};
\draw[fill=9x9x4color15] (1,-24) rectangle (2,-25);
\node at (1.5,-24.5) {15};
\draw[fill=9x9x4color14] (2,-24) rectangle (3,-25);
\node at (2.5,-24.5) {14};
\draw[fill=9x9x4color13] (3,-24) rectangle (4,-25);
\node at (3.5,-24.5) {13};
\draw[fill=9x9x4color12] (4,-24) rectangle (5,-25);
\node at (4.5,-24.5) {12};
\draw[fill=9x9x4color11] (5,-24) rectangle (6,-25);
\node at (5.5,-24.5) {11};
\draw[fill=9x9x4color8] (6,-24) rectangle (7,-25);
\node at (6.5,-24.5) {8};
\draw[fill=9x9x4color1] (7,-24) rectangle (8,-25);
\node at (7.5,-24.5) {1};
\draw[fill=red] (8,-24) rectangle (9,-25);
\draw[fill=9x9x4color15] (0,-25) rectangle (1,-26);
\node at (0.5,-25.5) {15};
\draw[fill=9x9x4color14] (1,-25) rectangle (2,-26);
\node at (1.5,-25.5) {14};
\draw[fill=9x9x4color13] (2,-25) rectangle (3,-26);
\node at (2.5,-25.5) {13};
\draw[fill=9x9x4color12] (3,-25) rectangle (4,-26);
\node at (3.5,-25.5) {12};
\draw[fill=9x9x4color11] (4,-25) rectangle (5,-26);
\node at (4.5,-25.5) {11};
\draw[fill=9x9x4color10] (5,-25) rectangle (6,-26);
\node at (5.5,-25.5) {10};
\draw[fill=9x9x4color9] (6,-25) rectangle (7,-26);
\node at (6.5,-25.5) {9};
\draw[fill=red] (7,-25) rectangle (8,-26);
\draw[fill=9x9x4color1] (8,-25) rectangle (9,-26);
\node at (8.5,-25.5) {1};
\draw[fill=9x9x4color6] (0,-26) rectangle (1,-27);
\node at (0.5,-26.5) {6};
\draw[fill=9x9x4color3] (1,-26) rectangle (2,-27);
\node at (1.5,-26.5) {3};
\draw[fill=red] (2,-26) rectangle (3,-27);
\draw[fill=9x9x4color1] (3,-26) rectangle (4,-27);
\node at (3.5,-26.5) {1};
\draw[fill=9x9x4color8] (4,-26) rectangle (5,-27);
\node at (4.5,-26.5) {8};
\draw[fill=9x9x4color9] (5,-26) rectangle (6,-27);
\node at (5.5,-26.5) {9};
\draw[fill=9x9x4color12] (6,-26) rectangle (7,-27);
\node at (6.5,-26.5) {12};
\draw[fill=9x9x4color13] (7,-26) rectangle (8,-27);
\node at (7.5,-26.5) {13};
\draw[fill=9x9x4color14] (8,-26) rectangle (9,-27);
\node at (8.5,-26.5) {14};
\draw[fill=9x9x4color1] (0,-27) rectangle (1,-28);
\node at (0.5,-27.5) {1};
\draw[fill=9x9x4color2] (1,-27) rectangle (2,-28);
\node at (1.5,-27.5) {2};
\draw[fill=9x9x4color1] (2,-27) rectangle (3,-28);
\node at (2.5,-27.5) {1};
\draw[fill=red] (3,-27) rectangle (4,-28);
\draw[fill=9x9x4color1] (4,-27) rectangle (5,-28);
\node at (4.5,-27.5) {1};
\draw[fill=red] (5,-27) rectangle (6,-28);
\draw[fill=9x9x4color13] (6,-27) rectangle (7,-28);
\node at (6.5,-27.5) {13};
\draw[fill=9x9x4color14] (7,-27) rectangle (8,-28);
\node at (7.5,-27.5) {14};
\draw[fill=9x9x4color15] (8,-27) rectangle (9,-28);
\node at (8.5,-27.5) {15};
\draw[fill=red] (0,-28) rectangle (1,-29);
\draw[fill=9x9x4color1] (1,-28) rectangle (2,-29);
\node at (1.5,-28.5) {1};
\draw[fill=red] (2,-28) rectangle (3,-29);
\draw[fill=9x9x4color1] (3,-28) rectangle (4,-29);
\node at (3.5,-28.5) {1};
\draw[fill=red] (4,-28) rectangle (5,-29);
\draw[fill=9x9x4color1] (5,-28) rectangle (6,-29);
\node at (5.5,-28.5) {1};
\draw[fill=9x9x4color14] (6,-28) rectangle (7,-29);
\node at (6.5,-28.5) {14};
\draw[fill=9x9x4color15] (7,-28) rectangle (8,-29);
\node at (7.5,-28.5) {15};
\draw[fill=9x9x4color16] (8,-28) rectangle (9,-29);
\node at (8.5,-28.5) {16};
\draw[fill=9x9x4color23] (0,-30) rectangle (1,-31);
\node at (0.5,-30.5) {23};
\draw[fill=9x9x4color20] (1,-30) rectangle (2,-31);
\node at (1.5,-30.5) {20};
\draw[fill=9x9x4color19] (2,-30) rectangle (3,-31);
\node at (2.5,-30.5) {19};
\draw[fill=9x9x4color18] (3,-30) rectangle (4,-31);
\node at (3.5,-30.5) {18};
\draw[fill=9x9x4color17] (4,-30) rectangle (5,-31);
\node at (4.5,-30.5) {17};
\draw[fill=9x9x4color16] (5,-30) rectangle (6,-31);
\node at (5.5,-30.5) {16};
\draw[fill=9x9x4color1] (6,-30) rectangle (7,-31);
\node at (6.5,-30.5) {1};
\draw[fill=red] (7,-30) rectangle (8,-31);
\draw[fill=9x9x4color11] (8,-30) rectangle (9,-31);
\node at (8.5,-30.5) {11};
\draw[fill=9x9x4color22] (0,-31) rectangle (1,-32);
\node at (0.5,-31.5) {22};
\draw[fill=9x9x4color19] (1,-31) rectangle (2,-32);
\node at (1.5,-31.5) {19};
\draw[fill=9x9x4color18] (2,-31) rectangle (3,-32);
\node at (2.5,-31.5) {18};
\draw[fill=9x9x4color17] (3,-31) rectangle (4,-32);
\node at (3.5,-31.5) {17};
\draw[fill=9x9x4color16] (4,-31) rectangle (5,-32);
\node at (4.5,-31.5) {16};
\draw[fill=9x9x4color15] (5,-31) rectangle (6,-32);
\node at (5.5,-31.5) {15};
\draw[fill=red] (6,-31) rectangle (7,-32);
\draw[fill=9x9x4color3] (7,-31) rectangle (8,-32);
\node at (7.5,-31.5) {3};
\draw[fill=9x9x4color10] (8,-31) rectangle (9,-32);
\node at (8.5,-31.5) {10};
\draw[fill=9x9x4color19] (0,-32) rectangle (1,-33);
\node at (0.5,-32.5) {19};
\draw[fill=9x9x4color18] (1,-32) rectangle (2,-33);
\node at (1.5,-32.5) {18};
\draw[fill=9x9x4color17] (2,-32) rectangle (3,-33);
\node at (2.5,-32.5) {17};
\draw[fill=9x9x4color16] (3,-32) rectangle (4,-33);
\node at (3.5,-32.5) {16};
\draw[fill=9x9x4color15] (4,-32) rectangle (5,-33);
\node at (4.5,-32.5) {15};
\draw[fill=9x9x4color14] (5,-32) rectangle (6,-33);
\node at (5.5,-32.5) {14};
\draw[fill=9x9x4color1] (6,-32) rectangle (7,-33);
\node at (6.5,-32.5) {1};
\draw[fill=9x9x4color4] (7,-32) rectangle (8,-33);
\node at (7.5,-32.5) {4};
\draw[fill=9x9x4color9] (8,-32) rectangle (9,-33);
\node at (8.5,-32.5) {9};
\draw[fill=9x9x4color18] (0,-33) rectangle (1,-34);
\node at (0.5,-33.5) {18};
\draw[fill=9x9x4color17] (1,-33) rectangle (2,-34);
\node at (1.5,-33.5) {17};
\draw[fill=9x9x4color16] (2,-33) rectangle (3,-34);
\node at (2.5,-33.5) {16};
\draw[fill=9x9x4color15] (3,-33) rectangle (4,-34);
\node at (3.5,-33.5) {15};
\draw[fill=9x9x4color14] (4,-33) rectangle (5,-34);
\node at (4.5,-33.5) {14};
\draw[fill=9x9x4color13] (5,-33) rectangle (6,-34);
\node at (5.5,-33.5) {13};
\draw[fill=red] (6,-33) rectangle (7,-34);
\draw[fill=9x9x4color5] (7,-33) rectangle (8,-34);
\node at (7.5,-33.5) {5};
\draw[fill=9x9x4color8] (8,-33) rectangle (9,-34);
\node at (8.5,-33.5) {8};
\draw[fill=9x9x4color17] (0,-34) rectangle (1,-35);
\node at (0.5,-34.5) {17};
\draw[fill=9x9x4color16] (1,-34) rectangle (2,-35);
\node at (1.5,-34.5) {16};
\draw[fill=9x9x4color15] (2,-34) rectangle (3,-35);
\node at (2.5,-34.5) {15};
\draw[fill=9x9x4color14] (3,-34) rectangle (4,-35);
\node at (3.5,-34.5) {14};
\draw[fill=9x9x4color13] (4,-34) rectangle (5,-35);
\node at (4.5,-34.5) {13};
\draw[fill=9x9x4color12] (5,-34) rectangle (6,-35);
\node at (5.5,-34.5) {12};
\draw[fill=9x9x4color7] (6,-34) rectangle (7,-35);
\node at (6.5,-34.5) {7};
\draw[fill=9x9x4color6] (7,-34) rectangle (8,-35);
\node at (7.5,-34.5) {6};
\draw[fill=9x9x4color7] (8,-34) rectangle (9,-35);
\node at (8.5,-34.5) {7};
\draw[fill=9x9x4color16] (0,-35) rectangle (1,-36);
\node at (0.5,-35.5) {16};
\draw[fill=9x9x4color15] (1,-35) rectangle (2,-36);
\node at (1.5,-35.5) {15};
\draw[fill=9x9x4color14] (2,-35) rectangle (3,-36);
\node at (2.5,-35.5) {14};
\draw[fill=9x9x4color13] (3,-35) rectangle (4,-36);
\node at (3.5,-35.5) {13};
\draw[fill=9x9x4color12] (4,-35) rectangle (5,-36);
\node at (4.5,-35.5) {12};
\draw[fill=9x9x4color11] (5,-35) rectangle (6,-36);
\node at (5.5,-35.5) {11};
\draw[fill=red] (6,-35) rectangle (7,-36);
\draw[fill=9x9x4color1] (7,-35) rectangle (8,-36);
\node at (7.5,-35.5) {1};
\draw[fill=red] (8,-35) rectangle (9,-36);
\draw[fill=9x9x4color7] (0,-36) rectangle (1,-37);
\node at (0.5,-36.5) {7};
\draw[fill=red] (1,-36) rectangle (2,-37);
\draw[fill=9x9x4color1] (2,-36) rectangle (3,-37);
\node at (2.5,-36.5) {1};
\draw[fill=red] (3,-36) rectangle (4,-37);
\draw[fill=9x9x4color7] (4,-36) rectangle (5,-37);
\node at (4.5,-36.5) {7};
\draw[fill=red] (5,-36) rectangle (6,-37);
\draw[fill=9x9x4color13] (6,-36) rectangle (7,-37);
\node at (6.5,-36.5) {13};
\draw[fill=9x9x4color14] (7,-36) rectangle (8,-37);
\node at (7.5,-36.5) {14};
\draw[fill=9x9x4color15] (8,-36) rectangle (9,-37);
\node at (8.5,-36.5) {15};
\draw[fill=red] (0,-37) rectangle (1,-38);
\draw[fill=9x9x4color3] (1,-37) rectangle (2,-38);
\node at (1.5,-37.5) {3};
\draw[fill=9x9x4color4] (2,-37) rectangle (3,-38);
\node at (2.5,-37.5) {4};
\draw[fill=9x9x4color5] (3,-37) rectangle (4,-38);
\node at (3.5,-37.5) {5};
\draw[fill=9x9x4color6] (4,-37) rectangle (5,-38);
\node at (4.5,-37.5) {6};
\draw[fill=9x9x4color1] (5,-37) rectangle (6,-38);
\node at (5.5,-37.5) {1};
\draw[fill=9x9x4color14] (6,-37) rectangle (7,-38);
\node at (6.5,-37.5) {14};
\draw[fill=9x9x4color15] (7,-37) rectangle (8,-38);
\node at (7.5,-37.5) {15};
\draw[fill=9x9x4color16] (8,-37) rectangle (9,-38);
\node at (8.5,-37.5) {16};
\draw[fill=9x9x4color11] (0,-38) rectangle (1,-39);
\node at (0.5,-38.5) {11};
\draw[fill=9x9x4color10] (1,-38) rectangle (2,-39);
\node at (1.5,-38.5) {10};
\draw[fill=9x9x4color9] (2,-38) rectangle (3,-39);
\node at (2.5,-38.5) {9};
\draw[fill=9x9x4color8] (3,-38) rectangle (4,-39);
\node at (3.5,-38.5) {8};
\draw[fill=9x9x4color7] (4,-38) rectangle (5,-39);
\node at (4.5,-38.5) {7};
\draw[fill=red] (5,-38) rectangle (6,-39);
\draw[fill=9x9x4color15] (6,-38) rectangle (7,-39);
\node at (6.5,-38.5) {15};
\draw[fill=9x9x4color16] (7,-38) rectangle (8,-39);
\node at (7.5,-38.5) {16};
\draw[fill=9x9x4color17] (8,-38) rectangle (9,-39);
\node at (8.5,-38.5) {17};

\end{scope}
\end{tikzpicture}
\end{center}

%% file: 2-5-29-tikz.txt
\definecolor{29x5x2color1}{RGB}{255, 128, 128}
\definecolor{29x5x2color2}{RGB}{255, 130, 130}
\definecolor{29x5x2color3}{RGB}{255, 132, 132}
\definecolor{29x5x2color4}{RGB}{255, 134, 134}
\definecolor{29x5x2color5}{RGB}{255, 136, 136}
\definecolor{29x5x2color6}{RGB}{255, 138, 138}
\definecolor{29x5x2color7}{RGB}{255, 140, 140}
\definecolor{29x5x2color8}{RGB}{255, 142, 142}
\definecolor{29x5x2color9}{RGB}{255, 144, 144}
\definecolor{29x5x2color10}{RGB}{255, 146, 146}
\definecolor{29x5x2color11}{RGB}{255, 148, 148}
\definecolor{29x5x2color12}{RGB}{255, 150, 150}
\definecolor{29x5x2color13}{RGB}{255, 152, 152}
\definecolor{29x5x2color14}{RGB}{255, 154, 154}
\definecolor{29x5x2color15}{RGB}{255, 156, 156}
\definecolor{29x5x2color16}{RGB}{255, 158, 158}
\definecolor{29x5x2color17}{RGB}{255, 160, 160}
\definecolor{29x5x2color18}{RGB}{255, 162, 162}
\definecolor{29x5x2color19}{RGB}{255, 164, 164}
\definecolor{29x5x2color20}{RGB}{255, 166, 166}
\definecolor{29x5x2color21}{RGB}{255, 168, 168}
\definecolor{29x5x2color22}{RGB}{255, 171, 171}
\definecolor{29x5x2color23}{RGB}{255, 173, 173}
\definecolor{29x5x2color24}{RGB}{255, 175, 175}
\definecolor{29x5x2color25}{RGB}{255, 177, 177}
\definecolor{29x5x2color26}{RGB}{255, 179, 179}
\definecolor{29x5x2color27}{RGB}{255, 181, 181}
\definecolor{29x5x2color28}{RGB}{255, 183, 183}
\definecolor{29x5x2color29}{RGB}{255, 185, 185}
\definecolor{29x5x2color30}{RGB}{255, 187, 187}
\definecolor{29x5x2color31}{RGB}{255, 189, 189}
\definecolor{29x5x2color32}{RGB}{255, 191, 191}
\definecolor{29x5x2color33}{RGB}{255, 193, 193}
\definecolor{29x5x2color34}{RGB}{255, 195, 195}
\definecolor{29x5x2color35}{RGB}{255, 197, 197}
\definecolor{29x5x2color36}{RGB}{255, 199, 199}
\definecolor{29x5x2color37}{RGB}{255, 201, 201}
\definecolor{29x5x2color38}{RGB}{255, 203, 203}
\definecolor{29x5x2color39}{RGB}{255, 205, 205}
\definecolor{29x5x2color40}{RGB}{255, 207, 207}
\definecolor{29x5x2color41}{RGB}{255, 209, 209}
\definecolor{29x5x2color42}{RGB}{255, 211, 211}
\definecolor{29x5x2color43}{RGB}{255, 214, 214}
\definecolor{29x5x2color44}{RGB}{255, 216, 216}
\definecolor{29x5x2color45}{RGB}{255, 218, 218}
\definecolor{29x5x2color46}{RGB}{255, 220, 220}
\definecolor{29x5x2color47}{RGB}{255, 222, 222}
\definecolor{29x5x2color48}{RGB}{255, 224, 224}
\definecolor{29x5x2color49}{RGB}{255, 226, 226}
\definecolor{29x5x2color50}{RGB}{255, 228, 228}
\definecolor{29x5x2color51}{RGB}{255, 230, 230}
\definecolor{29x5x2color52}{RGB}{255, 232, 232}
\definecolor{29x5x2color53}{RGB}{255, 234, 234}
\definecolor{29x5x2color54}{RGB}{255, 236, 236}
\definecolor{29x5x2color55}{RGB}{255, 238, 238}
\definecolor{29x5x2color56}{RGB}{255, 240, 240}
\definecolor{29x5x2color57}{RGB}{255, 242, 242}
\definecolor{29x5x2color58}{RGB}{255, 244, 244}
\definecolor{29x5x2color59}{RGB}{255, 246, 246}
\definecolor{29x5x2color60}{RGB}{255, 248, 248}
\definecolor{29x5x2color61}{RGB}{255, 250, 250}
\definecolor{29x5x2color62}{RGB}{255, 255, 255}

\begin{center}
\begin{tikzpicture}[scale=0.5]
\begin{scope}
\draw[very thick] (0,-0) rectangle (29,-5);
\foreach \b in {0,...,4}
 \foreach \c in {0,...,28}
  \draw (\c,\b*-1-0) rectangle (\c+1,(\b*-1-1);
\draw[very thick] (0,-6) rectangle (29,-11);
\foreach \b in {0,...,4}
 \foreach \c in {0,...,28}
  \draw (\c,\b*-1-6) rectangle (\c+1,(\b*-1-7);

\draw[fill=red] (0,-0) rectangle (1,-1);
\draw[fill=29x5x2color5] (1,-0) rectangle (2,-1);
\node at (1.5,-0.5) {5};
\draw[fill=red] (2,-0) rectangle (3,-1);
\draw[fill=29x5x2color1] (3,-0) rectangle (4,-1);
\node at (3.5,-0.5) {1};
\draw[fill=red] (4,-0) rectangle (5,-1);
\draw[fill=29x5x2color1] (5,-0) rectangle (6,-1);
\node at (5.5,-0.5) {1};
\draw[fill=red] (6,-0) rectangle (7,-1);
\draw[fill=29x5x2color5] (7,-0) rectangle (8,-1);
\node at (7.5,-0.5) {5};
\draw[fill=red] (8,-0) rectangle (9,-1);
\draw[fill=29x5x2color1] (9,-0) rectangle (10,-1);
\node at (9.5,-0.5) {1};
\draw[fill=red] (10,-0) rectangle (11,-1);
\draw[fill=29x5x2color1] (11,-0) rectangle (12,-1);
\node at (11.5,-0.5) {1};
\draw[fill=red] (12,-0) rectangle (13,-1);
\draw[fill=29x5x2color5] (13,-0) rectangle (14,-1);
\node at (13.5,-0.5) {5};
\draw[fill=red] (14,-0) rectangle (15,-1);
\draw[fill=29x5x2color1] (15,-0) rectangle (16,-1);
\node at (15.5,-0.5) {1};
\draw[fill=red] (16,-0) rectangle (17,-1);
\draw[fill=29x5x2color1] (17,-0) rectangle (18,-1);
\node at (17.5,-0.5) {1};
\draw[fill=red] (18,-0) rectangle (19,-1);
\draw[fill=29x5x2color5] (19,-0) rectangle (20,-1);
\node at (19.5,-0.5) {5};
\draw[fill=red] (20,-0) rectangle (21,-1);
\draw[fill=29x5x2color1] (21,-0) rectangle (22,-1);
\node at (21.5,-0.5) {1};
\draw[fill=red] (22,-0) rectangle (23,-1);
\draw[fill=29x5x2color1] (23,-0) rectangle (24,-1);
\node at (23.5,-0.5) {1};
\draw[fill=red] (24,-0) rectangle (25,-1);
\draw[fill=29x5x2color33] (25,-0) rectangle (26,-1);
\node at (25.5,-0.5) {33};
\draw[fill=red] (26,-0) rectangle (27,-1);
\draw[fill=29x5x2color1] (27,-0) rectangle (28,-1);
\node at (27.5,-0.5) {1};
\draw[fill=red] (28,-0) rectangle (29,-1);
\draw[fill=29x5x2color1] (0,-1) rectangle (1,-2);
\node at (0.5,-1.5) {1};
\draw[fill=29x5x2color4] (1,-1) rectangle (2,-2);
\node at (1.5,-1.5) {4};
\draw[fill=29x5x2color1] (2,-1) rectangle (3,-2);
\node at (2.5,-1.5) {1};
\draw[fill=red] (3,-1) rectangle (4,-2);
\draw[fill=29x5x2color1] (4,-1) rectangle (5,-2);
\node at (4.5,-1.5) {1};
\draw[fill=red] (5,-1) rectangle (6,-2);
\draw[fill=29x5x2color1] (6,-1) rectangle (7,-2);
\node at (6.5,-1.5) {1};
\draw[fill=29x5x2color4] (7,-1) rectangle (8,-2);
\node at (7.5,-1.5) {4};
\draw[fill=29x5x2color1] (8,-1) rectangle (9,-2);
\node at (8.5,-1.5) {1};
\draw[fill=red] (9,-1) rectangle (10,-2);
\draw[fill=29x5x2color1] (10,-1) rectangle (11,-2);
\node at (10.5,-1.5) {1};
\draw[fill=red] (11,-1) rectangle (12,-2);
\draw[fill=29x5x2color1] (12,-1) rectangle (13,-2);
\node at (12.5,-1.5) {1};
\draw[fill=29x5x2color4] (13,-1) rectangle (14,-2);
\node at (13.5,-1.5) {4};
\draw[fill=29x5x2color1] (14,-1) rectangle (15,-2);
\node at (14.5,-1.5) {1};
\draw[fill=red] (15,-1) rectangle (16,-2);
\draw[fill=29x5x2color1] (16,-1) rectangle (17,-2);
\node at (16.5,-1.5) {1};
\draw[fill=red] (17,-1) rectangle (18,-2);
\draw[fill=29x5x2color1] (18,-1) rectangle (19,-2);
\node at (18.5,-1.5) {1};
\draw[fill=29x5x2color4] (19,-1) rectangle (20,-2);
\node at (19.5,-1.5) {4};
\draw[fill=29x5x2color1] (20,-1) rectangle (21,-2);
\node at (20.5,-1.5) {1};
\draw[fill=red] (21,-1) rectangle (22,-2);
\draw[fill=29x5x2color1] (22,-1) rectangle (23,-2);
\node at (22.5,-1.5) {1};
\draw[fill=red] (23,-1) rectangle (24,-2);
\draw[fill=29x5x2color1] (24,-1) rectangle (25,-2);
\node at (24.5,-1.5) {1};
\draw[fill=29x5x2color32] (25,-1) rectangle (26,-2);
\node at (25.5,-1.5) {32};
\draw[fill=29x5x2color3] (26,-1) rectangle (27,-2);
\node at (26.5,-1.5) {3};
\draw[fill=29x5x2color2] (27,-1) rectangle (28,-2);
\node at (27.5,-1.5) {2};
\draw[fill=29x5x2color1] (28,-1) rectangle (29,-2);
\node at (28.5,-1.5) {1};
\draw[fill=red] (0,-2) rectangle (1,-3);
\draw[fill=29x5x2color3] (1,-2) rectangle (2,-3);
\node at (1.5,-2.5) {3};
\draw[fill=red] (2,-2) rectangle (3,-3);
\draw[fill=29x5x2color1] (3,-2) rectangle (4,-3);
\node at (3.5,-2.5) {1};
\draw[fill=29x5x2color2] (4,-2) rectangle (5,-3);
\node at (4.5,-2.5) {2};
\draw[fill=29x5x2color1] (5,-2) rectangle (6,-3);
\node at (5.5,-2.5) {1};
\draw[fill=red] (6,-2) rectangle (7,-3);
\draw[fill=29x5x2color3] (7,-2) rectangle (8,-3);
\node at (7.5,-2.5) {3};
\draw[fill=red] (8,-2) rectangle (9,-3);
\draw[fill=29x5x2color1] (9,-2) rectangle (10,-3);
\node at (9.5,-2.5) {1};
\draw[fill=29x5x2color2] (10,-2) rectangle (11,-3);
\node at (10.5,-2.5) {2};
\draw[fill=29x5x2color1] (11,-2) rectangle (12,-3);
\node at (11.5,-2.5) {1};
\draw[fill=red] (12,-2) rectangle (13,-3);
\draw[fill=29x5x2color3] (13,-2) rectangle (14,-3);
\node at (13.5,-2.5) {3};
\draw[fill=red] (14,-2) rectangle (15,-3);
\draw[fill=29x5x2color1] (15,-2) rectangle (16,-3);
\node at (15.5,-2.5) {1};
\draw[fill=29x5x2color2] (16,-2) rectangle (17,-3);
\node at (16.5,-2.5) {2};
\draw[fill=29x5x2color1] (17,-2) rectangle (18,-3);
\node at (17.5,-2.5) {1};
\draw[fill=red] (18,-2) rectangle (19,-3);
\draw[fill=29x5x2color3] (19,-2) rectangle (20,-3);
\node at (19.5,-2.5) {3};
\draw[fill=red] (20,-2) rectangle (21,-3);
\draw[fill=29x5x2color1] (21,-2) rectangle (22,-3);
\node at (21.5,-2.5) {1};
\draw[fill=29x5x2color2] (22,-2) rectangle (23,-3);
\node at (22.5,-2.5) {2};
\draw[fill=29x5x2color1] (23,-2) rectangle (24,-3);
\node at (23.5,-2.5) {1};
\draw[fill=red] (24,-2) rectangle (25,-3);
\draw[fill=29x5x2color31] (25,-2) rectangle (26,-3);
\node at (25.5,-2.5) {31};
\draw[fill=red] (26,-2) rectangle (27,-3);
\draw[fill=29x5x2color1] (27,-2) rectangle (28,-3);
\node at (27.5,-2.5) {1};
\draw[fill=red] (28,-2) rectangle (29,-3);
\draw[fill=29x5x2color1] (0,-3) rectangle (1,-4);
\node at (0.5,-3.5) {1};
\draw[fill=29x5x2color4] (1,-3) rectangle (2,-4);
\node at (1.5,-3.5) {4};
\draw[fill=29x5x2color5] (2,-3) rectangle (3,-4);
\node at (2.5,-3.5) {5};
\draw[fill=29x5x2color6] (3,-3) rectangle (4,-4);
\node at (3.5,-3.5) {6};
\draw[fill=29x5x2color7] (4,-3) rectangle (5,-4);
\node at (4.5,-3.5) {7};
\draw[fill=29x5x2color8] (5,-3) rectangle (6,-4);
\node at (5.5,-3.5) {8};
\draw[fill=29x5x2color9] (6,-3) rectangle (7,-4);
\node at (6.5,-3.5) {9};
\draw[fill=29x5x2color10] (7,-3) rectangle (8,-4);
\node at (7.5,-3.5) {10};
\draw[fill=29x5x2color11] (8,-3) rectangle (9,-4);
\node at (8.5,-3.5) {11};
\draw[fill=29x5x2color12] (9,-3) rectangle (10,-4);
\node at (9.5,-3.5) {12};
\draw[fill=29x5x2color13] (10,-3) rectangle (11,-4);
\node at (10.5,-3.5) {13};
\draw[fill=29x5x2color14] (11,-3) rectangle (12,-4);
\node at (11.5,-3.5) {14};
\draw[fill=29x5x2color15] (12,-3) rectangle (13,-4);
\node at (12.5,-3.5) {15};
\draw[fill=29x5x2color16] (13,-3) rectangle (14,-4);
\node at (13.5,-3.5) {16};
\draw[fill=29x5x2color17] (14,-3) rectangle (15,-4);
\node at (14.5,-3.5) {17};
\draw[fill=29x5x2color18] (15,-3) rectangle (16,-4);
\node at (15.5,-3.5) {18};
\draw[fill=29x5x2color19] (16,-3) rectangle (17,-4);
\node at (16.5,-3.5) {19};
\draw[fill=29x5x2color20] (17,-3) rectangle (18,-4);
\node at (17.5,-3.5) {20};
\draw[fill=29x5x2color21] (18,-3) rectangle (19,-4);
\node at (18.5,-3.5) {21};
\draw[fill=29x5x2color22] (19,-3) rectangle (20,-4);
\node at (19.5,-3.5) {22};
\draw[fill=29x5x2color23] (20,-3) rectangle (21,-4);
\node at (20.5,-3.5) {23};
\draw[fill=29x5x2color24] (21,-3) rectangle (22,-4);
\node at (21.5,-3.5) {24};
\draw[fill=29x5x2color25] (22,-3) rectangle (23,-4);
\node at (22.5,-3.5) {25};
\draw[fill=29x5x2color26] (23,-3) rectangle (24,-4);
\node at (23.5,-3.5) {26};
\draw[fill=29x5x2color27] (24,-3) rectangle (25,-4);
\node at (24.5,-3.5) {27};
\draw[fill=29x5x2color30] (25,-3) rectangle (26,-4);
\node at (25.5,-3.5) {30};
\draw[fill=29x5x2color1] (26,-3) rectangle (27,-4);
\node at (26.5,-3.5) {1};
\draw[fill=red] (27,-3) rectangle (28,-4);
\draw[fill=29x5x2color1] (28,-3) rectangle (29,-4);
\node at (28.5,-3.5) {1};
\draw[fill=red] (0,-4) rectangle (1,-5);
\draw[fill=29x5x2color5] (1,-4) rectangle (2,-5);
\node at (1.5,-4.5) {5};
\draw[fill=29x5x2color6] (2,-4) rectangle (3,-5);
\node at (2.5,-4.5) {6};
\draw[fill=29x5x2color7] (3,-4) rectangle (4,-5);
\node at (3.5,-4.5) {7};
\draw[fill=29x5x2color8] (4,-4) rectangle (5,-5);
\node at (4.5,-4.5) {8};
\draw[fill=29x5x2color9] (5,-4) rectangle (6,-5);
\node at (5.5,-4.5) {9};
\draw[fill=29x5x2color10] (6,-4) rectangle (7,-5);
\node at (6.5,-4.5) {10};
\draw[fill=29x5x2color11] (7,-4) rectangle (8,-5);
\node at (7.5,-4.5) {11};
\draw[fill=29x5x2color12] (8,-4) rectangle (9,-5);
\node at (8.5,-4.5) {12};
\draw[fill=29x5x2color13] (9,-4) rectangle (10,-5);
\node at (9.5,-4.5) {13};
\draw[fill=29x5x2color14] (10,-4) rectangle (11,-5);
\node at (10.5,-4.5) {14};
\draw[fill=29x5x2color15] (11,-4) rectangle (12,-5);
\node at (11.5,-4.5) {15};
\draw[fill=29x5x2color16] (12,-4) rectangle (13,-5);
\node at (12.5,-4.5) {16};
\draw[fill=29x5x2color17] (13,-4) rectangle (14,-5);
\node at (13.5,-4.5) {17};
\draw[fill=29x5x2color18] (14,-4) rectangle (15,-5);
\node at (14.5,-4.5) {18};
\draw[fill=29x5x2color19] (15,-4) rectangle (16,-5);
\node at (15.5,-4.5) {19};
\draw[fill=29x5x2color20] (16,-4) rectangle (17,-5);
\node at (16.5,-4.5) {20};
\draw[fill=29x5x2color21] (17,-4) rectangle (18,-5);
\node at (17.5,-4.5) {21};
\draw[fill=29x5x2color22] (18,-4) rectangle (19,-5);
\node at (18.5,-4.5) {22};
\draw[fill=29x5x2color23] (19,-4) rectangle (20,-5);
\node at (19.5,-4.5) {23};
\draw[fill=29x5x2color24] (20,-4) rectangle (21,-5);
\node at (20.5,-4.5) {24};
\draw[fill=29x5x2color25] (21,-4) rectangle (22,-5);
\node at (21.5,-4.5) {25};
\draw[fill=29x5x2color26] (22,-4) rectangle (23,-5);
\node at (22.5,-4.5) {26};
\draw[fill=29x5x2color27] (23,-4) rectangle (24,-5);
\node at (23.5,-4.5) {27};
\draw[fill=29x5x2color28] (24,-4) rectangle (25,-5);
\node at (24.5,-4.5) {28};
\draw[fill=29x5x2color29] (25,-4) rectangle (26,-5);
\node at (25.5,-4.5) {29};
\draw[fill=red] (26,-4) rectangle (27,-5);
\draw[fill=29x5x2color1] (27,-4) rectangle (28,-5);
\node at (27.5,-4.5) {1};
\draw[fill=red] (28,-4) rectangle (29,-5);
\draw[fill=29x5x2color61] (0,-6) rectangle (1,-7);
\node at (0.5,-6.5) {61};
\draw[fill=29x5x2color60] (1,-6) rectangle (2,-7);
\node at (1.5,-6.5) {60};
\draw[fill=29x5x2color59] (2,-6) rectangle (3,-7);
\node at (2.5,-6.5) {59};
\draw[fill=29x5x2color58] (3,-6) rectangle (4,-7);
\node at (3.5,-6.5) {58};
\draw[fill=29x5x2color57] (4,-6) rectangle (5,-7);
\node at (4.5,-6.5) {57};
\draw[fill=29x5x2color56] (5,-6) rectangle (6,-7);
\node at (5.5,-6.5) {56};
\draw[fill=29x5x2color55] (6,-6) rectangle (7,-7);
\node at (6.5,-6.5) {55};
\draw[fill=29x5x2color54] (7,-6) rectangle (8,-7);
\node at (7.5,-6.5) {54};
\draw[fill=29x5x2color53] (8,-6) rectangle (9,-7);
\node at (8.5,-6.5) {53};
\draw[fill=29x5x2color52] (9,-6) rectangle (10,-7);
\node at (9.5,-6.5) {52};
\draw[fill=29x5x2color51] (10,-6) rectangle (11,-7);
\node at (10.5,-6.5) {51};
\draw[fill=29x5x2color50] (11,-6) rectangle (12,-7);
\node at (11.5,-6.5) {50};
\draw[fill=29x5x2color49] (12,-6) rectangle (13,-7);
\node at (12.5,-6.5) {49};
\draw[fill=29x5x2color48] (13,-6) rectangle (14,-7);
\node at (13.5,-6.5) {48};
\draw[fill=29x5x2color47] (14,-6) rectangle (15,-7);
\node at (14.5,-6.5) {47};
\draw[fill=29x5x2color46] (15,-6) rectangle (16,-7);
\node at (15.5,-6.5) {46};
\draw[fill=29x5x2color45] (16,-6) rectangle (17,-7);
\node at (16.5,-6.5) {45};
\draw[fill=29x5x2color44] (17,-6) rectangle (18,-7);
\node at (17.5,-6.5) {44};
\draw[fill=29x5x2color43] (18,-6) rectangle (19,-7);
\node at (18.5,-6.5) {43};
\draw[fill=29x5x2color42] (19,-6) rectangle (20,-7);
\node at (19.5,-6.5) {42};
\draw[fill=29x5x2color41] (20,-6) rectangle (21,-7);
\node at (20.5,-6.5) {41};
\draw[fill=29x5x2color40] (21,-6) rectangle (22,-7);
\node at (21.5,-6.5) {40};
\draw[fill=29x5x2color39] (22,-6) rectangle (23,-7);
\node at (22.5,-6.5) {39};
\draw[fill=29x5x2color38] (23,-6) rectangle (24,-7);
\node at (23.5,-6.5) {38};
\draw[fill=29x5x2color37] (24,-6) rectangle (25,-7);
\node at (24.5,-6.5) {37};
\draw[fill=29x5x2color36] (25,-6) rectangle (26,-7);
\node at (25.5,-6.5) {36};
\draw[fill=29x5x2color35] (26,-6) rectangle (27,-7);
\node at (26.5,-6.5) {35};
\draw[fill=red] (27,-6) rectangle (28,-7);
\draw[fill=29x5x2color1] (28,-6) rectangle (29,-7);
\node at (28.5,-6.5) {1};
\draw[fill=red] (0,-7) rectangle (1,-8);
\draw[fill=29x5x2color5] (1,-7) rectangle (2,-8);
\node at (1.5,-7.5) {5};
\draw[fill=29x5x2color6] (2,-7) rectangle (3,-8);
\node at (2.5,-7.5) {6};
\draw[fill=29x5x2color7] (3,-7) rectangle (4,-8);
\node at (3.5,-7.5) {7};
\draw[fill=29x5x2color8] (4,-7) rectangle (5,-8);
\node at (4.5,-7.5) {8};
\draw[fill=29x5x2color9] (5,-7) rectangle (6,-8);
\node at (5.5,-7.5) {9};
\draw[fill=29x5x2color10] (6,-7) rectangle (7,-8);
\node at (6.5,-7.5) {10};
\draw[fill=29x5x2color11] (7,-7) rectangle (8,-8);
\node at (7.5,-7.5) {11};
\draw[fill=29x5x2color12] (8,-7) rectangle (9,-8);
\node at (8.5,-7.5) {12};
\draw[fill=29x5x2color13] (9,-7) rectangle (10,-8);
\node at (9.5,-7.5) {13};
\draw[fill=29x5x2color14] (10,-7) rectangle (11,-8);
\node at (10.5,-7.5) {14};
\draw[fill=29x5x2color15] (11,-7) rectangle (12,-8);
\node at (11.5,-7.5) {15};
\draw[fill=29x5x2color16] (12,-7) rectangle (13,-8);
\node at (12.5,-7.5) {16};
\draw[fill=29x5x2color17] (13,-7) rectangle (14,-8);
\node at (13.5,-7.5) {17};
\draw[fill=29x5x2color18] (14,-7) rectangle (15,-8);
\node at (14.5,-7.5) {18};
\draw[fill=29x5x2color19] (15,-7) rectangle (16,-8);
\node at (15.5,-7.5) {19};
\draw[fill=29x5x2color20] (16,-7) rectangle (17,-8);
\node at (16.5,-7.5) {20};
\draw[fill=29x5x2color21] (17,-7) rectangle (18,-8);
\node at (17.5,-7.5) {21};
\draw[fill=29x5x2color22] (18,-7) rectangle (19,-8);
\node at (18.5,-7.5) {22};
\draw[fill=29x5x2color23] (19,-7) rectangle (20,-8);
\node at (19.5,-7.5) {23};
\draw[fill=29x5x2color24] (20,-7) rectangle (21,-8);
\node at (20.5,-7.5) {24};
\draw[fill=29x5x2color25] (21,-7) rectangle (22,-8);
\node at (21.5,-7.5) {25};
\draw[fill=29x5x2color26] (22,-7) rectangle (23,-8);
\node at (22.5,-7.5) {26};
\draw[fill=29x5x2color27] (23,-7) rectangle (24,-8);
\node at (23.5,-7.5) {27};
\draw[fill=29x5x2color28] (24,-7) rectangle (25,-8);
\node at (24.5,-7.5) {28};
\draw[fill=29x5x2color33] (25,-7) rectangle (26,-8);
\node at (25.5,-7.5) {33};
\draw[fill=29x5x2color34] (26,-7) rectangle (27,-8);
\node at (26.5,-7.5) {34};
\draw[fill=29x5x2color3] (27,-7) rectangle (28,-8);
\node at (27.5,-7.5) {3};
\draw[fill=red] (28,-7) rectangle (29,-8);
\draw[fill=29x5x2color1] (0,-8) rectangle (1,-9);
\node at (0.5,-8.5) {1};
\draw[fill=29x5x2color2] (1,-8) rectangle (2,-9);
\node at (1.5,-8.5) {2};
\draw[fill=29x5x2color1] (2,-8) rectangle (3,-9);
\node at (2.5,-8.5) {1};
\draw[fill=red] (3,-8) rectangle (4,-9);
\draw[fill=29x5x2color3] (4,-8) rectangle (5,-9);
\node at (4.5,-8.5) {3};
\draw[fill=red] (5,-8) rectangle (6,-9);
\draw[fill=29x5x2color1] (6,-8) rectangle (7,-9);
\node at (6.5,-8.5) {1};
\draw[fill=29x5x2color2] (7,-8) rectangle (8,-9);
\node at (7.5,-8.5) {2};
\draw[fill=29x5x2color1] (8,-8) rectangle (9,-9);
\node at (8.5,-8.5) {1};
\draw[fill=red] (9,-8) rectangle (10,-9);
\draw[fill=29x5x2color3] (10,-8) rectangle (11,-9);
\node at (10.5,-8.5) {3};
\draw[fill=red] (11,-8) rectangle (12,-9);
\draw[fill=29x5x2color1] (12,-8) rectangle (13,-9);
\node at (12.5,-8.5) {1};
\draw[fill=29x5x2color2] (13,-8) rectangle (14,-9);
\node at (13.5,-8.5) {2};
\draw[fill=29x5x2color1] (14,-8) rectangle (15,-9);
\node at (14.5,-8.5) {1};
\draw[fill=red] (15,-8) rectangle (16,-9);
\draw[fill=29x5x2color3] (16,-8) rectangle (17,-9);
\node at (16.5,-8.5) {3};
\draw[fill=red] (17,-8) rectangle (18,-9);
\draw[fill=29x5x2color1] (18,-8) rectangle (19,-9);
\node at (18.5,-8.5) {1};
\draw[fill=29x5x2color2] (19,-8) rectangle (20,-9);
\node at (19.5,-8.5) {2};
\draw[fill=29x5x2color1] (20,-8) rectangle (21,-9);
\node at (20.5,-8.5) {1};
\draw[fill=red] (21,-8) rectangle (22,-9);
\draw[fill=29x5x2color3] (22,-8) rectangle (23,-9);
\node at (22.5,-8.5) {3};
\draw[fill=red] (23,-8) rectangle (24,-9);
\draw[fill=29x5x2color1] (24,-8) rectangle (25,-9);
\node at (24.5,-8.5) {1};
\draw[fill=29x5x2color32] (25,-8) rectangle (26,-9);
\node at (25.5,-8.5) {32};
\draw[fill=29x5x2color35] (26,-8) rectangle (27,-9);
\node at (26.5,-8.5) {35};
\draw[fill=29x5x2color36] (27,-8) rectangle (28,-9);
\node at (27.5,-8.5) {36};
\draw[fill=29x5x2color37] (28,-8) rectangle (29,-9);
\node at (28.5,-8.5) {37};
\draw[fill=red] (0,-9) rectangle (1,-10);
\draw[fill=29x5x2color1] (1,-9) rectangle (2,-10);
\node at (1.5,-9.5) {1};
\draw[fill=red] (2,-9) rectangle (3,-10);
\draw[fill=29x5x2color1] (3,-9) rectangle (4,-10);
\node at (3.5,-9.5) {1};
\draw[fill=29x5x2color4] (4,-9) rectangle (5,-10);
\node at (4.5,-9.5) {4};
\draw[fill=29x5x2color1] (5,-9) rectangle (6,-10);
\node at (5.5,-9.5) {1};
\draw[fill=red] (6,-9) rectangle (7,-10);
\draw[fill=29x5x2color1] (7,-9) rectangle (8,-10);
\node at (7.5,-9.5) {1};
\draw[fill=red] (8,-9) rectangle (9,-10);
\draw[fill=29x5x2color1] (9,-9) rectangle (10,-10);
\node at (9.5,-9.5) {1};
\draw[fill=29x5x2color4] (10,-9) rectangle (11,-10);
\node at (10.5,-9.5) {4};
\draw[fill=29x5x2color1] (11,-9) rectangle (12,-10);
\node at (11.5,-9.5) {1};
\draw[fill=red] (12,-9) rectangle (13,-10);
\draw[fill=29x5x2color1] (13,-9) rectangle (14,-10);
\node at (13.5,-9.5) {1};
\draw[fill=red] (14,-9) rectangle (15,-10);
\draw[fill=29x5x2color1] (15,-9) rectangle (16,-10);
\node at (15.5,-9.5) {1};
\draw[fill=29x5x2color4] (16,-9) rectangle (17,-10);
\node at (16.5,-9.5) {4};
\draw[fill=29x5x2color1] (17,-9) rectangle (18,-10);
\node at (17.5,-9.5) {1};
\draw[fill=red] (18,-9) rectangle (19,-10);
\draw[fill=29x5x2color1] (19,-9) rectangle (20,-10);
\node at (19.5,-9.5) {1};
\draw[fill=red] (20,-9) rectangle (21,-10);
\draw[fill=29x5x2color1] (21,-9) rectangle (22,-10);
\node at (21.5,-9.5) {1};
\draw[fill=29x5x2color4] (22,-9) rectangle (23,-10);
\node at (22.5,-9.5) {4};
\draw[fill=29x5x2color1] (23,-9) rectangle (24,-10);
\node at (23.5,-9.5) {1};
\draw[fill=red] (24,-9) rectangle (25,-10);
\draw[fill=29x5x2color31] (25,-9) rectangle (26,-10);
\node at (25.5,-9.5) {31};
\draw[fill=29x5x2color36] (26,-9) rectangle (27,-10);
\node at (26.5,-9.5) {36};
\draw[fill=29x5x2color37] (27,-9) rectangle (28,-10);
\node at (27.5,-9.5) {37};
\draw[fill=29x5x2color38] (28,-9) rectangle (29,-10);
\node at (28.5,-9.5) {38};
\draw[fill=29x5x2color1] (0,-10) rectangle (1,-11);
\node at (0.5,-10.5) {1};
\draw[fill=red] (1,-10) rectangle (2,-11);
\draw[fill=29x5x2color1] (2,-10) rectangle (3,-11);
\node at (2.5,-10.5) {1};
\draw[fill=red] (3,-10) rectangle (4,-11);
\draw[fill=29x5x2color5] (4,-10) rectangle (5,-11);
\node at (4.5,-10.5) {5};
\draw[fill=red] (5,-10) rectangle (6,-11);
\draw[fill=29x5x2color1] (6,-10) rectangle (7,-11);
\node at (6.5,-10.5) {1};
\draw[fill=red] (7,-10) rectangle (8,-11);
\draw[fill=29x5x2color1] (8,-10) rectangle (9,-11);
\node at (8.5,-10.5) {1};
\draw[fill=red] (9,-10) rectangle (10,-11);
\draw[fill=29x5x2color5] (10,-10) rectangle (11,-11);
\node at (10.5,-10.5) {5};
\draw[fill=red] (11,-10) rectangle (12,-11);
\draw[fill=29x5x2color1] (12,-10) rectangle (13,-11);
\node at (12.5,-10.5) {1};
\draw[fill=red] (13,-10) rectangle (14,-11);
\draw[fill=29x5x2color1] (14,-10) rectangle (15,-11);
\node at (14.5,-10.5) {1};
\draw[fill=red] (15,-10) rectangle (16,-11);
\draw[fill=29x5x2color5] (16,-10) rectangle (17,-11);
\node at (16.5,-10.5) {5};
\draw[fill=red] (17,-10) rectangle (18,-11);
\draw[fill=29x5x2color1] (18,-10) rectangle (19,-11);
\node at (18.5,-10.5) {1};
\draw[fill=red] (19,-10) rectangle (20,-11);
\draw[fill=29x5x2color1] (20,-10) rectangle (21,-11);
\node at (20.5,-10.5) {1};
\draw[fill=red] (21,-10) rectangle (22,-11);
\draw[fill=29x5x2color5] (22,-10) rectangle (23,-11);
\node at (22.5,-10.5) {5};
\draw[fill=red] (23,-10) rectangle (24,-11);
\draw[fill=29x5x2color1] (24,-10) rectangle (25,-11);
\node at (24.5,-10.5) {1};
\draw[fill=red] (25,-10) rectangle (26,-11);
\draw[fill=29x5x2color37] (26,-10) rectangle (27,-11);
\node at (26.5,-10.5) {37};
\draw[fill=29x5x2color38] (27,-10) rectangle (28,-11);
\node at (27.5,-10.5) {38};
\draw[fill=29x5x2color39] (28,-10) rectangle (29,-11);
\node at (28.5,-10.5) {39};

\end{scope}
\end{tikzpicture}
\end{center}

%% file: 2-6-24-tikz.txt
\definecolor{24x6x2color1}{RGB}{255, 128, 128}
\definecolor{24x6x2color2}{RGB}{255, 130, 130}
\definecolor{24x6x2color3}{RGB}{255, 132, 132}
\definecolor{24x6x2color4}{RGB}{255, 135, 135}
\definecolor{24x6x2color5}{RGB}{255, 137, 137}
\definecolor{24x6x2color6}{RGB}{255, 139, 139}
\definecolor{24x6x2color7}{RGB}{255, 142, 142}
\definecolor{24x6x2color8}{RGB}{255, 144, 144}
\definecolor{24x6x2color9}{RGB}{255, 146, 146}
\definecolor{24x6x2color10}{RGB}{255, 149, 149}
\definecolor{24x6x2color11}{RGB}{255, 151, 151}
\definecolor{24x6x2color12}{RGB}{255, 153, 153}
\definecolor{24x6x2color13}{RGB}{255, 156, 156}
\definecolor{24x6x2color14}{RGB}{255, 158, 158}
\definecolor{24x6x2color15}{RGB}{255, 160, 160}
\definecolor{24x6x2color16}{RGB}{255, 163, 163}
\definecolor{24x6x2color17}{RGB}{255, 165, 165}
\definecolor{24x6x2color18}{RGB}{255, 167, 167}
\definecolor{24x6x2color19}{RGB}{255, 170, 170}
\definecolor{24x6x2color20}{RGB}{255, 172, 172}
\definecolor{24x6x2color21}{RGB}{255, 175, 175}
\definecolor{24x6x2color22}{RGB}{255, 177, 177}
\definecolor{24x6x2color23}{RGB}{255, 179, 179}
\definecolor{24x6x2color24}{RGB}{255, 182, 182}
\definecolor{24x6x2color25}{RGB}{255, 184, 184}
\definecolor{24x6x2color26}{RGB}{255, 186, 186}
\definecolor{24x6x2color27}{RGB}{255, 189, 189}
\definecolor{24x6x2color28}{RGB}{255, 191, 191}
\definecolor{24x6x2color29}{RGB}{255, 193, 193}
\definecolor{24x6x2color30}{RGB}{255, 196, 196}
\definecolor{24x6x2color31}{RGB}{255, 198, 198}
\definecolor{24x6x2color32}{RGB}{255, 200, 200}
\definecolor{24x6x2color33}{RGB}{255, 203, 203}
\definecolor{24x6x2color34}{RGB}{255, 205, 205}
\definecolor{24x6x2color35}{RGB}{255, 207, 207}
\definecolor{24x6x2color36}{RGB}{255, 210, 210}
\definecolor{24x6x2color37}{RGB}{255, 212, 212}
\definecolor{24x6x2color38}{RGB}{255, 215, 215}
\definecolor{24x6x2color39}{RGB}{255, 217, 217}
\definecolor{24x6x2color40}{RGB}{255, 219, 219}
\definecolor{24x6x2color41}{RGB}{255, 222, 222}
\definecolor{24x6x2color42}{RGB}{255, 224, 224}
\definecolor{24x6x2color43}{RGB}{255, 226, 226}
\definecolor{24x6x2color44}{RGB}{255, 229, 229}
\definecolor{24x6x2color45}{RGB}{255, 231, 231}
\definecolor{24x6x2color46}{RGB}{255, 233, 233}
\definecolor{24x6x2color47}{RGB}{255, 236, 236}
\definecolor{24x6x2color48}{RGB}{255, 238, 238}
\definecolor{24x6x2color49}{RGB}{255, 240, 240}
\definecolor{24x6x2color50}{RGB}{255, 243, 243}
\definecolor{24x6x2color51}{RGB}{255, 245, 245}
\definecolor{24x6x2color52}{RGB}{255, 247, 247}
\definecolor{24x6x2color53}{RGB}{255, 250, 250}
\definecolor{24x6x2color54}{RGB}{255, 255, 255}

\begin{center}
\begin{tikzpicture}[scale=0.5]
\begin{scope}
\draw[very thick] (0,-0) rectangle (24,-6);
\foreach \b in {0,...,5}
 \foreach \c in {0,...,23}
  \draw (\c,\b*-1-0) rectangle (\c+1,(\b*-1-1);
\draw[very thick] (0,-7) rectangle (24,-13);
\foreach \b in {0,...,5}
 \foreach \c in {0,...,23}
  \draw (\c,\b*-1-7) rectangle (\c+1,(\b*-1-8);

\draw[fill=red] (0,-0) rectangle (1,-1);
\draw[fill=24x6x2color31] (1,-0) rectangle (2,-1);
\node at (1.5,-0.5) {31};
\draw[fill=24x6x2color32] (2,-0) rectangle (3,-1);
\node at (2.5,-0.5) {32};
\draw[fill=24x6x2color33] (3,-0) rectangle (4,-1);
\node at (3.5,-0.5) {33};
\draw[fill=24x6x2color34] (4,-0) rectangle (5,-1);
\node at (4.5,-0.5) {34};
\draw[fill=24x6x2color35] (5,-0) rectangle (6,-1);
\node at (5.5,-0.5) {35};
\draw[fill=24x6x2color36] (6,-0) rectangle (7,-1);
\node at (6.5,-0.5) {36};
\draw[fill=24x6x2color37] (7,-0) rectangle (8,-1);
\node at (7.5,-0.5) {37};
\draw[fill=24x6x2color38] (8,-0) rectangle (9,-1);
\node at (8.5,-0.5) {38};
\draw[fill=24x6x2color39] (9,-0) rectangle (10,-1);
\node at (9.5,-0.5) {39};
\draw[fill=24x6x2color40] (10,-0) rectangle (11,-1);
\node at (10.5,-0.5) {40};
\draw[fill=24x6x2color41] (11,-0) rectangle (12,-1);
\node at (11.5,-0.5) {41};
\draw[fill=24x6x2color42] (12,-0) rectangle (13,-1);
\node at (12.5,-0.5) {42};
\draw[fill=24x6x2color43] (13,-0) rectangle (14,-1);
\node at (13.5,-0.5) {43};
\draw[fill=24x6x2color44] (14,-0) rectangle (15,-1);
\node at (14.5,-0.5) {44};
\draw[fill=24x6x2color45] (15,-0) rectangle (16,-1);
\node at (15.5,-0.5) {45};
\draw[fill=24x6x2color46] (16,-0) rectangle (17,-1);
\node at (16.5,-0.5) {46};
\draw[fill=24x6x2color47] (17,-0) rectangle (18,-1);
\node at (17.5,-0.5) {47};
\draw[fill=24x6x2color48] (18,-0) rectangle (19,-1);
\node at (18.5,-0.5) {48};
\draw[fill=24x6x2color49] (19,-0) rectangle (20,-1);
\node at (19.5,-0.5) {49};
\draw[fill=24x6x2color50] (20,-0) rectangle (21,-1);
\node at (20.5,-0.5) {50};
\draw[fill=24x6x2color51] (21,-0) rectangle (22,-1);
\node at (21.5,-0.5) {51};
\draw[fill=24x6x2color52] (22,-0) rectangle (23,-1);
\node at (22.5,-0.5) {52};
\draw[fill=24x6x2color53] (23,-0) rectangle (24,-1);
\node at (23.5,-0.5) {53};
\draw[fill=24x6x2color31] (0,-1) rectangle (1,-2);
\node at (0.5,-1.5) {31};
\draw[fill=24x6x2color30] (1,-1) rectangle (2,-2);
\node at (1.5,-1.5) {30};
\draw[fill=24x6x2color29] (2,-1) rectangle (3,-2);
\node at (2.5,-1.5) {29};
\draw[fill=24x6x2color28] (3,-1) rectangle (4,-2);
\node at (3.5,-1.5) {28};
\draw[fill=24x6x2color27] (4,-1) rectangle (5,-2);
\node at (4.5,-1.5) {27};
\draw[fill=24x6x2color26] (5,-1) rectangle (6,-2);
\node at (5.5,-1.5) {26};
\draw[fill=24x6x2color25] (6,-1) rectangle (7,-2);
\node at (6.5,-1.5) {25};
\draw[fill=24x6x2color24] (7,-1) rectangle (8,-2);
\node at (7.5,-1.5) {24};
\draw[fill=24x6x2color23] (8,-1) rectangle (9,-2);
\node at (8.5,-1.5) {23};
\draw[fill=24x6x2color22] (9,-1) rectangle (10,-2);
\node at (9.5,-1.5) {22};
\draw[fill=24x6x2color21] (10,-1) rectangle (11,-2);
\node at (10.5,-1.5) {21};
\draw[fill=24x6x2color20] (11,-1) rectangle (12,-2);
\node at (11.5,-1.5) {20};
\draw[fill=24x6x2color19] (12,-1) rectangle (13,-2);
\node at (12.5,-1.5) {19};
\draw[fill=24x6x2color18] (13,-1) rectangle (14,-2);
\node at (13.5,-1.5) {18};
\draw[fill=24x6x2color17] (14,-1) rectangle (15,-2);
\node at (14.5,-1.5) {17};
\draw[fill=24x6x2color16] (15,-1) rectangle (16,-2);
\node at (15.5,-1.5) {16};
\draw[fill=24x6x2color15] (16,-1) rectangle (17,-2);
\node at (16.5,-1.5) {15};
\draw[fill=24x6x2color14] (17,-1) rectangle (18,-2);
\node at (17.5,-1.5) {14};
\draw[fill=24x6x2color13] (18,-1) rectangle (19,-2);
\node at (18.5,-1.5) {13};
\draw[fill=24x6x2color12] (19,-1) rectangle (20,-2);
\node at (19.5,-1.5) {12};
\draw[fill=24x6x2color11] (20,-1) rectangle (21,-2);
\node at (20.5,-1.5) {11};
\draw[fill=24x6x2color10] (21,-1) rectangle (22,-2);
\node at (21.5,-1.5) {10};
\draw[fill=24x6x2color9] (22,-1) rectangle (23,-2);
\node at (22.5,-1.5) {9};
\draw[fill=red] (23,-1) rectangle (24,-2);
\draw[fill=24x6x2color32] (0,-2) rectangle (1,-3);
\node at (0.5,-2.5) {32};
\draw[fill=24x6x2color29] (1,-2) rectangle (2,-3);
\node at (1.5,-2.5) {29};
\draw[fill=24x6x2color28] (2,-2) rectangle (3,-3);
\node at (2.5,-2.5) {28};
\draw[fill=24x6x2color27] (3,-2) rectangle (4,-3);
\node at (3.5,-2.5) {27};
\draw[fill=24x6x2color26] (4,-2) rectangle (5,-3);
\node at (4.5,-2.5) {26};
\draw[fill=24x6x2color25] (5,-2) rectangle (6,-3);
\node at (5.5,-2.5) {25};
\draw[fill=24x6x2color24] (6,-2) rectangle (7,-3);
\node at (6.5,-2.5) {24};
\draw[fill=24x6x2color23] (7,-2) rectangle (8,-3);
\node at (7.5,-2.5) {23};
\draw[fill=24x6x2color22] (8,-2) rectangle (9,-3);
\node at (8.5,-2.5) {22};
\draw[fill=24x6x2color21] (9,-2) rectangle (10,-3);
\node at (9.5,-2.5) {21};
\draw[fill=24x6x2color20] (10,-2) rectangle (11,-3);
\node at (10.5,-2.5) {20};
\draw[fill=24x6x2color19] (11,-2) rectangle (12,-3);
\node at (11.5,-2.5) {19};
\draw[fill=24x6x2color18] (12,-2) rectangle (13,-3);
\node at (12.5,-2.5) {18};
\draw[fill=24x6x2color17] (13,-2) rectangle (14,-3);
\node at (13.5,-2.5) {17};
\draw[fill=24x6x2color16] (14,-2) rectangle (15,-3);
\node at (14.5,-2.5) {16};
\draw[fill=24x6x2color15] (15,-2) rectangle (16,-3);
\node at (15.5,-2.5) {15};
\draw[fill=24x6x2color14] (16,-2) rectangle (17,-3);
\node at (16.5,-2.5) {14};
\draw[fill=24x6x2color13] (17,-2) rectangle (18,-3);
\node at (17.5,-2.5) {13};
\draw[fill=24x6x2color12] (18,-2) rectangle (19,-3);
\node at (18.5,-2.5) {12};
\draw[fill=24x6x2color11] (19,-2) rectangle (20,-3);
\node at (19.5,-2.5) {11};
\draw[fill=24x6x2color10] (20,-2) rectangle (21,-3);
\node at (20.5,-2.5) {10};
\draw[fill=24x6x2color9] (21,-2) rectangle (22,-3);
\node at (21.5,-2.5) {9};
\draw[fill=24x6x2color8] (22,-2) rectangle (23,-3);
\node at (22.5,-2.5) {8};
\draw[fill=24x6x2color7] (23,-2) rectangle (24,-3);
\node at (23.5,-2.5) {7};
\draw[fill=24x6x2color33] (0,-3) rectangle (1,-4);
\node at (0.5,-3.5) {33};
\draw[fill=24x6x2color28] (1,-3) rectangle (2,-4);
\node at (1.5,-3.5) {28};
\draw[fill=24x6x2color19] (2,-3) rectangle (3,-4);
\node at (2.5,-3.5) {19};
\draw[fill=24x6x2color18] (3,-3) rectangle (4,-4);
\node at (3.5,-3.5) {18};
\draw[fill=24x6x2color17] (4,-3) rectangle (5,-4);
\node at (4.5,-3.5) {17};
\draw[fill=24x6x2color16] (5,-3) rectangle (6,-4);
\node at (5.5,-3.5) {16};
\draw[fill=24x6x2color15] (6,-3) rectangle (7,-4);
\node at (6.5,-3.5) {15};
\draw[fill=24x6x2color14] (7,-3) rectangle (8,-4);
\node at (7.5,-3.5) {14};
\draw[fill=24x6x2color13] (8,-3) rectangle (9,-4);
\node at (8.5,-3.5) {13};
\draw[fill=24x6x2color12] (9,-3) rectangle (10,-4);
\node at (9.5,-3.5) {12};
\draw[fill=24x6x2color11] (10,-3) rectangle (11,-4);
\node at (10.5,-3.5) {11};
\draw[fill=24x6x2color10] (11,-3) rectangle (12,-4);
\node at (11.5,-3.5) {10};
\draw[fill=24x6x2color9] (12,-3) rectangle (13,-4);
\node at (12.5,-3.5) {9};
\draw[fill=24x6x2color8] (13,-3) rectangle (14,-4);
\node at (13.5,-3.5) {8};
\draw[fill=24x6x2color7] (14,-3) rectangle (15,-4);
\node at (14.5,-3.5) {7};
\draw[fill=24x6x2color6] (15,-3) rectangle (16,-4);
\node at (15.5,-3.5) {6};
\draw[fill=24x6x2color5] (16,-3) rectangle (17,-4);
\node at (16.5,-3.5) {5};
\draw[fill=24x6x2color4] (17,-3) rectangle (18,-4);
\node at (17.5,-3.5) {4};
\draw[fill=24x6x2color3] (18,-3) rectangle (19,-4);
\node at (18.5,-3.5) {3};
\draw[fill=24x6x2color2] (19,-3) rectangle (20,-4);
\node at (19.5,-3.5) {2};
\draw[fill=24x6x2color1] (20,-3) rectangle (21,-4);
\node at (20.5,-3.5) {1};
\draw[fill=red] (21,-3) rectangle (22,-4);
\draw[fill=24x6x2color3] (22,-3) rectangle (23,-4);
\node at (22.5,-3.5) {3};
\draw[fill=red] (23,-3) rectangle (24,-4);
\draw[fill=24x6x2color34] (0,-4) rectangle (1,-5);
\node at (0.5,-4.5) {34};
\draw[fill=24x6x2color27] (1,-4) rectangle (2,-5);
\node at (1.5,-4.5) {27};
\draw[fill=red] (2,-4) rectangle (3,-5);
\draw[fill=24x6x2color1] (3,-4) rectangle (4,-5);
\node at (3.5,-4.5) {1};
\draw[fill=red] (4,-4) rectangle (5,-5);
\draw[fill=24x6x2color1] (5,-4) rectangle (6,-5);
\node at (5.5,-4.5) {1};
\draw[fill=red] (6,-4) rectangle (7,-5);
\draw[fill=24x6x2color1] (7,-4) rectangle (8,-5);
\node at (7.5,-4.5) {1};
\draw[fill=red] (8,-4) rectangle (9,-5);
\draw[fill=24x6x2color1] (9,-4) rectangle (10,-5);
\node at (9.5,-4.5) {1};
\draw[fill=red] (10,-4) rectangle (11,-5);
\draw[fill=24x6x2color1] (11,-4) rectangle (12,-5);
\node at (11.5,-4.5) {1};
\draw[fill=red] (12,-4) rectangle (13,-5);
\draw[fill=24x6x2color1] (13,-4) rectangle (14,-5);
\node at (13.5,-4.5) {1};
\draw[fill=red] (14,-4) rectangle (15,-5);
\draw[fill=24x6x2color1] (15,-4) rectangle (16,-5);
\node at (15.5,-4.5) {1};
\draw[fill=red] (16,-4) rectangle (17,-5);
\draw[fill=24x6x2color1] (17,-4) rectangle (18,-5);
\node at (17.5,-4.5) {1};
\draw[fill=red] (18,-4) rectangle (19,-5);
\draw[fill=24x6x2color1] (19,-4) rectangle (20,-5);
\node at (19.5,-4.5) {1};
\draw[fill=red] (20,-4) rectangle (21,-5);
\draw[fill=24x6x2color1] (21,-4) rectangle (22,-5);
\node at (21.5,-4.5) {1};
\draw[fill=24x6x2color2] (22,-4) rectangle (23,-5);
\node at (22.5,-4.5) {2};
\draw[fill=24x6x2color1] (23,-4) rectangle (24,-5);
\node at (23.5,-4.5) {1};
\draw[fill=24x6x2color35] (0,-5) rectangle (1,-6);
\node at (0.5,-5.5) {35};
\draw[fill=red] (1,-5) rectangle (2,-6);
\draw[fill=24x6x2color1] (2,-5) rectangle (3,-6);
\node at (2.5,-5.5) {1};
\draw[fill=red] (3,-5) rectangle (4,-6);
\draw[fill=24x6x2color1] (4,-5) rectangle (5,-6);
\node at (4.5,-5.5) {1};
\draw[fill=red] (5,-5) rectangle (6,-6);
\draw[fill=24x6x2color1] (6,-5) rectangle (7,-6);
\node at (6.5,-5.5) {1};
\draw[fill=red] (7,-5) rectangle (8,-6);
\draw[fill=24x6x2color1] (8,-5) rectangle (9,-6);
\node at (8.5,-5.5) {1};
\draw[fill=red] (9,-5) rectangle (10,-6);
\draw[fill=24x6x2color1] (10,-5) rectangle (11,-6);
\node at (10.5,-5.5) {1};
\draw[fill=red] (11,-5) rectangle (12,-6);
\draw[fill=24x6x2color1] (12,-5) rectangle (13,-6);
\node at (12.5,-5.5) {1};
\draw[fill=red] (13,-5) rectangle (14,-6);
\draw[fill=24x6x2color1] (14,-5) rectangle (15,-6);
\node at (14.5,-5.5) {1};
\draw[fill=red] (15,-5) rectangle (16,-6);
\draw[fill=24x6x2color1] (16,-5) rectangle (17,-6);
\node at (16.5,-5.5) {1};
\draw[fill=red] (17,-5) rectangle (18,-6);
\draw[fill=24x6x2color1] (18,-5) rectangle (19,-6);
\node at (18.5,-5.5) {1};
\draw[fill=red] (19,-5) rectangle (20,-6);
\draw[fill=24x6x2color1] (20,-5) rectangle (21,-6);
\node at (20.5,-5.5) {1};
\draw[fill=red] (21,-5) rectangle (22,-6);
\draw[fill=24x6x2color1] (22,-5) rectangle (23,-6);
\node at (22.5,-5.5) {1};
\draw[fill=red] (23,-5) rectangle (24,-6);
\draw[fill=24x6x2color1] (0,-7) rectangle (1,-8);
\node at (0.5,-7.5) {1};
\draw[fill=red] (1,-7) rectangle (2,-8);
\draw[fill=24x6x2color3] (2,-7) rectangle (3,-8);
\node at (2.5,-7.5) {3};
\draw[fill=red] (3,-7) rectangle (4,-8);
\draw[fill=24x6x2color1] (4,-7) rectangle (5,-8);
\node at (4.5,-7.5) {1};
\draw[fill=red] (5,-7) rectangle (6,-8);
\draw[fill=24x6x2color1] (6,-7) rectangle (7,-8);
\node at (6.5,-7.5) {1};
\draw[fill=red] (7,-7) rectangle (8,-8);
\draw[fill=24x6x2color3] (8,-7) rectangle (9,-8);
\node at (8.5,-7.5) {3};
\draw[fill=red] (9,-7) rectangle (10,-8);
\draw[fill=24x6x2color1] (10,-7) rectangle (11,-8);
\node at (10.5,-7.5) {1};
\draw[fill=red] (11,-7) rectangle (12,-8);
\draw[fill=24x6x2color1] (12,-7) rectangle (13,-8);
\node at (12.5,-7.5) {1};
\draw[fill=red] (13,-7) rectangle (14,-8);
\draw[fill=24x6x2color3] (14,-7) rectangle (15,-8);
\node at (14.5,-7.5) {3};
\draw[fill=red] (15,-7) rectangle (16,-8);
\draw[fill=24x6x2color1] (16,-7) rectangle (17,-8);
\node at (16.5,-7.5) {1};
\draw[fill=red] (17,-7) rectangle (18,-8);
\draw[fill=24x6x2color1] (18,-7) rectangle (19,-8);
\node at (18.5,-7.5) {1};
\draw[fill=red] (19,-7) rectangle (20,-8);
\draw[fill=24x6x2color3] (20,-7) rectangle (21,-8);
\node at (20.5,-7.5) {3};
\draw[fill=red] (21,-7) rectangle (22,-8);
\draw[fill=24x6x2color1] (22,-7) rectangle (23,-8);
\node at (22.5,-7.5) {1};
\draw[fill=red] (23,-7) rectangle (24,-8);
\draw[fill=red] (0,-8) rectangle (1,-9);
\draw[fill=24x6x2color1] (1,-8) rectangle (2,-9);
\node at (1.5,-8.5) {1};
\draw[fill=24x6x2color2] (2,-8) rectangle (3,-9);
\node at (2.5,-8.5) {2};
\draw[fill=24x6x2color1] (3,-8) rectangle (4,-9);
\node at (3.5,-8.5) {1};
\draw[fill=red] (4,-8) rectangle (5,-9);
\draw[fill=24x6x2color1] (5,-8) rectangle (6,-9);
\node at (5.5,-8.5) {1};
\draw[fill=red] (6,-8) rectangle (7,-9);
\draw[fill=24x6x2color1] (7,-8) rectangle (8,-9);
\node at (7.5,-8.5) {1};
\draw[fill=24x6x2color2] (8,-8) rectangle (9,-9);
\node at (8.5,-8.5) {2};
\draw[fill=24x6x2color1] (9,-8) rectangle (10,-9);
\node at (9.5,-8.5) {1};
\draw[fill=red] (10,-8) rectangle (11,-9);
\draw[fill=24x6x2color1] (11,-8) rectangle (12,-9);
\node at (11.5,-8.5) {1};
\draw[fill=red] (12,-8) rectangle (13,-9);
\draw[fill=24x6x2color1] (13,-8) rectangle (14,-9);
\node at (13.5,-8.5) {1};
\draw[fill=24x6x2color2] (14,-8) rectangle (15,-9);
\node at (14.5,-8.5) {2};
\draw[fill=24x6x2color1] (15,-8) rectangle (16,-9);
\node at (15.5,-8.5) {1};
\draw[fill=red] (16,-8) rectangle (17,-9);
\draw[fill=24x6x2color1] (17,-8) rectangle (18,-9);
\node at (17.5,-8.5) {1};
\draw[fill=red] (18,-8) rectangle (19,-9);
\draw[fill=24x6x2color1] (19,-8) rectangle (20,-9);
\node at (19.5,-8.5) {1};
\draw[fill=24x6x2color2] (20,-8) rectangle (21,-9);
\node at (20.5,-8.5) {2};
\draw[fill=24x6x2color1] (21,-8) rectangle (22,-9);
\node at (21.5,-8.5) {1};
\draw[fill=red] (22,-8) rectangle (23,-9);
\draw[fill=24x6x2color1] (23,-8) rectangle (24,-9);
\node at (23.5,-8.5) {1};
\draw[fill=24x6x2color1] (0,-9) rectangle (1,-10);
\node at (0.5,-9.5) {1};
\draw[fill=red] (1,-9) rectangle (2,-10);
\draw[fill=24x6x2color1] (2,-9) rectangle (3,-10);
\node at (2.5,-9.5) {1};
\draw[fill=red] (3,-9) rectangle (4,-10);
\draw[fill=24x6x2color1] (4,-9) rectangle (5,-10);
\node at (4.5,-9.5) {1};
\draw[fill=24x6x2color2] (5,-9) rectangle (6,-10);
\node at (5.5,-9.5) {2};
\draw[fill=24x6x2color1] (6,-9) rectangle (7,-10);
\node at (6.5,-9.5) {1};
\draw[fill=red] (7,-9) rectangle (8,-10);
\draw[fill=24x6x2color1] (8,-9) rectangle (9,-10);
\node at (8.5,-9.5) {1};
\draw[fill=red] (9,-9) rectangle (10,-10);
\draw[fill=24x6x2color1] (10,-9) rectangle (11,-10);
\node at (10.5,-9.5) {1};
\draw[fill=24x6x2color2] (11,-9) rectangle (12,-10);
\node at (11.5,-9.5) {2};
\draw[fill=24x6x2color1] (12,-9) rectangle (13,-10);
\node at (12.5,-9.5) {1};
\draw[fill=red] (13,-9) rectangle (14,-10);
\draw[fill=24x6x2color1] (14,-9) rectangle (15,-10);
\node at (14.5,-9.5) {1};
\draw[fill=red] (15,-9) rectangle (16,-10);
\draw[fill=24x6x2color1] (16,-9) rectangle (17,-10);
\node at (16.5,-9.5) {1};
\draw[fill=24x6x2color2] (17,-9) rectangle (18,-10);
\node at (17.5,-9.5) {2};
\draw[fill=24x6x2color1] (18,-9) rectangle (19,-10);
\node at (18.5,-9.5) {1};
\draw[fill=red] (19,-9) rectangle (20,-10);
\draw[fill=24x6x2color1] (20,-9) rectangle (21,-10);
\node at (20.5,-9.5) {1};
\draw[fill=red] (21,-9) rectangle (22,-10);
\draw[fill=24x6x2color5] (22,-9) rectangle (23,-10);
\node at (22.5,-9.5) {5};
\draw[fill=24x6x2color6] (23,-9) rectangle (24,-10);
\node at (23.5,-9.5) {6};
\draw[fill=red] (0,-10) rectangle (1,-11);
\draw[fill=24x6x2color1] (1,-10) rectangle (2,-11);
\node at (1.5,-10.5) {1};
\draw[fill=red] (2,-10) rectangle (3,-11);
\draw[fill=24x6x2color1] (3,-10) rectangle (4,-11);
\node at (3.5,-10.5) {1};
\draw[fill=red] (4,-10) rectangle (5,-11);
\draw[fill=24x6x2color3] (5,-10) rectangle (6,-11);
\node at (5.5,-10.5) {3};
\draw[fill=red] (6,-10) rectangle (7,-11);
\draw[fill=24x6x2color1] (7,-10) rectangle (8,-11);
\node at (7.5,-10.5) {1};
\draw[fill=red] (8,-10) rectangle (9,-11);
\draw[fill=24x6x2color1] (9,-10) rectangle (10,-11);
\node at (9.5,-10.5) {1};
\draw[fill=red] (10,-10) rectangle (11,-11);
\draw[fill=24x6x2color3] (11,-10) rectangle (12,-11);
\node at (11.5,-10.5) {3};
\draw[fill=red] (12,-10) rectangle (13,-11);
\draw[fill=24x6x2color1] (13,-10) rectangle (14,-11);
\node at (13.5,-10.5) {1};
\draw[fill=red] (14,-10) rectangle (15,-11);
\draw[fill=24x6x2color1] (15,-10) rectangle (16,-11);
\node at (15.5,-10.5) {1};
\draw[fill=red] (16,-10) rectangle (17,-11);
\draw[fill=24x6x2color3] (17,-10) rectangle (18,-11);
\node at (17.5,-10.5) {3};
\draw[fill=red] (18,-10) rectangle (19,-11);
\draw[fill=24x6x2color1] (19,-10) rectangle (20,-11);
\node at (19.5,-10.5) {1};
\draw[fill=red] (20,-10) rectangle (21,-11);
\draw[fill=24x6x2color1] (21,-10) rectangle (22,-11);
\node at (21.5,-10.5) {1};
\draw[fill=24x6x2color4] (22,-10) rectangle (23,-11);
\node at (22.5,-10.5) {4};
\draw[fill=24x6x2color5] (23,-10) rectangle (24,-11);
\node at (23.5,-10.5) {5};
\draw[fill=24x6x2color27] (0,-11) rectangle (1,-12);
\node at (0.5,-11.5) {27};
\draw[fill=24x6x2color26] (1,-11) rectangle (2,-12);
\node at (1.5,-11.5) {26};
\draw[fill=24x6x2color23] (2,-11) rectangle (3,-12);
\node at (2.5,-11.5) {23};
\draw[fill=24x6x2color22] (3,-11) rectangle (4,-12);
\node at (3.5,-11.5) {22};
\draw[fill=24x6x2color21] (4,-11) rectangle (5,-12);
\node at (4.5,-11.5) {21};
\draw[fill=24x6x2color20] (5,-11) rectangle (6,-12);
\node at (5.5,-11.5) {20};
\draw[fill=24x6x2color19] (6,-11) rectangle (7,-12);
\node at (6.5,-11.5) {19};
\draw[fill=24x6x2color18] (7,-11) rectangle (8,-12);
\node at (7.5,-11.5) {18};
\draw[fill=24x6x2color17] (8,-11) rectangle (9,-12);
\node at (8.5,-11.5) {17};
\draw[fill=24x6x2color16] (9,-11) rectangle (10,-12);
\node at (9.5,-11.5) {16};
\draw[fill=24x6x2color15] (10,-11) rectangle (11,-12);
\node at (10.5,-11.5) {15};
\draw[fill=24x6x2color14] (11,-11) rectangle (12,-12);
\node at (11.5,-11.5) {14};
\draw[fill=24x6x2color13] (12,-11) rectangle (13,-12);
\node at (12.5,-11.5) {13};
\draw[fill=24x6x2color12] (13,-11) rectangle (14,-12);
\node at (13.5,-11.5) {12};
\draw[fill=24x6x2color11] (14,-11) rectangle (15,-12);
\node at (14.5,-11.5) {11};
\draw[fill=24x6x2color10] (15,-11) rectangle (16,-12);
\node at (15.5,-11.5) {10};
\draw[fill=24x6x2color9] (16,-11) rectangle (17,-12);
\node at (16.5,-11.5) {9};
\draw[fill=24x6x2color8] (17,-11) rectangle (18,-12);
\node at (17.5,-11.5) {8};
\draw[fill=24x6x2color7] (18,-11) rectangle (19,-12);
\node at (18.5,-11.5) {7};
\draw[fill=24x6x2color6] (19,-11) rectangle (20,-12);
\node at (19.5,-11.5) {6};
\draw[fill=24x6x2color5] (20,-11) rectangle (21,-12);
\node at (20.5,-11.5) {5};
\draw[fill=24x6x2color4] (21,-11) rectangle (22,-12);
\node at (21.5,-11.5) {4};
\draw[fill=24x6x2color3] (22,-11) rectangle (23,-12);
\node at (22.5,-11.5) {3};
\draw[fill=red] (23,-11) rectangle (24,-12);
\draw[fill=red] (0,-12) rectangle (1,-13);
\draw[fill=24x6x2color25] (1,-12) rectangle (2,-13);
\node at (1.5,-12.5) {25};
\draw[fill=24x6x2color24] (2,-12) rectangle (3,-13);
\node at (2.5,-12.5) {24};
\draw[fill=24x6x2color23] (3,-12) rectangle (4,-13);
\node at (3.5,-12.5) {23};
\draw[fill=24x6x2color22] (4,-12) rectangle (5,-13);
\node at (4.5,-12.5) {22};
\draw[fill=24x6x2color21] (5,-12) rectangle (6,-13);
\node at (5.5,-12.5) {21};
\draw[fill=24x6x2color20] (6,-12) rectangle (7,-13);
\node at (6.5,-12.5) {20};
\draw[fill=24x6x2color19] (7,-12) rectangle (8,-13);
\node at (7.5,-12.5) {19};
\draw[fill=24x6x2color18] (8,-12) rectangle (9,-13);
\node at (8.5,-12.5) {18};
\draw[fill=24x6x2color17] (9,-12) rectangle (10,-13);
\node at (9.5,-12.5) {17};
\draw[fill=24x6x2color16] (10,-12) rectangle (11,-13);
\node at (10.5,-12.5) {16};
\draw[fill=24x6x2color15] (11,-12) rectangle (12,-13);
\node at (11.5,-12.5) {15};
\draw[fill=24x6x2color14] (12,-12) rectangle (13,-13);
\node at (12.5,-12.5) {14};
\draw[fill=24x6x2color13] (13,-12) rectangle (14,-13);
\node at (13.5,-12.5) {13};
\draw[fill=24x6x2color12] (14,-12) rectangle (15,-13);
\node at (14.5,-12.5) {12};
\draw[fill=24x6x2color11] (15,-12) rectangle (16,-13);
\node at (15.5,-12.5) {11};
\draw[fill=24x6x2color10] (16,-12) rectangle (17,-13);
\node at (16.5,-12.5) {10};
\draw[fill=24x6x2color9] (17,-12) rectangle (18,-13);
\node at (17.5,-12.5) {9};
\draw[fill=24x6x2color8] (18,-12) rectangle (19,-13);
\node at (18.5,-12.5) {8};
\draw[fill=24x6x2color7] (19,-12) rectangle (20,-13);
\node at (19.5,-12.5) {7};
\draw[fill=24x6x2color6] (20,-12) rectangle (21,-13);
\node at (20.5,-12.5) {6};
\draw[fill=24x6x2color5] (21,-12) rectangle (22,-13);
\node at (21.5,-12.5) {5};
\draw[fill=red] (22,-12) rectangle (23,-13);
\draw[fill=24x6x2color1] (23,-12) rectangle (24,-13);
\node at (23.5,-12.5) {1};

\end{scope}
\end{tikzpicture}
\end{center}

%% file: 2-8-26-tikz.txt
\definecolor{26x8x2color1}{RGB}{255, 128, 128}
\definecolor{26x8x2color2}{RGB}{255, 130, 130}
\definecolor{26x8x2color3}{RGB}{255, 132, 132}
\definecolor{26x8x2color4}{RGB}{255, 134, 134}
\definecolor{26x8x2color5}{RGB}{255, 136, 136}
\definecolor{26x8x2color6}{RGB}{255, 138, 138}
\definecolor{26x8x2color7}{RGB}{255, 141, 141}
\definecolor{26x8x2color8}{RGB}{255, 143, 143}
\definecolor{26x8x2color9}{RGB}{255, 145, 145}
\definecolor{26x8x2color10}{RGB}{255, 147, 147}
\definecolor{26x8x2color11}{RGB}{255, 149, 149}
\definecolor{26x8x2color12}{RGB}{255, 152, 152}
\definecolor{26x8x2color13}{RGB}{255, 154, 154}
\definecolor{26x8x2color14}{RGB}{255, 156, 156}
\definecolor{26x8x2color15}{RGB}{255, 158, 158}
\definecolor{26x8x2color16}{RGB}{255, 160, 160}
\definecolor{26x8x2color17}{RGB}{255, 163, 163}
\definecolor{26x8x2color18}{RGB}{255, 165, 165}
\definecolor{26x8x2color19}{RGB}{255, 167, 167}
\definecolor{26x8x2color20}{RGB}{255, 169, 169}
\definecolor{26x8x2color21}{RGB}{255, 171, 171}
\definecolor{26x8x2color22}{RGB}{255, 173, 173}
\definecolor{26x8x2color23}{RGB}{255, 176, 176}
\definecolor{26x8x2color24}{RGB}{255, 178, 178}
\definecolor{26x8x2color25}{RGB}{255, 180, 180}
\definecolor{26x8x2color26}{RGB}{255, 182, 182}
\definecolor{26x8x2color27}{RGB}{255, 184, 184}
\definecolor{26x8x2color28}{RGB}{255, 187, 187}
\definecolor{26x8x2color29}{RGB}{255, 189, 189}
\definecolor{26x8x2color30}{RGB}{255, 191, 191}
\definecolor{26x8x2color31}{RGB}{255, 193, 193}
\definecolor{26x8x2color32}{RGB}{255, 195, 195}
\definecolor{26x8x2color33}{RGB}{255, 198, 198}
\definecolor{26x8x2color34}{RGB}{255, 200, 200}
\definecolor{26x8x2color35}{RGB}{255, 202, 202}
\definecolor{26x8x2color36}{RGB}{255, 204, 204}
\definecolor{26x8x2color37}{RGB}{255, 206, 206}
\definecolor{26x8x2color38}{RGB}{255, 209, 209}
\definecolor{26x8x2color39}{RGB}{255, 211, 211}
\definecolor{26x8x2color40}{RGB}{255, 213, 213}
\definecolor{26x8x2color41}{RGB}{255, 215, 215}
\definecolor{26x8x2color42}{RGB}{255, 217, 217}
\definecolor{26x8x2color43}{RGB}{255, 219, 219}
\definecolor{26x8x2color44}{RGB}{255, 222, 222}
\definecolor{26x8x2color45}{RGB}{255, 224, 224}
\definecolor{26x8x2color46}{RGB}{255, 226, 226}
\definecolor{26x8x2color47}{RGB}{255, 228, 228}
\definecolor{26x8x2color48}{RGB}{255, 230, 230}
\definecolor{26x8x2color49}{RGB}{255, 233, 233}
\definecolor{26x8x2color50}{RGB}{255, 235, 235}
\definecolor{26x8x2color51}{RGB}{255, 237, 237}
\definecolor{26x8x2color52}{RGB}{255, 239, 239}
\definecolor{26x8x2color53}{RGB}{255, 241, 241}
\definecolor{26x8x2color54}{RGB}{255, 244, 244}
\definecolor{26x8x2color55}{RGB}{255, 246, 246}
\definecolor{26x8x2color56}{RGB}{255, 248, 248}
\definecolor{26x8x2color57}{RGB}{255, 250, 250}
\definecolor{26x8x2color58}{RGB}{255, 255, 255}

\begin{center}
\begin{tikzpicture}[scale=0.5]
\begin{scope}
\draw[very thick] (0,-0) rectangle (26,-8);
\foreach \b in {0,...,7}
 \foreach \c in {0,...,25}
  \draw (\c,\b*-1-0) rectangle (\c+1,(\b*-1-1);
\draw[very thick] (0,-9) rectangle (26,-17);
\foreach \b in {0,...,7}
 \foreach \c in {0,...,25}
  \draw (\c,\b*-1-9) rectangle (\c+1,(\b*-1-10);

\draw[fill=red] (0,-0) rectangle (1,-1);
\draw[fill=26x8x2color33] (1,-0) rectangle (2,-1);
\node at (1.5,-0.5) {33};
\draw[fill=26x8x2color34] (2,-0) rectangle (3,-1);
\node at (2.5,-0.5) {34};
\draw[fill=26x8x2color35] (3,-0) rectangle (4,-1);
\node at (3.5,-0.5) {35};
\draw[fill=26x8x2color36] (4,-0) rectangle (5,-1);
\node at (4.5,-0.5) {36};
\draw[fill=26x8x2color37] (5,-0) rectangle (6,-1);
\node at (5.5,-0.5) {37};
\draw[fill=26x8x2color38] (6,-0) rectangle (7,-1);
\node at (6.5,-0.5) {38};
\draw[fill=26x8x2color39] (7,-0) rectangle (8,-1);
\node at (7.5,-0.5) {39};
\draw[fill=26x8x2color40] (8,-0) rectangle (9,-1);
\node at (8.5,-0.5) {40};
\draw[fill=26x8x2color41] (9,-0) rectangle (10,-1);
\node at (9.5,-0.5) {41};
\draw[fill=26x8x2color42] (10,-0) rectangle (11,-1);
\node at (10.5,-0.5) {42};
\draw[fill=26x8x2color43] (11,-0) rectangle (12,-1);
\node at (11.5,-0.5) {43};
\draw[fill=26x8x2color44] (12,-0) rectangle (13,-1);
\node at (12.5,-0.5) {44};
\draw[fill=26x8x2color45] (13,-0) rectangle (14,-1);
\node at (13.5,-0.5) {45};
\draw[fill=26x8x2color46] (14,-0) rectangle (15,-1);
\node at (14.5,-0.5) {46};
\draw[fill=26x8x2color47] (15,-0) rectangle (16,-1);
\node at (15.5,-0.5) {47};
\draw[fill=26x8x2color48] (16,-0) rectangle (17,-1);
\node at (16.5,-0.5) {48};
\draw[fill=26x8x2color49] (17,-0) rectangle (18,-1);
\node at (17.5,-0.5) {49};
\draw[fill=26x8x2color50] (18,-0) rectangle (19,-1);
\node at (18.5,-0.5) {50};
\draw[fill=26x8x2color51] (19,-0) rectangle (20,-1);
\node at (19.5,-0.5) {51};
\draw[fill=26x8x2color52] (20,-0) rectangle (21,-1);
\node at (20.5,-0.5) {52};
\draw[fill=26x8x2color53] (21,-0) rectangle (22,-1);
\node at (21.5,-0.5) {53};
\draw[fill=26x8x2color54] (22,-0) rectangle (23,-1);
\node at (22.5,-0.5) {54};
\draw[fill=26x8x2color55] (23,-0) rectangle (24,-1);
\node at (23.5,-0.5) {55};
\draw[fill=26x8x2color56] (24,-0) rectangle (25,-1);
\node at (24.5,-0.5) {56};
\draw[fill=26x8x2color57] (25,-0) rectangle (26,-1);
\node at (25.5,-0.5) {57};
\draw[fill=26x8x2color33] (0,-1) rectangle (1,-2);
\node at (0.5,-1.5) {33};
\draw[fill=26x8x2color32] (1,-1) rectangle (2,-2);
\node at (1.5,-1.5) {32};
\draw[fill=26x8x2color31] (2,-1) rectangle (3,-2);
\node at (2.5,-1.5) {31};
\draw[fill=26x8x2color30] (3,-1) rectangle (4,-2);
\node at (3.5,-1.5) {30};
\draw[fill=26x8x2color29] (4,-1) rectangle (5,-2);
\node at (4.5,-1.5) {29};
\draw[fill=26x8x2color28] (5,-1) rectangle (6,-2);
\node at (5.5,-1.5) {28};
\draw[fill=26x8x2color27] (6,-1) rectangle (7,-2);
\node at (6.5,-1.5) {27};
\draw[fill=26x8x2color26] (7,-1) rectangle (8,-2);
\node at (7.5,-1.5) {26};
\draw[fill=26x8x2color25] (8,-1) rectangle (9,-2);
\node at (8.5,-1.5) {25};
\draw[fill=26x8x2color24] (9,-1) rectangle (10,-2);
\node at (9.5,-1.5) {24};
\draw[fill=26x8x2color23] (10,-1) rectangle (11,-2);
\node at (10.5,-1.5) {23};
\draw[fill=26x8x2color22] (11,-1) rectangle (12,-2);
\node at (11.5,-1.5) {22};
\draw[fill=26x8x2color21] (12,-1) rectangle (13,-2);
\node at (12.5,-1.5) {21};
\draw[fill=26x8x2color20] (13,-1) rectangle (14,-2);
\node at (13.5,-1.5) {20};
\draw[fill=26x8x2color19] (14,-1) rectangle (15,-2);
\node at (14.5,-1.5) {19};
\draw[fill=26x8x2color18] (15,-1) rectangle (16,-2);
\node at (15.5,-1.5) {18};
\draw[fill=26x8x2color17] (16,-1) rectangle (17,-2);
\node at (16.5,-1.5) {17};
\draw[fill=26x8x2color16] (17,-1) rectangle (18,-2);
\node at (17.5,-1.5) {16};
\draw[fill=26x8x2color15] (18,-1) rectangle (19,-2);
\node at (18.5,-1.5) {15};
\draw[fill=26x8x2color14] (19,-1) rectangle (20,-2);
\node at (19.5,-1.5) {14};
\draw[fill=26x8x2color13] (20,-1) rectangle (21,-2);
\node at (20.5,-1.5) {13};
\draw[fill=26x8x2color12] (21,-1) rectangle (22,-2);
\node at (21.5,-1.5) {12};
\draw[fill=26x8x2color11] (22,-1) rectangle (23,-2);
\node at (22.5,-1.5) {11};
\draw[fill=26x8x2color10] (23,-1) rectangle (24,-2);
\node at (23.5,-1.5) {10};
\draw[fill=26x8x2color9] (24,-1) rectangle (25,-2);
\node at (24.5,-1.5) {9};
\draw[fill=red] (25,-1) rectangle (26,-2);
\draw[fill=26x8x2color34] (0,-2) rectangle (1,-3);
\node at (0.5,-2.5) {34};
\draw[fill=26x8x2color31] (1,-2) rectangle (2,-3);
\node at (1.5,-2.5) {31};
\draw[fill=26x8x2color30] (2,-2) rectangle (3,-3);
\node at (2.5,-2.5) {30};
\draw[fill=26x8x2color29] (3,-2) rectangle (4,-3);
\node at (3.5,-2.5) {29};
\draw[fill=26x8x2color28] (4,-2) rectangle (5,-3);
\node at (4.5,-2.5) {28};
\draw[fill=26x8x2color27] (5,-2) rectangle (6,-3);
\node at (5.5,-2.5) {27};
\draw[fill=26x8x2color26] (6,-2) rectangle (7,-3);
\node at (6.5,-2.5) {26};
\draw[fill=26x8x2color25] (7,-2) rectangle (8,-3);
\node at (7.5,-2.5) {25};
\draw[fill=26x8x2color24] (8,-2) rectangle (9,-3);
\node at (8.5,-2.5) {24};
\draw[fill=26x8x2color23] (9,-2) rectangle (10,-3);
\node at (9.5,-2.5) {23};
\draw[fill=26x8x2color22] (10,-2) rectangle (11,-3);
\node at (10.5,-2.5) {22};
\draw[fill=26x8x2color21] (11,-2) rectangle (12,-3);
\node at (11.5,-2.5) {21};
\draw[fill=26x8x2color20] (12,-2) rectangle (13,-3);
\node at (12.5,-2.5) {20};
\draw[fill=26x8x2color19] (13,-2) rectangle (14,-3);
\node at (13.5,-2.5) {19};
\draw[fill=26x8x2color18] (14,-2) rectangle (15,-3);
\node at (14.5,-2.5) {18};
\draw[fill=26x8x2color17] (15,-2) rectangle (16,-3);
\node at (15.5,-2.5) {17};
\draw[fill=26x8x2color16] (16,-2) rectangle (17,-3);
\node at (16.5,-2.5) {16};
\draw[fill=26x8x2color15] (17,-2) rectangle (18,-3);
\node at (17.5,-2.5) {15};
\draw[fill=26x8x2color14] (18,-2) rectangle (19,-3);
\node at (18.5,-2.5) {14};
\draw[fill=26x8x2color13] (19,-2) rectangle (20,-3);
\node at (19.5,-2.5) {13};
\draw[fill=26x8x2color12] (20,-2) rectangle (21,-3);
\node at (20.5,-2.5) {12};
\draw[fill=26x8x2color11] (21,-2) rectangle (22,-3);
\node at (21.5,-2.5) {11};
\draw[fill=26x8x2color10] (22,-2) rectangle (23,-3);
\node at (22.5,-2.5) {10};
\draw[fill=26x8x2color9] (23,-2) rectangle (24,-3);
\node at (23.5,-2.5) {9};
\draw[fill=26x8x2color8] (24,-2) rectangle (25,-3);
\node at (24.5,-2.5) {8};
\draw[fill=26x8x2color7] (25,-2) rectangle (26,-3);
\node at (25.5,-2.5) {7};
\draw[fill=26x8x2color35] (0,-3) rectangle (1,-4);
\node at (0.5,-3.5) {35};
\draw[fill=26x8x2color30] (1,-3) rectangle (2,-4);
\node at (1.5,-3.5) {30};
\draw[fill=26x8x2color29] (2,-3) rectangle (3,-4);
\node at (2.5,-3.5) {29};
\draw[fill=26x8x2color28] (3,-3) rectangle (4,-4);
\node at (3.5,-3.5) {28};
\draw[fill=26x8x2color27] (4,-3) rectangle (5,-4);
\node at (4.5,-3.5) {27};
\draw[fill=26x8x2color26] (5,-3) rectangle (6,-4);
\node at (5.5,-3.5) {26};
\draw[fill=26x8x2color25] (6,-3) rectangle (7,-4);
\node at (6.5,-3.5) {25};
\draw[fill=26x8x2color24] (7,-3) rectangle (8,-4);
\node at (7.5,-3.5) {24};
\draw[fill=26x8x2color23] (8,-3) rectangle (9,-4);
\node at (8.5,-3.5) {23};
\draw[fill=26x8x2color22] (9,-3) rectangle (10,-4);
\node at (9.5,-3.5) {22};
\draw[fill=26x8x2color21] (10,-3) rectangle (11,-4);
\node at (10.5,-3.5) {21};
\draw[fill=26x8x2color20] (11,-3) rectangle (12,-4);
\node at (11.5,-3.5) {20};
\draw[fill=26x8x2color19] (12,-3) rectangle (13,-4);
\node at (12.5,-3.5) {19};
\draw[fill=26x8x2color18] (13,-3) rectangle (14,-4);
\node at (13.5,-3.5) {18};
\draw[fill=26x8x2color17] (14,-3) rectangle (15,-4);
\node at (14.5,-3.5) {17};
\draw[fill=26x8x2color16] (15,-3) rectangle (16,-4);
\node at (15.5,-3.5) {16};
\draw[fill=26x8x2color15] (16,-3) rectangle (17,-4);
\node at (16.5,-3.5) {15};
\draw[fill=26x8x2color14] (17,-3) rectangle (18,-4);
\node at (17.5,-3.5) {14};
\draw[fill=26x8x2color13] (18,-3) rectangle (19,-4);
\node at (18.5,-3.5) {13};
\draw[fill=26x8x2color12] (19,-3) rectangle (20,-4);
\node at (19.5,-3.5) {12};
\draw[fill=26x8x2color11] (20,-3) rectangle (21,-4);
\node at (20.5,-3.5) {11};
\draw[fill=26x8x2color10] (21,-3) rectangle (22,-4);
\node at (21.5,-3.5) {10};
\draw[fill=26x8x2color9] (22,-3) rectangle (23,-4);
\node at (22.5,-3.5) {9};
\draw[fill=26x8x2color8] (23,-3) rectangle (24,-4);
\node at (23.5,-3.5) {8};
\draw[fill=26x8x2color7] (24,-3) rectangle (25,-4);
\node at (24.5,-3.5) {7};
\draw[fill=red] (25,-3) rectangle (26,-4);
\draw[fill=26x8x2color36] (0,-4) rectangle (1,-5);
\node at (0.5,-4.5) {36};
\draw[fill=26x8x2color29] (1,-4) rectangle (2,-5);
\node at (1.5,-4.5) {29};
\draw[fill=26x8x2color28] (2,-4) rectangle (3,-5);
\node at (2.5,-4.5) {28};
\draw[fill=26x8x2color27] (3,-4) rectangle (4,-5);
\node at (3.5,-4.5) {27};
\draw[fill=26x8x2color26] (4,-4) rectangle (5,-5);
\node at (4.5,-4.5) {26};
\draw[fill=26x8x2color25] (5,-4) rectangle (6,-5);
\node at (5.5,-4.5) {25};
\draw[fill=26x8x2color24] (6,-4) rectangle (7,-5);
\node at (6.5,-4.5) {24};
\draw[fill=26x8x2color23] (7,-4) rectangle (8,-5);
\node at (7.5,-4.5) {23};
\draw[fill=26x8x2color22] (8,-4) rectangle (9,-5);
\node at (8.5,-4.5) {22};
\draw[fill=26x8x2color21] (9,-4) rectangle (10,-5);
\node at (9.5,-4.5) {21};
\draw[fill=26x8x2color20] (10,-4) rectangle (11,-5);
\node at (10.5,-4.5) {20};
\draw[fill=26x8x2color19] (11,-4) rectangle (12,-5);
\node at (11.5,-4.5) {19};
\draw[fill=26x8x2color18] (12,-4) rectangle (13,-5);
\node at (12.5,-4.5) {18};
\draw[fill=26x8x2color17] (13,-4) rectangle (14,-5);
\node at (13.5,-4.5) {17};
\draw[fill=26x8x2color16] (14,-4) rectangle (15,-5);
\node at (14.5,-4.5) {16};
\draw[fill=26x8x2color15] (15,-4) rectangle (16,-5);
\node at (15.5,-4.5) {15};
\draw[fill=26x8x2color14] (16,-4) rectangle (17,-5);
\node at (16.5,-4.5) {14};
\draw[fill=26x8x2color13] (17,-4) rectangle (18,-5);
\node at (17.5,-4.5) {13};
\draw[fill=26x8x2color12] (18,-4) rectangle (19,-5);
\node at (18.5,-4.5) {12};
\draw[fill=26x8x2color11] (19,-4) rectangle (20,-5);
\node at (19.5,-4.5) {11};
\draw[fill=26x8x2color10] (20,-4) rectangle (21,-5);
\node at (20.5,-4.5) {10};
\draw[fill=26x8x2color9] (21,-4) rectangle (22,-5);
\node at (21.5,-4.5) {9};
\draw[fill=26x8x2color8] (22,-4) rectangle (23,-5);
\node at (22.5,-4.5) {8};
\draw[fill=26x8x2color7] (23,-4) rectangle (24,-5);
\node at (23.5,-4.5) {7};
\draw[fill=26x8x2color6] (24,-4) rectangle (25,-5);
\node at (24.5,-4.5) {6};
\draw[fill=26x8x2color5] (25,-4) rectangle (26,-5);
\node at (25.5,-4.5) {5};
\draw[fill=26x8x2color37] (0,-5) rectangle (1,-6);
\node at (0.5,-5.5) {37};
\draw[fill=26x8x2color28] (1,-5) rectangle (2,-6);
\node at (1.5,-5.5) {28};
\draw[fill=26x8x2color21] (2,-5) rectangle (3,-6);
\node at (2.5,-5.5) {21};
\draw[fill=26x8x2color20] (3,-5) rectangle (4,-6);
\node at (3.5,-5.5) {20};
\draw[fill=26x8x2color19] (4,-5) rectangle (5,-6);
\node at (4.5,-5.5) {19};
\draw[fill=26x8x2color18] (5,-5) rectangle (6,-6);
\node at (5.5,-5.5) {18};
\draw[fill=26x8x2color17] (6,-5) rectangle (7,-6);
\node at (6.5,-5.5) {17};
\draw[fill=26x8x2color16] (7,-5) rectangle (8,-6);
\node at (7.5,-5.5) {16};
\draw[fill=26x8x2color15] (8,-5) rectangle (9,-6);
\node at (8.5,-5.5) {15};
\draw[fill=26x8x2color14] (9,-5) rectangle (10,-6);
\node at (9.5,-5.5) {14};
\draw[fill=26x8x2color13] (10,-5) rectangle (11,-6);
\node at (10.5,-5.5) {13};
\draw[fill=26x8x2color12] (11,-5) rectangle (12,-6);
\node at (11.5,-5.5) {12};
\draw[fill=26x8x2color11] (12,-5) rectangle (13,-6);
\node at (12.5,-5.5) {11};
\draw[fill=26x8x2color10] (13,-5) rectangle (14,-6);
\node at (13.5,-5.5) {10};
\draw[fill=26x8x2color9] (14,-5) rectangle (15,-6);
\node at (14.5,-5.5) {9};
\draw[fill=26x8x2color8] (15,-5) rectangle (16,-6);
\node at (15.5,-5.5) {8};
\draw[fill=26x8x2color7] (16,-5) rectangle (17,-6);
\node at (16.5,-5.5) {7};
\draw[fill=26x8x2color6] (17,-5) rectangle (18,-6);
\node at (17.5,-5.5) {6};
\draw[fill=26x8x2color5] (18,-5) rectangle (19,-6);
\node at (18.5,-5.5) {5};
\draw[fill=26x8x2color4] (19,-5) rectangle (20,-6);
\node at (19.5,-5.5) {4};
\draw[fill=26x8x2color3] (20,-5) rectangle (21,-6);
\node at (20.5,-5.5) {3};
\draw[fill=26x8x2color2] (21,-5) rectangle (22,-6);
\node at (21.5,-5.5) {2};
\draw[fill=26x8x2color1] (22,-5) rectangle (23,-6);
\node at (22.5,-5.5) {1};
\draw[fill=red] (23,-5) rectangle (24,-6);
\draw[fill=26x8x2color3] (24,-5) rectangle (25,-6);
\node at (24.5,-5.5) {3};
\draw[fill=26x8x2color4] (25,-5) rectangle (26,-6);
\node at (25.5,-5.5) {4};
\draw[fill=26x8x2color38] (0,-6) rectangle (1,-7);
\node at (0.5,-6.5) {38};
\draw[fill=26x8x2color27] (1,-6) rectangle (2,-7);
\node at (1.5,-6.5) {27};
\draw[fill=red] (2,-6) rectangle (3,-7);
\draw[fill=26x8x2color1] (3,-6) rectangle (4,-7);
\node at (3.5,-6.5) {1};
\draw[fill=red] (4,-6) rectangle (5,-7);
\draw[fill=26x8x2color1] (5,-6) rectangle (6,-7);
\node at (5.5,-6.5) {1};
\draw[fill=red] (6,-6) rectangle (7,-7);
\draw[fill=26x8x2color1] (7,-6) rectangle (8,-7);
\node at (7.5,-6.5) {1};
\draw[fill=red] (8,-6) rectangle (9,-7);
\draw[fill=26x8x2color1] (9,-6) rectangle (10,-7);
\node at (9.5,-6.5) {1};
\draw[fill=red] (10,-6) rectangle (11,-7);
\draw[fill=26x8x2color1] (11,-6) rectangle (12,-7);
\node at (11.5,-6.5) {1};
\draw[fill=red] (12,-6) rectangle (13,-7);
\draw[fill=26x8x2color1] (13,-6) rectangle (14,-7);
\node at (13.5,-6.5) {1};
\draw[fill=red] (14,-6) rectangle (15,-7);
\draw[fill=26x8x2color1] (15,-6) rectangle (16,-7);
\node at (15.5,-6.5) {1};
\draw[fill=red] (16,-6) rectangle (17,-7);
\draw[fill=26x8x2color1] (17,-6) rectangle (18,-7);
\node at (17.5,-6.5) {1};
\draw[fill=red] (18,-6) rectangle (19,-7);
\draw[fill=26x8x2color1] (19,-6) rectangle (20,-7);
\node at (19.5,-6.5) {1};
\draw[fill=red] (20,-6) rectangle (21,-7);
\draw[fill=26x8x2color1] (21,-6) rectangle (22,-7);
\node at (21.5,-6.5) {1};
\draw[fill=red] (22,-6) rectangle (23,-7);
\draw[fill=26x8x2color1] (23,-6) rectangle (24,-7);
\node at (23.5,-6.5) {1};
\draw[fill=26x8x2color2] (24,-6) rectangle (25,-7);
\node at (24.5,-6.5) {2};
\draw[fill=26x8x2color3] (25,-6) rectangle (26,-7);
\node at (25.5,-6.5) {3};
\draw[fill=26x8x2color39] (0,-7) rectangle (1,-8);
\node at (0.5,-7.5) {39};
\draw[fill=red] (1,-7) rectangle (2,-8);
\draw[fill=26x8x2color1] (2,-7) rectangle (3,-8);
\node at (2.5,-7.5) {1};
\draw[fill=red] (3,-7) rectangle (4,-8);
\draw[fill=26x8x2color1] (4,-7) rectangle (5,-8);
\node at (4.5,-7.5) {1};
\draw[fill=red] (5,-7) rectangle (6,-8);
\draw[fill=26x8x2color1] (6,-7) rectangle (7,-8);
\node at (6.5,-7.5) {1};
\draw[fill=red] (7,-7) rectangle (8,-8);
\draw[fill=26x8x2color1] (8,-7) rectangle (9,-8);
\node at (8.5,-7.5) {1};
\draw[fill=red] (9,-7) rectangle (10,-8);
\draw[fill=26x8x2color1] (10,-7) rectangle (11,-8);
\node at (10.5,-7.5) {1};
\draw[fill=red] (11,-7) rectangle (12,-8);
\draw[fill=26x8x2color1] (12,-7) rectangle (13,-8);
\node at (12.5,-7.5) {1};
\draw[fill=red] (13,-7) rectangle (14,-8);
\draw[fill=26x8x2color1] (14,-7) rectangle (15,-8);
\node at (14.5,-7.5) {1};
\draw[fill=red] (15,-7) rectangle (16,-8);
\draw[fill=26x8x2color1] (16,-7) rectangle (17,-8);
\node at (16.5,-7.5) {1};
\draw[fill=red] (17,-7) rectangle (18,-8);
\draw[fill=26x8x2color1] (18,-7) rectangle (19,-8);
\node at (18.5,-7.5) {1};
\draw[fill=red] (19,-7) rectangle (20,-8);
\draw[fill=26x8x2color1] (20,-7) rectangle (21,-8);
\node at (20.5,-7.5) {1};
\draw[fill=red] (21,-7) rectangle (22,-8);
\draw[fill=26x8x2color1] (22,-7) rectangle (23,-8);
\node at (22.5,-7.5) {1};
\draw[fill=red] (23,-7) rectangle (24,-8);
\draw[fill=26x8x2color1] (24,-7) rectangle (25,-8);
\node at (24.5,-7.5) {1};
\draw[fill=red] (25,-7) rectangle (26,-8);
\draw[fill=26x8x2color1] (0,-9) rectangle (1,-10);
\node at (0.5,-9.5) {1};
\draw[fill=red] (1,-9) rectangle (2,-10);
\draw[fill=26x8x2color5] (2,-9) rectangle (3,-10);
\node at (2.5,-9.5) {5};
\draw[fill=red] (3,-9) rectangle (4,-10);
\draw[fill=26x8x2color1] (4,-9) rectangle (5,-10);
\node at (4.5,-9.5) {1};
\draw[fill=red] (5,-9) rectangle (6,-10);
\draw[fill=26x8x2color1] (6,-9) rectangle (7,-10);
\node at (6.5,-9.5) {1};
\draw[fill=red] (7,-9) rectangle (8,-10);
\draw[fill=26x8x2color5] (8,-9) rectangle (9,-10);
\node at (8.5,-9.5) {5};
\draw[fill=red] (9,-9) rectangle (10,-10);
\draw[fill=26x8x2color1] (10,-9) rectangle (11,-10);
\node at (10.5,-9.5) {1};
\draw[fill=red] (11,-9) rectangle (12,-10);
\draw[fill=26x8x2color1] (12,-9) rectangle (13,-10);
\node at (12.5,-9.5) {1};
\draw[fill=red] (13,-9) rectangle (14,-10);
\draw[fill=26x8x2color5] (14,-9) rectangle (15,-10);
\node at (14.5,-9.5) {5};
\draw[fill=red] (15,-9) rectangle (16,-10);
\draw[fill=26x8x2color1] (16,-9) rectangle (17,-10);
\node at (16.5,-9.5) {1};
\draw[fill=red] (17,-9) rectangle (18,-10);
\draw[fill=26x8x2color1] (18,-9) rectangle (19,-10);
\node at (18.5,-9.5) {1};
\draw[fill=red] (19,-9) rectangle (20,-10);
\draw[fill=26x8x2color5] (20,-9) rectangle (21,-10);
\node at (20.5,-9.5) {5};
\draw[fill=red] (21,-9) rectangle (22,-10);
\draw[fill=26x8x2color1] (22,-9) rectangle (23,-10);
\node at (22.5,-9.5) {1};
\draw[fill=red] (23,-9) rectangle (24,-10);
\draw[fill=26x8x2color1] (24,-9) rectangle (25,-10);
\node at (24.5,-9.5) {1};
\draw[fill=red] (25,-9) rectangle (26,-10);
\draw[fill=red] (0,-10) rectangle (1,-11);
\draw[fill=26x8x2color1] (1,-10) rectangle (2,-11);
\node at (1.5,-10.5) {1};
\draw[fill=26x8x2color4] (2,-10) rectangle (3,-11);
\node at (2.5,-10.5) {4};
\draw[fill=26x8x2color1] (3,-10) rectangle (4,-11);
\node at (3.5,-10.5) {1};
\draw[fill=red] (4,-10) rectangle (5,-11);
\draw[fill=26x8x2color1] (5,-10) rectangle (6,-11);
\node at (5.5,-10.5) {1};
\draw[fill=red] (6,-10) rectangle (7,-11);
\draw[fill=26x8x2color1] (7,-10) rectangle (8,-11);
\node at (7.5,-10.5) {1};
\draw[fill=26x8x2color4] (8,-10) rectangle (9,-11);
\node at (8.5,-10.5) {4};
\draw[fill=26x8x2color1] (9,-10) rectangle (10,-11);
\node at (9.5,-10.5) {1};
\draw[fill=red] (10,-10) rectangle (11,-11);
\draw[fill=26x8x2color1] (11,-10) rectangle (12,-11);
\node at (11.5,-10.5) {1};
\draw[fill=red] (12,-10) rectangle (13,-11);
\draw[fill=26x8x2color1] (13,-10) rectangle (14,-11);
\node at (13.5,-10.5) {1};
\draw[fill=26x8x2color4] (14,-10) rectangle (15,-11);
\node at (14.5,-10.5) {4};
\draw[fill=26x8x2color1] (15,-10) rectangle (16,-11);
\node at (15.5,-10.5) {1};
\draw[fill=red] (16,-10) rectangle (17,-11);
\draw[fill=26x8x2color1] (17,-10) rectangle (18,-11);
\node at (17.5,-10.5) {1};
\draw[fill=red] (18,-10) rectangle (19,-11);
\draw[fill=26x8x2color1] (19,-10) rectangle (20,-11);
\node at (19.5,-10.5) {1};
\draw[fill=26x8x2color4] (20,-10) rectangle (21,-11);
\node at (20.5,-10.5) {4};
\draw[fill=26x8x2color1] (21,-10) rectangle (22,-11);
\node at (21.5,-10.5) {1};
\draw[fill=red] (22,-10) rectangle (23,-11);
\draw[fill=26x8x2color1] (23,-10) rectangle (24,-11);
\node at (23.5,-10.5) {1};
\draw[fill=red] (24,-10) rectangle (25,-11);
\draw[fill=26x8x2color1] (25,-10) rectangle (26,-11);
\node at (25.5,-10.5) {1};
\draw[fill=26x8x2color1] (0,-11) rectangle (1,-12);
\node at (0.5,-11.5) {1};
\draw[fill=red] (1,-11) rectangle (2,-12);
\draw[fill=26x8x2color3] (2,-11) rectangle (3,-12);
\node at (2.5,-11.5) {3};
\draw[fill=red] (3,-11) rectangle (4,-12);
\draw[fill=26x8x2color1] (4,-11) rectangle (5,-12);
\node at (4.5,-11.5) {1};
\draw[fill=26x8x2color2] (5,-11) rectangle (6,-12);
\node at (5.5,-11.5) {2};
\draw[fill=26x8x2color1] (6,-11) rectangle (7,-12);
\node at (6.5,-11.5) {1};
\draw[fill=red] (7,-11) rectangle (8,-12);
\draw[fill=26x8x2color3] (8,-11) rectangle (9,-12);
\node at (8.5,-11.5) {3};
\draw[fill=red] (9,-11) rectangle (10,-12);
\draw[fill=26x8x2color1] (10,-11) rectangle (11,-12);
\node at (10.5,-11.5) {1};
\draw[fill=26x8x2color2] (11,-11) rectangle (12,-12);
\node at (11.5,-11.5) {2};
\draw[fill=26x8x2color1] (12,-11) rectangle (13,-12);
\node at (12.5,-11.5) {1};
\draw[fill=red] (13,-11) rectangle (14,-12);
\draw[fill=26x8x2color3] (14,-11) rectangle (15,-12);
\node at (14.5,-11.5) {3};
\draw[fill=red] (15,-11) rectangle (16,-12);
\draw[fill=26x8x2color1] (16,-11) rectangle (17,-12);
\node at (16.5,-11.5) {1};
\draw[fill=26x8x2color2] (17,-11) rectangle (18,-12);
\node at (17.5,-11.5) {2};
\draw[fill=26x8x2color1] (18,-11) rectangle (19,-12);
\node at (18.5,-11.5) {1};
\draw[fill=red] (19,-11) rectangle (20,-12);
\draw[fill=26x8x2color3] (20,-11) rectangle (21,-12);
\node at (20.5,-11.5) {3};
\draw[fill=red] (21,-11) rectangle (22,-12);
\draw[fill=26x8x2color1] (22,-11) rectangle (23,-12);
\node at (22.5,-11.5) {1};
\draw[fill=26x8x2color4] (23,-11) rectangle (24,-12);
\node at (23.5,-11.5) {4};
\draw[fill=26x8x2color5] (24,-11) rectangle (25,-12);
\node at (24.5,-11.5) {5};
\draw[fill=26x8x2color6] (25,-11) rectangle (26,-12);
\node at (25.5,-11.5) {6};
\draw[fill=red] (0,-12) rectangle (1,-13);
\draw[fill=26x8x2color1] (1,-12) rectangle (2,-13);
\node at (1.5,-12.5) {1};
\draw[fill=26x8x2color2] (2,-12) rectangle (3,-13);
\node at (2.5,-12.5) {2};
\draw[fill=26x8x2color1] (3,-12) rectangle (4,-13);
\node at (3.5,-12.5) {1};
\draw[fill=red] (4,-12) rectangle (5,-13);
\draw[fill=26x8x2color3] (5,-12) rectangle (6,-13);
\node at (5.5,-12.5) {3};
\draw[fill=red] (6,-12) rectangle (7,-13);
\draw[fill=26x8x2color1] (7,-12) rectangle (8,-13);
\node at (7.5,-12.5) {1};
\draw[fill=26x8x2color2] (8,-12) rectangle (9,-13);
\node at (8.5,-12.5) {2};
\draw[fill=26x8x2color1] (9,-12) rectangle (10,-13);
\node at (9.5,-12.5) {1};
\draw[fill=red] (10,-12) rectangle (11,-13);
\draw[fill=26x8x2color3] (11,-12) rectangle (12,-13);
\node at (11.5,-12.5) {3};
\draw[fill=red] (12,-12) rectangle (13,-13);
\draw[fill=26x8x2color1] (13,-12) rectangle (14,-13);
\node at (13.5,-12.5) {1};
\draw[fill=26x8x2color2] (14,-12) rectangle (15,-13);
\node at (14.5,-12.5) {2};
\draw[fill=26x8x2color1] (15,-12) rectangle (16,-13);
\node at (15.5,-12.5) {1};
\draw[fill=red] (16,-12) rectangle (17,-13);
\draw[fill=26x8x2color3] (17,-12) rectangle (18,-13);
\node at (17.5,-12.5) {3};
\draw[fill=red] (18,-12) rectangle (19,-13);
\draw[fill=26x8x2color1] (19,-12) rectangle (20,-13);
\node at (19.5,-12.5) {1};
\draw[fill=26x8x2color2] (20,-12) rectangle (21,-13);
\node at (20.5,-12.5) {2};
\draw[fill=26x8x2color1] (21,-12) rectangle (22,-13);
\node at (21.5,-12.5) {1};
\draw[fill=red] (22,-12) rectangle (23,-13);
\draw[fill=26x8x2color3] (23,-12) rectangle (24,-13);
\node at (23.5,-12.5) {3};
\draw[fill=red] (24,-12) rectangle (25,-13);
\draw[fill=26x8x2color1] (25,-12) rectangle (26,-13);
\node at (25.5,-12.5) {1};
\draw[fill=26x8x2color1] (0,-13) rectangle (1,-14);
\node at (0.5,-13.5) {1};
\draw[fill=red] (1,-13) rectangle (2,-14);
\draw[fill=26x8x2color1] (2,-13) rectangle (3,-14);
\node at (2.5,-13.5) {1};
\draw[fill=red] (3,-13) rectangle (4,-14);
\draw[fill=26x8x2color1] (4,-13) rectangle (5,-14);
\node at (4.5,-13.5) {1};
\draw[fill=26x8x2color4] (5,-13) rectangle (6,-14);
\node at (5.5,-13.5) {4};
\draw[fill=26x8x2color1] (6,-13) rectangle (7,-14);
\node at (6.5,-13.5) {1};
\draw[fill=red] (7,-13) rectangle (8,-14);
\draw[fill=26x8x2color1] (8,-13) rectangle (9,-14);
\node at (8.5,-13.5) {1};
\draw[fill=red] (9,-13) rectangle (10,-14);
\draw[fill=26x8x2color1] (10,-13) rectangle (11,-14);
\node at (10.5,-13.5) {1};
\draw[fill=26x8x2color4] (11,-13) rectangle (12,-14);
\node at (11.5,-13.5) {4};
\draw[fill=26x8x2color1] (12,-13) rectangle (13,-14);
\node at (12.5,-13.5) {1};
\draw[fill=red] (13,-13) rectangle (14,-14);
\draw[fill=26x8x2color1] (14,-13) rectangle (15,-14);
\node at (14.5,-13.5) {1};
\draw[fill=red] (15,-13) rectangle (16,-14);
\draw[fill=26x8x2color1] (16,-13) rectangle (17,-14);
\node at (16.5,-13.5) {1};
\draw[fill=26x8x2color4] (17,-13) rectangle (18,-14);
\node at (17.5,-13.5) {4};
\draw[fill=26x8x2color1] (18,-13) rectangle (19,-14);
\node at (18.5,-13.5) {1};
\draw[fill=red] (19,-13) rectangle (20,-14);
\draw[fill=26x8x2color1] (20,-13) rectangle (21,-14);
\node at (20.5,-13.5) {1};
\draw[fill=red] (21,-13) rectangle (22,-14);
\draw[fill=26x8x2color1] (22,-13) rectangle (23,-14);
\node at (22.5,-13.5) {1};
\draw[fill=26x8x2color2] (23,-13) rectangle (24,-14);
\node at (23.5,-13.5) {2};
\draw[fill=26x8x2color1] (24,-13) rectangle (25,-14);
\node at (24.5,-13.5) {1};
\draw[fill=red] (25,-13) rectangle (26,-14);
\draw[fill=red] (0,-14) rectangle (1,-15);
\draw[fill=26x8x2color1] (1,-14) rectangle (2,-15);
\node at (1.5,-14.5) {1};
\draw[fill=red] (2,-14) rectangle (3,-15);
\draw[fill=26x8x2color1] (3,-14) rectangle (4,-15);
\node at (3.5,-14.5) {1};
\draw[fill=red] (4,-14) rectangle (5,-15);
\draw[fill=26x8x2color5] (5,-14) rectangle (6,-15);
\node at (5.5,-14.5) {5};
\draw[fill=red] (6,-14) rectangle (7,-15);
\draw[fill=26x8x2color1] (7,-14) rectangle (8,-15);
\node at (7.5,-14.5) {1};
\draw[fill=red] (8,-14) rectangle (9,-15);
\draw[fill=26x8x2color1] (9,-14) rectangle (10,-15);
\node at (9.5,-14.5) {1};
\draw[fill=red] (10,-14) rectangle (11,-15);
\draw[fill=26x8x2color5] (11,-14) rectangle (12,-15);
\node at (11.5,-14.5) {5};
\draw[fill=red] (12,-14) rectangle (13,-15);
\draw[fill=26x8x2color1] (13,-14) rectangle (14,-15);
\node at (13.5,-14.5) {1};
\draw[fill=red] (14,-14) rectangle (15,-15);
\draw[fill=26x8x2color1] (15,-14) rectangle (16,-15);
\node at (15.5,-14.5) {1};
\draw[fill=red] (16,-14) rectangle (17,-15);
\draw[fill=26x8x2color5] (17,-14) rectangle (18,-15);
\node at (17.5,-14.5) {5};
\draw[fill=red] (18,-14) rectangle (19,-15);
\draw[fill=26x8x2color1] (19,-14) rectangle (20,-15);
\node at (19.5,-14.5) {1};
\draw[fill=red] (20,-14) rectangle (21,-15);
\draw[fill=26x8x2color1] (21,-14) rectangle (22,-15);
\node at (21.5,-14.5) {1};
\draw[fill=red] (22,-14) rectangle (23,-15);
\draw[fill=26x8x2color1] (23,-14) rectangle (24,-15);
\node at (23.5,-14.5) {1};
\draw[fill=red] (24,-14) rectangle (25,-15);
\draw[fill=26x8x2color1] (25,-14) rectangle (26,-15);
\node at (25.5,-14.5) {1};
\draw[fill=26x8x2color27] (0,-15) rectangle (1,-16);
\node at (0.5,-15.5) {27};
\draw[fill=26x8x2color26] (1,-15) rectangle (2,-16);
\node at (1.5,-15.5) {26};
\draw[fill=26x8x2color23] (2,-15) rectangle (3,-16);
\node at (2.5,-15.5) {23};
\draw[fill=26x8x2color22] (3,-15) rectangle (4,-16);
\node at (3.5,-15.5) {22};
\draw[fill=26x8x2color21] (4,-15) rectangle (5,-16);
\node at (4.5,-15.5) {21};
\draw[fill=26x8x2color20] (5,-15) rectangle (6,-16);
\node at (5.5,-15.5) {20};
\draw[fill=26x8x2color19] (6,-15) rectangle (7,-16);
\node at (6.5,-15.5) {19};
\draw[fill=26x8x2color18] (7,-15) rectangle (8,-16);
\node at (7.5,-15.5) {18};
\draw[fill=26x8x2color17] (8,-15) rectangle (9,-16);
\node at (8.5,-15.5) {17};
\draw[fill=26x8x2color16] (9,-15) rectangle (10,-16);
\node at (9.5,-15.5) {16};
\draw[fill=26x8x2color15] (10,-15) rectangle (11,-16);
\node at (10.5,-15.5) {15};
\draw[fill=26x8x2color14] (11,-15) rectangle (12,-16);
\node at (11.5,-15.5) {14};
\draw[fill=26x8x2color13] (12,-15) rectangle (13,-16);
\node at (12.5,-15.5) {13};
\draw[fill=26x8x2color12] (13,-15) rectangle (14,-16);
\node at (13.5,-15.5) {12};
\draw[fill=26x8x2color11] (14,-15) rectangle (15,-16);
\node at (14.5,-15.5) {11};
\draw[fill=26x8x2color10] (15,-15) rectangle (16,-16);
\node at (15.5,-15.5) {10};
\draw[fill=26x8x2color9] (16,-15) rectangle (17,-16);
\node at (16.5,-15.5) {9};
\draw[fill=26x8x2color8] (17,-15) rectangle (18,-16);
\node at (17.5,-15.5) {8};
\draw[fill=26x8x2color7] (18,-15) rectangle (19,-16);
\node at (18.5,-15.5) {7};
\draw[fill=26x8x2color6] (19,-15) rectangle (20,-16);
\node at (19.5,-15.5) {6};
\draw[fill=26x8x2color5] (20,-15) rectangle (21,-16);
\node at (20.5,-15.5) {5};
\draw[fill=26x8x2color4] (21,-15) rectangle (22,-16);
\node at (21.5,-15.5) {4};
\draw[fill=26x8x2color3] (22,-15) rectangle (23,-16);
\node at (22.5,-15.5) {3};
\draw[fill=26x8x2color2] (23,-15) rectangle (24,-16);
\node at (23.5,-15.5) {2};
\draw[fill=26x8x2color1] (24,-15) rectangle (25,-16);
\node at (24.5,-15.5) {1};
\draw[fill=red] (25,-15) rectangle (26,-16);
\draw[fill=red] (0,-16) rectangle (1,-17);
\draw[fill=26x8x2color25] (1,-16) rectangle (2,-17);
\node at (1.5,-16.5) {25};
\draw[fill=26x8x2color24] (2,-16) rectangle (3,-17);
\node at (2.5,-16.5) {24};
\draw[fill=26x8x2color23] (3,-16) rectangle (4,-17);
\node at (3.5,-16.5) {23};
\draw[fill=26x8x2color22] (4,-16) rectangle (5,-17);
\node at (4.5,-16.5) {22};
\draw[fill=26x8x2color21] (5,-16) rectangle (6,-17);
\node at (5.5,-16.5) {21};
\draw[fill=26x8x2color20] (6,-16) rectangle (7,-17);
\node at (6.5,-16.5) {20};
\draw[fill=26x8x2color19] (7,-16) rectangle (8,-17);
\node at (7.5,-16.5) {19};
\draw[fill=26x8x2color18] (8,-16) rectangle (9,-17);
\node at (8.5,-16.5) {18};
\draw[fill=26x8x2color17] (9,-16) rectangle (10,-17);
\node at (9.5,-16.5) {17};
\draw[fill=26x8x2color16] (10,-16) rectangle (11,-17);
\node at (10.5,-16.5) {16};
\draw[fill=26x8x2color15] (11,-16) rectangle (12,-17);
\node at (11.5,-16.5) {15};
\draw[fill=26x8x2color14] (12,-16) rectangle (13,-17);
\node at (12.5,-16.5) {14};
\draw[fill=26x8x2color13] (13,-16) rectangle (14,-17);
\node at (13.5,-16.5) {13};
\draw[fill=26x8x2color12] (14,-16) rectangle (15,-17);
\node at (14.5,-16.5) {12};
\draw[fill=26x8x2color11] (15,-16) rectangle (16,-17);
\node at (15.5,-16.5) {11};
\draw[fill=26x8x2color10] (16,-16) rectangle (17,-17);
\node at (16.5,-16.5) {10};
\draw[fill=26x8x2color9] (17,-16) rectangle (18,-17);
\node at (17.5,-16.5) {9};
\draw[fill=26x8x2color8] (18,-16) rectangle (19,-17);
\node at (18.5,-16.5) {8};
\draw[fill=26x8x2color7] (19,-16) rectangle (20,-17);
\node at (19.5,-16.5) {7};
\draw[fill=26x8x2color6] (20,-16) rectangle (21,-17);
\node at (20.5,-16.5) {6};
\draw[fill=26x8x2color5] (21,-16) rectangle (22,-17);
\node at (21.5,-16.5) {5};
\draw[fill=26x8x2color4] (22,-16) rectangle (23,-17);
\node at (22.5,-16.5) {4};
\draw[fill=26x8x2color3] (23,-16) rectangle (24,-17);
\node at (23.5,-16.5) {3};
\draw[fill=red] (24,-16) rectangle (25,-17);
\draw[fill=26x8x2color1] (25,-16) rectangle (26,-17);
\node at (25.5,-16.5) {1};

\end{scope}
\end{tikzpicture}
\end{center}

%% file: 4-4-25-tikz.txt
\definecolor{25x4x4color1}{RGB}{255, 128, 128}
\definecolor{25x4x4color2}{RGB}{255, 129, 129}
\definecolor{25x4x4color3}{RGB}{255, 130, 130}
\definecolor{25x4x4color4}{RGB}{255, 132, 132}
\definecolor{25x4x4color5}{RGB}{255, 133, 133}
\definecolor{25x4x4color6}{RGB}{255, 134, 134}
\definecolor{25x4x4color7}{RGB}{255, 136, 136}
\definecolor{25x4x4color8}{RGB}{255, 137, 137}
\definecolor{25x4x4color9}{RGB}{255, 138, 138}
\definecolor{25x4x4color10}{RGB}{255, 140, 140}
\definecolor{25x4x4color11}{RGB}{255, 141, 141}
\definecolor{25x4x4color12}{RGB}{255, 142, 142}
\definecolor{25x4x4color13}{RGB}{255, 144, 144}
\definecolor{25x4x4color14}{RGB}{255, 145, 145}
\definecolor{25x4x4color15}{RGB}{255, 146, 146}
\definecolor{25x4x4color16}{RGB}{255, 148, 148}
\definecolor{25x4x4color17}{RGB}{255, 149, 149}
\definecolor{25x4x4color18}{RGB}{255, 150, 150}
\definecolor{25x4x4color19}{RGB}{255, 152, 152}
\definecolor{25x4x4color20}{RGB}{255, 153, 153}
\definecolor{25x4x4color21}{RGB}{255, 154, 154}
\definecolor{25x4x4color22}{RGB}{255, 156, 156}
\definecolor{25x4x4color23}{RGB}{255, 157, 157}
\definecolor{25x4x4color24}{RGB}{255, 158, 158}
\definecolor{25x4x4color25}{RGB}{255, 160, 160}
\definecolor{25x4x4color26}{RGB}{255, 161, 161}
\definecolor{25x4x4color27}{RGB}{255, 162, 162}
\definecolor{25x4x4color28}{RGB}{255, 164, 164}
\definecolor{25x4x4color29}{RGB}{255, 165, 165}
\definecolor{25x4x4color30}{RGB}{255, 166, 166}
\definecolor{25x4x4color31}{RGB}{255, 168, 168}
\definecolor{25x4x4color32}{RGB}{255, 169, 169}
\definecolor{25x4x4color33}{RGB}{255, 170, 170}
\definecolor{25x4x4color34}{RGB}{255, 172, 172}
\definecolor{25x4x4color35}{RGB}{255, 173, 173}
\definecolor{25x4x4color36}{RGB}{255, 174, 174}
\definecolor{25x4x4color37}{RGB}{255, 176, 176}
\definecolor{25x4x4color38}{RGB}{255, 177, 177}
\definecolor{25x4x4color39}{RGB}{255, 178, 178}
\definecolor{25x4x4color40}{RGB}{255, 180, 180}
\definecolor{25x4x4color41}{RGB}{255, 181, 181}
\definecolor{25x4x4color42}{RGB}{255, 182, 182}
\definecolor{25x4x4color43}{RGB}{255, 184, 184}
\definecolor{25x4x4color44}{RGB}{255, 185, 185}
\definecolor{25x4x4color45}{RGB}{255, 186, 186}
\definecolor{25x4x4color46}{RGB}{255, 188, 188}
\definecolor{25x4x4color47}{RGB}{255, 189, 189}
\definecolor{25x4x4color48}{RGB}{255, 190, 190}
\definecolor{25x4x4color49}{RGB}{255, 192, 192}
\definecolor{25x4x4color50}{RGB}{255, 193, 193}
\definecolor{25x4x4color51}{RGB}{255, 194, 194}
\definecolor{25x4x4color52}{RGB}{255, 196, 196}
\definecolor{25x4x4color53}{RGB}{255, 197, 197}
\definecolor{25x4x4color54}{RGB}{255, 198, 198}
\definecolor{25x4x4color55}{RGB}{255, 200, 200}
\definecolor{25x4x4color56}{RGB}{255, 201, 201}
\definecolor{25x4x4color57}{RGB}{255, 202, 202}
\definecolor{25x4x4color58}{RGB}{255, 204, 204}
\definecolor{25x4x4color59}{RGB}{255, 205, 205}
\definecolor{25x4x4color60}{RGB}{255, 206, 206}
\definecolor{25x4x4color61}{RGB}{255, 208, 208}
\definecolor{25x4x4color62}{RGB}{255, 209, 209}
\definecolor{25x4x4color63}{RGB}{255, 210, 210}
\definecolor{25x4x4color64}{RGB}{255, 212, 212}
\definecolor{25x4x4color65}{RGB}{255, 213, 213}
\definecolor{25x4x4color66}{RGB}{255, 214, 214}
\definecolor{25x4x4color67}{RGB}{255, 216, 216}
\definecolor{25x4x4color68}{RGB}{255, 217, 217}
\definecolor{25x4x4color69}{RGB}{255, 218, 218}
\definecolor{25x4x4color70}{RGB}{255, 220, 220}
\definecolor{25x4x4color71}{RGB}{255, 221, 221}
\definecolor{25x4x4color72}{RGB}{255, 222, 222}
\definecolor{25x4x4color73}{RGB}{255, 224, 224}
\definecolor{25x4x4color74}{RGB}{255, 225, 225}
\definecolor{25x4x4color75}{RGB}{255, 226, 226}
\definecolor{25x4x4color76}{RGB}{255, 228, 228}
\definecolor{25x4x4color77}{RGB}{255, 229, 229}
\definecolor{25x4x4color78}{RGB}{255, 230, 230}
\definecolor{25x4x4color79}{RGB}{255, 232, 232}
\definecolor{25x4x4color80}{RGB}{255, 233, 233}
\definecolor{25x4x4color81}{RGB}{255, 234, 234}
\definecolor{25x4x4color82}{RGB}{255, 236, 236}
\definecolor{25x4x4color83}{RGB}{255, 237, 237}
\definecolor{25x4x4color84}{RGB}{255, 238, 238}
\definecolor{25x4x4color85}{RGB}{255, 240, 240}
\definecolor{25x4x4color86}{RGB}{255, 241, 241}
\definecolor{25x4x4color87}{RGB}{255, 242, 242}
\definecolor{25x4x4color88}{RGB}{255, 244, 244}
\definecolor{25x4x4color89}{RGB}{255, 245, 245}
\definecolor{25x4x4color90}{RGB}{255, 246, 246}
\definecolor{25x4x4color91}{RGB}{255, 248, 248}
\definecolor{25x4x4color92}{RGB}{255, 249, 249}
\definecolor{25x4x4color93}{RGB}{255, 250, 250}
\definecolor{25x4x4color94}{RGB}{255, 252, 252}
\definecolor{25x4x4color95}{RGB}{255, 255, 255}

\begin{center}
\begin{tikzpicture}[scale=0.5]
\begin{scope}
\draw[very thick] (0,-0) rectangle (25,-4);
\foreach \b in {0,...,3}
 \foreach \c in {0,...,24}
  \draw (\c,\b*-1-0) rectangle (\c+1,(\b*-1-1);
\draw[very thick] (0,-5) rectangle (25,-9);
\foreach \b in {0,...,3}
 \foreach \c in {0,...,24}
  \draw (\c,\b*-1-5) rectangle (\c+1,(\b*-1-6);
\draw[very thick] (0,-10) rectangle (25,-14);
\foreach \b in {0,...,3}
 \foreach \c in {0,...,24}
  \draw (\c,\b*-1-10) rectangle (\c+1,(\b*-1-11);
\draw[very thick] (0,-15) rectangle (25,-19);
\foreach \b in {0,...,3}
 \foreach \c in {0,...,24}
  \draw (\c,\b*-1-15) rectangle (\c+1,(\b*-1-16);

\draw[fill=red] (0,-0) rectangle (1,-1);
\draw[fill=25x4x4color71] (1,-0) rectangle (2,-1);
\node at (1.5,-0.5) {71};
\draw[fill=25x4x4color72] (2,-0) rectangle (3,-1);
\node at (2.5,-0.5) {72};
\draw[fill=25x4x4color73] (3,-0) rectangle (4,-1);
\node at (3.5,-0.5) {73};
\draw[fill=25x4x4color74] (4,-0) rectangle (5,-1);
\node at (4.5,-0.5) {74};
\draw[fill=25x4x4color75] (5,-0) rectangle (6,-1);
\node at (5.5,-0.5) {75};
\draw[fill=25x4x4color76] (6,-0) rectangle (7,-1);
\node at (6.5,-0.5) {76};
\draw[fill=25x4x4color77] (7,-0) rectangle (8,-1);
\node at (7.5,-0.5) {77};
\draw[fill=25x4x4color78] (8,-0) rectangle (9,-1);
\node at (8.5,-0.5) {78};
\draw[fill=25x4x4color79] (9,-0) rectangle (10,-1);
\node at (9.5,-0.5) {79};
\draw[fill=25x4x4color80] (10,-0) rectangle (11,-1);
\node at (10.5,-0.5) {80};
\draw[fill=25x4x4color81] (11,-0) rectangle (12,-1);
\node at (11.5,-0.5) {81};
\draw[fill=25x4x4color82] (12,-0) rectangle (13,-1);
\node at (12.5,-0.5) {82};
\draw[fill=25x4x4color83] (13,-0) rectangle (14,-1);
\node at (13.5,-0.5) {83};
\draw[fill=25x4x4color84] (14,-0) rectangle (15,-1);
\node at (14.5,-0.5) {84};
\draw[fill=25x4x4color85] (15,-0) rectangle (16,-1);
\node at (15.5,-0.5) {85};
\draw[fill=25x4x4color86] (16,-0) rectangle (17,-1);
\node at (16.5,-0.5) {86};
\draw[fill=25x4x4color87] (17,-0) rectangle (18,-1);
\node at (17.5,-0.5) {87};
\draw[fill=25x4x4color88] (18,-0) rectangle (19,-1);
\node at (18.5,-0.5) {88};
\draw[fill=25x4x4color89] (19,-0) rectangle (20,-1);
\node at (19.5,-0.5) {89};
\draw[fill=25x4x4color90] (20,-0) rectangle (21,-1);
\node at (20.5,-0.5) {90};
\draw[fill=25x4x4color91] (21,-0) rectangle (22,-1);
\node at (21.5,-0.5) {91};
\draw[fill=25x4x4color92] (22,-0) rectangle (23,-1);
\node at (22.5,-0.5) {92};
\draw[fill=25x4x4color93] (23,-0) rectangle (24,-1);
\node at (23.5,-0.5) {93};
\draw[fill=25x4x4color94] (24,-0) rectangle (25,-1);
\node at (24.5,-0.5) {94};
\draw[fill=25x4x4color69] (0,-1) rectangle (1,-2);
\node at (0.5,-1.5) {69};
\draw[fill=25x4x4color70] (1,-1) rectangle (2,-2);
\node at (1.5,-1.5) {70};
\draw[fill=25x4x4color71] (2,-1) rectangle (3,-2);
\node at (2.5,-1.5) {71};
\draw[fill=25x4x4color72] (3,-1) rectangle (4,-2);
\node at (3.5,-1.5) {72};
\draw[fill=25x4x4color73] (4,-1) rectangle (5,-2);
\node at (4.5,-1.5) {73};
\draw[fill=25x4x4color74] (5,-1) rectangle (6,-2);
\node at (5.5,-1.5) {74};
\draw[fill=25x4x4color75] (6,-1) rectangle (7,-2);
\node at (6.5,-1.5) {75};
\draw[fill=25x4x4color76] (7,-1) rectangle (8,-2);
\node at (7.5,-1.5) {76};
\draw[fill=25x4x4color77] (8,-1) rectangle (9,-2);
\node at (8.5,-1.5) {77};
\draw[fill=25x4x4color78] (9,-1) rectangle (10,-2);
\node at (9.5,-1.5) {78};
\draw[fill=25x4x4color79] (10,-1) rectangle (11,-2);
\node at (10.5,-1.5) {79};
\draw[fill=25x4x4color80] (11,-1) rectangle (12,-2);
\node at (11.5,-1.5) {80};
\draw[fill=25x4x4color81] (12,-1) rectangle (13,-2);
\node at (12.5,-1.5) {81};
\draw[fill=25x4x4color82] (13,-1) rectangle (14,-2);
\node at (13.5,-1.5) {82};
\draw[fill=25x4x4color83] (14,-1) rectangle (15,-2);
\node at (14.5,-1.5) {83};
\draw[fill=25x4x4color84] (15,-1) rectangle (16,-2);
\node at (15.5,-1.5) {84};
\draw[fill=25x4x4color85] (16,-1) rectangle (17,-2);
\node at (16.5,-1.5) {85};
\draw[fill=25x4x4color86] (17,-1) rectangle (18,-2);
\node at (17.5,-1.5) {86};
\draw[fill=25x4x4color87] (18,-1) rectangle (19,-2);
\node at (18.5,-1.5) {87};
\draw[fill=25x4x4color88] (19,-1) rectangle (20,-2);
\node at (19.5,-1.5) {88};
\draw[fill=25x4x4color89] (20,-1) rectangle (21,-2);
\node at (20.5,-1.5) {89};
\draw[fill=25x4x4color90] (21,-1) rectangle (22,-2);
\node at (21.5,-1.5) {90};
\draw[fill=25x4x4color91] (22,-1) rectangle (23,-2);
\node at (22.5,-1.5) {91};
\draw[fill=25x4x4color92] (23,-1) rectangle (24,-2);
\node at (23.5,-1.5) {92};
\draw[fill=25x4x4color93] (24,-1) rectangle (25,-2);
\node at (24.5,-1.5) {93};
\draw[fill=red] (0,-2) rectangle (1,-3);
\draw[fill=25x4x4color65] (1,-2) rectangle (2,-3);
\node at (1.5,-2.5) {65};
\draw[fill=25x4x4color64] (2,-2) rectangle (3,-3);
\node at (2.5,-2.5) {64};
\draw[fill=25x4x4color61] (3,-2) rectangle (4,-3);
\node at (3.5,-2.5) {61};
\draw[fill=red] (4,-2) rectangle (5,-3);
\draw[fill=25x4x4color1] (5,-2) rectangle (6,-3);
\node at (5.5,-2.5) {1};
\draw[fill=red] (6,-2) rectangle (7,-3);
\draw[fill=25x4x4color1] (7,-2) rectangle (8,-3);
\node at (7.5,-2.5) {1};
\draw[fill=red] (8,-2) rectangle (9,-3);
\draw[fill=25x4x4color1] (9,-2) rectangle (10,-3);
\node at (9.5,-2.5) {1};
\draw[fill=red] (10,-2) rectangle (11,-3);
\draw[fill=25x4x4color1] (11,-2) rectangle (12,-3);
\node at (11.5,-2.5) {1};
\draw[fill=red] (12,-2) rectangle (13,-3);
\draw[fill=25x4x4color1] (13,-2) rectangle (14,-3);
\node at (13.5,-2.5) {1};
\draw[fill=red] (14,-2) rectangle (15,-3);
\draw[fill=25x4x4color1] (15,-2) rectangle (16,-3);
\node at (15.5,-2.5) {1};
\draw[fill=red] (16,-2) rectangle (17,-3);
\draw[fill=25x4x4color1] (17,-2) rectangle (18,-3);
\node at (17.5,-2.5) {1};
\draw[fill=red] (18,-2) rectangle (19,-3);
\draw[fill=25x4x4color1] (19,-2) rectangle (20,-3);
\node at (19.5,-2.5) {1};
\draw[fill=red] (20,-2) rectangle (21,-3);
\draw[fill=25x4x4color1] (21,-2) rectangle (22,-3);
\node at (21.5,-2.5) {1};
\draw[fill=red] (22,-2) rectangle (23,-3);
\draw[fill=25x4x4color1] (23,-2) rectangle (24,-3);
\node at (23.5,-2.5) {1};
\draw[fill=red] (24,-2) rectangle (25,-3);
\draw[fill=25x4x4color1] (0,-3) rectangle (1,-4);
\node at (0.5,-3.5) {1};
\draw[fill=red] (1,-3) rectangle (2,-4);
\draw[fill=25x4x4color63] (2,-3) rectangle (3,-4);
\node at (2.5,-3.5) {63};
\draw[fill=red] (3,-3) rectangle (4,-4);
\draw[fill=25x4x4color1] (4,-3) rectangle (5,-4);
\node at (4.5,-3.5) {1};
\draw[fill=red] (5,-3) rectangle (6,-4);
\draw[fill=25x4x4color1] (6,-3) rectangle (7,-4);
\node at (6.5,-3.5) {1};
\draw[fill=red] (7,-3) rectangle (8,-4);
\draw[fill=25x4x4color1] (8,-3) rectangle (9,-4);
\node at (8.5,-3.5) {1};
\draw[fill=red] (9,-3) rectangle (10,-4);
\draw[fill=25x4x4color1] (10,-3) rectangle (11,-4);
\node at (10.5,-3.5) {1};
\draw[fill=red] (11,-3) rectangle (12,-4);
\draw[fill=25x4x4color1] (12,-3) rectangle (13,-4);
\node at (12.5,-3.5) {1};
\draw[fill=red] (13,-3) rectangle (14,-4);
\draw[fill=25x4x4color1] (14,-3) rectangle (15,-4);
\node at (14.5,-3.5) {1};
\draw[fill=red] (15,-3) rectangle (16,-4);
\draw[fill=25x4x4color1] (16,-3) rectangle (17,-4);
\node at (16.5,-3.5) {1};
\draw[fill=red] (17,-3) rectangle (18,-4);
\draw[fill=25x4x4color1] (18,-3) rectangle (19,-4);
\node at (18.5,-3.5) {1};
\draw[fill=red] (19,-3) rectangle (20,-4);
\draw[fill=25x4x4color1] (20,-3) rectangle (21,-4);
\node at (20.5,-3.5) {1};
\draw[fill=red] (21,-3) rectangle (22,-4);
\draw[fill=25x4x4color1] (22,-3) rectangle (23,-4);
\node at (22.5,-3.5) {1};
\draw[fill=red] (23,-3) rectangle (24,-4);
\draw[fill=25x4x4color1] (24,-3) rectangle (25,-4);
\node at (24.5,-3.5) {1};
\draw[fill=25x4x4color69] (0,-5) rectangle (1,-6);
\node at (0.5,-5.5) {69};
\draw[fill=25x4x4color70] (1,-5) rectangle (2,-6);
\node at (1.5,-5.5) {70};
\draw[fill=25x4x4color71] (2,-5) rectangle (3,-6);
\node at (2.5,-5.5) {71};
\draw[fill=25x4x4color72] (3,-5) rectangle (4,-6);
\node at (3.5,-5.5) {72};
\draw[fill=25x4x4color73] (4,-5) rectangle (5,-6);
\node at (4.5,-5.5) {73};
\draw[fill=25x4x4color74] (5,-5) rectangle (6,-6);
\node at (5.5,-5.5) {74};
\draw[fill=25x4x4color75] (6,-5) rectangle (7,-6);
\node at (6.5,-5.5) {75};
\draw[fill=25x4x4color76] (7,-5) rectangle (8,-6);
\node at (7.5,-5.5) {76};
\draw[fill=25x4x4color77] (8,-5) rectangle (9,-6);
\node at (8.5,-5.5) {77};
\draw[fill=25x4x4color78] (9,-5) rectangle (10,-6);
\node at (9.5,-5.5) {78};
\draw[fill=25x4x4color79] (10,-5) rectangle (11,-6);
\node at (10.5,-5.5) {79};
\draw[fill=25x4x4color80] (11,-5) rectangle (12,-6);
\node at (11.5,-5.5) {80};
\draw[fill=25x4x4color81] (12,-5) rectangle (13,-6);
\node at (12.5,-5.5) {81};
\draw[fill=25x4x4color82] (13,-5) rectangle (14,-6);
\node at (13.5,-5.5) {82};
\draw[fill=25x4x4color83] (14,-5) rectangle (15,-6);
\node at (14.5,-5.5) {83};
\draw[fill=25x4x4color84] (15,-5) rectangle (16,-6);
\node at (15.5,-5.5) {84};
\draw[fill=25x4x4color85] (16,-5) rectangle (17,-6);
\node at (16.5,-5.5) {85};
\draw[fill=25x4x4color86] (17,-5) rectangle (18,-6);
\node at (17.5,-5.5) {86};
\draw[fill=25x4x4color87] (18,-5) rectangle (19,-6);
\node at (18.5,-5.5) {87};
\draw[fill=25x4x4color88] (19,-5) rectangle (20,-6);
\node at (19.5,-5.5) {88};
\draw[fill=25x4x4color89] (20,-5) rectangle (21,-6);
\node at (20.5,-5.5) {89};
\draw[fill=25x4x4color90] (21,-5) rectangle (22,-6);
\node at (21.5,-5.5) {90};
\draw[fill=25x4x4color91] (22,-5) rectangle (23,-6);
\node at (22.5,-5.5) {91};
\draw[fill=25x4x4color92] (23,-5) rectangle (24,-6);
\node at (23.5,-5.5) {92};
\draw[fill=25x4x4color93] (24,-5) rectangle (25,-6);
\node at (24.5,-5.5) {93};
\draw[fill=25x4x4color68] (0,-6) rectangle (1,-7);
\node at (0.5,-6.5) {68};
\draw[fill=25x4x4color67] (1,-6) rectangle (2,-7);
\node at (1.5,-6.5) {67};
\draw[fill=red] (2,-6) rectangle (3,-7);
\draw[fill=25x4x4color61] (3,-6) rectangle (4,-7);
\node at (3.5,-6.5) {61};
\draw[fill=25x4x4color62] (4,-6) rectangle (5,-7);
\node at (4.5,-6.5) {62};
\draw[fill=25x4x4color63] (5,-6) rectangle (6,-7);
\node at (5.5,-6.5) {63};
\draw[fill=25x4x4color64] (6,-6) rectangle (7,-7);
\node at (6.5,-6.5) {64};
\draw[fill=25x4x4color65] (7,-6) rectangle (8,-7);
\node at (7.5,-6.5) {65};
\draw[fill=25x4x4color66] (8,-6) rectangle (9,-7);
\node at (8.5,-6.5) {66};
\draw[fill=25x4x4color67] (9,-6) rectangle (10,-7);
\node at (9.5,-6.5) {67};
\draw[fill=25x4x4color68] (10,-6) rectangle (11,-7);
\node at (10.5,-6.5) {68};
\draw[fill=25x4x4color69] (11,-6) rectangle (12,-7);
\node at (11.5,-6.5) {69};
\draw[fill=25x4x4color70] (12,-6) rectangle (13,-7);
\node at (12.5,-6.5) {70};
\draw[fill=25x4x4color71] (13,-6) rectangle (14,-7);
\node at (13.5,-6.5) {71};
\draw[fill=25x4x4color72] (14,-6) rectangle (15,-7);
\node at (14.5,-6.5) {72};
\draw[fill=25x4x4color73] (15,-6) rectangle (16,-7);
\node at (15.5,-6.5) {73};
\draw[fill=25x4x4color74] (16,-6) rectangle (17,-7);
\node at (16.5,-6.5) {74};
\draw[fill=25x4x4color75] (17,-6) rectangle (18,-7);
\node at (17.5,-6.5) {75};
\draw[fill=25x4x4color76] (18,-6) rectangle (19,-7);
\node at (18.5,-6.5) {76};
\draw[fill=25x4x4color77] (19,-6) rectangle (20,-7);
\node at (19.5,-6.5) {77};
\draw[fill=25x4x4color78] (20,-6) rectangle (21,-7);
\node at (20.5,-6.5) {78};
\draw[fill=25x4x4color79] (21,-6) rectangle (22,-7);
\node at (21.5,-6.5) {79};
\draw[fill=25x4x4color80] (22,-6) rectangle (23,-7);
\node at (22.5,-6.5) {80};
\draw[fill=25x4x4color81] (23,-6) rectangle (24,-7);
\node at (23.5,-6.5) {81};
\draw[fill=25x4x4color82] (24,-6) rectangle (25,-7);
\node at (24.5,-6.5) {82};
\draw[fill=25x4x4color67] (0,-7) rectangle (1,-8);
\node at (0.5,-7.5) {67};
\draw[fill=25x4x4color66] (1,-7) rectangle (2,-8);
\node at (1.5,-7.5) {66};
\draw[fill=25x4x4color61] (2,-7) rectangle (3,-8);
\node at (2.5,-7.5) {61};
\draw[fill=25x4x4color60] (3,-7) rectangle (4,-8);
\node at (3.5,-7.5) {60};
\draw[fill=25x4x4color53] (4,-7) rectangle (5,-8);
\node at (4.5,-7.5) {53};
\draw[fill=25x4x4color52] (5,-7) rectangle (6,-8);
\node at (5.5,-7.5) {52};
\draw[fill=25x4x4color51] (6,-7) rectangle (7,-8);
\node at (6.5,-7.5) {51};
\draw[fill=25x4x4color50] (7,-7) rectangle (8,-8);
\node at (7.5,-7.5) {50};
\draw[fill=25x4x4color49] (8,-7) rectangle (9,-8);
\node at (8.5,-7.5) {49};
\draw[fill=25x4x4color48] (9,-7) rectangle (10,-8);
\node at (9.5,-7.5) {48};
\draw[fill=25x4x4color47] (10,-7) rectangle (11,-8);
\node at (10.5,-7.5) {47};
\draw[fill=25x4x4color46] (11,-7) rectangle (12,-8);
\node at (11.5,-7.5) {46};
\draw[fill=25x4x4color45] (12,-7) rectangle (13,-8);
\node at (12.5,-7.5) {45};
\draw[fill=25x4x4color44] (13,-7) rectangle (14,-8);
\node at (13.5,-7.5) {44};
\draw[fill=25x4x4color43] (14,-7) rectangle (15,-8);
\node at (14.5,-7.5) {43};
\draw[fill=25x4x4color42] (15,-7) rectangle (16,-8);
\node at (15.5,-7.5) {42};
\draw[fill=25x4x4color41] (16,-7) rectangle (17,-8);
\node at (16.5,-7.5) {41};
\draw[fill=25x4x4color40] (17,-7) rectangle (18,-8);
\node at (17.5,-7.5) {40};
\draw[fill=25x4x4color39] (18,-7) rectangle (19,-8);
\node at (18.5,-7.5) {39};
\draw[fill=25x4x4color38] (19,-7) rectangle (20,-8);
\node at (19.5,-7.5) {38};
\draw[fill=25x4x4color37] (20,-7) rectangle (21,-8);
\node at (20.5,-7.5) {37};
\draw[fill=25x4x4color36] (21,-7) rectangle (22,-8);
\node at (21.5,-7.5) {36};
\draw[fill=25x4x4color35] (22,-7) rectangle (23,-8);
\node at (22.5,-7.5) {35};
\draw[fill=25x4x4color34] (23,-7) rectangle (24,-8);
\node at (23.5,-7.5) {34};
\draw[fill=25x4x4color33] (24,-7) rectangle (25,-8);
\node at (24.5,-7.5) {33};
\draw[fill=red] (0,-8) rectangle (1,-9);
\draw[fill=25x4x4color63] (1,-8) rectangle (2,-9);
\node at (1.5,-8.5) {63};
\draw[fill=25x4x4color62] (2,-8) rectangle (3,-9);
\node at (2.5,-8.5) {62};
\draw[fill=25x4x4color59] (3,-8) rectangle (4,-9);
\node at (3.5,-8.5) {59};
\draw[fill=25x4x4color54] (4,-8) rectangle (5,-9);
\node at (4.5,-8.5) {54};
\draw[fill=25x4x4color53] (5,-8) rectangle (6,-9);
\node at (5.5,-8.5) {53};
\draw[fill=25x4x4color52] (6,-8) rectangle (7,-9);
\node at (6.5,-8.5) {52};
\draw[fill=25x4x4color51] (7,-8) rectangle (8,-9);
\node at (7.5,-8.5) {51};
\draw[fill=25x4x4color50] (8,-8) rectangle (9,-9);
\node at (8.5,-8.5) {50};
\draw[fill=25x4x4color49] (9,-8) rectangle (10,-9);
\node at (9.5,-8.5) {49};
\draw[fill=25x4x4color48] (10,-8) rectangle (11,-9);
\node at (10.5,-8.5) {48};
\draw[fill=25x4x4color47] (11,-8) rectangle (12,-9);
\node at (11.5,-8.5) {47};
\draw[fill=25x4x4color46] (12,-8) rectangle (13,-9);
\node at (12.5,-8.5) {46};
\draw[fill=25x4x4color45] (13,-8) rectangle (14,-9);
\node at (13.5,-8.5) {45};
\draw[fill=25x4x4color44] (14,-8) rectangle (15,-9);
\node at (14.5,-8.5) {44};
\draw[fill=25x4x4color43] (15,-8) rectangle (16,-9);
\node at (15.5,-8.5) {43};
\draw[fill=25x4x4color42] (16,-8) rectangle (17,-9);
\node at (16.5,-8.5) {42};
\draw[fill=25x4x4color41] (17,-8) rectangle (18,-9);
\node at (17.5,-8.5) {41};
\draw[fill=25x4x4color40] (18,-8) rectangle (19,-9);
\node at (18.5,-8.5) {40};
\draw[fill=25x4x4color39] (19,-8) rectangle (20,-9);
\node at (19.5,-8.5) {39};
\draw[fill=25x4x4color38] (20,-8) rectangle (21,-9);
\node at (20.5,-8.5) {38};
\draw[fill=25x4x4color37] (21,-8) rectangle (22,-9);
\node at (21.5,-8.5) {37};
\draw[fill=25x4x4color36] (22,-8) rectangle (23,-9);
\node at (22.5,-8.5) {36};
\draw[fill=25x4x4color35] (23,-8) rectangle (24,-9);
\node at (23.5,-8.5) {35};
\draw[fill=red] (24,-8) rectangle (25,-9);
\draw[fill=red] (0,-10) rectangle (1,-11);
\draw[fill=25x4x4color1] (1,-10) rectangle (2,-11);
\node at (1.5,-10.5) {1};
\draw[fill=red] (2,-10) rectangle (3,-11);
\draw[fill=25x4x4color1] (3,-10) rectangle (4,-11);
\node at (3.5,-10.5) {1};
\draw[fill=red] (4,-10) rectangle (5,-11);
\draw[fill=25x4x4color3] (5,-10) rectangle (6,-11);
\node at (5.5,-10.5) {3};
\draw[fill=25x4x4color4] (6,-10) rectangle (7,-11);
\node at (6.5,-10.5) {4};
\draw[fill=25x4x4color5] (7,-10) rectangle (8,-11);
\node at (7.5,-10.5) {5};
\draw[fill=25x4x4color6] (8,-10) rectangle (9,-11);
\node at (8.5,-10.5) {6};
\draw[fill=25x4x4color7] (9,-10) rectangle (10,-11);
\node at (9.5,-10.5) {7};
\draw[fill=red] (10,-10) rectangle (11,-11);
\draw[fill=25x4x4color13] (11,-10) rectangle (12,-11);
\node at (11.5,-10.5) {13};
\draw[fill=25x4x4color14] (12,-10) rectangle (13,-11);
\node at (12.5,-10.5) {14};
\draw[fill=25x4x4color15] (13,-10) rectangle (14,-11);
\node at (13.5,-10.5) {15};
\draw[fill=25x4x4color16] (14,-10) rectangle (15,-11);
\node at (14.5,-10.5) {16};
\draw[fill=25x4x4color17] (15,-10) rectangle (16,-11);
\node at (15.5,-10.5) {17};
\draw[fill=red] (16,-10) rectangle (17,-11);
\draw[fill=25x4x4color23] (17,-10) rectangle (18,-11);
\node at (17.5,-10.5) {23};
\draw[fill=25x4x4color24] (18,-10) rectangle (19,-11);
\node at (18.5,-10.5) {24};
\draw[fill=25x4x4color25] (19,-10) rectangle (20,-11);
\node at (19.5,-10.5) {25};
\draw[fill=25x4x4color26] (20,-10) rectangle (21,-11);
\node at (20.5,-10.5) {26};
\draw[fill=25x4x4color27] (21,-10) rectangle (22,-11);
\node at (21.5,-10.5) {27};
\draw[fill=red] (22,-10) rectangle (23,-11);
\draw[fill=25x4x4color1] (23,-10) rectangle (24,-11);
\node at (23.5,-10.5) {1};
\draw[fill=red] (24,-10) rectangle (25,-11);
\draw[fill=25x4x4color9] (0,-11) rectangle (1,-12);
\node at (0.5,-11.5) {9};
\draw[fill=25x4x4color8] (1,-11) rectangle (2,-12);
\node at (1.5,-11.5) {8};
\draw[fill=25x4x4color1] (2,-11) rectangle (3,-12);
\node at (2.5,-11.5) {1};
\draw[fill=red] (3,-11) rectangle (4,-12);
\draw[fill=25x4x4color1] (4,-11) rectangle (5,-12);
\node at (4.5,-11.5) {1};
\draw[fill=25x4x4color2] (5,-11) rectangle (6,-12);
\node at (5.5,-11.5) {2};
\draw[fill=25x4x4color3] (6,-11) rectangle (7,-12);
\node at (6.5,-11.5) {3};
\draw[fill=25x4x4color4] (7,-11) rectangle (8,-12);
\node at (7.5,-11.5) {4};
\draw[fill=25x4x4color5] (8,-11) rectangle (9,-12);
\node at (8.5,-11.5) {5};
\draw[fill=25x4x4color8] (9,-11) rectangle (10,-12);
\node at (9.5,-11.5) {8};
\draw[fill=25x4x4color9] (10,-11) rectangle (11,-12);
\node at (10.5,-11.5) {9};
\draw[fill=25x4x4color12] (11,-11) rectangle (12,-12);
\node at (11.5,-11.5) {12};
\draw[fill=25x4x4color13] (12,-11) rectangle (13,-12);
\node at (12.5,-11.5) {13};
\draw[fill=25x4x4color14] (13,-11) rectangle (14,-12);
\node at (13.5,-11.5) {14};
\draw[fill=25x4x4color15] (14,-11) rectangle (15,-12);
\node at (14.5,-11.5) {15};
\draw[fill=25x4x4color18] (15,-11) rectangle (16,-12);
\node at (15.5,-11.5) {18};
\draw[fill=25x4x4color19] (16,-11) rectangle (17,-12);
\node at (16.5,-11.5) {19};
\draw[fill=25x4x4color22] (17,-11) rectangle (18,-12);
\node at (17.5,-11.5) {22};
\draw[fill=25x4x4color23] (18,-11) rectangle (19,-12);
\node at (18.5,-11.5) {23};
\draw[fill=25x4x4color24] (19,-11) rectangle (20,-12);
\node at (19.5,-11.5) {24};
\draw[fill=25x4x4color25] (20,-11) rectangle (21,-12);
\node at (20.5,-11.5) {25};
\draw[fill=25x4x4color28] (21,-11) rectangle (22,-12);
\node at (21.5,-11.5) {28};
\draw[fill=25x4x4color29] (22,-11) rectangle (23,-12);
\node at (22.5,-11.5) {29};
\draw[fill=25x4x4color30] (23,-11) rectangle (24,-12);
\node at (23.5,-11.5) {30};
\draw[fill=25x4x4color31] (24,-11) rectangle (25,-12);
\node at (24.5,-11.5) {31};
\draw[fill=25x4x4color68] (0,-12) rectangle (1,-13);
\node at (0.5,-12.5) {68};
\draw[fill=25x4x4color67] (1,-12) rectangle (2,-13);
\node at (1.5,-12.5) {67};
\draw[fill=25x4x4color8] (2,-12) rectangle (3,-13);
\node at (2.5,-12.5) {8};
\draw[fill=25x4x4color1] (3,-12) rectangle (4,-13);
\node at (3.5,-12.5) {1};
\draw[fill=red] (4,-12) rectangle (5,-13);
\draw[fill=25x4x4color1] (5,-12) rectangle (6,-13);
\node at (5.5,-12.5) {1};
\draw[fill=red] (6,-12) rectangle (7,-13);
\draw[fill=25x4x4color3] (7,-12) rectangle (8,-13);
\node at (7.5,-12.5) {3};
\draw[fill=red] (8,-12) rectangle (9,-13);
\draw[fill=25x4x4color9] (9,-12) rectangle (10,-13);
\node at (9.5,-12.5) {9};
\draw[fill=25x4x4color10] (10,-12) rectangle (11,-13);
\node at (10.5,-12.5) {10};
\draw[fill=25x4x4color11] (11,-12) rectangle (12,-13);
\node at (11.5,-12.5) {11};
\draw[fill=red] (12,-12) rectangle (13,-13);
\draw[fill=25x4x4color3] (13,-12) rectangle (14,-13);
\node at (13.5,-12.5) {3};
\draw[fill=red] (14,-12) rectangle (15,-13);
\draw[fill=25x4x4color19] (15,-12) rectangle (16,-13);
\node at (15.5,-12.5) {19};
\draw[fill=25x4x4color20] (16,-12) rectangle (17,-13);
\node at (16.5,-12.5) {20};
\draw[fill=25x4x4color21] (17,-12) rectangle (18,-13);
\node at (17.5,-12.5) {21};
\draw[fill=red] (18,-12) rectangle (19,-13);
\draw[fill=25x4x4color3] (19,-12) rectangle (20,-13);
\node at (19.5,-12.5) {3};
\draw[fill=red] (20,-12) rectangle (21,-13);
\draw[fill=25x4x4color29] (21,-12) rectangle (22,-13);
\node at (21.5,-12.5) {29};
\draw[fill=25x4x4color30] (22,-12) rectangle (23,-13);
\node at (22.5,-12.5) {30};
\draw[fill=25x4x4color31] (23,-12) rectangle (24,-13);
\node at (23.5,-12.5) {31};
\draw[fill=25x4x4color32] (24,-12) rectangle (25,-13);
\node at (24.5,-12.5) {32};
\draw[fill=25x4x4color69] (0,-13) rectangle (1,-14);
\node at (0.5,-13.5) {69};
\draw[fill=25x4x4color68] (1,-13) rectangle (2,-14);
\node at (1.5,-13.5) {68};
\draw[fill=25x4x4color59] (2,-13) rectangle (3,-14);
\node at (2.5,-13.5) {59};
\draw[fill=25x4x4color58] (3,-13) rectangle (4,-14);
\node at (3.5,-13.5) {58};
\draw[fill=25x4x4color55] (4,-13) rectangle (5,-14);
\node at (4.5,-13.5) {55};
\draw[fill=25x4x4color54] (5,-13) rectangle (6,-14);
\node at (5.5,-13.5) {54};
\draw[fill=25x4x4color53] (6,-13) rectangle (7,-14);
\node at (6.5,-13.5) {53};
\draw[fill=25x4x4color52] (7,-13) rectangle (8,-14);
\node at (7.5,-13.5) {52};
\draw[fill=25x4x4color51] (8,-13) rectangle (9,-14);
\node at (8.5,-13.5) {51};
\draw[fill=25x4x4color50] (9,-13) rectangle (10,-14);
\node at (9.5,-13.5) {50};
\draw[fill=25x4x4color49] (10,-13) rectangle (11,-14);
\node at (10.5,-13.5) {49};
\draw[fill=25x4x4color48] (11,-13) rectangle (12,-14);
\node at (11.5,-13.5) {48};
\draw[fill=25x4x4color47] (12,-13) rectangle (13,-14);
\node at (12.5,-13.5) {47};
\draw[fill=25x4x4color46] (13,-13) rectangle (14,-14);
\node at (13.5,-13.5) {46};
\draw[fill=25x4x4color45] (14,-13) rectangle (15,-14);
\node at (14.5,-13.5) {45};
\draw[fill=25x4x4color44] (15,-13) rectangle (16,-14);
\node at (15.5,-13.5) {44};
\draw[fill=25x4x4color43] (16,-13) rectangle (17,-14);
\node at (16.5,-13.5) {43};
\draw[fill=25x4x4color42] (17,-13) rectangle (18,-14);
\node at (17.5,-13.5) {42};
\draw[fill=25x4x4color41] (18,-13) rectangle (19,-14);
\node at (18.5,-13.5) {41};
\draw[fill=25x4x4color40] (19,-13) rectangle (20,-14);
\node at (19.5,-13.5) {40};
\draw[fill=25x4x4color39] (20,-13) rectangle (21,-14);
\node at (20.5,-13.5) {39};
\draw[fill=25x4x4color38] (21,-13) rectangle (22,-14);
\node at (21.5,-13.5) {38};
\draw[fill=25x4x4color37] (22,-13) rectangle (23,-14);
\node at (22.5,-13.5) {37};
\draw[fill=25x4x4color36] (23,-13) rectangle (24,-14);
\node at (23.5,-13.5) {36};
\draw[fill=25x4x4color33] (24,-13) rectangle (25,-14);
\node at (24.5,-13.5) {33};
\draw[fill=25x4x4color1] (0,-15) rectangle (1,-16);
\node at (0.5,-15.5) {1};
\draw[fill=red] (1,-15) rectangle (2,-16);
\draw[fill=25x4x4color5] (2,-15) rectangle (3,-16);
\node at (2.5,-15.5) {5};
\draw[fill=25x4x4color4] (3,-15) rectangle (4,-16);
\node at (3.5,-15.5) {4};
\draw[fill=25x4x4color3] (4,-15) rectangle (5,-16);
\node at (4.5,-15.5) {3};
\draw[fill=red] (5,-15) rectangle (6,-16);
\draw[fill=25x4x4color1] (6,-15) rectangle (7,-16);
\node at (6.5,-15.5) {1};
\draw[fill=red] (7,-15) rectangle (8,-16);
\draw[fill=25x4x4color1] (8,-15) rectangle (9,-16);
\node at (8.5,-15.5) {1};
\draw[fill=red] (9,-15) rectangle (10,-16);
\draw[fill=25x4x4color1] (10,-15) rectangle (11,-16);
\node at (10.5,-15.5) {1};
\draw[fill=red] (11,-15) rectangle (12,-16);
\draw[fill=25x4x4color1] (12,-15) rectangle (13,-16);
\node at (12.5,-15.5) {1};
\draw[fill=red] (13,-15) rectangle (14,-16);
\draw[fill=25x4x4color1] (14,-15) rectangle (15,-16);
\node at (14.5,-15.5) {1};
\draw[fill=red] (15,-15) rectangle (16,-16);
\draw[fill=25x4x4color1] (16,-15) rectangle (17,-16);
\node at (16.5,-15.5) {1};
\draw[fill=red] (17,-15) rectangle (18,-16);
\draw[fill=25x4x4color1] (18,-15) rectangle (19,-16);
\node at (18.5,-15.5) {1};
\draw[fill=red] (19,-15) rectangle (20,-16);
\draw[fill=25x4x4color1] (20,-15) rectangle (21,-16);
\node at (20.5,-15.5) {1};
\draw[fill=red] (21,-15) rectangle (22,-16);
\draw[fill=25x4x4color1] (22,-15) rectangle (23,-16);
\node at (22.5,-15.5) {1};
\draw[fill=red] (23,-15) rectangle (24,-16);
\draw[fill=25x4x4color1] (24,-15) rectangle (25,-16);
\node at (24.5,-15.5) {1};
\draw[fill=red] (0,-16) rectangle (1,-17);
\draw[fill=25x4x4color7] (1,-16) rectangle (2,-17);
\node at (1.5,-16.5) {7};
\draw[fill=25x4x4color6] (2,-16) rectangle (3,-17);
\node at (2.5,-16.5) {6};
\draw[fill=25x4x4color3] (3,-16) rectangle (4,-17);
\node at (3.5,-16.5) {3};
\draw[fill=25x4x4color2] (4,-16) rectangle (5,-17);
\node at (4.5,-16.5) {2};
\draw[fill=25x4x4color1] (5,-16) rectangle (6,-17);
\node at (5.5,-16.5) {1};
\draw[fill=red] (6,-16) rectangle (7,-17);
\draw[fill=25x4x4color1] (7,-16) rectangle (8,-17);
\node at (7.5,-16.5) {1};
\draw[fill=red] (8,-16) rectangle (9,-17);
\draw[fill=25x4x4color1] (9,-16) rectangle (10,-17);
\node at (9.5,-16.5) {1};
\draw[fill=25x4x4color2] (10,-16) rectangle (11,-17);
\node at (10.5,-16.5) {2};
\draw[fill=25x4x4color1] (11,-16) rectangle (12,-17);
\node at (11.5,-16.5) {1};
\draw[fill=red] (12,-16) rectangle (13,-17);
\draw[fill=25x4x4color1] (13,-16) rectangle (14,-17);
\node at (13.5,-16.5) {1};
\draw[fill=red] (14,-16) rectangle (15,-17);
\draw[fill=25x4x4color1] (15,-16) rectangle (16,-17);
\node at (15.5,-16.5) {1};
\draw[fill=25x4x4color2] (16,-16) rectangle (17,-17);
\node at (16.5,-16.5) {2};
\draw[fill=25x4x4color1] (17,-16) rectangle (18,-17);
\node at (17.5,-16.5) {1};
\draw[fill=red] (18,-16) rectangle (19,-17);
\draw[fill=25x4x4color1] (19,-16) rectangle (20,-17);
\node at (19.5,-16.5) {1};
\draw[fill=red] (20,-16) rectangle (21,-17);
\draw[fill=25x4x4color1] (21,-16) rectangle (22,-17);
\node at (21.5,-16.5) {1};
\draw[fill=25x4x4color2] (22,-16) rectangle (23,-17);
\node at (22.5,-16.5) {2};
\draw[fill=25x4x4color1] (23,-16) rectangle (24,-17);
\node at (23.5,-16.5) {1};
\draw[fill=red] (24,-16) rectangle (25,-17);
\draw[fill=25x4x4color69] (0,-17) rectangle (1,-18);
\node at (0.5,-17.5) {69};
\draw[fill=25x4x4color68] (1,-17) rectangle (2,-18);
\node at (1.5,-17.5) {68};
\draw[fill=25x4x4color7] (2,-17) rectangle (3,-18);
\node at (2.5,-17.5) {7};
\draw[fill=red] (3,-17) rectangle (4,-18);
\draw[fill=25x4x4color1] (4,-17) rectangle (5,-18);
\node at (4.5,-17.5) {1};
\draw[fill=red] (5,-17) rectangle (6,-18);
\draw[fill=25x4x4color1] (6,-17) rectangle (7,-18);
\node at (6.5,-17.5) {1};
\draw[fill=25x4x4color2] (7,-17) rectangle (8,-18);
\node at (7.5,-17.5) {2};
\draw[fill=25x4x4color1] (8,-17) rectangle (9,-18);
\node at (8.5,-17.5) {1};
\draw[fill=red] (9,-17) rectangle (10,-18);
\draw[fill=25x4x4color3] (10,-17) rectangle (11,-18);
\node at (10.5,-17.5) {3};
\draw[fill=red] (11,-17) rectangle (12,-18);
\draw[fill=25x4x4color1] (12,-17) rectangle (13,-18);
\node at (12.5,-17.5) {1};
\draw[fill=25x4x4color2] (13,-17) rectangle (14,-18);
\node at (13.5,-17.5) {2};
\draw[fill=25x4x4color1] (14,-17) rectangle (15,-18);
\node at (14.5,-17.5) {1};
\draw[fill=red] (15,-17) rectangle (16,-18);
\draw[fill=25x4x4color3] (16,-17) rectangle (17,-18);
\node at (16.5,-17.5) {3};
\draw[fill=red] (17,-17) rectangle (18,-18);
\draw[fill=25x4x4color1] (18,-17) rectangle (19,-18);
\node at (18.5,-17.5) {1};
\draw[fill=25x4x4color2] (19,-17) rectangle (20,-18);
\node at (19.5,-17.5) {2};
\draw[fill=25x4x4color1] (20,-17) rectangle (21,-18);
\node at (20.5,-17.5) {1};
\draw[fill=red] (21,-17) rectangle (22,-18);
\draw[fill=25x4x4color3] (22,-17) rectangle (23,-18);
\node at (22.5,-17.5) {3};
\draw[fill=red] (23,-17) rectangle (24,-18);
\draw[fill=25x4x4color1] (24,-17) rectangle (25,-18);
\node at (24.5,-17.5) {1};
\draw[fill=25x4x4color70] (0,-18) rectangle (1,-19);
\node at (0.5,-18.5) {70};
\draw[fill=25x4x4color69] (1,-18) rectangle (2,-19);
\node at (1.5,-18.5) {69};
\draw[fill=red] (2,-18) rectangle (3,-19);
\draw[fill=25x4x4color57] (3,-18) rectangle (4,-19);
\node at (3.5,-18.5) {57};
\draw[fill=25x4x4color56] (4,-18) rectangle (5,-19);
\node at (4.5,-18.5) {56};
\draw[fill=25x4x4color55] (5,-18) rectangle (6,-19);
\node at (5.5,-18.5) {55};
\draw[fill=25x4x4color54] (6,-18) rectangle (7,-19);
\node at (6.5,-18.5) {54};
\draw[fill=25x4x4color53] (7,-18) rectangle (8,-19);
\node at (7.5,-18.5) {53};
\draw[fill=25x4x4color52] (8,-18) rectangle (9,-19);
\node at (8.5,-18.5) {52};
\draw[fill=25x4x4color51] (9,-18) rectangle (10,-19);
\node at (9.5,-18.5) {51};
\draw[fill=25x4x4color50] (10,-18) rectangle (11,-19);
\node at (10.5,-18.5) {50};
\draw[fill=25x4x4color49] (11,-18) rectangle (12,-19);
\node at (11.5,-18.5) {49};
\draw[fill=25x4x4color48] (12,-18) rectangle (13,-19);
\node at (12.5,-18.5) {48};
\draw[fill=25x4x4color47] (13,-18) rectangle (14,-19);
\node at (13.5,-18.5) {47};
\draw[fill=25x4x4color46] (14,-18) rectangle (15,-19);
\node at (14.5,-18.5) {46};
\draw[fill=25x4x4color45] (15,-18) rectangle (16,-19);
\node at (15.5,-18.5) {45};
\draw[fill=25x4x4color44] (16,-18) rectangle (17,-19);
\node at (16.5,-18.5) {44};
\draw[fill=25x4x4color43] (17,-18) rectangle (18,-19);
\node at (17.5,-18.5) {43};
\draw[fill=25x4x4color42] (18,-18) rectangle (19,-19);
\node at (18.5,-18.5) {42};
\draw[fill=25x4x4color41] (19,-18) rectangle (20,-19);
\node at (19.5,-18.5) {41};
\draw[fill=25x4x4color40] (20,-18) rectangle (21,-19);
\node at (20.5,-18.5) {40};
\draw[fill=25x4x4color39] (21,-18) rectangle (22,-19);
\node at (21.5,-18.5) {39};
\draw[fill=25x4x4color38] (22,-18) rectangle (23,-19);
\node at (22.5,-18.5) {38};
\draw[fill=25x4x4color37] (23,-18) rectangle (24,-19);
\node at (23.5,-18.5) {37};
\draw[fill=red] (24,-18) rectangle (25,-19);

\end{scope}
\end{tikzpicture}
\end{center}

%% file: 4-4-28-tikz.txt
\definecolor{28x4x4color1}{RGB}{255, 128, 128}
\definecolor{28x4x4color2}{RGB}{255, 130, 130}
\definecolor{28x4x4color3}{RGB}{255, 132, 132}
\definecolor{28x4x4color4}{RGB}{255, 134, 134}
\definecolor{28x4x4color5}{RGB}{255, 137, 137}
\definecolor{28x4x4color6}{RGB}{255, 139, 139}
\definecolor{28x4x4color7}{RGB}{255, 141, 141}
\definecolor{28x4x4color8}{RGB}{255, 143, 143}
\definecolor{28x4x4color9}{RGB}{255, 146, 146}
\definecolor{28x4x4color10}{RGB}{255, 148, 148}
\definecolor{28x4x4color11}{RGB}{255, 150, 150}
\definecolor{28x4x4color12}{RGB}{255, 152, 152}
\definecolor{28x4x4color13}{RGB}{255, 155, 155}
\definecolor{28x4x4color14}{RGB}{255, 157, 157}
\definecolor{28x4x4color15}{RGB}{255, 159, 159}
\definecolor{28x4x4color16}{RGB}{255, 162, 162}
\definecolor{28x4x4color17}{RGB}{255, 164, 164}
\definecolor{28x4x4color18}{RGB}{255, 166, 166}
\definecolor{28x4x4color19}{RGB}{255, 168, 168}
\definecolor{28x4x4color20}{RGB}{255, 171, 171}
\definecolor{28x4x4color21}{RGB}{255, 173, 173}
\definecolor{28x4x4color22}{RGB}{255, 175, 175}
\definecolor{28x4x4color23}{RGB}{255, 177, 177}
\definecolor{28x4x4color24}{RGB}{255, 180, 180}
\definecolor{28x4x4color25}{RGB}{255, 182, 182}
\definecolor{28x4x4color26}{RGB}{255, 184, 184}
\definecolor{28x4x4color27}{RGB}{255, 186, 186}
\definecolor{28x4x4color28}{RGB}{255, 189, 189}
\definecolor{28x4x4color29}{RGB}{255, 191, 191}
\definecolor{28x4x4color30}{RGB}{255, 193, 193}
\definecolor{28x4x4color31}{RGB}{255, 196, 196}
\definecolor{28x4x4color32}{RGB}{255, 198, 198}
\definecolor{28x4x4color33}{RGB}{255, 200, 200}
\definecolor{28x4x4color34}{RGB}{255, 202, 202}
\definecolor{28x4x4color35}{RGB}{255, 205, 205}
\definecolor{28x4x4color36}{RGB}{255, 207, 207}
\definecolor{28x4x4color37}{RGB}{255, 209, 209}
\definecolor{28x4x4color38}{RGB}{255, 211, 211}
\definecolor{28x4x4color39}{RGB}{255, 214, 214}
\definecolor{28x4x4color40}{RGB}{255, 216, 216}
\definecolor{28x4x4color41}{RGB}{255, 218, 218}
\definecolor{28x4x4color42}{RGB}{255, 220, 220}
\definecolor{28x4x4color43}{RGB}{255, 223, 223}
\definecolor{28x4x4color44}{RGB}{255, 225, 225}
\definecolor{28x4x4color45}{RGB}{255, 227, 227}
\definecolor{28x4x4color46}{RGB}{255, 230, 230}
\definecolor{28x4x4color47}{RGB}{255, 232, 232}
\definecolor{28x4x4color48}{RGB}{255, 234, 234}
\definecolor{28x4x4color49}{RGB}{255, 236, 236}
\definecolor{28x4x4color50}{RGB}{255, 239, 239}
\definecolor{28x4x4color51}{RGB}{255, 241, 241}
\definecolor{28x4x4color52}{RGB}{255, 243, 243}
\definecolor{28x4x4color53}{RGB}{255, 245, 245}
\definecolor{28x4x4color54}{RGB}{255, 248, 248}
\definecolor{28x4x4color55}{RGB}{255, 250, 250}
\definecolor{28x4x4color56}{RGB}{255, 255, 255}

\begin{center}
\begin{tikzpicture}[scale=0.5]
\begin{scope}
\draw[very thick] (0,-0) rectangle (28,-4);
\foreach \b in {0,...,3}
 \foreach \c in {0,...,27}
  \draw (\c,\b*-1-0) rectangle (\c+1,(\b*-1-1);
\draw[very thick] (0,-5) rectangle (28,-9);
\foreach \b in {0,...,3}
 \foreach \c in {0,...,27}
  \draw (\c,\b*-1-5) rectangle (\c+1,(\b*-1-6);
\draw[very thick] (0,-10) rectangle (28,-14);
\foreach \b in {0,...,3}
 \foreach \c in {0,...,27}
  \draw (\c,\b*-1-10) rectangle (\c+1,(\b*-1-11);
\draw[very thick] (0,-15) rectangle (28,-19);
\foreach \b in {0,...,3}
 \foreach \c in {0,...,27}
  \draw (\c,\b*-1-15) rectangle (\c+1,(\b*-1-16);

\draw[fill=red] (0,-0) rectangle (1,-1);
\draw[fill=28x4x4color9] (1,-0) rectangle (2,-1);
\node at (1.5,-0.5) {9};
\draw[fill=28x4x4color10] (2,-0) rectangle (3,-1);
\node at (2.5,-0.5) {10};
\draw[fill=28x4x4color29] (3,-0) rectangle (4,-1);
\node at (3.5,-0.5) {29};
\draw[fill=28x4x4color30] (4,-0) rectangle (5,-1);
\node at (4.5,-0.5) {30};
\draw[fill=28x4x4color31] (5,-0) rectangle (6,-1);
\node at (5.5,-0.5) {31};
\draw[fill=28x4x4color32] (6,-0) rectangle (7,-1);
\node at (6.5,-0.5) {32};
\draw[fill=28x4x4color33] (7,-0) rectangle (8,-1);
\node at (7.5,-0.5) {33};
\draw[fill=28x4x4color34] (8,-0) rectangle (9,-1);
\node at (8.5,-0.5) {34};
\draw[fill=28x4x4color35] (9,-0) rectangle (10,-1);
\node at (9.5,-0.5) {35};
\draw[fill=28x4x4color36] (10,-0) rectangle (11,-1);
\node at (10.5,-0.5) {36};
\draw[fill=28x4x4color37] (11,-0) rectangle (12,-1);
\node at (11.5,-0.5) {37};
\draw[fill=28x4x4color38] (12,-0) rectangle (13,-1);
\node at (12.5,-0.5) {38};
\draw[fill=28x4x4color39] (13,-0) rectangle (14,-1);
\node at (13.5,-0.5) {39};
\draw[fill=28x4x4color40] (14,-0) rectangle (15,-1);
\node at (14.5,-0.5) {40};
\draw[fill=28x4x4color41] (15,-0) rectangle (16,-1);
\node at (15.5,-0.5) {41};
\draw[fill=28x4x4color42] (16,-0) rectangle (17,-1);
\node at (16.5,-0.5) {42};
\draw[fill=28x4x4color43] (17,-0) rectangle (18,-1);
\node at (17.5,-0.5) {43};
\draw[fill=28x4x4color44] (18,-0) rectangle (19,-1);
\node at (18.5,-0.5) {44};
\draw[fill=28x4x4color45] (19,-0) rectangle (20,-1);
\node at (19.5,-0.5) {45};
\draw[fill=28x4x4color46] (20,-0) rectangle (21,-1);
\node at (20.5,-0.5) {46};
\draw[fill=28x4x4color47] (21,-0) rectangle (22,-1);
\node at (21.5,-0.5) {47};
\draw[fill=28x4x4color48] (22,-0) rectangle (23,-1);
\node at (22.5,-0.5) {48};
\draw[fill=28x4x4color49] (23,-0) rectangle (24,-1);
\node at (23.5,-0.5) {49};
\draw[fill=28x4x4color50] (24,-0) rectangle (25,-1);
\node at (24.5,-0.5) {50};
\draw[fill=28x4x4color51] (25,-0) rectangle (26,-1);
\node at (25.5,-0.5) {51};
\draw[fill=28x4x4color52] (26,-0) rectangle (27,-1);
\node at (26.5,-0.5) {52};
\draw[fill=28x4x4color53] (27,-0) rectangle (28,-1);
\node at (27.5,-0.5) {53};
\draw[fill=28x4x4color1] (0,-1) rectangle (1,-2);
\node at (0.5,-1.5) {1};
\draw[fill=28x4x4color8] (1,-1) rectangle (2,-2);
\node at (1.5,-1.5) {8};
\draw[fill=28x4x4color9] (2,-1) rectangle (3,-2);
\node at (2.5,-1.5) {9};
\draw[fill=28x4x4color28] (3,-1) rectangle (4,-2);
\node at (3.5,-1.5) {28};
\draw[fill=28x4x4color29] (4,-1) rectangle (5,-2);
\node at (4.5,-1.5) {29};
\draw[fill=28x4x4color30] (5,-1) rectangle (6,-2);
\node at (5.5,-1.5) {30};
\draw[fill=28x4x4color31] (6,-1) rectangle (7,-2);
\node at (6.5,-1.5) {31};
\draw[fill=28x4x4color32] (7,-1) rectangle (8,-2);
\node at (7.5,-1.5) {32};
\draw[fill=28x4x4color33] (8,-1) rectangle (9,-2);
\node at (8.5,-1.5) {33};
\draw[fill=28x4x4color34] (9,-1) rectangle (10,-2);
\node at (9.5,-1.5) {34};
\draw[fill=28x4x4color35] (10,-1) rectangle (11,-2);
\node at (10.5,-1.5) {35};
\draw[fill=28x4x4color36] (11,-1) rectangle (12,-2);
\node at (11.5,-1.5) {36};
\draw[fill=28x4x4color37] (12,-1) rectangle (13,-2);
\node at (12.5,-1.5) {37};
\draw[fill=28x4x4color38] (13,-1) rectangle (14,-2);
\node at (13.5,-1.5) {38};
\draw[fill=28x4x4color39] (14,-1) rectangle (15,-2);
\node at (14.5,-1.5) {39};
\draw[fill=28x4x4color40] (15,-1) rectangle (16,-2);
\node at (15.5,-1.5) {40};
\draw[fill=28x4x4color41] (16,-1) rectangle (17,-2);
\node at (16.5,-1.5) {41};
\draw[fill=28x4x4color42] (17,-1) rectangle (18,-2);
\node at (17.5,-1.5) {42};
\draw[fill=28x4x4color43] (18,-1) rectangle (19,-2);
\node at (18.5,-1.5) {43};
\draw[fill=28x4x4color44] (19,-1) rectangle (20,-2);
\node at (19.5,-1.5) {44};
\draw[fill=28x4x4color45] (20,-1) rectangle (21,-2);
\node at (20.5,-1.5) {45};
\draw[fill=28x4x4color46] (21,-1) rectangle (22,-2);
\node at (21.5,-1.5) {46};
\draw[fill=28x4x4color47] (22,-1) rectangle (23,-2);
\node at (22.5,-1.5) {47};
\draw[fill=28x4x4color48] (23,-1) rectangle (24,-2);
\node at (23.5,-1.5) {48};
\draw[fill=28x4x4color49] (24,-1) rectangle (25,-2);
\node at (24.5,-1.5) {49};
\draw[fill=28x4x4color50] (25,-1) rectangle (26,-2);
\node at (25.5,-1.5) {50};
\draw[fill=28x4x4color51] (26,-1) rectangle (27,-2);
\node at (26.5,-1.5) {51};
\draw[fill=28x4x4color52] (27,-1) rectangle (28,-2);
\node at (27.5,-1.5) {52};
\draw[fill=red] (0,-2) rectangle (1,-3);
\draw[fill=28x4x4color1] (1,-2) rectangle (2,-3);
\node at (1.5,-2.5) {1};
\draw[fill=red] (2,-2) rectangle (3,-3);
\draw[fill=28x4x4color27] (3,-2) rectangle (4,-3);
\node at (3.5,-2.5) {27};
\draw[fill=red] (4,-2) rectangle (5,-3);
\draw[fill=28x4x4color1] (5,-2) rectangle (6,-3);
\node at (5.5,-2.5) {1};
\draw[fill=red] (6,-2) rectangle (7,-3);
\draw[fill=28x4x4color1] (7,-2) rectangle (8,-3);
\node at (7.5,-2.5) {1};
\draw[fill=red] (8,-2) rectangle (9,-3);
\draw[fill=28x4x4color1] (9,-2) rectangle (10,-3);
\node at (9.5,-2.5) {1};
\draw[fill=red] (10,-2) rectangle (11,-3);
\draw[fill=28x4x4color1] (11,-2) rectangle (12,-3);
\node at (11.5,-2.5) {1};
\draw[fill=red] (12,-2) rectangle (13,-3);
\draw[fill=28x4x4color1] (13,-2) rectangle (14,-3);
\node at (13.5,-2.5) {1};
\draw[fill=red] (14,-2) rectangle (15,-3);
\draw[fill=28x4x4color1] (15,-2) rectangle (16,-3);
\node at (15.5,-2.5) {1};
\draw[fill=red] (16,-2) rectangle (17,-3);
\draw[fill=28x4x4color1] (17,-2) rectangle (18,-3);
\node at (17.5,-2.5) {1};
\draw[fill=red] (18,-2) rectangle (19,-3);
\draw[fill=28x4x4color1] (19,-2) rectangle (20,-3);
\node at (19.5,-2.5) {1};
\draw[fill=red] (20,-2) rectangle (21,-3);
\draw[fill=28x4x4color1] (21,-2) rectangle (22,-3);
\node at (21.5,-2.5) {1};
\draw[fill=red] (22,-2) rectangle (23,-3);
\draw[fill=28x4x4color1] (23,-2) rectangle (24,-3);
\node at (23.5,-2.5) {1};
\draw[fill=red] (24,-2) rectangle (25,-3);
\draw[fill=28x4x4color1] (25,-2) rectangle (26,-3);
\node at (25.5,-2.5) {1};
\draw[fill=red] (26,-2) rectangle (27,-3);
\draw[fill=28x4x4color1] (27,-2) rectangle (28,-3);
\node at (27.5,-2.5) {1};
\draw[fill=28x4x4color1] (0,-3) rectangle (1,-4);
\node at (0.5,-3.5) {1};
\draw[fill=red] (1,-3) rectangle (2,-4);
\draw[fill=28x4x4color7] (2,-3) rectangle (3,-4);
\node at (2.5,-3.5) {7};
\draw[fill=28x4x4color28] (3,-3) rectangle (4,-4);
\node at (3.5,-3.5) {28};
\draw[fill=28x4x4color29] (4,-3) rectangle (5,-4);
\node at (4.5,-3.5) {29};
\draw[fill=red] (5,-3) rectangle (6,-4);
\draw[fill=28x4x4color1] (6,-3) rectangle (7,-4);
\node at (6.5,-3.5) {1};
\draw[fill=red] (7,-3) rectangle (8,-4);
\draw[fill=28x4x4color1] (8,-3) rectangle (9,-4);
\node at (8.5,-3.5) {1};
\draw[fill=red] (9,-3) rectangle (10,-4);
\draw[fill=28x4x4color1] (10,-3) rectangle (11,-4);
\node at (10.5,-3.5) {1};
\draw[fill=red] (11,-3) rectangle (12,-4);
\draw[fill=28x4x4color1] (12,-3) rectangle (13,-4);
\node at (12.5,-3.5) {1};
\draw[fill=red] (13,-3) rectangle (14,-4);
\draw[fill=28x4x4color1] (14,-3) rectangle (15,-4);
\node at (14.5,-3.5) {1};
\draw[fill=red] (15,-3) rectangle (16,-4);
\draw[fill=28x4x4color1] (16,-3) rectangle (17,-4);
\node at (16.5,-3.5) {1};
\draw[fill=red] (17,-3) rectangle (18,-4);
\draw[fill=28x4x4color1] (18,-3) rectangle (19,-4);
\node at (18.5,-3.5) {1};
\draw[fill=red] (19,-3) rectangle (20,-4);
\draw[fill=28x4x4color1] (20,-3) rectangle (21,-4);
\node at (20.5,-3.5) {1};
\draw[fill=red] (21,-3) rectangle (22,-4);
\draw[fill=28x4x4color1] (22,-3) rectangle (23,-4);
\node at (22.5,-3.5) {1};
\draw[fill=red] (23,-3) rectangle (24,-4);
\draw[fill=28x4x4color1] (24,-3) rectangle (25,-4);
\node at (24.5,-3.5) {1};
\draw[fill=red] (25,-3) rectangle (26,-4);
\draw[fill=28x4x4color1] (26,-3) rectangle (27,-4);
\node at (26.5,-3.5) {1};
\draw[fill=red] (27,-3) rectangle (28,-4);
\draw[fill=28x4x4color1] (0,-5) rectangle (1,-6);
\node at (0.5,-5.5) {1};
\draw[fill=28x4x4color8] (1,-5) rectangle (2,-6);
\node at (1.5,-5.5) {8};
\draw[fill=28x4x4color9] (2,-5) rectangle (3,-6);
\node at (2.5,-5.5) {9};
\draw[fill=28x4x4color28] (3,-5) rectangle (4,-6);
\node at (3.5,-5.5) {28};
\draw[fill=28x4x4color29] (4,-5) rectangle (5,-6);
\node at (4.5,-5.5) {29};
\draw[fill=28x4x4color30] (5,-5) rectangle (6,-6);
\node at (5.5,-5.5) {30};
\draw[fill=28x4x4color31] (6,-5) rectangle (7,-6);
\node at (6.5,-5.5) {31};
\draw[fill=28x4x4color32] (7,-5) rectangle (8,-6);
\node at (7.5,-5.5) {32};
\draw[fill=28x4x4color33] (8,-5) rectangle (9,-6);
\node at (8.5,-5.5) {33};
\draw[fill=28x4x4color34] (9,-5) rectangle (10,-6);
\node at (9.5,-5.5) {34};
\draw[fill=28x4x4color35] (10,-5) rectangle (11,-6);
\node at (10.5,-5.5) {35};
\draw[fill=28x4x4color36] (11,-5) rectangle (12,-6);
\node at (11.5,-5.5) {36};
\draw[fill=28x4x4color37] (12,-5) rectangle (13,-6);
\node at (12.5,-5.5) {37};
\draw[fill=28x4x4color38] (13,-5) rectangle (14,-6);
\node at (13.5,-5.5) {38};
\draw[fill=28x4x4color39] (14,-5) rectangle (15,-6);
\node at (14.5,-5.5) {39};
\draw[fill=28x4x4color40] (15,-5) rectangle (16,-6);
\node at (15.5,-5.5) {40};
\draw[fill=28x4x4color41] (16,-5) rectangle (17,-6);
\node at (16.5,-5.5) {41};
\draw[fill=28x4x4color42] (17,-5) rectangle (18,-6);
\node at (17.5,-5.5) {42};
\draw[fill=28x4x4color43] (18,-5) rectangle (19,-6);
\node at (18.5,-5.5) {43};
\draw[fill=28x4x4color44] (19,-5) rectangle (20,-6);
\node at (19.5,-5.5) {44};
\draw[fill=28x4x4color45] (20,-5) rectangle (21,-6);
\node at (20.5,-5.5) {45};
\draw[fill=28x4x4color46] (21,-5) rectangle (22,-6);
\node at (21.5,-5.5) {46};
\draw[fill=28x4x4color47] (22,-5) rectangle (23,-6);
\node at (22.5,-5.5) {47};
\draw[fill=28x4x4color48] (23,-5) rectangle (24,-6);
\node at (23.5,-5.5) {48};
\draw[fill=28x4x4color49] (24,-5) rectangle (25,-6);
\node at (24.5,-5.5) {49};
\draw[fill=28x4x4color50] (25,-5) rectangle (26,-6);
\node at (25.5,-5.5) {50};
\draw[fill=28x4x4color51] (26,-5) rectangle (27,-6);
\node at (26.5,-5.5) {51};
\draw[fill=28x4x4color52] (27,-5) rectangle (28,-6);
\node at (27.5,-5.5) {52};
\draw[fill=red] (0,-6) rectangle (1,-7);
\draw[fill=28x4x4color7] (1,-6) rectangle (2,-7);
\node at (1.5,-6.5) {7};
\draw[fill=28x4x4color8] (2,-6) rectangle (3,-7);
\node at (2.5,-6.5) {8};
\draw[fill=28x4x4color27] (3,-6) rectangle (4,-7);
\node at (3.5,-6.5) {27};
\draw[fill=28x4x4color28] (4,-6) rectangle (5,-7);
\node at (4.5,-6.5) {28};
\draw[fill=28x4x4color29] (5,-6) rectangle (6,-7);
\node at (5.5,-6.5) {29};
\draw[fill=28x4x4color30] (6,-6) rectangle (7,-7);
\node at (6.5,-6.5) {30};
\draw[fill=28x4x4color31] (7,-6) rectangle (8,-7);
\node at (7.5,-6.5) {31};
\draw[fill=28x4x4color32] (8,-6) rectangle (9,-7);
\node at (8.5,-6.5) {32};
\draw[fill=28x4x4color33] (9,-6) rectangle (10,-7);
\node at (9.5,-6.5) {33};
\draw[fill=28x4x4color34] (10,-6) rectangle (11,-7);
\node at (10.5,-6.5) {34};
\draw[fill=28x4x4color35] (11,-6) rectangle (12,-7);
\node at (11.5,-6.5) {35};
\draw[fill=28x4x4color36] (12,-6) rectangle (13,-7);
\node at (12.5,-6.5) {36};
\draw[fill=28x4x4color37] (13,-6) rectangle (14,-7);
\node at (13.5,-6.5) {37};
\draw[fill=28x4x4color38] (14,-6) rectangle (15,-7);
\node at (14.5,-6.5) {38};
\draw[fill=28x4x4color39] (15,-6) rectangle (16,-7);
\node at (15.5,-6.5) {39};
\draw[fill=28x4x4color40] (16,-6) rectangle (17,-7);
\node at (16.5,-6.5) {40};
\draw[fill=28x4x4color41] (17,-6) rectangle (18,-7);
\node at (17.5,-6.5) {41};
\draw[fill=28x4x4color42] (18,-6) rectangle (19,-7);
\node at (18.5,-6.5) {42};
\draw[fill=28x4x4color43] (19,-6) rectangle (20,-7);
\node at (19.5,-6.5) {43};
\draw[fill=28x4x4color44] (20,-6) rectangle (21,-7);
\node at (20.5,-6.5) {44};
\draw[fill=28x4x4color45] (21,-6) rectangle (22,-7);
\node at (21.5,-6.5) {45};
\draw[fill=28x4x4color46] (22,-6) rectangle (23,-7);
\node at (22.5,-6.5) {46};
\draw[fill=28x4x4color47] (23,-6) rectangle (24,-7);
\node at (23.5,-6.5) {47};
\draw[fill=28x4x4color48] (24,-6) rectangle (25,-7);
\node at (24.5,-6.5) {48};
\draw[fill=28x4x4color49] (25,-6) rectangle (26,-7);
\node at (25.5,-6.5) {49};
\draw[fill=28x4x4color50] (26,-6) rectangle (27,-7);
\node at (26.5,-6.5) {50};
\draw[fill=28x4x4color51] (27,-6) rectangle (28,-7);
\node at (27.5,-6.5) {51};
\draw[fill=28x4x4color1] (0,-7) rectangle (1,-8);
\node at (0.5,-7.5) {1};
\draw[fill=28x4x4color2] (1,-7) rectangle (2,-8);
\node at (1.5,-7.5) {2};
\draw[fill=28x4x4color5] (2,-7) rectangle (3,-8);
\node at (2.5,-7.5) {5};
\draw[fill=28x4x4color26] (3,-7) rectangle (4,-8);
\node at (3.5,-7.5) {26};
\draw[fill=28x4x4color25] (4,-7) rectangle (5,-8);
\node at (4.5,-7.5) {25};
\draw[fill=28x4x4color24] (5,-7) rectangle (6,-8);
\node at (5.5,-7.5) {24};
\draw[fill=28x4x4color23] (6,-7) rectangle (7,-8);
\node at (6.5,-7.5) {23};
\draw[fill=28x4x4color22] (7,-7) rectangle (8,-8);
\node at (7.5,-7.5) {22};
\draw[fill=28x4x4color21] (8,-7) rectangle (9,-8);
\node at (8.5,-7.5) {21};
\draw[fill=28x4x4color20] (9,-7) rectangle (10,-8);
\node at (9.5,-7.5) {20};
\draw[fill=28x4x4color19] (10,-7) rectangle (11,-8);
\node at (10.5,-7.5) {19};
\draw[fill=28x4x4color18] (11,-7) rectangle (12,-8);
\node at (11.5,-7.5) {18};
\draw[fill=28x4x4color17] (12,-7) rectangle (13,-8);
\node at (12.5,-7.5) {17};
\draw[fill=28x4x4color16] (13,-7) rectangle (14,-8);
\node at (13.5,-7.5) {16};
\draw[fill=28x4x4color15] (14,-7) rectangle (15,-8);
\node at (14.5,-7.5) {15};
\draw[fill=28x4x4color14] (15,-7) rectangle (16,-8);
\node at (15.5,-7.5) {14};
\draw[fill=28x4x4color13] (16,-7) rectangle (17,-8);
\node at (16.5,-7.5) {13};
\draw[fill=28x4x4color12] (17,-7) rectangle (18,-8);
\node at (17.5,-7.5) {12};
\draw[fill=28x4x4color11] (18,-7) rectangle (19,-8);
\node at (18.5,-7.5) {11};
\draw[fill=28x4x4color10] (19,-7) rectangle (20,-8);
\node at (19.5,-7.5) {10};
\draw[fill=28x4x4color9] (20,-7) rectangle (21,-8);
\node at (20.5,-7.5) {9};
\draw[fill=28x4x4color8] (21,-7) rectangle (22,-8);
\node at (21.5,-7.5) {8};
\draw[fill=28x4x4color7] (22,-7) rectangle (23,-8);
\node at (22.5,-7.5) {7};
\draw[fill=28x4x4color6] (23,-7) rectangle (24,-8);
\node at (23.5,-7.5) {6};
\draw[fill=28x4x4color5] (24,-7) rectangle (25,-8);
\node at (24.5,-7.5) {5};
\draw[fill=28x4x4color4] (25,-7) rectangle (26,-8);
\node at (25.5,-7.5) {4};
\draw[fill=28x4x4color1] (26,-7) rectangle (27,-8);
\node at (26.5,-7.5) {1};
\draw[fill=red] (27,-7) rectangle (28,-8);
\draw[fill=red] (0,-8) rectangle (1,-9);
\draw[fill=28x4x4color1] (1,-8) rectangle (2,-9);
\node at (1.5,-8.5) {1};
\draw[fill=28x4x4color6] (2,-8) rectangle (3,-9);
\node at (2.5,-8.5) {6};
\draw[fill=28x4x4color27] (3,-8) rectangle (4,-9);
\node at (3.5,-8.5) {27};
\draw[fill=28x4x4color30] (4,-8) rectangle (5,-9);
\node at (4.5,-8.5) {30};
\draw[fill=28x4x4color31] (5,-8) rectangle (6,-9);
\node at (5.5,-8.5) {31};
\draw[fill=28x4x4color32] (6,-8) rectangle (7,-9);
\node at (6.5,-8.5) {32};
\draw[fill=28x4x4color33] (7,-8) rectangle (8,-9);
\node at (7.5,-8.5) {33};
\draw[fill=28x4x4color34] (8,-8) rectangle (9,-9);
\node at (8.5,-8.5) {34};
\draw[fill=28x4x4color35] (9,-8) rectangle (10,-9);
\node at (9.5,-8.5) {35};
\draw[fill=28x4x4color36] (10,-8) rectangle (11,-9);
\node at (10.5,-8.5) {36};
\draw[fill=28x4x4color37] (11,-8) rectangle (12,-9);
\node at (11.5,-8.5) {37};
\draw[fill=28x4x4color38] (12,-8) rectangle (13,-9);
\node at (12.5,-8.5) {38};
\draw[fill=28x4x4color39] (13,-8) rectangle (14,-9);
\node at (13.5,-8.5) {39};
\draw[fill=28x4x4color40] (14,-8) rectangle (15,-9);
\node at (14.5,-8.5) {40};
\draw[fill=28x4x4color41] (15,-8) rectangle (16,-9);
\node at (15.5,-8.5) {41};
\draw[fill=28x4x4color42] (16,-8) rectangle (17,-9);
\node at (16.5,-8.5) {42};
\draw[fill=28x4x4color43] (17,-8) rectangle (18,-9);
\node at (17.5,-8.5) {43};
\draw[fill=28x4x4color44] (18,-8) rectangle (19,-9);
\node at (18.5,-8.5) {44};
\draw[fill=28x4x4color45] (19,-8) rectangle (20,-9);
\node at (19.5,-8.5) {45};
\draw[fill=28x4x4color46] (20,-8) rectangle (21,-9);
\node at (20.5,-8.5) {46};
\draw[fill=28x4x4color47] (21,-8) rectangle (22,-9);
\node at (21.5,-8.5) {47};
\draw[fill=28x4x4color48] (22,-8) rectangle (23,-9);
\node at (22.5,-8.5) {48};
\draw[fill=28x4x4color49] (23,-8) rectangle (24,-9);
\node at (23.5,-8.5) {49};
\draw[fill=28x4x4color50] (24,-8) rectangle (25,-9);
\node at (24.5,-8.5) {50};
\draw[fill=28x4x4color51] (25,-8) rectangle (26,-9);
\node at (25.5,-8.5) {51};
\draw[fill=28x4x4color52] (26,-8) rectangle (27,-9);
\node at (26.5,-8.5) {52};
\draw[fill=28x4x4color53] (27,-8) rectangle (28,-9);
\node at (27.5,-8.5) {53};
\draw[fill=red] (0,-10) rectangle (1,-11);
\draw[fill=28x4x4color7] (1,-10) rectangle (2,-11);
\node at (1.5,-10.5) {7};
\draw[fill=28x4x4color8] (2,-10) rectangle (3,-11);
\node at (2.5,-10.5) {8};
\draw[fill=28x4x4color9] (3,-10) rectangle (4,-11);
\node at (3.5,-10.5) {9};
\draw[fill=28x4x4color10] (4,-10) rectangle (5,-11);
\node at (4.5,-10.5) {10};
\draw[fill=28x4x4color11] (5,-10) rectangle (6,-11);
\node at (5.5,-10.5) {11};
\draw[fill=28x4x4color12] (6,-10) rectangle (7,-11);
\node at (6.5,-10.5) {12};
\draw[fill=28x4x4color13] (7,-10) rectangle (8,-11);
\node at (7.5,-10.5) {13};
\draw[fill=28x4x4color14] (8,-10) rectangle (9,-11);
\node at (8.5,-10.5) {14};
\draw[fill=28x4x4color15] (9,-10) rectangle (10,-11);
\node at (9.5,-10.5) {15};
\draw[fill=28x4x4color16] (10,-10) rectangle (11,-11);
\node at (10.5,-10.5) {16};
\draw[fill=28x4x4color17] (11,-10) rectangle (12,-11);
\node at (11.5,-10.5) {17};
\draw[fill=28x4x4color18] (12,-10) rectangle (13,-11);
\node at (12.5,-10.5) {18};
\draw[fill=28x4x4color19] (13,-10) rectangle (14,-11);
\node at (13.5,-10.5) {19};
\draw[fill=28x4x4color20] (14,-10) rectangle (15,-11);
\node at (14.5,-10.5) {20};
\draw[fill=28x4x4color21] (15,-10) rectangle (16,-11);
\node at (15.5,-10.5) {21};
\draw[fill=28x4x4color22] (16,-10) rectangle (17,-11);
\node at (16.5,-10.5) {22};
\draw[fill=28x4x4color23] (17,-10) rectangle (18,-11);
\node at (17.5,-10.5) {23};
\draw[fill=28x4x4color24] (18,-10) rectangle (19,-11);
\node at (18.5,-10.5) {24};
\draw[fill=28x4x4color25] (19,-10) rectangle (20,-11);
\node at (19.5,-10.5) {25};
\draw[fill=28x4x4color26] (20,-10) rectangle (21,-11);
\node at (20.5,-10.5) {26};
\draw[fill=28x4x4color27] (21,-10) rectangle (22,-11);
\node at (21.5,-10.5) {27};
\draw[fill=28x4x4color28] (22,-10) rectangle (23,-11);
\node at (22.5,-10.5) {28};
\draw[fill=28x4x4color29] (23,-10) rectangle (24,-11);
\node at (23.5,-10.5) {29};
\draw[fill=28x4x4color30] (24,-10) rectangle (25,-11);
\node at (24.5,-10.5) {30};
\draw[fill=28x4x4color31] (25,-10) rectangle (26,-11);
\node at (25.5,-10.5) {31};
\draw[fill=28x4x4color32] (26,-10) rectangle (27,-11);
\node at (26.5,-10.5) {32};
\draw[fill=28x4x4color33] (27,-10) rectangle (28,-11);
\node at (27.5,-10.5) {33};
\draw[fill=28x4x4color5] (0,-11) rectangle (1,-12);
\node at (0.5,-11.5) {5};
\draw[fill=28x4x4color6] (1,-11) rectangle (2,-12);
\node at (1.5,-11.5) {6};
\draw[fill=28x4x4color7] (2,-11) rectangle (3,-12);
\node at (2.5,-11.5) {7};
\draw[fill=28x4x4color8] (3,-11) rectangle (4,-12);
\node at (3.5,-11.5) {8};
\draw[fill=28x4x4color9] (4,-11) rectangle (5,-12);
\node at (4.5,-11.5) {9};
\draw[fill=28x4x4color10] (5,-11) rectangle (6,-12);
\node at (5.5,-11.5) {10};
\draw[fill=28x4x4color11] (6,-11) rectangle (7,-12);
\node at (6.5,-11.5) {11};
\draw[fill=28x4x4color12] (7,-11) rectangle (8,-12);
\node at (7.5,-11.5) {12};
\draw[fill=28x4x4color13] (8,-11) rectangle (9,-12);
\node at (8.5,-11.5) {13};
\draw[fill=28x4x4color14] (9,-11) rectangle (10,-12);
\node at (9.5,-11.5) {14};
\draw[fill=28x4x4color15] (10,-11) rectangle (11,-12);
\node at (10.5,-11.5) {15};
\draw[fill=28x4x4color16] (11,-11) rectangle (12,-12);
\node at (11.5,-11.5) {16};
\draw[fill=28x4x4color17] (12,-11) rectangle (13,-12);
\node at (12.5,-11.5) {17};
\draw[fill=28x4x4color18] (13,-11) rectangle (14,-12);
\node at (13.5,-11.5) {18};
\draw[fill=28x4x4color19] (14,-11) rectangle (15,-12);
\node at (14.5,-11.5) {19};
\draw[fill=28x4x4color20] (15,-11) rectangle (16,-12);
\node at (15.5,-11.5) {20};
\draw[fill=28x4x4color21] (16,-11) rectangle (17,-12);
\node at (16.5,-11.5) {21};
\draw[fill=28x4x4color22] (17,-11) rectangle (18,-12);
\node at (17.5,-11.5) {22};
\draw[fill=28x4x4color23] (18,-11) rectangle (19,-12);
\node at (18.5,-11.5) {23};
\draw[fill=28x4x4color24] (19,-11) rectangle (20,-12);
\node at (19.5,-11.5) {24};
\draw[fill=28x4x4color25] (20,-11) rectangle (21,-12);
\node at (20.5,-11.5) {25};
\draw[fill=28x4x4color26] (21,-11) rectangle (22,-12);
\node at (21.5,-11.5) {26};
\draw[fill=28x4x4color27] (22,-11) rectangle (23,-12);
\node at (22.5,-11.5) {27};
\draw[fill=28x4x4color28] (23,-11) rectangle (24,-12);
\node at (23.5,-11.5) {28};
\draw[fill=28x4x4color29] (24,-11) rectangle (25,-12);
\node at (24.5,-11.5) {29};
\draw[fill=28x4x4color30] (25,-11) rectangle (26,-12);
\node at (25.5,-11.5) {30};
\draw[fill=28x4x4color31] (26,-11) rectangle (27,-12);
\node at (26.5,-11.5) {31};
\draw[fill=28x4x4color32] (27,-11) rectangle (28,-12);
\node at (27.5,-11.5) {32};
\draw[fill=28x4x4color4] (0,-12) rectangle (1,-13);
\node at (0.5,-12.5) {4};
\draw[fill=28x4x4color3] (1,-12) rectangle (2,-13);
\node at (1.5,-12.5) {3};
\draw[fill=28x4x4color4] (2,-12) rectangle (3,-13);
\node at (2.5,-12.5) {4};
\draw[fill=28x4x4color1] (3,-12) rectangle (4,-13);
\node at (3.5,-12.5) {1};
\draw[fill=red] (4,-12) rectangle (5,-13);
\draw[fill=28x4x4color1] (5,-12) rectangle (6,-13);
\node at (5.5,-12.5) {1};
\draw[fill=red] (6,-12) rectangle (7,-13);
\draw[fill=28x4x4color3] (7,-12) rectangle (8,-13);
\node at (7.5,-12.5) {3};
\draw[fill=red] (8,-12) rectangle (9,-13);
\draw[fill=28x4x4color1] (9,-12) rectangle (10,-13);
\node at (9.5,-12.5) {1};
\draw[fill=red] (10,-12) rectangle (11,-13);
\draw[fill=28x4x4color1] (11,-12) rectangle (12,-13);
\node at (11.5,-12.5) {1};
\draw[fill=red] (12,-12) rectangle (13,-13);
\draw[fill=28x4x4color3] (13,-12) rectangle (14,-13);
\node at (13.5,-12.5) {3};
\draw[fill=red] (14,-12) rectangle (15,-13);
\draw[fill=28x4x4color1] (15,-12) rectangle (16,-13);
\node at (15.5,-12.5) {1};
\draw[fill=red] (16,-12) rectangle (17,-13);
\draw[fill=28x4x4color1] (17,-12) rectangle (18,-13);
\node at (17.5,-12.5) {1};
\draw[fill=red] (18,-12) rectangle (19,-13);
\draw[fill=28x4x4color3] (19,-12) rectangle (20,-13);
\node at (19.5,-12.5) {3};
\draw[fill=red] (20,-12) rectangle (21,-13);
\draw[fill=28x4x4color1] (21,-12) rectangle (22,-13);
\node at (21.5,-12.5) {1};
\draw[fill=red] (22,-12) rectangle (23,-13);
\draw[fill=28x4x4color1] (23,-12) rectangle (24,-13);
\node at (23.5,-12.5) {1};
\draw[fill=red] (24,-12) rectangle (25,-13);
\draw[fill=28x4x4color3] (25,-12) rectangle (26,-13);
\node at (25.5,-12.5) {3};
\draw[fill=red] (26,-12) rectangle (27,-13);
\draw[fill=28x4x4color1] (27,-12) rectangle (28,-13);
\node at (27.5,-12.5) {1};
\draw[fill=28x4x4color1] (0,-13) rectangle (1,-14);
\node at (0.5,-13.5) {1};
\draw[fill=red] (1,-13) rectangle (2,-14);
\draw[fill=28x4x4color5] (2,-13) rectangle (3,-14);
\node at (2.5,-13.5) {5};
\draw[fill=red] (3,-13) rectangle (4,-14);
\draw[fill=28x4x4color31] (4,-13) rectangle (5,-14);
\node at (4.5,-13.5) {31};
\draw[fill=28x4x4color32] (5,-13) rectangle (6,-14);
\node at (5.5,-13.5) {32};
\draw[fill=28x4x4color33] (6,-13) rectangle (7,-14);
\node at (6.5,-13.5) {33};
\draw[fill=28x4x4color34] (7,-13) rectangle (8,-14);
\node at (7.5,-13.5) {34};
\draw[fill=28x4x4color35] (8,-13) rectangle (9,-14);
\node at (8.5,-13.5) {35};
\draw[fill=28x4x4color36] (9,-13) rectangle (10,-14);
\node at (9.5,-13.5) {36};
\draw[fill=28x4x4color37] (10,-13) rectangle (11,-14);
\node at (10.5,-13.5) {37};
\draw[fill=28x4x4color38] (11,-13) rectangle (12,-14);
\node at (11.5,-13.5) {38};
\draw[fill=28x4x4color39] (12,-13) rectangle (13,-14);
\node at (12.5,-13.5) {39};
\draw[fill=28x4x4color40] (13,-13) rectangle (14,-14);
\node at (13.5,-13.5) {40};
\draw[fill=28x4x4color41] (14,-13) rectangle (15,-14);
\node at (14.5,-13.5) {41};
\draw[fill=28x4x4color42] (15,-13) rectangle (16,-14);
\node at (15.5,-13.5) {42};
\draw[fill=28x4x4color43] (16,-13) rectangle (17,-14);
\node at (16.5,-13.5) {43};
\draw[fill=28x4x4color44] (17,-13) rectangle (18,-14);
\node at (17.5,-13.5) {44};
\draw[fill=28x4x4color45] (18,-13) rectangle (19,-14);
\node at (18.5,-13.5) {45};
\draw[fill=28x4x4color46] (19,-13) rectangle (20,-14);
\node at (19.5,-13.5) {46};
\draw[fill=28x4x4color47] (20,-13) rectangle (21,-14);
\node at (20.5,-13.5) {47};
\draw[fill=28x4x4color48] (21,-13) rectangle (22,-14);
\node at (21.5,-13.5) {48};
\draw[fill=28x4x4color49] (22,-13) rectangle (23,-14);
\node at (22.5,-13.5) {49};
\draw[fill=28x4x4color50] (23,-13) rectangle (24,-14);
\node at (23.5,-13.5) {50};
\draw[fill=28x4x4color51] (24,-13) rectangle (25,-14);
\node at (24.5,-13.5) {51};
\draw[fill=28x4x4color52] (25,-13) rectangle (26,-14);
\node at (25.5,-13.5) {52};
\draw[fill=28x4x4color53] (26,-13) rectangle (27,-14);
\node at (26.5,-13.5) {53};
\draw[fill=28x4x4color54] (27,-13) rectangle (28,-14);
\node at (27.5,-13.5) {54};
\draw[fill=28x4x4color7] (0,-15) rectangle (1,-16);
\node at (0.5,-15.5) {7};
\draw[fill=red] (1,-15) rectangle (2,-16);
\draw[fill=28x4x4color1] (2,-15) rectangle (3,-16);
\node at (2.5,-15.5) {1};
\draw[fill=red] (3,-15) rectangle (4,-16);
\draw[fill=28x4x4color3] (4,-15) rectangle (5,-16);
\node at (4.5,-15.5) {3};
\draw[fill=red] (5,-15) rectangle (6,-16);
\draw[fill=28x4x4color1] (6,-15) rectangle (7,-16);
\node at (6.5,-15.5) {1};
\draw[fill=red] (7,-15) rectangle (8,-16);
\draw[fill=28x4x4color1] (8,-15) rectangle (9,-16);
\node at (8.5,-15.5) {1};
\draw[fill=red] (9,-15) rectangle (10,-16);
\draw[fill=28x4x4color3] (10,-15) rectangle (11,-16);
\node at (10.5,-15.5) {3};
\draw[fill=red] (11,-15) rectangle (12,-16);
\draw[fill=28x4x4color1] (12,-15) rectangle (13,-16);
\node at (12.5,-15.5) {1};
\draw[fill=red] (13,-15) rectangle (14,-16);
\draw[fill=28x4x4color1] (14,-15) rectangle (15,-16);
\node at (14.5,-15.5) {1};
\draw[fill=red] (15,-15) rectangle (16,-16);
\draw[fill=28x4x4color3] (16,-15) rectangle (17,-16);
\node at (16.5,-15.5) {3};
\draw[fill=red] (17,-15) rectangle (18,-16);
\draw[fill=28x4x4color1] (18,-15) rectangle (19,-16);
\node at (18.5,-15.5) {1};
\draw[fill=red] (19,-15) rectangle (20,-16);
\draw[fill=28x4x4color1] (20,-15) rectangle (21,-16);
\node at (20.5,-15.5) {1};
\draw[fill=red] (21,-15) rectangle (22,-16);
\draw[fill=28x4x4color3] (22,-15) rectangle (23,-16);
\node at (22.5,-15.5) {3};
\draw[fill=red] (23,-15) rectangle (24,-16);
\draw[fill=28x4x4color1] (24,-15) rectangle (25,-16);
\node at (24.5,-15.5) {1};
\draw[fill=red] (25,-15) rectangle (26,-16);
\draw[fill=28x4x4color1] (26,-15) rectangle (27,-16);
\node at (26.5,-15.5) {1};
\draw[fill=red] (27,-15) rectangle (28,-16);
\draw[fill=28x4x4color6] (0,-16) rectangle (1,-17);
\node at (0.5,-16.5) {6};
\draw[fill=28x4x4color1] (1,-16) rectangle (2,-17);
\node at (1.5,-16.5) {1};
\draw[fill=red] (2,-16) rectangle (3,-17);
\draw[fill=28x4x4color1] (3,-16) rectangle (4,-17);
\node at (3.5,-16.5) {1};
\draw[fill=28x4x4color2] (4,-16) rectangle (5,-17);
\node at (4.5,-16.5) {2};
\draw[fill=28x4x4color1] (5,-16) rectangle (6,-17);
\node at (5.5,-16.5) {1};
\draw[fill=red] (6,-16) rectangle (7,-17);
\draw[fill=28x4x4color1] (7,-16) rectangle (8,-17);
\node at (7.5,-16.5) {1};
\draw[fill=red] (8,-16) rectangle (9,-17);
\draw[fill=28x4x4color1] (9,-16) rectangle (10,-17);
\node at (9.5,-16.5) {1};
\draw[fill=28x4x4color2] (10,-16) rectangle (11,-17);
\node at (10.5,-16.5) {2};
\draw[fill=28x4x4color1] (11,-16) rectangle (12,-17);
\node at (11.5,-16.5) {1};
\draw[fill=red] (12,-16) rectangle (13,-17);
\draw[fill=28x4x4color1] (13,-16) rectangle (14,-17);
\node at (13.5,-16.5) {1};
\draw[fill=red] (14,-16) rectangle (15,-17);
\draw[fill=28x4x4color1] (15,-16) rectangle (16,-17);
\node at (15.5,-16.5) {1};
\draw[fill=28x4x4color2] (16,-16) rectangle (17,-17);
\node at (16.5,-16.5) {2};
\draw[fill=28x4x4color1] (17,-16) rectangle (18,-17);
\node at (17.5,-16.5) {1};
\draw[fill=red] (18,-16) rectangle (19,-17);
\draw[fill=28x4x4color1] (19,-16) rectangle (20,-17);
\node at (19.5,-16.5) {1};
\draw[fill=red] (20,-16) rectangle (21,-17);
\draw[fill=28x4x4color1] (21,-16) rectangle (22,-17);
\node at (21.5,-16.5) {1};
\draw[fill=28x4x4color2] (22,-16) rectangle (23,-17);
\node at (22.5,-16.5) {2};
\draw[fill=28x4x4color1] (23,-16) rectangle (24,-17);
\node at (23.5,-16.5) {1};
\draw[fill=red] (24,-16) rectangle (25,-17);
\draw[fill=28x4x4color1] (25,-16) rectangle (26,-17);
\node at (25.5,-16.5) {1};
\draw[fill=red] (26,-16) rectangle (27,-17);
\draw[fill=28x4x4color1] (27,-16) rectangle (28,-17);
\node at (27.5,-16.5) {1};
\draw[fill=28x4x4color5] (0,-17) rectangle (1,-18);
\node at (0.5,-17.5) {5};
\draw[fill=red] (1,-17) rectangle (2,-18);
\draw[fill=28x4x4color1] (2,-17) rectangle (3,-18);
\node at (2.5,-17.5) {1};
\draw[fill=red] (3,-17) rectangle (4,-18);
\draw[fill=28x4x4color1] (4,-17) rectangle (5,-18);
\node at (4.5,-17.5) {1};
\draw[fill=red] (5,-17) rectangle (6,-18);
\draw[fill=28x4x4color1] (6,-17) rectangle (7,-18);
\node at (6.5,-17.5) {1};
\draw[fill=28x4x4color2] (7,-17) rectangle (8,-18);
\node at (7.5,-17.5) {2};
\draw[fill=28x4x4color1] (8,-17) rectangle (9,-18);
\node at (8.5,-17.5) {1};
\draw[fill=red] (9,-17) rectangle (10,-18);
\draw[fill=28x4x4color1] (10,-17) rectangle (11,-18);
\node at (10.5,-17.5) {1};
\draw[fill=red] (11,-17) rectangle (12,-18);
\draw[fill=28x4x4color1] (12,-17) rectangle (13,-18);
\node at (12.5,-17.5) {1};
\draw[fill=28x4x4color2] (13,-17) rectangle (14,-18);
\node at (13.5,-17.5) {2};
\draw[fill=28x4x4color1] (14,-17) rectangle (15,-18);
\node at (14.5,-17.5) {1};
\draw[fill=red] (15,-17) rectangle (16,-18);
\draw[fill=28x4x4color1] (16,-17) rectangle (17,-18);
\node at (16.5,-17.5) {1};
\draw[fill=red] (17,-17) rectangle (18,-18);
\draw[fill=28x4x4color1] (18,-17) rectangle (19,-18);
\node at (18.5,-17.5) {1};
\draw[fill=28x4x4color2] (19,-17) rectangle (20,-18);
\node at (19.5,-17.5) {2};
\draw[fill=28x4x4color1] (20,-17) rectangle (21,-18);
\node at (20.5,-17.5) {1};
\draw[fill=red] (21,-17) rectangle (22,-18);
\draw[fill=28x4x4color1] (22,-17) rectangle (23,-18);
\node at (22.5,-17.5) {1};
\draw[fill=red] (23,-17) rectangle (24,-18);
\draw[fill=28x4x4color1] (24,-17) rectangle (25,-18);
\node at (24.5,-17.5) {1};
\draw[fill=28x4x4color2] (25,-17) rectangle (26,-18);
\node at (25.5,-17.5) {2};
\draw[fill=28x4x4color1] (26,-17) rectangle (27,-18);
\node at (26.5,-17.5) {1};
\draw[fill=red] (27,-17) rectangle (28,-18);
\draw[fill=red] (0,-18) rectangle (1,-19);
\draw[fill=28x4x4color1] (1,-18) rectangle (2,-19);
\node at (1.5,-18.5) {1};
\draw[fill=28x4x4color6] (2,-18) rectangle (3,-19);
\node at (2.5,-18.5) {6};
\draw[fill=28x4x4color7] (3,-18) rectangle (4,-19);
\node at (3.5,-18.5) {7};
\draw[fill=28x4x4color32] (4,-18) rectangle (5,-19);
\node at (4.5,-18.5) {32};
\draw[fill=28x4x4color33] (5,-18) rectangle (6,-19);
\node at (5.5,-18.5) {33};
\draw[fill=28x4x4color34] (6,-18) rectangle (7,-19);
\node at (6.5,-18.5) {34};
\draw[fill=28x4x4color35] (7,-18) rectangle (8,-19);
\node at (7.5,-18.5) {35};
\draw[fill=28x4x4color36] (8,-18) rectangle (9,-19);
\node at (8.5,-18.5) {36};
\draw[fill=28x4x4color37] (9,-18) rectangle (10,-19);
\node at (9.5,-18.5) {37};
\draw[fill=28x4x4color38] (10,-18) rectangle (11,-19);
\node at (10.5,-18.5) {38};
\draw[fill=28x4x4color39] (11,-18) rectangle (12,-19);
\node at (11.5,-18.5) {39};
\draw[fill=28x4x4color40] (12,-18) rectangle (13,-19);
\node at (12.5,-18.5) {40};
\draw[fill=28x4x4color41] (13,-18) rectangle (14,-19);
\node at (13.5,-18.5) {41};
\draw[fill=28x4x4color42] (14,-18) rectangle (15,-19);
\node at (14.5,-18.5) {42};
\draw[fill=28x4x4color43] (15,-18) rectangle (16,-19);
\node at (15.5,-18.5) {43};
\draw[fill=28x4x4color44] (16,-18) rectangle (17,-19);
\node at (16.5,-18.5) {44};
\draw[fill=28x4x4color45] (17,-18) rectangle (18,-19);
\node at (17.5,-18.5) {45};
\draw[fill=28x4x4color46] (18,-18) rectangle (19,-19);
\node at (18.5,-18.5) {46};
\draw[fill=28x4x4color47] (19,-18) rectangle (20,-19);
\node at (19.5,-18.5) {47};
\draw[fill=28x4x4color48] (20,-18) rectangle (21,-19);
\node at (20.5,-18.5) {48};
\draw[fill=28x4x4color49] (21,-18) rectangle (22,-19);
\node at (21.5,-18.5) {49};
\draw[fill=28x4x4color50] (22,-18) rectangle (23,-19);
\node at (22.5,-18.5) {50};
\draw[fill=28x4x4color51] (23,-18) rectangle (24,-19);
\node at (23.5,-18.5) {51};
\draw[fill=28x4x4color52] (24,-18) rectangle (25,-19);
\node at (24.5,-18.5) {52};
\draw[fill=28x4x4color53] (25,-18) rectangle (26,-19);
\node at (25.5,-18.5) {53};
\draw[fill=28x4x4color54] (26,-18) rectangle (27,-19);
\node at (26.5,-18.5) {54};
\draw[fill=28x4x4color55] (27,-18) rectangle (28,-19);
\node at (27.5,-18.5) {55};

\end{scope}
\end{tikzpicture}
\end{center}

%% file: 2-3-4-tikz.txt
\definecolor{4x3x2color1}{RGB}{255, 128, 128}
\definecolor{4x3x2color2}{RGB}{255, 149, 149}
\definecolor{4x3x2color3}{RGB}{255, 170, 170}
\definecolor{4x3x2color4}{RGB}{255, 191, 191}
\definecolor{4x3x2color5}{RGB}{255, 212, 212}
\definecolor{4x3x2color6}{RGB}{255, 255, 255}

\begin{center}
\begin{tikzpicture}[scale=0.5]
\begin{scope}[xshift=0cm,yshift=0cm]
\draw[very thick] (0,0) rectangle (4,-3);
\foreach \b in {0,...,2}
 \foreach \c in {0,...,3}
  \draw (\c,\b*-1) rectangle (\c+1,(\b*-1-1);

\draw[fill=red] (0,-0) rectangle (1,-1);
\draw[fill=4x3x2color3] (1,-0) rectangle (2,-1);
\node at (1.5,-0.5) {3};
\draw[fill=4x3x2color4] (2,-0) rectangle (3,-1);
\node at (2.5,-0.5) {4};
\draw[fill=4x3x2color5] (3,-0) rectangle (4,-1);
\node at (3.5,-0.5) {5};
\draw[fill=4x3x2color1] (0,-1) rectangle (1,-2);
\node at (0.5,-1.5) {1};
\draw[fill=4x3x2color2] (1,-1) rectangle (2,-2);
\node at (1.5,-1.5) {2};
\draw[fill=4x3x2color3] (2,-1) rectangle (3,-2);
\node at (2.5,-1.5) {3};
\draw[fill=red] (3,-1) rectangle (4,-2);
\draw[fill=red] (0,-2) rectangle (1,-3);
\draw[fill=4x3x2color1] (1,-2) rectangle (2,-3);
\node at (1.5,-2.5) {1};
\draw[fill=red] (2,-2) rectangle (3,-3);
\draw[fill=4x3x2color1] (3,-2) rectangle (4,-3);
\node at (3.5,-2.5) {1};
\end{scope}

\begin{scope}[xshift=5cm,yshift=0cm]
\draw[very thick] (0,0) rectangle (4,-3);
\foreach \b in {0,...,2}
 \foreach \c in {0,...,3}
  \draw (\c,\b*-1) rectangle (\c+1,(\b*-1-1);

\draw[fill=4x3x2color1] (0,-0) rectangle (1,-1);
\node at (0.5,-0.5) {1};
\draw[fill=red] (1,-0) rectangle (2,-1);
\draw[fill=4x3x2color3] (2,-0) rectangle (3,-1);
\node at (2.5,-0.5) {3};
\draw[fill=red] (3,-0) rectangle (4,-1);
\draw[fill=red] (0,-1) rectangle (1,-2);
\draw[fill=4x3x2color1] (1,-1) rectangle (2,-2);
\node at (1.5,-1.5) {1};
\draw[fill=4x3x2color2] (2,-1) rectangle (3,-2);
\node at (2.5,-1.5) {2};
\draw[fill=4x3x2color1] (3,-1) rectangle (4,-2);
\node at (3.5,-1.5) {1};
\draw[fill=4x3x2color1] (0,-2) rectangle (1,-3);
\node at (0.5,-2.5) {1};
\draw[fill=red] (1,-2) rectangle (2,-3);
\draw[fill=4x3x2color1] (2,-2) rectangle (3,-3);
\node at (2.5,-2.5) {1};
\draw[fill=red] (3,-2) rectangle (4,-3);
\end{scope}

\end{tikzpicture}
\end{center}

%% file: 2-3-5-tikz.txt
\definecolor{5x3x2color1}{RGB}{255, 128, 128}
\definecolor{5x3x2color2}{RGB}{255, 143, 143}
\definecolor{5x3x2color3}{RGB}{255, 159, 159}
\definecolor{5x3x2color4}{RGB}{255, 175, 175}
\definecolor{5x3x2color5}{RGB}{255, 191, 191}
\definecolor{5x3x2color6}{RGB}{255, 207, 207}
\definecolor{5x3x2color7}{RGB}{255, 223, 223}
\definecolor{5x3x2color8}{RGB}{255, 255, 255}

\begin{center}
\begin{tikzpicture}[scale=0.5]
\begin{scope}[xshift=0cm,yshift=0cm]
\draw[very thick] (0,0) rectangle (5,-3);
\foreach \b in {0,...,2}
 \foreach \c in {0,...,4}
  \draw (\c,\b*-1) rectangle (\c+1,(\b*-1-1);

\draw[fill=red] (0,-0) rectangle (1,-1);
\draw[fill=5x3x2color3] (1,-0) rectangle (2,-1);
\node at (1.5,-0.5) {3};
\draw[fill=red] (2,-0) rectangle (3,-1);
\draw[fill=5x3x2color1] (3,-0) rectangle (4,-1);
\node at (3.5,-0.5) {1};
\draw[fill=red] (4,-0) rectangle (5,-1);
\draw[fill=5x3x2color1] (0,-1) rectangle (1,-2);
\node at (0.5,-1.5) {1};
\draw[fill=5x3x2color2] (1,-1) rectangle (2,-2);
\node at (1.5,-1.5) {2};
\draw[fill=5x3x2color1] (2,-1) rectangle (3,-2);
\node at (2.5,-1.5) {1};
\draw[fill=red] (3,-1) rectangle (4,-2);
\draw[fill=5x3x2color1] (4,-1) rectangle (5,-2);
\node at (4.5,-1.5) {1};
\draw[fill=red] (0,-2) rectangle (1,-3);
\draw[fill=5x3x2color1] (1,-2) rectangle (2,-3);
\node at (1.5,-2.5) {1};
\draw[fill=red] (2,-2) rectangle (3,-3);
\draw[fill=5x3x2color1] (3,-2) rectangle (4,-3);
\node at (3.5,-2.5) {1};
\draw[fill=red] (4,-2) rectangle (5,-3);
\end{scope}

\begin{scope}[xshift=6cm,yshift=0cm]
\draw[very thick] (0,0) rectangle (5,-3);
\foreach \b in {0,...,2}
 \foreach \c in {0,...,4}
  \draw (\c,\b*-1) rectangle (\c+1,(\b*-1-1);

\draw[fill=5x3x2color7] (0,-0) rectangle (1,-1);
\node at (0.5,-0.5) {7};
\draw[fill=5x3x2color6] (1,-0) rectangle (2,-1);
\node at (1.5,-0.5) {6};
\draw[fill=5x3x2color5] (2,-0) rectangle (3,-1);
\node at (2.5,-0.5) {5};
\draw[fill=red] (3,-0) rectangle (4,-1);
\draw[fill=5x3x2color1] (4,-0) rectangle (5,-1);
\node at (4.5,-0.5) {1};
\draw[fill=red] (0,-1) rectangle (1,-2);
\draw[fill=5x3x2color3] (1,-1) rectangle (2,-2);
\node at (1.5,-1.5) {3};
\draw[fill=5x3x2color4] (2,-1) rectangle (3,-2);
\node at (2.5,-1.5) {4};
\draw[fill=5x3x2color1] (3,-1) rectangle (4,-2);
\node at (3.5,-1.5) {1};
\draw[fill=red] (4,-1) rectangle (5,-2);
\draw[fill=5x3x2color1] (0,-2) rectangle (1,-3);
\node at (0.5,-2.5) {1};
\draw[fill=red] (1,-2) rectangle (2,-3);
\draw[fill=5x3x2color5] (2,-2) rectangle (3,-3);
\node at (2.5,-2.5) {5};
\draw[fill=5x3x2color6] (3,-2) rectangle (4,-3);
\node at (3.5,-2.5) {6};
\draw[fill=5x3x2color7] (4,-2) rectangle (5,-3);
\node at (4.5,-2.5) {7};
\end{scope}

\end{tikzpicture}
\end{center}

%% file: 2-4-4-tikz.txt
\definecolor{4x4x2color1}{RGB}{255, 128, 128}
\definecolor{4x4x2color2}{RGB}{255, 143, 143}
\definecolor{4x4x2color3}{RGB}{255, 159, 159}
\definecolor{4x4x2color4}{RGB}{255, 175, 175}
\definecolor{4x4x2color5}{RGB}{255, 191, 191}
\definecolor{4x4x2color6}{RGB}{255, 207, 207}
\definecolor{4x4x2color7}{RGB}{255, 223, 223}
\definecolor{4x4x2color8}{RGB}{255, 255, 255}

\begin{center}
\begin{tikzpicture}[scale=0.5]
\begin{scope}[xshift=0cm,yshift=0cm]
\draw[very thick] (0,0) rectangle (4,-4);
\foreach \b in {0,...,3}
 \foreach \c in {0,...,3}
  \draw (\c,\b*-1) rectangle (\c+1,(\b*-1-1);

\draw[fill=red] (0,-0) rectangle (1,-1);
\draw[fill=4x4x2color3] (1,-0) rectangle (2,-1);
\node at (1.5,-0.5) {3};
\draw[fill=4x4x2color6] (2,-0) rectangle (3,-1);
\node at (2.5,-0.5) {6};
\draw[fill=4x4x2color7] (3,-0) rectangle (4,-1);
\node at (3.5,-0.5) {7};
\draw[fill=4x4x2color1] (0,-1) rectangle (1,-2);
\node at (0.5,-1.5) {1};
\draw[fill=4x4x2color2] (1,-1) rectangle (2,-2);
\node at (1.5,-1.5) {2};
\draw[fill=4x4x2color3] (2,-1) rectangle (3,-2);
\node at (2.5,-1.5) {3};
\draw[fill=red] (3,-1) rectangle (4,-2);
\draw[fill=red] (0,-2) rectangle (1,-3);
\draw[fill=4x4x2color1] (1,-2) rectangle (2,-3);
\node at (1.5,-2.5) {1};
\draw[fill=red] (2,-2) rectangle (3,-3);
\draw[fill=4x4x2color1] (3,-2) rectangle (4,-3);
\node at (3.5,-2.5) {1};
\draw[fill=4x4x2color3] (0,-3) rectangle (1,-4);
\node at (0.5,-3.5) {3};
\draw[fill=4x4x2color2] (1,-3) rectangle (2,-4);
\node at (1.5,-3.5) {2};
\draw[fill=4x4x2color1] (2,-3) rectangle (3,-4);
\node at (2.5,-3.5) {1};
\draw[fill=red] (3,-3) rectangle (4,-4);
\end{scope}

\begin{scope}[xshift=5cm,yshift=0cm]
\draw[very thick] (0,0) rectangle (4,-4);
\foreach \b in {0,...,3}
 \foreach \c in {0,...,3}
  \draw (\c,\b*-1) rectangle (\c+1,(\b*-1-1);

\draw[fill=4x4x2color1] (0,-0) rectangle (1,-1);
\node at (0.5,-0.5) {1};
\draw[fill=red] (1,-0) rectangle (2,-1);
\draw[fill=4x4x2color5] (2,-0) rectangle (3,-1);
\node at (2.5,-0.5) {5};
\draw[fill=red] (3,-0) rectangle (4,-1);
\draw[fill=red] (0,-1) rectangle (1,-2);
\draw[fill=4x4x2color1] (1,-1) rectangle (2,-2);
\node at (1.5,-1.5) {1};
\draw[fill=4x4x2color4] (2,-1) rectangle (3,-2);
\node at (2.5,-1.5) {4};
\draw[fill=4x4x2color5] (3,-1) rectangle (4,-2);
\node at (3.5,-1.5) {5};
\draw[fill=4x4x2color1] (0,-2) rectangle (1,-3);
\node at (0.5,-2.5) {1};
\draw[fill=red] (1,-2) rectangle (2,-3);
\draw[fill=4x4x2color1] (2,-2) rectangle (3,-3);
\node at (2.5,-2.5) {1};
\draw[fill=4x4x2color6] (3,-2) rectangle (4,-3);
\node at (3.5,-2.5) {6};
\draw[fill=red] (0,-3) rectangle (1,-4);
\draw[fill=4x4x2color1] (1,-3) rectangle (2,-4);
\node at (1.5,-3.5) {1};
\draw[fill=red] (2,-3) rectangle (3,-4);
\draw[fill=4x4x2color7] (3,-3) rectangle (4,-4);
\node at (3.5,-3.5) {7};
\end{scope}

\end{tikzpicture}
\end{center}

%% file: 2-4-5-tikz.txt
\definecolor{5x4x2color1}{RGB}{255, 128, 128}
\definecolor{5x4x2color2}{RGB}{255, 138, 138}
\definecolor{5x4x2color3}{RGB}{255, 149, 149}
\definecolor{5x4x2color4}{RGB}{255, 159, 159}
\definecolor{5x4x2color5}{RGB}{255, 170, 170}
\definecolor{5x4x2color6}{RGB}{255, 180, 180}
\definecolor{5x4x2color7}{RGB}{255, 191, 191}
\definecolor{5x4x2color8}{RGB}{255, 202, 202}
\definecolor{5x4x2color9}{RGB}{255, 212, 212}
\definecolor{5x4x2color10}{RGB}{255, 223, 223}
\definecolor{5x4x2color11}{RGB}{255, 233, 233}
\definecolor{5x4x2color12}{RGB}{255, 255, 255}

\begin{center}
\begin{tikzpicture}[scale=0.5]
\begin{scope}[xshift=0cm,yshift=0cm]
\draw[very thick] (0,0) rectangle (5,-4);
\foreach \b in {0,...,3}
 \foreach \c in {0,...,4}
  \draw (\c,\b*-1) rectangle (\c+1,(\b*-1-1);

\draw[fill=red] (0,-0) rectangle (1,-1);
\draw[fill=5x4x2color5] (1,-0) rectangle (2,-1);
\node at (1.5,-0.5) {5};
\draw[fill=5x4x2color6] (2,-0) rectangle (3,-1);
\node at (2.5,-0.5) {6};
\draw[fill=5x4x2color7] (3,-0) rectangle (4,-1);
\node at (3.5,-0.5) {7};
\draw[fill=red] (4,-0) rectangle (5,-1);
\draw[fill=5x4x2color1] (0,-1) rectangle (1,-2);
\node at (0.5,-1.5) {1};
\draw[fill=5x4x2color4] (1,-1) rectangle (2,-2);
\node at (1.5,-1.5) {4};
\draw[fill=5x4x2color5] (2,-1) rectangle (3,-2);
\node at (2.5,-1.5) {5};
\draw[fill=5x4x2color8] (3,-1) rectangle (4,-2);
\node at (3.5,-1.5) {8};
\draw[fill=5x4x2color9] (4,-1) rectangle (5,-2);
\node at (4.5,-1.5) {9};
\draw[fill=red] (0,-2) rectangle (1,-3);
\draw[fill=5x4x2color3] (1,-2) rectangle (2,-3);
\node at (1.5,-2.5) {3};
\draw[fill=5x4x2color4] (2,-2) rectangle (3,-3);
\node at (2.5,-2.5) {4};
\draw[fill=5x4x2color5] (3,-2) rectangle (4,-3);
\node at (3.5,-2.5) {5};
\draw[fill=red] (4,-2) rectangle (5,-3);
\draw[fill=5x4x2color1] (0,-3) rectangle (1,-4);
\node at (0.5,-3.5) {1};
\draw[fill=red] (1,-3) rectangle (2,-4);
\draw[fill=5x4x2color1] (2,-3) rectangle (3,-4);
\node at (2.5,-3.5) {1};
\draw[fill=red] (3,-3) rectangle (4,-4);
\draw[fill=5x4x2color1] (4,-3) rectangle (5,-4);
\node at (4.5,-3.5) {1};
\end{scope}

\begin{scope}[xshift=6cm,yshift=0cm]
\draw[very thick] (0,0) rectangle (5,-4);
\foreach \b in {0,...,3}
 \foreach \c in {0,...,4}
  \draw (\c,\b*-1) rectangle (\c+1,(\b*-1-1);

\draw[fill=5x4x2color1] (0,-0) rectangle (1,-1);
\node at (0.5,-0.5) {1};
\draw[fill=red] (1,-0) rectangle (2,-1);
\draw[fill=5x4x2color1] (2,-0) rectangle (3,-1);
\node at (2.5,-0.5) {1};
\draw[fill=red] (3,-0) rectangle (4,-1);
\draw[fill=5x4x2color11] (4,-0) rectangle (5,-1);
\node at (4.5,-0.5) {11};
\draw[fill=red] (0,-1) rectangle (1,-2);
\draw[fill=5x4x2color1] (1,-1) rectangle (2,-2);
\node at (1.5,-1.5) {1};
\draw[fill=red] (2,-1) rectangle (3,-2);
\draw[fill=5x4x2color7] (3,-1) rectangle (4,-2);
\node at (3.5,-1.5) {7};
\draw[fill=5x4x2color10] (4,-1) rectangle (5,-2);
\node at (4.5,-1.5) {10};
\draw[fill=5x4x2color1] (0,-2) rectangle (1,-3);
\node at (0.5,-2.5) {1};
\draw[fill=5x4x2color2] (1,-2) rectangle (2,-3);
\node at (1.5,-2.5) {2};
\draw[fill=5x4x2color3] (2,-2) rectangle (3,-3);
\node at (2.5,-2.5) {3};
\draw[fill=5x4x2color6] (3,-2) rectangle (4,-3);
\node at (3.5,-2.5) {6};
\draw[fill=5x4x2color7] (4,-2) rectangle (5,-3);
\node at (4.5,-2.5) {7};
\draw[fill=red] (0,-3) rectangle (1,-4);
\draw[fill=5x4x2color1] (1,-3) rectangle (2,-4);
\node at (1.5,-3.5) {1};
\draw[fill=red] (2,-3) rectangle (3,-4);
\draw[fill=5x4x2color1] (3,-3) rectangle (4,-4);
\node at (3.5,-3.5) {1};
\draw[fill=red] (4,-3) rectangle (5,-4);
\end{scope}

\end{tikzpicture}
\end{center}

%% file: 2-4-6-tikz.txt
\definecolor{6x4x2color1}{RGB}{255, 128, 128}
\definecolor{6x4x2color2}{RGB}{255, 140, 140}
\definecolor{6x4x2color3}{RGB}{255, 153, 153}
\definecolor{6x4x2color4}{RGB}{255, 166, 166}
\definecolor{6x4x2color5}{RGB}{255, 178, 178}
\definecolor{6x4x2color6}{RGB}{255, 191, 191}
\definecolor{6x4x2color7}{RGB}{255, 204, 204}
\definecolor{6x4x2color8}{RGB}{255, 216, 216}
\definecolor{6x4x2color9}{RGB}{255, 229, 229}
\definecolor{6x4x2color10}{RGB}{255, 255, 255}

\begin{center}
\begin{tikzpicture}[scale=0.5]
\begin{scope}[xshift=0cm,yshift=0cm]
\draw[very thick] (0,0) rectangle (6,-4);
\foreach \b in {0,...,3}
 \foreach \c in {0,...,5}
  \draw (\c,\b*-1) rectangle (\c+1,(\b*-1-1);

\draw[fill=red] (0,-0) rectangle (1,-1);
\draw[fill=6x4x2color1] (1,-0) rectangle (2,-1);
\node at (1.5,-0.5) {1};
\draw[fill=6x4x2color2] (2,-0) rectangle (3,-1);
\node at (2.5,-0.5) {2};
\draw[fill=6x4x2color3] (3,-0) rectangle (4,-1);
\node at (3.5,-0.5) {3};
\draw[fill=red] (4,-0) rectangle (5,-1);
\draw[fill=6x4x2color9] (5,-0) rectangle (6,-1);
\node at (5.5,-0.5) {9};
\draw[fill=6x4x2color1] (0,-1) rectangle (1,-2);
\node at (0.5,-1.5) {1};
\draw[fill=red] (1,-1) rectangle (2,-2);
\draw[fill=6x4x2color1] (2,-1) rectangle (3,-2);
\node at (2.5,-1.5) {1};
\draw[fill=6x4x2color4] (3,-1) rectangle (4,-2);
\node at (3.5,-1.5) {4};
\draw[fill=6x4x2color5] (4,-1) rectangle (5,-2);
\node at (4.5,-1.5) {5};
\draw[fill=6x4x2color8] (5,-1) rectangle (6,-2);
\node at (5.5,-1.5) {8};
\draw[fill=red] (0,-2) rectangle (1,-3);
\draw[fill=6x4x2color1] (1,-2) rectangle (2,-3);
\node at (1.5,-2.5) {1};
\draw[fill=red] (2,-2) rectangle (3,-3);
\draw[fill=6x4x2color1] (3,-2) rectangle (4,-3);
\node at (3.5,-2.5) {1};
\draw[fill=red] (4,-2) rectangle (5,-3);
\draw[fill=6x4x2color1] (5,-2) rectangle (6,-3);
\node at (5.5,-2.5) {1};
\draw[fill=6x4x2color5] (0,-3) rectangle (1,-4);
\node at (0.5,-3.5) {5};
\draw[fill=6x4x2color4] (1,-3) rectangle (2,-4);
\node at (1.5,-3.5) {4};
\draw[fill=6x4x2color1] (2,-3) rectangle (3,-4);
\node at (2.5,-3.5) {1};
\draw[fill=red] (3,-3) rectangle (4,-4);
\draw[fill=6x4x2color1] (4,-3) rectangle (5,-4);
\node at (4.5,-3.5) {1};
\draw[fill=red] (5,-3) rectangle (6,-4);
\end{scope}

\begin{scope}[xshift=7cm,yshift=0cm]
\draw[very thick] (0,0) rectangle (6,-4);
\foreach \b in {0,...,3}
 \foreach \c in {0,...,5}
  \draw (\c,\b*-1) rectangle (\c+1,(\b*-1-1);

\draw[fill=6x4x2color5] (0,-0) rectangle (1,-1);
\node at (0.5,-0.5) {5};
\draw[fill=red] (1,-0) rectangle (2,-1);
\draw[fill=6x4x2color1] (2,-0) rectangle (3,-1);
\node at (2.5,-0.5) {1};
\draw[fill=red] (3,-0) rectangle (4,-1);
\draw[fill=6x4x2color1] (4,-0) rectangle (5,-1);
\node at (4.5,-0.5) {1};
\draw[fill=red] (5,-0) rectangle (6,-1);
\draw[fill=6x4x2color4] (0,-1) rectangle (1,-2);
\node at (0.5,-1.5) {4};
\draw[fill=6x4x2color1] (1,-1) rectangle (2,-2);
\node at (1.5,-1.5) {1};
\draw[fill=red] (2,-1) rectangle (3,-2);
\draw[fill=6x4x2color5] (3,-1) rectangle (4,-2);
\node at (3.5,-1.5) {5};
\draw[fill=6x4x2color6] (4,-1) rectangle (5,-2);
\node at (4.5,-1.5) {6};
\draw[fill=6x4x2color7] (5,-1) rectangle (6,-2);
\node at (5.5,-1.5) {7};
\draw[fill=6x4x2color3] (0,-2) rectangle (1,-3);
\node at (0.5,-2.5) {3};
\draw[fill=6x4x2color2] (1,-2) rectangle (2,-3);
\node at (1.5,-2.5) {2};
\draw[fill=6x4x2color1] (2,-2) rectangle (3,-3);
\node at (2.5,-2.5) {1};
\draw[fill=6x4x2color6] (3,-2) rectangle (4,-3);
\node at (3.5,-2.5) {6};
\draw[fill=6x4x2color7] (4,-2) rectangle (5,-3);
\node at (4.5,-2.5) {7};
\draw[fill=red] (5,-2) rectangle (6,-3);
\draw[fill=red] (0,-3) rectangle (1,-4);
\draw[fill=6x4x2color3] (1,-3) rectangle (2,-4);
\node at (1.5,-3.5) {3};
\draw[fill=red] (2,-3) rectangle (3,-4);
\draw[fill=6x4x2color7] (3,-3) rectangle (4,-4);
\node at (3.5,-3.5) {7};
\draw[fill=6x4x2color8] (4,-3) rectangle (5,-4);
\node at (4.5,-3.5) {8};
\draw[fill=6x4x2color9] (5,-3) rectangle (6,-4);
\node at (5.5,-3.5) {9};
\end{scope}

\end{tikzpicture}
\end{center}

%% file: 2-4-7-tikz.txt
\definecolor{7x4x2color1}{RGB}{255, 128, 128}
\definecolor{7x4x2color2}{RGB}{255, 138, 138}
\definecolor{7x4x2color3}{RGB}{255, 149, 149}
\definecolor{7x4x2color4}{RGB}{255, 159, 159}
\definecolor{7x4x2color5}{RGB}{255, 170, 170}
\definecolor{7x4x2color6}{RGB}{255, 180, 180}
\definecolor{7x4x2color7}{RGB}{255, 191, 191}
\definecolor{7x4x2color8}{RGB}{255, 202, 202}
\definecolor{7x4x2color9}{RGB}{255, 212, 212}
\definecolor{7x4x2color10}{RGB}{255, 223, 223}
\definecolor{7x4x2color11}{RGB}{255, 233, 233}
\definecolor{7x4x2color12}{RGB}{255, 255, 255}

\begin{center}
\begin{tikzpicture}[scale=0.5]
\begin{scope}[xshift=0cm,yshift=0cm]
\draw[very thick] (0,0) rectangle (7,-4);
\foreach \b in {0,...,3}
 \foreach \c in {0,...,6}
  \draw (\c,\b*-1) rectangle (\c+1,(\b*-1-1);

\draw[fill=red] (0,-0) rectangle (1,-1);
\draw[fill=7x4x2color9] (1,-0) rectangle (2,-1);
\node at (1.5,-0.5) {9};
\draw[fill=red] (2,-0) rectangle (3,-1);
\draw[fill=7x4x2color7] (3,-0) rectangle (4,-1);
\node at (3.5,-0.5) {7};
\draw[fill=red] (4,-0) rectangle (5,-1);
\draw[fill=7x4x2color1] (5,-0) rectangle (6,-1);
\node at (5.5,-0.5) {1};
\draw[fill=red] (6,-0) rectangle (7,-1);
\draw[fill=7x4x2color9] (0,-1) rectangle (1,-2);
\node at (0.5,-1.5) {9};
\draw[fill=7x4x2color8] (1,-1) rectangle (2,-2);
\node at (1.5,-1.5) {8};
\draw[fill=7x4x2color7] (2,-1) rectangle (3,-2);
\node at (2.5,-1.5) {7};
\draw[fill=7x4x2color6] (3,-1) rectangle (4,-2);
\node at (3.5,-1.5) {6};
\draw[fill=7x4x2color3] (4,-1) rectangle (5,-2);
\node at (4.5,-1.5) {3};
\draw[fill=7x4x2color2] (5,-1) rectangle (6,-2);
\node at (5.5,-1.5) {2};
\draw[fill=7x4x2color3] (6,-1) rectangle (7,-2);
\node at (6.5,-1.5) {3};
\draw[fill=7x4x2color10] (0,-2) rectangle (1,-3);
\node at (0.5,-2.5) {10};
\draw[fill=7x4x2color7] (1,-2) rectangle (2,-3);
\node at (1.5,-2.5) {7};
\draw[fill=7x4x2color6] (2,-2) rectangle (3,-3);
\node at (2.5,-2.5) {6};
\draw[fill=7x4x2color1] (3,-2) rectangle (4,-3);
\node at (3.5,-2.5) {1};
\draw[fill=red] (4,-2) rectangle (5,-3);
\draw[fill=7x4x2color1] (5,-2) rectangle (6,-3);
\node at (5.5,-2.5) {1};
\draw[fill=7x4x2color4] (6,-2) rectangle (7,-3);
\node at (6.5,-2.5) {4};
\draw[fill=7x4x2color11] (0,-3) rectangle (1,-4);
\node at (0.5,-3.5) {11};
\draw[fill=red] (1,-3) rectangle (2,-4);
\draw[fill=7x4x2color5] (2,-3) rectangle (3,-4);
\node at (2.5,-3.5) {5};
\draw[fill=red] (3,-3) rectangle (4,-4);
\draw[fill=7x4x2color1] (4,-3) rectangle (5,-4);
\node at (4.5,-3.5) {1};
\draw[fill=red] (5,-3) rectangle (6,-4);
\draw[fill=7x4x2color5] (6,-3) rectangle (7,-4);
\node at (6.5,-3.5) {5};
\end{scope}

\begin{scope}[xshift=8cm,yshift=0cm]
\draw[very thick] (0,0) rectangle (7,-4);
\foreach \b in {0,...,3}
 \foreach \c in {0,...,6}
  \draw (\c,\b*-1) rectangle (\c+1,(\b*-1-1);

\draw[fill=7x4x2color11] (0,-0) rectangle (1,-1);
\node at (0.5,-0.5) {11};
\draw[fill=7x4x2color10] (1,-0) rectangle (2,-1);
\node at (1.5,-0.5) {10};
\draw[fill=7x4x2color9] (2,-0) rectangle (3,-1);
\node at (2.5,-0.5) {9};
\draw[fill=7x4x2color8] (3,-0) rectangle (4,-1);
\node at (3.5,-0.5) {8};
\draw[fill=7x4x2color5] (4,-0) rectangle (5,-1);
\node at (4.5,-0.5) {5};
\draw[fill=red] (5,-0) rectangle (6,-1);
\draw[fill=7x4x2color1] (6,-0) rectangle (7,-1);
\node at (6.5,-0.5) {1};
\draw[fill=red] (0,-1) rectangle (1,-2);
\draw[fill=7x4x2color1] (1,-1) rectangle (2,-2);
\node at (1.5,-1.5) {1};
\draw[fill=red] (2,-1) rectangle (3,-2);
\draw[fill=7x4x2color5] (3,-1) rectangle (4,-2);
\node at (3.5,-1.5) {5};
\draw[fill=7x4x2color4] (4,-1) rectangle (5,-2);
\node at (4.5,-1.5) {4};
\draw[fill=7x4x2color1] (5,-1) rectangle (6,-2);
\node at (5.5,-1.5) {1};
\draw[fill=red] (6,-1) rectangle (7,-2);
\draw[fill=7x4x2color1] (0,-2) rectangle (1,-3);
\node at (0.5,-2.5) {1};
\draw[fill=red] (1,-2) rectangle (2,-3);
\draw[fill=7x4x2color1] (2,-2) rectangle (3,-3);
\node at (2.5,-2.5) {1};
\draw[fill=red] (3,-2) rectangle (4,-3);
\draw[fill=7x4x2color1] (4,-2) rectangle (5,-3);
\node at (4.5,-2.5) {1};
\draw[fill=red] (5,-2) rectangle (6,-3);
\draw[fill=7x4x2color1] (6,-2) rectangle (7,-3);
\node at (6.5,-2.5) {1};
\draw[fill=red] (0,-3) rectangle (1,-4);
\draw[fill=7x4x2color1] (1,-3) rectangle (2,-4);
\node at (1.5,-3.5) {1};
\draw[fill=7x4x2color4] (2,-3) rectangle (3,-4);
\node at (2.5,-3.5) {4};
\draw[fill=7x4x2color3] (3,-3) rectangle (4,-4);
\node at (3.5,-3.5) {3};
\draw[fill=7x4x2color2] (4,-3) rectangle (5,-4);
\node at (4.5,-3.5) {2};
\draw[fill=7x4x2color1] (5,-3) rectangle (6,-4);
\node at (5.5,-3.5) {1};
\draw[fill=red] (6,-3) rectangle (7,-4);
\end{scope}

\end{tikzpicture}
\end{center}

%% file: 2-5-6-tikz.txt
\definecolor{6x5x2color1}{RGB}{255, 128, 128}
\definecolor{6x5x2color2}{RGB}{255, 140, 140}
\definecolor{6x5x2color3}{RGB}{255, 153, 153}
\definecolor{6x5x2color4}{RGB}{255, 166, 166}
\definecolor{6x5x2color5}{RGB}{255, 178, 178}
\definecolor{6x5x2color6}{RGB}{255, 191, 191}
\definecolor{6x5x2color7}{RGB}{255, 204, 204}
\definecolor{6x5x2color8}{RGB}{255, 216, 216}
\definecolor{6x5x2color9}{RGB}{255, 229, 229}
\definecolor{6x5x2color10}{RGB}{255, 255, 255}

\begin{center}
\begin{tikzpicture}[scale=0.5]
\begin{scope}[xshift=0cm,yshift=0cm]
\draw[very thick] (0,0) rectangle (6,-5);
\foreach \b in {0,...,4}
 \foreach \c in {0,...,5}
  \draw (\c,\b*-1) rectangle (\c+1,(\b*-1-1);

\draw[fill=red] (0,-0) rectangle (1,-1);
\draw[fill=6x5x2color5] (1,-0) rectangle (2,-1);
\node at (1.5,-0.5) {5};
\draw[fill=6x5x2color6] (2,-0) rectangle (3,-1);
\node at (2.5,-0.5) {6};
\draw[fill=6x5x2color7] (3,-0) rectangle (4,-1);
\node at (3.5,-0.5) {7};
\draw[fill=6x5x2color8] (4,-0) rectangle (5,-1);
\node at (4.5,-0.5) {8};
\draw[fill=6x5x2color9] (5,-0) rectangle (6,-1);
\node at (5.5,-0.5) {9};
\draw[fill=6x5x2color1] (0,-1) rectangle (1,-2);
\node at (0.5,-1.5) {1};
\draw[fill=6x5x2color4] (1,-1) rectangle (2,-2);
\node at (1.5,-1.5) {4};
\draw[fill=6x5x2color5] (2,-1) rectangle (3,-2);
\node at (2.5,-1.5) {5};
\draw[fill=6x5x2color6] (3,-1) rectangle (4,-2);
\node at (3.5,-1.5) {6};
\draw[fill=6x5x2color7] (4,-1) rectangle (5,-2);
\node at (4.5,-1.5) {7};
\draw[fill=red] (5,-1) rectangle (6,-2);
\draw[fill=red] (0,-2) rectangle (1,-3);
\draw[fill=6x5x2color3] (1,-2) rectangle (2,-3);
\node at (1.5,-2.5) {3};
\draw[fill=6x5x2color4] (2,-2) rectangle (3,-3);
\node at (2.5,-2.5) {4};
\draw[fill=6x5x2color5] (3,-2) rectangle (4,-3);
\node at (3.5,-2.5) {5};
\draw[fill=6x5x2color6] (4,-2) rectangle (5,-3);
\node at (4.5,-2.5) {6};
\draw[fill=6x5x2color7] (5,-2) rectangle (6,-3);
\node at (5.5,-2.5) {7};
\draw[fill=6x5x2color1] (0,-3) rectangle (1,-4);
\node at (0.5,-3.5) {1};
\draw[fill=6x5x2color2] (1,-3) rectangle (2,-4);
\node at (1.5,-3.5) {2};
\draw[fill=6x5x2color3] (2,-3) rectangle (3,-4);
\node at (2.5,-3.5) {3};
\draw[fill=6x5x2color4] (3,-3) rectangle (4,-4);
\node at (3.5,-3.5) {4};
\draw[fill=6x5x2color5] (4,-3) rectangle (5,-4);
\node at (4.5,-3.5) {5};
\draw[fill=6x5x2color8] (5,-3) rectangle (6,-4);
\node at (5.5,-3.5) {8};
\draw[fill=red] (0,-4) rectangle (1,-5);
\draw[fill=6x5x2color1] (1,-4) rectangle (2,-5);
\node at (1.5,-4.5) {1};
\draw[fill=red] (2,-4) rectangle (3,-5);
\draw[fill=6x5x2color3] (3,-4) rectangle (4,-5);
\node at (3.5,-4.5) {3};
\draw[fill=red] (4,-4) rectangle (5,-5);
\draw[fill=6x5x2color9] (5,-4) rectangle (6,-5);
\node at (5.5,-4.5) {9};
\end{scope}

\begin{scope}[xshift=7cm,yshift=0cm]
\draw[very thick] (0,0) rectangle (6,-5);
\foreach \b in {0,...,4}
 \foreach \c in {0,...,5}
  \draw (\c,\b*-1) rectangle (\c+1,(\b*-1-1);

\draw[fill=6x5x2color1] (0,-0) rectangle (1,-1);
\node at (0.5,-0.5) {1};
\draw[fill=red] (1,-0) rectangle (2,-1);
\draw[fill=6x5x2color1] (2,-0) rectangle (3,-1);
\node at (2.5,-0.5) {1};
\draw[fill=red] (3,-0) rectangle (4,-1);
\draw[fill=6x5x2color3] (4,-0) rectangle (5,-1);
\node at (4.5,-0.5) {3};
\draw[fill=red] (5,-0) rectangle (6,-1);
\draw[fill=red] (0,-1) rectangle (1,-2);
\draw[fill=6x5x2color1] (1,-1) rectangle (2,-2);
\node at (1.5,-1.5) {1};
\draw[fill=red] (2,-1) rectangle (3,-2);
\draw[fill=6x5x2color1] (3,-1) rectangle (4,-2);
\node at (3.5,-1.5) {1};
\draw[fill=6x5x2color2] (4,-1) rectangle (5,-2);
\node at (4.5,-1.5) {2};
\draw[fill=6x5x2color1] (5,-1) rectangle (6,-2);
\node at (5.5,-1.5) {1};
\draw[fill=6x5x2color1] (0,-2) rectangle (1,-3);
\node at (0.5,-2.5) {1};
\draw[fill=6x5x2color2] (1,-2) rectangle (2,-3);
\node at (1.5,-2.5) {2};
\draw[fill=6x5x2color1] (2,-2) rectangle (3,-3);
\node at (2.5,-2.5) {1};
\draw[fill=red] (3,-2) rectangle (4,-3);
\draw[fill=6x5x2color1] (4,-2) rectangle (5,-3);
\node at (4.5,-2.5) {1};
\draw[fill=red] (5,-2) rectangle (6,-3);
\draw[fill=red] (0,-3) rectangle (1,-4);
\draw[fill=6x5x2color1] (1,-3) rectangle (2,-4);
\node at (1.5,-3.5) {1};
\draw[fill=red] (2,-3) rectangle (3,-4);
\draw[fill=6x5x2color1] (3,-3) rectangle (4,-4);
\node at (3.5,-3.5) {1};
\draw[fill=red] (4,-3) rectangle (5,-4);
\draw[fill=6x5x2color1] (5,-3) rectangle (6,-4);
\node at (5.5,-3.5) {1};
\draw[fill=6x5x2color1] (0,-4) rectangle (1,-5);
\node at (0.5,-4.5) {1};
\draw[fill=red] (1,-4) rectangle (2,-5);
\draw[fill=6x5x2color1] (2,-4) rectangle (3,-5);
\node at (2.5,-4.5) {1};
\draw[fill=6x5x2color2] (3,-4) rectangle (4,-5);
\node at (3.5,-4.5) {2};
\draw[fill=6x5x2color1] (4,-4) rectangle (5,-5);
\node at (4.5,-4.5) {1};
\draw[fill=red] (5,-4) rectangle (6,-5);
\end{scope}

\end{tikzpicture}
\end{center}

%% file: 2-5-7-tikz.txt
\definecolor{7x5x2color1}{RGB}{255, 128, 128}
\definecolor{7x5x2color2}{RGB}{255, 137, 137}
\definecolor{7x5x2color3}{RGB}{255, 146, 146}
\definecolor{7x5x2color4}{RGB}{255, 155, 155}
\definecolor{7x5x2color5}{RGB}{255, 164, 164}
\definecolor{7x5x2color6}{RGB}{255, 173, 173}
\definecolor{7x5x2color7}{RGB}{255, 182, 182}
\definecolor{7x5x2color8}{RGB}{255, 191, 191}
\definecolor{7x5x2color9}{RGB}{255, 200, 200}
\definecolor{7x5x2color10}{RGB}{255, 209, 209}
\definecolor{7x5x2color11}{RGB}{255, 218, 218}
\definecolor{7x5x2color12}{RGB}{255, 227, 227}
\definecolor{7x5x2color13}{RGB}{255, 236, 236}
\definecolor{7x5x2color14}{RGB}{255, 255, 255}

\begin{center}
\begin{tikzpicture}[scale=0.5]
\begin{scope}[xshift=0cm,yshift=0cm]
\draw[very thick] (0,0) rectangle (7,-5);
\foreach \b in {0,...,4}
 \foreach \c in {0,...,6}
  \draw (\c,\b*-1) rectangle (\c+1,(\b*-1-1);

\draw[fill=red] (0,-0) rectangle (1,-1);
\draw[fill=7x5x2color9] (1,-0) rectangle (2,-1);
\node at (1.5,-0.5) {9};
\draw[fill=7x5x2color8] (2,-0) rectangle (3,-1);
\node at (2.5,-0.5) {8};
\draw[fill=7x5x2color7] (3,-0) rectangle (4,-1);
\node at (3.5,-0.5) {7};
\draw[fill=red] (4,-0) rectangle (5,-1);
\draw[fill=7x5x2color3] (5,-0) rectangle (6,-1);
\node at (5.5,-0.5) {3};
\draw[fill=red] (6,-0) rectangle (7,-1);
\draw[fill=7x5x2color11] (0,-1) rectangle (1,-2);
\node at (0.5,-1.5) {11};
\draw[fill=7x5x2color10] (1,-1) rectangle (2,-2);
\node at (1.5,-1.5) {10};
\draw[fill=7x5x2color7] (2,-1) rectangle (3,-2);
\node at (2.5,-1.5) {7};
\draw[fill=7x5x2color6] (3,-1) rectangle (4,-2);
\node at (3.5,-1.5) {6};
\draw[fill=7x5x2color3] (4,-1) rectangle (5,-2);
\node at (4.5,-1.5) {3};
\draw[fill=7x5x2color2] (5,-1) rectangle (6,-2);
\node at (5.5,-1.5) {2};
\draw[fill=7x5x2color1] (6,-1) rectangle (7,-2);
\node at (6.5,-1.5) {1};
\draw[fill=red] (0,-2) rectangle (1,-3);
\draw[fill=7x5x2color3] (1,-2) rectangle (2,-3);
\node at (1.5,-2.5) {3};
\draw[fill=red] (2,-2) rectangle (3,-3);
\draw[fill=7x5x2color5] (3,-2) rectangle (4,-3);
\node at (3.5,-2.5) {5};
\draw[fill=7x5x2color4] (4,-2) rectangle (5,-3);
\node at (4.5,-2.5) {4};
\draw[fill=7x5x2color1] (5,-2) rectangle (6,-3);
\node at (5.5,-2.5) {1};
\draw[fill=red] (6,-2) rectangle (7,-3);
\draw[fill=7x5x2color1] (0,-3) rectangle (1,-4);
\node at (0.5,-3.5) {1};
\draw[fill=7x5x2color2] (1,-3) rectangle (2,-4);
\node at (1.5,-3.5) {2};
\draw[fill=7x5x2color1] (2,-3) rectangle (3,-4);
\node at (2.5,-3.5) {1};
\draw[fill=red] (3,-3) rectangle (4,-4);
\draw[fill=7x5x2color1] (4,-3) rectangle (5,-4);
\node at (4.5,-3.5) {1};
\draw[fill=red] (5,-3) rectangle (6,-4);
\draw[fill=7x5x2color1] (6,-3) rectangle (7,-4);
\node at (6.5,-3.5) {1};
\draw[fill=red] (0,-4) rectangle (1,-5);
\draw[fill=7x5x2color1] (1,-4) rectangle (2,-5);
\node at (1.5,-4.5) {1};
\draw[fill=red] (2,-4) rectangle (3,-5);
\draw[fill=7x5x2color1] (3,-4) rectangle (4,-5);
\node at (3.5,-4.5) {1};
\draw[fill=red] (4,-4) rectangle (5,-5);
\draw[fill=7x5x2color1] (5,-4) rectangle (6,-5);
\node at (5.5,-4.5) {1};
\draw[fill=red] (6,-4) rectangle (7,-5);
\end{scope}

\begin{scope}[xshift=8cm,yshift=0cm]
\draw[very thick] (0,0) rectangle (7,-5);
\foreach \b in {0,...,4}
 \foreach \c in {0,...,6}
  \draw (\c,\b*-1) rectangle (\c+1,(\b*-1-1);

\draw[fill=7x5x2color13] (0,-0) rectangle (1,-1);
\node at (0.5,-0.5) {13};
\draw[fill=red] (1,-0) rectangle (2,-1);
\draw[fill=7x5x2color1] (2,-0) rectangle (3,-1);
\node at (2.5,-0.5) {1};
\draw[fill=red] (3,-0) rectangle (4,-1);
\draw[fill=7x5x2color1] (4,-0) rectangle (5,-1);
\node at (4.5,-0.5) {1};
\draw[fill=7x5x2color4] (5,-0) rectangle (6,-1);
\node at (5.5,-0.5) {4};
\draw[fill=7x5x2color5] (6,-0) rectangle (7,-1);
\node at (6.5,-0.5) {5};
\draw[fill=7x5x2color12] (0,-1) rectangle (1,-2);
\node at (0.5,-1.5) {12};
\draw[fill=7x5x2color9] (1,-1) rectangle (2,-2);
\node at (1.5,-1.5) {9};
\draw[fill=red] (2,-1) rectangle (3,-2);
\draw[fill=7x5x2color1] (3,-1) rectangle (4,-2);
\node at (3.5,-1.5) {1};
\draw[fill=red] (4,-1) rectangle (5,-2);
\draw[fill=7x5x2color1] (5,-1) rectangle (6,-2);
\node at (5.5,-1.5) {1};
\draw[fill=red] (6,-1) rectangle (7,-2);
\draw[fill=7x5x2color9] (0,-2) rectangle (1,-3);
\node at (0.5,-2.5) {9};
\draw[fill=7x5x2color8] (1,-2) rectangle (2,-3);
\node at (1.5,-2.5) {8};
\draw[fill=7x5x2color7] (2,-2) rectangle (3,-3);
\node at (2.5,-2.5) {7};
\draw[fill=7x5x2color6] (3,-2) rectangle (4,-3);
\node at (3.5,-2.5) {6};
\draw[fill=7x5x2color5] (4,-2) rectangle (5,-3);
\node at (4.5,-2.5) {5};
\draw[fill=red] (5,-2) rectangle (6,-3);
\draw[fill=7x5x2color1] (6,-2) rectangle (7,-3);
\node at (6.5,-2.5) {1};
\draw[fill=red] (0,-3) rectangle (1,-4);
\draw[fill=7x5x2color3] (1,-3) rectangle (2,-4);
\node at (1.5,-3.5) {3};
\draw[fill=7x5x2color8] (2,-3) rectangle (3,-4);
\node at (2.5,-3.5) {8};
\draw[fill=7x5x2color9] (3,-3) rectangle (4,-4);
\node at (3.5,-3.5) {9};
\draw[fill=7x5x2color10] (4,-3) rectangle (5,-4);
\node at (4.5,-3.5) {10};
\draw[fill=7x5x2color11] (5,-3) rectangle (6,-4);
\node at (5.5,-3.5) {11};
\draw[fill=7x5x2color12] (6,-3) rectangle (7,-4);
\node at (6.5,-3.5) {12};
\draw[fill=7x5x2color1] (0,-4) rectangle (1,-5);
\node at (0.5,-4.5) {1};
\draw[fill=red] (1,-4) rectangle (2,-5);
\draw[fill=7x5x2color9] (2,-4) rectangle (3,-5);
\node at (2.5,-4.5) {9};
\draw[fill=7x5x2color10] (3,-4) rectangle (4,-5);
\node at (3.5,-4.5) {10};
\draw[fill=7x5x2color11] (4,-4) rectangle (5,-5);
\node at (4.5,-4.5) {11};
\draw[fill=7x5x2color12] (5,-4) rectangle (6,-5);
\node at (5.5,-4.5) {12};
\draw[fill=7x5x2color13] (6,-4) rectangle (7,-5);
\node at (6.5,-4.5) {13};
\end{scope}

\end{tikzpicture}
\end{center}

%% file: 2-6-7-tikz.txt
\definecolor{7x6x2color1}{RGB}{255, 128, 128}
\definecolor{7x6x2color2}{RGB}{255, 133, 133}
\definecolor{7x6x2color3}{RGB}{255, 138, 138}
\definecolor{7x6x2color4}{RGB}{255, 143, 143}
\definecolor{7x6x2color5}{RGB}{255, 149, 149}
\definecolor{7x6x2color6}{RGB}{255, 154, 154}
\definecolor{7x6x2color7}{RGB}{255, 159, 159}
\definecolor{7x6x2color8}{RGB}{255, 165, 165}
\definecolor{7x6x2color9}{RGB}{255, 170, 170}
\definecolor{7x6x2color10}{RGB}{255, 175, 175}
\definecolor{7x6x2color11}{RGB}{255, 180, 180}
\definecolor{7x6x2color12}{RGB}{255, 186, 186}
\definecolor{7x6x2color13}{RGB}{255, 191, 191}
\definecolor{7x6x2color14}{RGB}{255, 196, 196}
\definecolor{7x6x2color15}{RGB}{255, 202, 202}
\definecolor{7x6x2color16}{RGB}{255, 207, 207}
\definecolor{7x6x2color17}{RGB}{255, 212, 212}
\definecolor{7x6x2color18}{RGB}{255, 217, 217}
\definecolor{7x6x2color19}{RGB}{255, 223, 223}
\definecolor{7x6x2color20}{RGB}{255, 228, 228}
\definecolor{7x6x2color21}{RGB}{255, 233, 233}
\definecolor{7x6x2color22}{RGB}{255, 239, 239}
\definecolor{7x6x2color23}{RGB}{255, 244, 244}
\definecolor{7x6x2color24}{RGB}{255, 255, 255}

\begin{center}
\begin{tikzpicture}[scale=0.5]
\begin{scope}[xshift=0cm,yshift=0cm]
\draw[very thick] (0,0) rectangle (7,-6);
\foreach \b in {0,...,5}
 \foreach \c in {0,...,6}
  \draw (\c,\b*-1) rectangle (\c+1,(\b*-1-1);

\draw[fill=red] (0,-0) rectangle (1,-1);
\draw[fill=7x6x2color19] (1,-0) rectangle (2,-1);
\node at (1.5,-0.5) {19};
\draw[fill=red] (2,-0) rectangle (3,-1);
\draw[fill=7x6x2color1] (3,-0) rectangle (4,-1);
\node at (3.5,-0.5) {1};
\draw[fill=red] (4,-0) rectangle (5,-1);
\draw[fill=7x6x2color5] (5,-0) rectangle (6,-1);
\node at (5.5,-0.5) {5};
\draw[fill=red] (6,-0) rectangle (7,-1);
\draw[fill=7x6x2color19] (0,-1) rectangle (1,-2);
\node at (0.5,-1.5) {19};
\draw[fill=7x6x2color18] (1,-1) rectangle (2,-2);
\node at (1.5,-1.5) {18};
\draw[fill=7x6x2color15] (2,-1) rectangle (3,-2);
\node at (2.5,-1.5) {15};
\draw[fill=7x6x2color6] (3,-1) rectangle (4,-2);
\node at (3.5,-1.5) {6};
\draw[fill=7x6x2color5] (4,-1) rectangle (5,-2);
\node at (4.5,-1.5) {5};
\draw[fill=7x6x2color4] (5,-1) rectangle (6,-2);
\node at (5.5,-1.5) {4};
\draw[fill=7x6x2color1] (6,-1) rectangle (7,-2);
\node at (6.5,-1.5) {1};
\draw[fill=7x6x2color20] (0,-2) rectangle (1,-3);
\node at (0.5,-2.5) {20};
\draw[fill=7x6x2color15] (1,-2) rectangle (2,-3);
\node at (1.5,-2.5) {15};
\draw[fill=7x6x2color14] (2,-2) rectangle (3,-3);
\node at (2.5,-2.5) {14};
\draw[fill=7x6x2color7] (3,-2) rectangle (4,-3);
\node at (3.5,-2.5) {7};
\draw[fill=7x6x2color6] (4,-2) rectangle (5,-3);
\node at (4.5,-2.5) {6};
\draw[fill=7x6x2color3] (5,-2) rectangle (6,-3);
\node at (5.5,-2.5) {3};
\draw[fill=red] (6,-2) rectangle (7,-3);
\draw[fill=7x6x2color21] (0,-3) rectangle (1,-4);
\node at (0.5,-3.5) {21};
\draw[fill=7x6x2color14] (1,-3) rectangle (2,-4);
\node at (1.5,-3.5) {14};
\draw[fill=7x6x2color13] (2,-3) rectangle (3,-4);
\node at (2.5,-3.5) {13};
\draw[fill=7x6x2color8] (3,-3) rectangle (4,-4);
\node at (3.5,-3.5) {8};
\draw[fill=7x6x2color7] (4,-3) rectangle (5,-4);
\node at (4.5,-3.5) {7};
\draw[fill=red] (5,-3) rectangle (6,-4);
\draw[fill=7x6x2color1] (6,-3) rectangle (7,-4);
\node at (6.5,-3.5) {1};
\draw[fill=7x6x2color22] (0,-4) rectangle (1,-5);
\node at (0.5,-4.5) {22};
\draw[fill=7x6x2color13] (1,-4) rectangle (2,-5);
\node at (1.5,-4.5) {13};
\draw[fill=7x6x2color12] (2,-4) rectangle (3,-5);
\node at (2.5,-4.5) {12};
\draw[fill=7x6x2color9] (3,-4) rectangle (4,-5);
\node at (3.5,-4.5) {9};
\draw[fill=7x6x2color8] (4,-4) rectangle (5,-5);
\node at (4.5,-4.5) {8};
\draw[fill=7x6x2color3] (5,-4) rectangle (6,-5);
\node at (5.5,-4.5) {3};
\draw[fill=7x6x2color4] (6,-4) rectangle (7,-5);
\node at (6.5,-4.5) {4};
\draw[fill=7x6x2color23] (0,-5) rectangle (1,-6);
\node at (0.5,-5.5) {23};
\draw[fill=red] (1,-5) rectangle (2,-6);
\draw[fill=7x6x2color11] (2,-5) rectangle (3,-6);
\node at (2.5,-5.5) {11};
\draw[fill=7x6x2color10] (3,-5) rectangle (4,-6);
\node at (3.5,-5.5) {10};
\draw[fill=7x6x2color9] (4,-5) rectangle (5,-6);
\node at (4.5,-5.5) {9};
\draw[fill=red] (5,-5) rectangle (6,-6);
\draw[fill=7x6x2color5] (6,-5) rectangle (7,-6);
\node at (6.5,-5.5) {5};
\end{scope}

\begin{scope}[xshift=8cm,yshift=0cm]
\draw[very thick] (0,0) rectangle (7,-6);
\foreach \b in {0,...,5}
 \foreach \c in {0,...,6}
  \draw (\c,\b*-1) rectangle (\c+1,(\b*-1-1);

\draw[fill=7x6x2color21] (0,-0) rectangle (1,-1);
\node at (0.5,-0.5) {21};
\draw[fill=7x6x2color20] (1,-0) rectangle (2,-1);
\node at (1.5,-0.5) {20};
\draw[fill=7x6x2color17] (2,-0) rectangle (3,-1);
\node at (2.5,-0.5) {17};
\draw[fill=red] (3,-0) rectangle (4,-1);
\draw[fill=7x6x2color1] (4,-0) rectangle (5,-1);
\node at (4.5,-0.5) {1};
\draw[fill=7x6x2color6] (5,-0) rectangle (6,-1);
\node at (5.5,-0.5) {6};
\draw[fill=7x6x2color7] (6,-0) rectangle (7,-1);
\node at (6.5,-0.5) {7};
\draw[fill=red] (0,-1) rectangle (1,-2);
\draw[fill=7x6x2color17] (1,-1) rectangle (2,-2);
\node at (1.5,-1.5) {17};
\draw[fill=7x6x2color16] (2,-1) rectangle (3,-2);
\node at (2.5,-1.5) {16};
\draw[fill=7x6x2color1] (3,-1) rectangle (4,-2);
\node at (3.5,-1.5) {1};
\draw[fill=red] (4,-1) rectangle (5,-2);
\draw[fill=7x6x2color3] (5,-1) rectangle (6,-2);
\node at (5.5,-1.5) {3};
\draw[fill=red] (6,-1) rectangle (7,-2);
\draw[fill=7x6x2color1] (0,-2) rectangle (1,-3);
\node at (0.5,-2.5) {1};
\draw[fill=red] (1,-2) rectangle (2,-3);
\draw[fill=7x6x2color3] (2,-2) rectangle (3,-3);
\node at (2.5,-2.5) {3};
\draw[fill=red] (3,-2) rectangle (4,-3);
\draw[fill=7x6x2color1] (4,-2) rectangle (5,-3);
\node at (4.5,-2.5) {1};
\draw[fill=7x6x2color2] (5,-2) rectangle (6,-3);
\node at (5.5,-2.5) {2};
\draw[fill=7x6x2color1] (6,-2) rectangle (7,-3);
\node at (6.5,-2.5) {1};
\draw[fill=red] (0,-3) rectangle (1,-4);
\draw[fill=7x6x2color1] (1,-3) rectangle (2,-4);
\node at (1.5,-3.5) {1};
\draw[fill=7x6x2color2] (2,-3) rectangle (3,-4);
\node at (2.5,-3.5) {2};
\draw[fill=7x6x2color1] (3,-3) rectangle (4,-4);
\node at (3.5,-3.5) {1};
\draw[fill=red] (4,-3) rectangle (5,-4);
\draw[fill=7x6x2color1] (5,-3) rectangle (6,-4);
\node at (5.5,-3.5) {1};
\draw[fill=red] (6,-3) rectangle (7,-4);
\draw[fill=7x6x2color1] (0,-4) rectangle (1,-5);
\node at (0.5,-4.5) {1};
\draw[fill=red] (1,-4) rectangle (2,-5);
\draw[fill=7x6x2color1] (2,-4) rectangle (3,-5);
\node at (2.5,-4.5) {1};
\draw[fill=red] (3,-4) rectangle (4,-5);
\draw[fill=7x6x2color1] (4,-4) rectangle (5,-5);
\node at (4.5,-4.5) {1};
\draw[fill=7x6x2color2] (5,-4) rectangle (6,-5);
\node at (5.5,-4.5) {2};
\draw[fill=7x6x2color3] (6,-4) rectangle (7,-5);
\node at (6.5,-4.5) {3};
\draw[fill=red] (0,-5) rectangle (1,-6);
\draw[fill=7x6x2color1] (1,-5) rectangle (2,-6);
\node at (1.5,-5.5) {1};
\draw[fill=red] (2,-5) rectangle (3,-6);
\draw[fill=7x6x2color1] (3,-5) rectangle (4,-6);
\node at (3.5,-5.5) {1};
\draw[fill=red] (4,-5) rectangle (5,-6);
\draw[fill=7x6x2color1] (5,-5) rectangle (6,-6);
\node at (5.5,-5.5) {1};
\draw[fill=red] (6,-5) rectangle (7,-6);
\end{scope}

\end{tikzpicture}
\end{center} 

%% file: 3-4-4-tikz.txt
\definecolor{4x4x3color1}{RGB}{255, 128, 128}
\definecolor{4x4x3color2}{RGB}{255, 140, 140}
\definecolor{4x4x3color3}{RGB}{255, 153, 153}
\definecolor{4x4x3color4}{RGB}{255, 166, 166}
\definecolor{4x4x3color5}{RGB}{255, 178, 178}
\definecolor{4x4x3color6}{RGB}{255, 191, 191}
\definecolor{4x4x3color7}{RGB}{255, 204, 204}
\definecolor{4x4x3color8}{RGB}{255, 216, 216}
\definecolor{4x4x3color9}{RGB}{255, 229, 229}
\definecolor{4x4x3color10}{RGB}{255, 255, 255}

\begin{center}
\begin{tikzpicture}[scale=0.5]
\begin{scope}[xshift=0cm,yshift=0cm]
\draw[very thick] (0,0) rectangle (4,-4);
\foreach \b in {0,...,3}
 \foreach \c in {0,...,3}
  \draw (\c,\b*-1) rectangle (\c+1,(\b*-1-1);

\draw[fill=red] (0,-0) rectangle (1,-1);
\draw[fill=4x4x3color1] (1,-0) rectangle (2,-1);
\node at (1.5,-0.5) {1};
\draw[fill=4x4x3color2] (2,-0) rectangle (3,-1);
\node at (2.5,-0.5) {2};
\draw[fill=4x4x3color3] (3,-0) rectangle (4,-1);
\node at (3.5,-0.5) {3};
\draw[fill=4x4x3color1] (0,-1) rectangle (1,-2);
\node at (0.5,-1.5) {1};
\draw[fill=red] (1,-1) rectangle (2,-2);
\draw[fill=4x4x3color1] (2,-1) rectangle (3,-2);
\node at (2.5,-1.5) {1};
\draw[fill=red] (3,-1) rectangle (4,-2);
\draw[fill=4x4x3color8] (0,-2) rectangle (1,-3);
\node at (0.5,-2.5) {8};
\draw[fill=4x4x3color7] (1,-2) rectangle (2,-3);
\node at (1.5,-2.5) {7};
\draw[fill=4x4x3color6] (2,-2) rectangle (3,-3);
\node at (2.5,-2.5) {6};
\draw[fill=4x4x3color5] (3,-2) rectangle (4,-3);
\node at (3.5,-2.5) {5};
\draw[fill=4x4x3color9] (0,-3) rectangle (1,-4);
\node at (0.5,-3.5) {9};
\draw[fill=4x4x3color8] (1,-3) rectangle (2,-4);
\node at (1.5,-3.5) {8};
\draw[fill=4x4x3color7] (2,-3) rectangle (3,-4);
\node at (2.5,-3.5) {7};
\draw[fill=red] (3,-3) rectangle (4,-4);
\end{scope}

\begin{scope}[xshift=5cm,yshift=0cm]
\draw[very thick] (0,0) rectangle (4,-4);
\foreach \b in {0,...,3}
 \foreach \c in {0,...,3}
  \draw (\c,\b*-1) rectangle (\c+1,(\b*-1-1);

\draw[fill=4x4x3color1] (0,-0) rectangle (1,-1);
\node at (0.5,-0.5) {1};
\draw[fill=red] (1,-0) rectangle (2,-1);
\draw[fill=4x4x3color1] (2,-0) rectangle (3,-1);
\node at (2.5,-0.5) {1};
\draw[fill=red] (3,-0) rectangle (4,-1);
\draw[fill=red] (0,-1) rectangle (1,-2);
\draw[fill=4x4x3color1] (1,-1) rectangle (2,-2);
\node at (1.5,-1.5) {1};
\draw[fill=red] (2,-1) rectangle (3,-2);
\draw[fill=4x4x3color1] (3,-1) rectangle (4,-2);
\node at (3.5,-1.5) {1};
\draw[fill=4x4x3color1] (0,-2) rectangle (1,-3);
\node at (0.5,-2.5) {1};
\draw[fill=4x4x3color2] (1,-2) rectangle (2,-3);
\node at (1.5,-2.5) {2};
\draw[fill=4x4x3color3] (2,-2) rectangle (3,-3);
\node at (2.5,-2.5) {3};
\draw[fill=4x4x3color4] (3,-2) rectangle (4,-3);
\node at (3.5,-2.5) {4};
\draw[fill=red] (0,-3) rectangle (1,-4);
\draw[fill=4x4x3color1] (1,-3) rectangle (2,-4);
\node at (1.5,-3.5) {1};
\draw[fill=red] (2,-3) rectangle (3,-4);
\draw[fill=4x4x3color1] (3,-3) rectangle (4,-4);
\node at (3.5,-3.5) {1};
\end{scope}

\begin{scope}[xshift=10cm,yshift=0cm]
\draw[very thick] (0,0) rectangle (4,-4);
\foreach \b in {0,...,3}
 \foreach \c in {0,...,3}
  \draw (\c,\b*-1) rectangle (\c+1,(\b*-1-1);

\draw[fill=red] (0,-0) rectangle (1,-1);
\draw[fill=4x4x3color5] (1,-0) rectangle (2,-1);
\node at (1.5,-0.5) {5};
\draw[fill=4x4x3color6] (2,-0) rectangle (3,-1);
\node at (2.5,-0.5) {6};
\draw[fill=4x4x3color7] (3,-0) rectangle (4,-1);
\node at (3.5,-0.5) {7};
\draw[fill=4x4x3color1] (0,-1) rectangle (1,-2);
\node at (0.5,-1.5) {1};
\draw[fill=4x4x3color4] (1,-1) rectangle (2,-2);
\node at (1.5,-1.5) {4};
\draw[fill=4x4x3color5] (2,-1) rectangle (3,-2);
\node at (2.5,-1.5) {5};
\draw[fill=4x4x3color6] (3,-1) rectangle (4,-2);
\node at (3.5,-1.5) {6};
\draw[fill=red] (0,-2) rectangle (1,-3);
\draw[fill=4x4x3color3] (1,-2) rectangle (2,-3);
\node at (1.5,-2.5) {3};
\draw[fill=4x4x3color4] (2,-2) rectangle (3,-3);
\node at (2.5,-2.5) {4};
\draw[fill=4x4x3color5] (3,-2) rectangle (4,-3);
\node at (3.5,-2.5) {5};
\draw[fill=4x4x3color1] (0,-3) rectangle (1,-4);
\node at (0.5,-3.5) {1};
\draw[fill=red] (1,-3) rectangle (2,-4);
\draw[fill=4x4x3color1] (2,-3) rectangle (3,-4);
\node at (2.5,-3.5) {1};
\draw[fill=red] (3,-3) rectangle (4,-4);
\end{scope}

\end{tikzpicture}
\end{center}

%% file: 3-4-5-tikz.txt
\definecolor{5x4x3color1}{RGB}{255, 128, 128}
\definecolor{5x4x3color2}{RGB}{255, 136, 136}
\definecolor{5x4x3color3}{RGB}{255, 144, 144}
\definecolor{5x4x3color4}{RGB}{255, 153, 153}
\definecolor{5x4x3color5}{RGB}{255, 161, 161}
\definecolor{5x4x3color6}{RGB}{255, 170, 170}
\definecolor{5x4x3color7}{RGB}{255, 178, 178}
\definecolor{5x4x3color8}{RGB}{255, 187, 187}
\definecolor{5x4x3color9}{RGB}{255, 195, 195}
\definecolor{5x4x3color10}{RGB}{255, 204, 204}
\definecolor{5x4x3color11}{RGB}{255, 212, 212}
\definecolor{5x4x3color12}{RGB}{255, 221, 221}
\definecolor{5x4x3color13}{RGB}{255, 229, 229}
\definecolor{5x4x3color14}{RGB}{255, 238, 238}
\definecolor{5x4x3color15}{RGB}{255, 255, 255}

\begin{center}
\begin{tikzpicture}[scale=0.5]
\begin{scope}[xshift=0cm,yshift=0cm]
\draw[very thick] (0,0) rectangle (5,-4);
\foreach \b in {0,...,3}
 \foreach \c in {0,...,4}
  \draw (\c,\b*-1) rectangle (\c+1,(\b*-1-1);

\draw[fill=red] (0,-0) rectangle (1,-1);
\draw[fill=5x4x3color11] (1,-0) rectangle (2,-1);
\node at (1.5,-0.5) {11};
\draw[fill=5x4x3color10] (2,-0) rectangle (3,-1);
\node at (2.5,-0.5) {10};
\draw[fill=5x4x3color9] (3,-0) rectangle (4,-1);
\node at (3.5,-0.5) {9};
\draw[fill=red] (4,-0) rectangle (5,-1);
\draw[fill=5x4x3color9] (0,-1) rectangle (1,-2);
\node at (0.5,-1.5) {9};
\draw[fill=5x4x3color8] (1,-1) rectangle (2,-2);
\node at (1.5,-1.5) {8};
\draw[fill=5x4x3color7] (2,-1) rectangle (3,-2);
\node at (2.5,-1.5) {7};
\draw[fill=5x4x3color6] (3,-1) rectangle (4,-2);
\node at (3.5,-1.5) {6};
\draw[fill=5x4x3color1] (4,-1) rectangle (5,-2);
\node at (4.5,-1.5) {1};
\draw[fill=5x4x3color10] (0,-2) rectangle (1,-3);
\node at (0.5,-2.5) {10};
\draw[fill=5x4x3color3] (1,-2) rectangle (2,-3);
\node at (1.5,-2.5) {3};
\draw[fill=red] (2,-2) rectangle (3,-3);
\draw[fill=5x4x3color5] (3,-2) rectangle (4,-3);
\node at (3.5,-2.5) {5};
\draw[fill=red] (4,-2) rectangle (5,-3);
\draw[fill=5x4x3color11] (0,-3) rectangle (1,-4);
\node at (0.5,-3.5) {11};
\draw[fill=red] (1,-3) rectangle (2,-4);
\draw[fill=5x4x3color3] (2,-3) rectangle (3,-4);
\node at (2.5,-3.5) {3};
\draw[fill=5x4x3color6] (3,-3) rectangle (4,-4);
\node at (3.5,-3.5) {6};
\draw[fill=5x4x3color7] (4,-3) rectangle (5,-4);
\node at (4.5,-3.5) {7};
\end{scope}

\begin{scope}[xshift=6cm,yshift=0cm]
\draw[very thick] (0,0) rectangle (5,-4);
\foreach \b in {0,...,3}
 \foreach \c in {0,...,4}
  \draw (\c,\b*-1) rectangle (\c+1,(\b*-1-1);

\draw[fill=5x4x3color13] (0,-0) rectangle (1,-1);
\node at (0.5,-0.5) {13};
\draw[fill=5x4x3color12] (1,-0) rectangle (2,-1);
\node at (1.5,-0.5) {12};
\draw[fill=5x4x3color9] (2,-0) rectangle (3,-1);
\node at (2.5,-0.5) {9};
\draw[fill=5x4x3color8] (3,-0) rectangle (4,-1);
\node at (3.5,-0.5) {8};
\draw[fill=5x4x3color1] (4,-0) rectangle (5,-1);
\node at (4.5,-0.5) {1};
\draw[fill=red] (0,-1) rectangle (1,-2);
\draw[fill=5x4x3color3] (1,-1) rectangle (2,-2);
\node at (1.5,-1.5) {3};
\draw[fill=red] (2,-1) rectangle (3,-2);
\draw[fill=5x4x3color5] (3,-1) rectangle (4,-2);
\node at (3.5,-1.5) {5};
\draw[fill=red] (4,-1) rectangle (5,-2);
\draw[fill=5x4x3color1] (0,-2) rectangle (1,-3);
\node at (0.5,-2.5) {1};
\draw[fill=5x4x3color2] (1,-2) rectangle (2,-3);
\node at (1.5,-2.5) {2};
\draw[fill=5x4x3color1] (2,-2) rectangle (3,-3);
\node at (2.5,-2.5) {1};
\draw[fill=5x4x3color4] (3,-2) rectangle (4,-3);
\node at (3.5,-2.5) {4};
\draw[fill=5x4x3color1] (4,-2) rectangle (5,-3);
\node at (4.5,-2.5) {1};
\draw[fill=red] (0,-3) rectangle (1,-4);
\draw[fill=5x4x3color1] (1,-3) rectangle (2,-4);
\node at (1.5,-3.5) {1};
\draw[fill=5x4x3color2] (2,-3) rectangle (3,-4);
\node at (2.5,-3.5) {2};
\draw[fill=5x4x3color3] (3,-3) rectangle (4,-4);
\node at (3.5,-3.5) {3};
\draw[fill=red] (4,-3) rectangle (5,-4);
\end{scope}

\begin{scope}[xshift=12cm,yshift=0cm]
\draw[very thick] (0,0) rectangle (5,-4);
\foreach \b in {0,...,3}
 \foreach \c in {0,...,4}
  \draw (\c,\b*-1) rectangle (\c+1,(\b*-1-1);

\draw[fill=5x4x3color14] (0,-0) rectangle (1,-1);
\node at (0.5,-0.5) {14};
\draw[fill=5x4x3color13] (1,-0) rectangle (2,-1);
\node at (1.5,-0.5) {13};
\draw[fill=red] (2,-0) rectangle (3,-1);
\draw[fill=5x4x3color7] (3,-0) rectangle (4,-1);
\node at (3.5,-0.5) {7};
\draw[fill=red] (4,-0) rectangle (5,-1);
\draw[fill=5x4x3color5] (0,-1) rectangle (1,-2);
\node at (0.5,-1.5) {5};
\draw[fill=5x4x3color4] (1,-1) rectangle (2,-2);
\node at (1.5,-1.5) {4};
\draw[fill=5x4x3color1] (2,-1) rectangle (3,-2);
\node at (2.5,-1.5) {1};
\draw[fill=5x4x3color6] (3,-1) rectangle (4,-2);
\node at (3.5,-1.5) {6};
\draw[fill=5x4x3color7] (4,-1) rectangle (5,-2);
\node at (4.5,-1.5) {7};
\draw[fill=red] (0,-2) rectangle (1,-3);
\draw[fill=5x4x3color1] (1,-2) rectangle (2,-3);
\node at (1.5,-2.5) {1};
\draw[fill=red] (2,-2) rectangle (3,-3);
\draw[fill=5x4x3color5] (3,-2) rectangle (4,-3);
\node at (3.5,-2.5) {5};
\draw[fill=5x4x3color8] (4,-2) rectangle (5,-3);
\node at (4.5,-2.5) {8};
\draw[fill=5x4x3color1] (0,-3) rectangle (1,-4);
\node at (0.5,-3.5) {1};
\draw[fill=red] (1,-3) rectangle (2,-4);
\draw[fill=5x4x3color1] (2,-3) rectangle (3,-4);
\node at (2.5,-3.5) {1};
\draw[fill=red] (3,-3) rectangle (4,-4);
\draw[fill=5x4x3color9] (4,-3) rectangle (5,-4);
\node at (4.5,-3.5) {9};
\end{scope}

\end{tikzpicture}
\end{center}

%% file: 3-4-7-tikz.txt
\definecolor{7x4x3color1}{RGB}{255, 128, 128}
\definecolor{7x4x3color2}{RGB}{255, 134, 134}
\definecolor{7x4x3color3}{RGB}{255, 141, 141}
\definecolor{7x4x3color4}{RGB}{255, 148, 148}
\definecolor{7x4x3color5}{RGB}{255, 154, 154}
\definecolor{7x4x3color6}{RGB}{255, 161, 161}
\definecolor{7x4x3color7}{RGB}{255, 168, 168}
\definecolor{7x4x3color8}{RGB}{255, 174, 174}
\definecolor{7x4x3color9}{RGB}{255, 181, 181}
\definecolor{7x4x3color10}{RGB}{255, 188, 188}
\definecolor{7x4x3color11}{RGB}{255, 194, 194}
\definecolor{7x4x3color12}{RGB}{255, 201, 201}
\definecolor{7x4x3color13}{RGB}{255, 208, 208}
\definecolor{7x4x3color14}{RGB}{255, 214, 214}
\definecolor{7x4x3color15}{RGB}{255, 221, 221}
\definecolor{7x4x3color16}{RGB}{255, 228, 228}
\definecolor{7x4x3color17}{RGB}{255, 234, 234}
\definecolor{7x4x3color18}{RGB}{255, 241, 241}
\definecolor{7x4x3color19}{RGB}{255, 255, 255}

\begin{center}
\begin{tikzpicture}[scale=0.5]
\begin{scope}[xshift=0cm,yshift=0cm]
\draw[very thick] (0,0) rectangle (7,-4);
\foreach \b in {0,...,3}
 \foreach \c in {0,...,6}
  \draw (\c,\b*-1) rectangle (\c+1,(\b*-1-1);

\draw[fill=red] (0,-0) rectangle (1,-1);
\draw[fill=7x4x3color15] (1,-0) rectangle (2,-1);
\node at (1.5,-0.5) {15};
\draw[fill=7x4x3color14] (2,-0) rectangle (3,-1);
\node at (2.5,-0.5) {14};
\draw[fill=7x4x3color13] (3,-0) rectangle (4,-1);
\node at (3.5,-0.5) {13};
\draw[fill=red] (4,-0) rectangle (5,-1);
\draw[fill=7x4x3color1] (5,-0) rectangle (6,-1);
\node at (5.5,-0.5) {1};
\draw[fill=red] (6,-0) rectangle (7,-1);
\draw[fill=7x4x3color11] (0,-1) rectangle (1,-2);
\node at (0.5,-1.5) {11};
\draw[fill=7x4x3color10] (1,-1) rectangle (2,-2);
\node at (1.5,-1.5) {10};
\draw[fill=7x4x3color7] (2,-1) rectangle (3,-2);
\node at (2.5,-1.5) {7};
\draw[fill=7x4x3color6] (3,-1) rectangle (4,-2);
\node at (3.5,-1.5) {6};
\draw[fill=7x4x3color3] (4,-1) rectangle (5,-2);
\node at (4.5,-1.5) {3};
\draw[fill=7x4x3color2] (5,-1) rectangle (6,-2);
\node at (5.5,-1.5) {2};
\draw[fill=7x4x3color1] (6,-1) rectangle (7,-2);
\node at (6.5,-1.5) {1};
\draw[fill=7x4x3color12] (0,-2) rectangle (1,-3);
\node at (0.5,-2.5) {12};
\draw[fill=7x4x3color9] (1,-2) rectangle (2,-3);
\node at (1.5,-2.5) {9};
\draw[fill=red] (2,-2) rectangle (3,-3);
\draw[fill=7x4x3color1] (3,-2) rectangle (4,-3);
\node at (3.5,-2.5) {1};
\draw[fill=red] (4,-2) rectangle (5,-3);
\draw[fill=7x4x3color1] (5,-2) rectangle (6,-3);
\node at (5.5,-2.5) {1};
\draw[fill=red] (6,-2) rectangle (7,-3);
\draw[fill=7x4x3color13] (0,-3) rectangle (1,-4);
\node at (0.5,-3.5) {13};
\draw[fill=red] (1,-3) rectangle (2,-4);
\draw[fill=7x4x3color1] (2,-3) rectangle (3,-4);
\node at (2.5,-3.5) {1};
\draw[fill=red] (3,-3) rectangle (4,-4);
\draw[fill=7x4x3color1] (4,-3) rectangle (5,-4);
\node at (4.5,-3.5) {1};
\draw[fill=red] (5,-3) rectangle (6,-4);
\draw[fill=7x4x3color1] (6,-3) rectangle (7,-4);
\node at (6.5,-3.5) {1};
\end{scope}

\begin{scope}[xshift=8cm,yshift=0cm]
\draw[very thick] (0,0) rectangle (7,-4);
\foreach \b in {0,...,3}
 \foreach \c in {0,...,6}
  \draw (\c,\b*-1) rectangle (\c+1,(\b*-1-1);

\draw[fill=7x4x3color17] (0,-0) rectangle (1,-1);
\node at (0.5,-0.5) {17};
\draw[fill=7x4x3color16] (1,-0) rectangle (2,-1);
\node at (1.5,-0.5) {16};
\draw[fill=7x4x3color13] (2,-0) rectangle (3,-1);
\node at (2.5,-0.5) {13};
\draw[fill=7x4x3color12] (3,-0) rectangle (4,-1);
\node at (3.5,-0.5) {12};
\draw[fill=7x4x3color5] (4,-0) rectangle (5,-1);
\node at (4.5,-0.5) {5};
\draw[fill=red] (5,-0) rectangle (6,-1);
\draw[fill=7x4x3color1] (6,-0) rectangle (7,-1);
\node at (6.5,-0.5) {1};
\draw[fill=red] (0,-1) rectangle (1,-2);
\draw[fill=7x4x3color9] (1,-1) rectangle (2,-2);
\node at (1.5,-1.5) {9};
\draw[fill=red] (2,-1) rectangle (3,-2);
\draw[fill=7x4x3color5] (3,-1) rectangle (4,-2);
\node at (3.5,-1.5) {5};
\draw[fill=7x4x3color4] (4,-1) rectangle (5,-2);
\node at (4.5,-1.5) {4};
\draw[fill=7x4x3color3] (5,-1) rectangle (6,-2);
\node at (5.5,-1.5) {3};
\draw[fill=red] (6,-1) rectangle (7,-2);
\draw[fill=7x4x3color1] (0,-2) rectangle (1,-3);
\node at (0.5,-2.5) {1};
\draw[fill=7x4x3color8] (1,-2) rectangle (2,-3);
\node at (1.5,-2.5) {8};
\draw[fill=7x4x3color1] (2,-2) rectangle (3,-3);
\node at (2.5,-2.5) {1};
\draw[fill=7x4x3color4] (3,-2) rectangle (4,-3);
\node at (3.5,-2.5) {4};
\draw[fill=7x4x3color3] (4,-2) rectangle (5,-3);
\node at (4.5,-2.5) {3};
\draw[fill=7x4x3color2] (5,-2) rectangle (6,-3);
\node at (5.5,-2.5) {2};
\draw[fill=7x4x3color1] (6,-2) rectangle (7,-3);
\node at (6.5,-2.5) {1};
\draw[fill=red] (0,-3) rectangle (1,-4);
\draw[fill=7x4x3color7] (1,-3) rectangle (2,-4);
\node at (1.5,-3.5) {7};
\draw[fill=7x4x3color6] (2,-3) rectangle (3,-4);
\node at (2.5,-3.5) {6};
\draw[fill=7x4x3color5] (3,-3) rectangle (4,-4);
\node at (3.5,-3.5) {5};
\draw[fill=7x4x3color4] (4,-3) rectangle (5,-4);
\node at (4.5,-3.5) {4};
\draw[fill=7x4x3color1] (5,-3) rectangle (6,-4);
\node at (5.5,-3.5) {1};
\draw[fill=red] (6,-3) rectangle (7,-4);
\end{scope}

\begin{scope}[xshift=16cm,yshift=0cm]
\draw[very thick] (0,0) rectangle (7,-4);
\foreach \b in {0,...,3}
 \foreach \c in {0,...,6}
  \draw (\c,\b*-1) rectangle (\c+1,(\b*-1-1);

\draw[fill=7x4x3color18] (0,-0) rectangle (1,-1);
\node at (0.5,-0.5) {18};
\draw[fill=7x4x3color17] (1,-0) rectangle (2,-1);
\node at (1.5,-0.5) {17};
\draw[fill=red] (2,-0) rectangle (3,-1);
\draw[fill=7x4x3color11] (3,-0) rectangle (4,-1);
\node at (3.5,-0.5) {11};
\draw[fill=7x4x3color10] (4,-0) rectangle (5,-1);
\node at (4.5,-0.5) {10};
\draw[fill=7x4x3color9] (5,-0) rectangle (6,-1);
\node at (5.5,-0.5) {9};
\draw[fill=red] (6,-0) rectangle (7,-1);
\draw[fill=7x4x3color11] (0,-1) rectangle (1,-2);
\node at (0.5,-1.5) {11};
\draw[fill=7x4x3color10] (1,-1) rectangle (2,-2);
\node at (1.5,-1.5) {10};
\draw[fill=7x4x3color1] (2,-1) rectangle (3,-2);
\node at (2.5,-1.5) {1};
\draw[fill=7x4x3color6] (3,-1) rectangle (4,-2);
\node at (3.5,-1.5) {6};
\draw[fill=7x4x3color7] (4,-1) rectangle (5,-2);
\node at (4.5,-1.5) {7};
\draw[fill=7x4x3color8] (5,-1) rectangle (6,-2);
\node at (5.5,-1.5) {8};
\draw[fill=7x4x3color9] (6,-1) rectangle (7,-2);
\node at (6.5,-1.5) {9};
\draw[fill=red] (0,-2) rectangle (1,-3);
\draw[fill=7x4x3color9] (1,-2) rectangle (2,-3);
\node at (1.5,-2.5) {9};
\draw[fill=red] (2,-2) rectangle (3,-3);
\draw[fill=7x4x3color5] (3,-2) rectangle (4,-3);
\node at (3.5,-2.5) {5};
\draw[fill=red] (4,-2) rectangle (5,-3);
\draw[fill=7x4x3color3] (5,-2) rectangle (6,-3);
\node at (5.5,-2.5) {3};
\draw[fill=7x4x3color10] (6,-2) rectangle (7,-3);
\node at (6.5,-2.5) {10};
\draw[fill=7x4x3color11] (0,-3) rectangle (1,-4);
\node at (0.5,-3.5) {11};
\draw[fill=7x4x3color10] (1,-3) rectangle (2,-4);
\node at (1.5,-3.5) {10};
\draw[fill=7x4x3color7] (2,-3) rectangle (3,-4);
\node at (2.5,-3.5) {7};
\draw[fill=7x4x3color6] (3,-3) rectangle (4,-4);
\node at (3.5,-3.5) {6};
\draw[fill=7x4x3color5] (4,-3) rectangle (5,-4);
\node at (4.5,-3.5) {5};
\draw[fill=red] (5,-3) rectangle (6,-4);
\draw[fill=7x4x3color11] (6,-3) rectangle (7,-4);
\node at (6.5,-3.5) {11};
\end{scope}

\end{tikzpicture}
\end{center}

%% file: 3-5-5-tikz.txt
\definecolor{5x5x3color1}{RGB}{255, 128, 128}
\definecolor{5x5x3color2}{RGB}{255, 135, 135}
\definecolor{5x5x3color3}{RGB}{255, 142, 142}
\definecolor{5x5x3color4}{RGB}{255, 150, 150}
\definecolor{5x5x3color5}{RGB}{255, 157, 157}
\definecolor{5x5x3color6}{RGB}{255, 165, 165}
\definecolor{5x5x3color7}{RGB}{255, 172, 172}
\definecolor{5x5x3color8}{RGB}{255, 180, 180}
\definecolor{5x5x3color9}{RGB}{255, 187, 187}
\definecolor{5x5x3color10}{RGB}{255, 195, 195}
\definecolor{5x5x3color11}{RGB}{255, 202, 202}
\definecolor{5x5x3color12}{RGB}{255, 210, 210}
\definecolor{5x5x3color13}{RGB}{255, 217, 217}
\definecolor{5x5x3color14}{RGB}{255, 225, 225}
\definecolor{5x5x3color15}{RGB}{255, 232, 232}
\definecolor{5x5x3color16}{RGB}{255, 240, 240}
\definecolor{5x5x3color17}{RGB}{255, 255, 255}

\begin{center}
\begin{tikzpicture}[scale=0.5]
\begin{scope}[xshift=0cm,yshift=0cm]
\draw[very thick] (0,0) rectangle (5,-5);
\foreach \b in {0,...,4}
 \foreach \c in {0,...,4}
  \draw (\c,\b*-1) rectangle (\c+1,(\b*-1-1);

\draw[fill=5x5x3color16] (0,-0) rectangle (1,-1);
\node at (0.5,-0.5) {16};
\draw[fill=5x5x3color15] (1,-0) rectangle (2,-1);
\node at (1.5,-0.5) {15};
\draw[fill=5x5x3color14] (2,-0) rectangle (3,-1);
\node at (2.5,-0.5) {14};
\draw[fill=5x5x3color13] (3,-0) rectangle (4,-1);
\node at (3.5,-0.5) {13};
\draw[fill=red] (4,-0) rectangle (5,-1);
\draw[fill=5x5x3color15] (0,-1) rectangle (1,-2);
\node at (0.5,-1.5) {15};
\draw[fill=5x5x3color14] (1,-1) rectangle (2,-2);
\node at (1.5,-1.5) {14};
\draw[fill=5x5x3color13] (2,-1) rectangle (3,-2);
\node at (2.5,-1.5) {13};
\draw[fill=5x5x3color12] (3,-1) rectangle (4,-2);
\node at (3.5,-1.5) {12};
\draw[fill=5x5x3color11] (4,-1) rectangle (5,-2);
\node at (4.5,-1.5) {11};
\draw[fill=red] (0,-2) rectangle (1,-3);
\draw[fill=5x5x3color7] (1,-2) rectangle (2,-3);
\node at (1.5,-2.5) {7};
\draw[fill=5x5x3color8] (2,-2) rectangle (3,-3);
\node at (2.5,-2.5) {8};
\draw[fill=5x5x3color9] (3,-2) rectangle (4,-3);
\node at (3.5,-2.5) {9};
\draw[fill=5x5x3color10] (4,-2) rectangle (5,-3);
\node at (4.5,-2.5) {10};
\draw[fill=5x5x3color1] (0,-3) rectangle (1,-4);
\node at (0.5,-3.5) {1};
\draw[fill=red] (1,-3) rectangle (2,-4);
\draw[fill=5x5x3color1] (2,-3) rectangle (3,-4);
\node at (2.5,-3.5) {1};
\draw[fill=5x5x3color2] (3,-3) rectangle (4,-4);
\node at (3.5,-3.5) {2};
\draw[fill=5x5x3color3] (4,-3) rectangle (5,-4);
\node at (4.5,-3.5) {3};
\draw[fill=red] (0,-4) rectangle (1,-5);
\draw[fill=5x5x3color1] (1,-4) rectangle (2,-5);
\node at (1.5,-4.5) {1};
\draw[fill=red] (2,-4) rectangle (3,-5);
\draw[fill=5x5x3color1] (3,-4) rectangle (4,-5);
\node at (3.5,-4.5) {1};
\draw[fill=red] (4,-4) rectangle (5,-5);
\end{scope}

\begin{scope}[xshift=6cm,yshift=0cm]
\draw[very thick] (0,0) rectangle (5,-5);
\foreach \b in {0,...,4}
 \foreach \c in {0,...,4}
  \draw (\c,\b*-1) rectangle (\c+1,(\b*-1-1);

\draw[fill=5x5x3color7] (0,-0) rectangle (1,-1);
\node at (0.5,-0.5) {7};
\draw[fill=5x5x3color6] (1,-0) rectangle (2,-1);
\node at (1.5,-0.5) {6};
\draw[fill=5x5x3color5] (2,-0) rectangle (3,-1);
\node at (2.5,-0.5) {5};
\draw[fill=5x5x3color4] (3,-0) rectangle (4,-1);
\node at (3.5,-0.5) {4};
\draw[fill=5x5x3color1] (4,-0) rectangle (5,-1);
\node at (4.5,-0.5) {1};
\draw[fill=red] (0,-1) rectangle (1,-2);
\draw[fill=5x5x3color5] (1,-1) rectangle (2,-2);
\node at (1.5,-1.5) {5};
\draw[fill=5x5x3color4] (2,-1) rectangle (3,-2);
\node at (2.5,-1.5) {4};
\draw[fill=5x5x3color3] (3,-1) rectangle (4,-2);
\node at (3.5,-1.5) {3};
\draw[fill=red] (4,-1) rectangle (5,-2);
\draw[fill=5x5x3color7] (0,-2) rectangle (1,-3);
\node at (0.5,-2.5) {7};
\draw[fill=5x5x3color6] (1,-2) rectangle (2,-3);
\node at (1.5,-2.5) {6};
\draw[fill=5x5x3color1] (2,-2) rectangle (3,-3);
\node at (2.5,-2.5) {1};
\draw[fill=red] (3,-2) rectangle (4,-3);
\draw[fill=5x5x3color1] (4,-2) rectangle (5,-3);
\node at (4.5,-2.5) {1};
\draw[fill=5x5x3color8] (0,-3) rectangle (1,-4);
\node at (0.5,-3.5) {8};
\draw[fill=5x5x3color1] (1,-3) rectangle (2,-4);
\node at (1.5,-3.5) {1};
\draw[fill=red] (2,-3) rectangle (3,-4);
\draw[fill=5x5x3color1] (3,-3) rectangle (4,-4);
\node at (3.5,-3.5) {1};
\draw[fill=5x5x3color2] (4,-3) rectangle (5,-4);
\node at (4.5,-3.5) {2};
\draw[fill=5x5x3color9] (0,-4) rectangle (1,-5);
\node at (0.5,-4.5) {9};
\draw[fill=red] (1,-4) rectangle (2,-5);
\draw[fill=5x5x3color1] (2,-4) rectangle (3,-5);
\node at (2.5,-4.5) {1};
\draw[fill=red] (3,-4) rectangle (4,-5);
\draw[fill=5x5x3color1] (4,-4) rectangle (5,-5);
\node at (4.5,-4.5) {1};
\end{scope}

\begin{scope}[xshift=12cm,yshift=0cm]
\draw[very thick] (0,0) rectangle (5,-5);
\foreach \b in {0,...,4}
 \foreach \c in {0,...,4}
  \draw (\c,\b*-1) rectangle (\c+1,(\b*-1-1);

\draw[fill=red] (0,-0) rectangle (1,-1);
\draw[fill=5x5x3color1] (1,-0) rectangle (2,-1);
\node at (1.5,-0.5) {1};
\draw[fill=red] (2,-0) rectangle (3,-1);
\draw[fill=5x5x3color3] (3,-0) rectangle (4,-1);
\node at (3.5,-0.5) {3};
\draw[fill=red] (4,-0) rectangle (5,-1);
\draw[fill=5x5x3color1] (0,-1) rectangle (1,-2);
\node at (0.5,-1.5) {1};
\draw[fill=red] (1,-1) rectangle (2,-2);
\draw[fill=5x5x3color1] (2,-1) rectangle (3,-2);
\node at (2.5,-1.5) {1};
\draw[fill=5x5x3color2] (3,-1) rectangle (4,-2);
\node at (3.5,-1.5) {2};
\draw[fill=5x5x3color1] (4,-1) rectangle (5,-2);
\node at (4.5,-1.5) {1};
\draw[fill=5x5x3color8] (0,-2) rectangle (1,-3);
\node at (0.5,-2.5) {8};
\draw[fill=5x5x3color7] (1,-2) rectangle (2,-3);
\node at (1.5,-2.5) {7};
\draw[fill=red] (2,-2) rectangle (3,-3);
\draw[fill=5x5x3color1] (3,-2) rectangle (4,-3);
\node at (3.5,-2.5) {1};
\draw[fill=red] (4,-2) rectangle (5,-3);
\draw[fill=5x5x3color9] (0,-3) rectangle (1,-4);
\node at (0.5,-3.5) {9};
\draw[fill=5x5x3color8] (1,-3) rectangle (2,-4);
\node at (1.5,-3.5) {8};
\draw[fill=5x5x3color5] (2,-3) rectangle (3,-4);
\node at (2.5,-3.5) {5};
\draw[fill=5x5x3color4] (3,-3) rectangle (4,-4);
\node at (3.5,-3.5) {4};
\draw[fill=5x5x3color3] (4,-3) rectangle (5,-4);
\node at (4.5,-3.5) {3};
\draw[fill=5x5x3color10] (0,-4) rectangle (1,-5);
\node at (0.5,-4.5) {10};
\draw[fill=5x5x3color9] (1,-4) rectangle (2,-5);
\node at (1.5,-4.5) {9};
\draw[fill=5x5x3color6] (2,-4) rectangle (3,-5);
\node at (2.5,-4.5) {6};
\draw[fill=5x5x3color5] (3,-4) rectangle (4,-5);
\node at (3.5,-4.5) {5};
\draw[fill=red] (4,-4) rectangle (5,-5);
\end{scope}

\end{tikzpicture}
\end{center}

%% file: 3-5-7-tikz.txt
\definecolor{7x5x3color1}{RGB}{255, 128, 128}
\definecolor{7x5x3color2}{RGB}{255, 133, 133}
\definecolor{7x5x3color3}{RGB}{255, 139, 139}
\definecolor{7x5x3color4}{RGB}{255, 144, 144}
\definecolor{7x5x3color5}{RGB}{255, 150, 150}
\definecolor{7x5x3color6}{RGB}{255, 155, 155}
\definecolor{7x5x3color7}{RGB}{255, 161, 161}
\definecolor{7x5x3color8}{RGB}{255, 166, 166}
\definecolor{7x5x3color9}{RGB}{255, 172, 172}
\definecolor{7x5x3color10}{RGB}{255, 177, 177}
\definecolor{7x5x3color11}{RGB}{255, 183, 183}
\definecolor{7x5x3color12}{RGB}{255, 188, 188}
\definecolor{7x5x3color13}{RGB}{255, 194, 194}
\definecolor{7x5x3color14}{RGB}{255, 199, 199}
\definecolor{7x5x3color15}{RGB}{255, 205, 205}
\definecolor{7x5x3color16}{RGB}{255, 210, 210}
\definecolor{7x5x3color17}{RGB}{255, 216, 216}
\definecolor{7x5x3color18}{RGB}{255, 221, 221}
\definecolor{7x5x3color19}{RGB}{255, 227, 227}
\definecolor{7x5x3color20}{RGB}{255, 232, 232}
\definecolor{7x5x3color21}{RGB}{255, 238, 238}
\definecolor{7x5x3color22}{RGB}{255, 243, 243}
\definecolor{7x5x3color23}{RGB}{255, 255, 255}

\begin{center}
\begin{tikzpicture}[scale=0.5]
\begin{scope}[xshift=0cm,yshift=0cm]
\draw[very thick] (0,0) rectangle (7,-5);
\foreach \b in {0,...,4}
 \foreach \c in {0,...,6}
  \draw (\c,\b*-1) rectangle (\c+1,(\b*-1-1);

\draw[fill=red] (0,-0) rectangle (1,-1);
\draw[fill=7x5x3color11] (1,-0) rectangle (2,-1);
\node at (1.5,-0.5) {11};
\draw[fill=7x5x3color12] (2,-0) rectangle (3,-1);
\node at (2.5,-0.5) {12};
\draw[fill=7x5x3color13] (3,-0) rectangle (4,-1);
\node at (3.5,-0.5) {13};
\draw[fill=red] (4,-0) rectangle (5,-1);
\draw[fill=7x5x3color19] (5,-0) rectangle (6,-1);
\node at (5.5,-0.5) {19};
\draw[fill=red] (6,-0) rectangle (7,-1);
\draw[fill=7x5x3color5] (0,-1) rectangle (1,-2);
\node at (0.5,-1.5) {5};
\draw[fill=7x5x3color10] (1,-1) rectangle (2,-2);
\node at (1.5,-1.5) {10};
\draw[fill=7x5x3color9] (2,-1) rectangle (3,-2);
\node at (2.5,-1.5) {9};
\draw[fill=red] (3,-1) rectangle (4,-2);
\draw[fill=7x5x3color11] (4,-1) rectangle (5,-2);
\node at (4.5,-1.5) {11};
\draw[fill=7x5x3color18] (5,-1) rectangle (6,-2);
\node at (5.5,-1.5) {18};
\draw[fill=7x5x3color19] (6,-1) rectangle (7,-2);
\node at (6.5,-1.5) {19};
\draw[fill=red] (0,-2) rectangle (1,-3);
\draw[fill=7x5x3color9] (1,-2) rectangle (2,-3);
\node at (1.5,-2.5) {9};
\draw[fill=7x5x3color8] (2,-2) rectangle (3,-3);
\node at (2.5,-2.5) {8};
\draw[fill=7x5x3color7] (3,-2) rectangle (4,-3);
\node at (3.5,-2.5) {7};
\draw[fill=7x5x3color12] (4,-2) rectangle (5,-3);
\node at (4.5,-2.5) {12};
\draw[fill=7x5x3color17] (5,-2) rectangle (6,-3);
\node at (5.5,-2.5) {17};
\draw[fill=7x5x3color18] (6,-2) rectangle (7,-3);
\node at (6.5,-2.5) {18};
\draw[fill=7x5x3color11] (0,-3) rectangle (1,-4);
\node at (0.5,-3.5) {11};
\draw[fill=7x5x3color10] (1,-3) rectangle (2,-4);
\node at (1.5,-3.5) {10};
\draw[fill=7x5x3color5] (2,-3) rectangle (3,-4);
\node at (2.5,-3.5) {5};
\draw[fill=red] (3,-3) rectangle (4,-4);
\draw[fill=7x5x3color13] (4,-3) rectangle (5,-4);
\node at (4.5,-3.5) {13};
\draw[fill=7x5x3color16] (5,-3) rectangle (6,-4);
\node at (5.5,-3.5) {16};
\draw[fill=7x5x3color17] (6,-3) rectangle (7,-4);
\node at (6.5,-3.5) {17};
\draw[fill=7x5x3color12] (0,-4) rectangle (1,-5);
\node at (0.5,-4.5) {12};
\draw[fill=7x5x3color11] (1,-4) rectangle (2,-5);
\node at (1.5,-4.5) {11};
\draw[fill=red] (2,-4) rectangle (3,-5);
\draw[fill=7x5x3color7] (3,-4) rectangle (4,-5);
\node at (3.5,-4.5) {7};
\draw[fill=7x5x3color14] (4,-4) rectangle (5,-5);
\node at (4.5,-4.5) {14};
\draw[fill=7x5x3color15] (5,-4) rectangle (6,-5);
\node at (5.5,-4.5) {15};
\draw[fill=red] (6,-4) rectangle (7,-5);
\end{scope}

\begin{scope}[xshift=8cm,yshift=0cm]
\draw[very thick] (0,0) rectangle (7,-5);
\foreach \b in {0,...,4}
 \foreach \c in {0,...,6}
  \draw (\c,\b*-1) rectangle (\c+1,(\b*-1-1);

\draw[fill=7x5x3color5] (0,-0) rectangle (1,-1);
\node at (0.5,-0.5) {5};
\draw[fill=7x5x3color6] (1,-0) rectangle (2,-1);
\node at (1.5,-0.5) {6};
\draw[fill=7x5x3color7] (2,-0) rectangle (3,-1);
\node at (2.5,-0.5) {7};
\draw[fill=7x5x3color14] (3,-0) rectangle (4,-1);
\node at (3.5,-0.5) {14};
\draw[fill=7x5x3color15] (4,-0) rectangle (5,-1);
\node at (4.5,-0.5) {15};
\draw[fill=7x5x3color20] (5,-0) rectangle (6,-1);
\node at (5.5,-0.5) {20};
\draw[fill=7x5x3color21] (6,-0) rectangle (7,-1);
\node at (6.5,-0.5) {21};
\draw[fill=7x5x3color4] (0,-1) rectangle (1,-2);
\node at (0.5,-1.5) {4};
\draw[fill=7x5x3color1] (1,-1) rectangle (2,-2);
\node at (1.5,-1.5) {1};
\draw[fill=red] (2,-1) rectangle (3,-2);
\draw[fill=7x5x3color7] (3,-1) rectangle (4,-2);
\node at (3.5,-1.5) {7};
\draw[fill=7x5x3color10] (4,-1) rectangle (5,-2);
\node at (4.5,-1.5) {10};
\draw[fill=7x5x3color11] (5,-1) rectangle (6,-2);
\node at (5.5,-1.5) {11};
\draw[fill=7x5x3color20] (6,-1) rectangle (7,-2);
\node at (6.5,-1.5) {20};
\draw[fill=7x5x3color1] (0,-2) rectangle (1,-3);
\node at (0.5,-2.5) {1};
\draw[fill=red] (1,-2) rectangle (2,-3);
\draw[fill=7x5x3color3] (2,-2) rectangle (3,-3);
\node at (2.5,-2.5) {3};
\draw[fill=7x5x3color6] (3,-2) rectangle (4,-3);
\node at (3.5,-2.5) {6};
\draw[fill=7x5x3color9] (4,-2) rectangle (5,-3);
\node at (4.5,-2.5) {9};
\draw[fill=7x5x3color10] (5,-2) rectangle (6,-3);
\node at (5.5,-2.5) {10};
\draw[fill=7x5x3color11] (6,-2) rectangle (7,-3);
\node at (6.5,-2.5) {11};
\draw[fill=red] (0,-3) rectangle (1,-4);
\draw[fill=7x5x3color1] (1,-3) rectangle (2,-4);
\node at (1.5,-3.5) {1};
\draw[fill=7x5x3color4] (2,-3) rectangle (3,-4);
\node at (2.5,-3.5) {4};
\draw[fill=7x5x3color5] (3,-3) rectangle (4,-4);
\node at (3.5,-3.5) {5};
\draw[fill=7x5x3color8] (4,-3) rectangle (5,-4);
\node at (4.5,-3.5) {8};
\draw[fill=7x5x3color9] (5,-3) rectangle (6,-4);
\node at (5.5,-3.5) {9};
\draw[fill=7x5x3color10] (6,-3) rectangle (7,-4);
\node at (6.5,-3.5) {10};
\draw[fill=7x5x3color7] (0,-4) rectangle (1,-5);
\node at (0.5,-4.5) {7};
\draw[fill=7x5x3color6] (1,-4) rectangle (2,-5);
\node at (1.5,-4.5) {6};
\draw[fill=7x5x3color5] (2,-4) rectangle (3,-5);
\node at (2.5,-4.5) {5};
\draw[fill=7x5x3color6] (3,-4) rectangle (4,-5);
\node at (3.5,-4.5) {6};
\draw[fill=7x5x3color7] (4,-4) rectangle (5,-5);
\node at (4.5,-4.5) {7};
\draw[fill=red] (5,-4) rectangle (6,-5);
\draw[fill=7x5x3color1] (6,-4) rectangle (7,-5);
\node at (6.5,-4.5) {1};
\end{scope}

\begin{scope}[xshift=16cm,yshift=0cm]
\draw[very thick] (0,0) rectangle (7,-5);
\foreach \b in {0,...,4}
 \foreach \c in {0,...,6}
  \draw (\c,\b*-1) rectangle (\c+1,(\b*-1-1);

\draw[fill=red] (0,-0) rectangle (1,-1);
\draw[fill=7x5x3color1] (1,-0) rectangle (2,-1);
\node at (1.5,-0.5) {1};
\draw[fill=red] (2,-0) rectangle (3,-1);
\draw[fill=7x5x3color15] (3,-0) rectangle (4,-1);
\node at (3.5,-0.5) {15};
\draw[fill=7x5x3color16] (4,-0) rectangle (5,-1);
\node at (4.5,-0.5) {16};
\draw[fill=7x5x3color21] (5,-0) rectangle (6,-1);
\node at (5.5,-0.5) {21};
\draw[fill=7x5x3color22] (6,-0) rectangle (7,-1);
\node at (6.5,-0.5) {22};
\draw[fill=7x5x3color3] (0,-1) rectangle (1,-2);
\node at (0.5,-1.5) {3};
\draw[fill=red] (1,-1) rectangle (2,-2);
\draw[fill=7x5x3color1] (2,-1) rectangle (3,-2);
\node at (2.5,-1.5) {1};
\draw[fill=7x5x3color8] (3,-1) rectangle (4,-2);
\node at (3.5,-1.5) {8};
\draw[fill=7x5x3color9] (4,-1) rectangle (5,-2);
\node at (4.5,-1.5) {9};
\draw[fill=red] (5,-1) rectangle (6,-2);
\draw[fill=7x5x3color21] (6,-1) rectangle (7,-2);
\node at (6.5,-1.5) {21};
\draw[fill=7x5x3color2] (0,-2) rectangle (1,-3);
\node at (0.5,-2.5) {2};
\draw[fill=7x5x3color1] (1,-2) rectangle (2,-3);
\node at (1.5,-2.5) {1};
\draw[fill=7x5x3color2] (2,-2) rectangle (3,-3);
\node at (2.5,-2.5) {2};
\draw[fill=7x5x3color3] (3,-2) rectangle (4,-3);
\node at (3.5,-2.5) {3};
\draw[fill=red] (4,-2) rectangle (5,-3);
\draw[fill=7x5x3color1] (5,-2) rectangle (6,-3);
\node at (5.5,-2.5) {1};
\draw[fill=red] (6,-2) rectangle (7,-3);
\draw[fill=7x5x3color1] (0,-3) rectangle (1,-4);
\node at (0.5,-3.5) {1};
\draw[fill=red] (1,-3) rectangle (2,-4);
\draw[fill=7x5x3color1] (2,-3) rectangle (3,-4);
\node at (2.5,-3.5) {1};
\draw[fill=red] (3,-3) rectangle (4,-4);
\draw[fill=7x5x3color1] (4,-3) rectangle (5,-4);
\node at (4.5,-3.5) {1};
\draw[fill=7x5x3color2] (5,-3) rectangle (6,-4);
\node at (5.5,-3.5) {2};
\draw[fill=7x5x3color3] (6,-3) rectangle (7,-4);
\node at (6.5,-3.5) {3};
\draw[fill=red] (0,-4) rectangle (1,-5);
\draw[fill=7x5x3color1] (1,-4) rectangle (2,-5);
\node at (1.5,-4.5) {1};
\draw[fill=red] (2,-4) rectangle (3,-5);
\draw[fill=7x5x3color1] (3,-4) rectangle (4,-5);
\node at (3.5,-4.5) {1};
\draw[fill=red] (4,-4) rectangle (5,-5);
\draw[fill=7x5x3color1] (5,-4) rectangle (6,-5);
\node at (5.5,-4.5) {1};
\draw[fill=red] (6,-4) rectangle (7,-5);
\end{scope}

\end{tikzpicture}
\end{center}

%% file: 4-4-5-tikz.txt
\definecolor{5x4x4color1}{RGB}{255, 128, 128}
\definecolor{5x4x4color2}{RGB}{255, 137, 137}
\definecolor{5x4x4color3}{RGB}{255, 146, 146}
\definecolor{5x4x4color4}{RGB}{255, 155, 155}
\definecolor{5x4x4color5}{RGB}{255, 164, 164}
\definecolor{5x4x4color6}{RGB}{255, 173, 173}
\definecolor{5x4x4color7}{RGB}{255, 182, 182}
\definecolor{5x4x4color8}{RGB}{255, 191, 191}
\definecolor{5x4x4color9}{RGB}{255, 200, 200}
\definecolor{5x4x4color10}{RGB}{255, 209, 209}
\definecolor{5x4x4color11}{RGB}{255, 218, 218}
\definecolor{5x4x4color12}{RGB}{255, 227, 227}
\definecolor{5x4x4color13}{RGB}{255, 236, 236}
\definecolor{5x4x4color14}{RGB}{255, 255, 255}

\begin{center}
\begin{tikzpicture}[scale=0.5]
\begin{scope}[xshift=0cm,yshift=0cm]
\draw[very thick] (0,0) rectangle (5,-4);
\foreach \b in {0,...,3}
 \foreach \c in {0,...,4}
  \draw (\c,\b*-1) rectangle (\c+1,(\b*-1-1);

\draw[fill=5x4x4color6] (0,-0) rectangle (1,-1);
\node at (0.5,-0.5) {6};
\draw[fill=5x4x4color5] (1,-0) rectangle (2,-1);
\node at (1.5,-0.5) {5};
\draw[fill=red] (2,-0) rectangle (3,-1);
\draw[fill=5x4x4color11] (3,-0) rectangle (4,-1);
\node at (3.5,-0.5) {11};
\draw[fill=5x4x4color12] (4,-0) rectangle (5,-1);
\node at (4.5,-0.5) {12};
\draw[fill=5x4x4color1] (0,-1) rectangle (1,-2);
\node at (0.5,-1.5) {1};
\draw[fill=red] (1,-1) rectangle (2,-2);
\draw[fill=5x4x4color1] (2,-1) rectangle (3,-2);
\node at (2.5,-1.5) {1};
\draw[fill=5x4x4color10] (3,-1) rectangle (4,-2);
\node at (3.5,-1.5) {10};
\draw[fill=5x4x4color11] (4,-1) rectangle (5,-2);
\node at (4.5,-1.5) {11};
\draw[fill=red] (0,-2) rectangle (1,-3);
\draw[fill=5x4x4color1] (1,-2) rectangle (2,-3);
\node at (1.5,-2.5) {1};
\draw[fill=5x4x4color4] (2,-2) rectangle (3,-3);
\node at (2.5,-2.5) {4};
\draw[fill=5x4x4color9] (3,-2) rectangle (4,-3);
\node at (3.5,-2.5) {9};
\draw[fill=red] (4,-2) rectangle (5,-3);
\draw[fill=5x4x4color7] (0,-3) rectangle (1,-4);
\node at (0.5,-3.5) {7};
\draw[fill=red] (1,-3) rectangle (2,-4);
\draw[fill=5x4x4color5] (2,-3) rectangle (3,-4);
\node at (2.5,-3.5) {5};
\draw[fill=5x4x4color12] (3,-3) rectangle (4,-4);
\node at (3.5,-3.5) {12};
\draw[fill=5x4x4color13] (4,-3) rectangle (5,-4);
\node at (4.5,-3.5) {13};
\end{scope}

\begin{scope}[xshift=6cm,yshift=0cm]
\draw[very thick] (0,0) rectangle (5,-4);
\foreach \b in {0,...,3}
 \foreach \c in {0,...,4}
  \draw (\c,\b*-1) rectangle (\c+1,(\b*-1-1);

\draw[fill=5x4x4color5] (0,-0) rectangle (1,-1);
\node at (0.5,-0.5) {5};
\draw[fill=5x4x4color4] (1,-0) rectangle (2,-1);
\node at (1.5,-0.5) {4};
\draw[fill=5x4x4color1] (2,-0) rectangle (3,-1);
\node at (2.5,-0.5) {1};
\draw[fill=red] (3,-0) rectangle (4,-1);
\draw[fill=5x4x4color11] (4,-0) rectangle (5,-1);
\node at (4.5,-0.5) {11};
\draw[fill=red] (0,-1) rectangle (1,-2);
\draw[fill=5x4x4color1] (1,-1) rectangle (2,-2);
\node at (1.5,-1.5) {1};
\draw[fill=red] (2,-1) rectangle (3,-2);
\draw[fill=5x4x4color7] (3,-1) rectangle (4,-2);
\node at (3.5,-1.5) {7};
\draw[fill=5x4x4color10] (4,-1) rectangle (5,-2);
\node at (4.5,-1.5) {10};
\draw[fill=5x4x4color3] (0,-2) rectangle (1,-3);
\node at (0.5,-2.5) {3};
\draw[fill=5x4x4color2] (1,-2) rectangle (2,-3);
\node at (1.5,-2.5) {2};
\draw[fill=5x4x4color3] (2,-2) rectangle (3,-3);
\node at (2.5,-2.5) {3};
\draw[fill=5x4x4color8] (3,-2) rectangle (4,-3);
\node at (3.5,-2.5) {8};
\draw[fill=5x4x4color9] (4,-2) rectangle (5,-3);
\node at (4.5,-2.5) {9};
\draw[fill=5x4x4color6] (0,-3) rectangle (1,-4);
\node at (0.5,-3.5) {6};
\draw[fill=5x4x4color1] (1,-3) rectangle (2,-4);
\node at (1.5,-3.5) {1};
\draw[fill=red] (2,-3) rectangle (3,-4);
\draw[fill=5x4x4color11] (3,-3) rectangle (4,-4);
\node at (3.5,-3.5) {11};
\draw[fill=5x4x4color12] (4,-3) rectangle (5,-4);
\node at (4.5,-3.5) {12};
\end{scope}

\begin{scope}[xshift=12cm,yshift=0cm]
\draw[very thick] (0,0) rectangle (5,-4);
\foreach \b in {0,...,3}
 \foreach \c in {0,...,4}
  \draw (\c,\b*-1) rectangle (\c+1,(\b*-1-1);

\draw[fill=red] (0,-0) rectangle (1,-1);
\draw[fill=5x4x4color3] (1,-0) rectangle (2,-1);
\node at (1.5,-0.5) {3};
\draw[fill=5x4x4color4] (2,-0) rectangle (3,-1);
\node at (2.5,-0.5) {4};
\draw[fill=5x4x4color5] (3,-0) rectangle (4,-1);
\node at (3.5,-0.5) {5};
\draw[fill=red] (4,-0) rectangle (5,-1);
\draw[fill=5x4x4color1] (0,-1) rectangle (1,-2);
\node at (0.5,-1.5) {1};
\draw[fill=5x4x4color2] (1,-1) rectangle (2,-2);
\node at (1.5,-1.5) {2};
\draw[fill=5x4x4color3] (2,-1) rectangle (3,-2);
\node at (2.5,-1.5) {3};
\draw[fill=5x4x4color6] (3,-1) rectangle (4,-2);
\node at (3.5,-1.5) {6};
\draw[fill=5x4x4color1] (4,-1) rectangle (5,-2);
\node at (4.5,-1.5) {1};
\draw[fill=5x4x4color4] (0,-2) rectangle (1,-3);
\node at (0.5,-2.5) {4};
\draw[fill=5x4x4color3] (1,-2) rectangle (2,-3);
\node at (1.5,-2.5) {3};
\draw[fill=5x4x4color4] (2,-2) rectangle (3,-3);
\node at (2.5,-2.5) {4};
\draw[fill=5x4x4color7] (3,-2) rectangle (4,-3);
\node at (3.5,-2.5) {7};
\draw[fill=red] (4,-2) rectangle (5,-3);
\draw[fill=5x4x4color5] (0,-3) rectangle (1,-4);
\node at (0.5,-3.5) {5};
\draw[fill=red] (1,-3) rectangle (2,-4);
\draw[fill=5x4x4color5] (2,-3) rectangle (3,-4);
\node at (2.5,-3.5) {5};
\draw[fill=5x4x4color10] (3,-3) rectangle (4,-4);
\node at (3.5,-3.5) {10};
\draw[fill=5x4x4color11] (4,-3) rectangle (5,-4);
\node at (4.5,-3.5) {11};
\end{scope}

\begin{scope}[xshift=18cm,yshift=0cm]
\draw[very thick] (0,0) rectangle (5,-4);
\foreach \b in {0,...,3}
 \foreach \c in {0,...,4}
  \draw (\c,\b*-1) rectangle (\c+1,(\b*-1-1);

\draw[fill=5x4x4color1] (0,-0) rectangle (1,-1);
\node at (0.5,-0.5) {1};
\draw[fill=red] (1,-0) rectangle (2,-1);
\draw[fill=5x4x4color5] (2,-0) rectangle (3,-1);
\node at (2.5,-0.5) {5};
\draw[fill=5x4x4color8] (3,-0) rectangle (4,-1);
\node at (3.5,-0.5) {8};
\draw[fill=5x4x4color9] (4,-0) rectangle (5,-1);
\node at (4.5,-0.5) {9};
\draw[fill=red] (0,-1) rectangle (1,-2);
\draw[fill=5x4x4color1] (1,-1) rectangle (2,-2);
\node at (1.5,-1.5) {1};
\draw[fill=red] (2,-1) rectangle (3,-2);
\draw[fill=5x4x4color7] (3,-1) rectangle (4,-2);
\node at (3.5,-1.5) {7};
\draw[fill=red] (4,-1) rectangle (5,-2);
\draw[fill=5x4x4color5] (0,-2) rectangle (1,-3);
\node at (0.5,-2.5) {5};
\draw[fill=5x4x4color6] (1,-2) rectangle (2,-3);
\node at (1.5,-2.5) {6};
\draw[fill=5x4x4color7] (2,-2) rectangle (3,-3);
\node at (2.5,-2.5) {7};
\draw[fill=5x4x4color8] (3,-2) rectangle (4,-3);
\node at (3.5,-2.5) {8};
\draw[fill=5x4x4color1] (4,-2) rectangle (5,-3);
\node at (4.5,-2.5) {1};
\draw[fill=red] (0,-3) rectangle (1,-4);
\draw[fill=5x4x4color7] (1,-3) rectangle (2,-4);
\node at (1.5,-3.5) {7};
\draw[fill=5x4x4color8] (2,-3) rectangle (3,-4);
\node at (2.5,-3.5) {8};
\draw[fill=5x4x4color9] (3,-3) rectangle (4,-4);
\node at (3.5,-3.5) {9};
\draw[fill=red] (4,-3) rectangle (5,-4);
\end{scope}

\end{tikzpicture}
\end{center}

%% file: 4-4-6-tikz.txt
\definecolor{6x4x4color1}{RGB}{255, 128, 128}
\definecolor{6x4x4color2}{RGB}{255, 133, 133}
\definecolor{6x4x4color3}{RGB}{255, 139, 139}
\definecolor{6x4x4color4}{RGB}{255, 145, 145}
\definecolor{6x4x4color5}{RGB}{255, 151, 151}
\definecolor{6x4x4color6}{RGB}{255, 156, 156}
\definecolor{6x4x4color7}{RGB}{255, 162, 162}
\definecolor{6x4x4color8}{RGB}{255, 168, 168}
\definecolor{6x4x4color9}{RGB}{255, 174, 174}
\definecolor{6x4x4color10}{RGB}{255, 179, 179}
\definecolor{6x4x4color11}{RGB}{255, 185, 185}
\definecolor{6x4x4color12}{RGB}{255, 191, 191}
\definecolor{6x4x4color13}{RGB}{255, 197, 197}
\definecolor{6x4x4color14}{RGB}{255, 203, 203}
\definecolor{6x4x4color15}{RGB}{255, 208, 208}
\definecolor{6x4x4color16}{RGB}{255, 214, 214}
\definecolor{6x4x4color17}{RGB}{255, 220, 220}
\definecolor{6x4x4color18}{RGB}{255, 226, 226}
\definecolor{6x4x4color19}{RGB}{255, 231, 231}
\definecolor{6x4x4color20}{RGB}{255, 237, 237}
\definecolor{6x4x4color21}{RGB}{255, 243, 243}
\definecolor{6x4x4color22}{RGB}{255, 255, 255}

\begin{center}
\begin{tikzpicture}[scale=0.5]
\begin{scope}[xshift=0cm,yshift=0cm]
\draw[very thick] (0,0) rectangle (6,-4);
\foreach \b in {0,...,3}
 \foreach \c in {0,...,5}
  \draw (\c,\b*-1) rectangle (\c+1,(\b*-1-1);

\draw[fill=red] (0,-0) rectangle (1,-1);
\draw[fill=6x4x4color1] (1,-0) rectangle (2,-1);
\node at (1.5,-0.5) {1};
\draw[fill=red] (2,-0) rectangle (3,-1);
\draw[fill=6x4x4color3] (3,-0) rectangle (4,-1);
\node at (3.5,-0.5) {3};
\draw[fill=red] (4,-0) rectangle (5,-1);
\draw[fill=6x4x4color1] (5,-0) rectangle (6,-1);
\node at (5.5,-0.5) {1};
\draw[fill=6x4x4color7] (0,-1) rectangle (1,-2);
\node at (0.5,-1.5) {7};
\draw[fill=6x4x4color6] (1,-1) rectangle (2,-2);
\node at (1.5,-1.5) {6};
\draw[fill=6x4x4color5] (2,-1) rectangle (3,-2);
\node at (2.5,-1.5) {5};
\draw[fill=6x4x4color4] (3,-1) rectangle (4,-2);
\node at (3.5,-1.5) {4};
\draw[fill=6x4x4color1] (4,-1) rectangle (5,-2);
\node at (4.5,-1.5) {1};
\draw[fill=red] (5,-1) rectangle (6,-2);
\draw[fill=6x4x4color8] (0,-2) rectangle (1,-3);
\node at (0.5,-2.5) {8};
\draw[fill=6x4x4color7] (1,-2) rectangle (2,-3);
\node at (1.5,-2.5) {7};
\draw[fill=6x4x4color8] (2,-2) rectangle (3,-3);
\node at (2.5,-2.5) {8};
\draw[fill=6x4x4color5] (3,-2) rectangle (4,-3);
\node at (3.5,-2.5) {5};
\draw[fill=red] (4,-2) rectangle (5,-3);
\draw[fill=6x4x4color13] (5,-2) rectangle (6,-3);
\node at (5.5,-2.5) {13};
\draw[fill=6x4x4color9] (0,-3) rectangle (1,-4);
\node at (0.5,-3.5) {9};
\draw[fill=red] (1,-3) rectangle (2,-4);
\draw[fill=6x4x4color11] (2,-3) rectangle (3,-4);
\node at (2.5,-3.5) {11};
\draw[fill=6x4x4color14] (3,-3) rectangle (4,-4);
\node at (3.5,-3.5) {14};
\draw[fill=6x4x4color15] (4,-3) rectangle (5,-4);
\node at (4.5,-3.5) {15};
\draw[fill=6x4x4color16] (5,-3) rectangle (6,-4);
\node at (5.5,-3.5) {16};
\end{scope}

\begin{scope}[xshift=7cm,yshift=0cm]
\draw[very thick] (0,0) rectangle (6,-4);
\foreach \b in {0,...,3}
 \foreach \c in {0,...,5}
  \draw (\c,\b*-1) rectangle (\c+1,(\b*-1-1);

\draw[fill=6x4x4color1] (0,-0) rectangle (1,-1);
\node at (0.5,-0.5) {1};
\draw[fill=red] (1,-0) rectangle (2,-1);
\draw[fill=6x4x4color1] (2,-0) rectangle (3,-1);
\node at (2.5,-0.5) {1};
\draw[fill=6x4x4color2] (3,-0) rectangle (4,-1);
\node at (3.5,-0.5) {2};
\draw[fill=6x4x4color1] (4,-0) rectangle (5,-1);
\node at (4.5,-0.5) {1};
\draw[fill=red] (5,-0) rectangle (6,-1);
\draw[fill=6x4x4color6] (0,-1) rectangle (1,-2);
\node at (0.5,-1.5) {6};
\draw[fill=6x4x4color5] (1,-1) rectangle (2,-2);
\node at (1.5,-1.5) {5};
\draw[fill=6x4x4color4] (2,-1) rectangle (3,-2);
\node at (2.5,-1.5) {4};
\draw[fill=6x4x4color3] (3,-1) rectangle (4,-2);
\node at (3.5,-1.5) {3};
\draw[fill=6x4x4color4] (4,-1) rectangle (5,-2);
\node at (4.5,-1.5) {4};
\draw[fill=6x4x4color5] (5,-1) rectangle (6,-2);
\node at (5.5,-1.5) {5};
\draw[fill=6x4x4color7] (0,-2) rectangle (1,-3);
\node at (0.5,-2.5) {7};
\draw[fill=6x4x4color6] (1,-2) rectangle (2,-3);
\node at (1.5,-2.5) {6};
\draw[fill=6x4x4color7] (2,-2) rectangle (3,-3);
\node at (2.5,-2.5) {7};
\draw[fill=red] (3,-2) rectangle (4,-3);
\draw[fill=6x4x4color5] (4,-2) rectangle (5,-3);
\node at (4.5,-2.5) {5};
\draw[fill=6x4x4color12] (5,-2) rectangle (6,-3);
\node at (5.5,-2.5) {12};
\draw[fill=red] (0,-3) rectangle (1,-4);
\draw[fill=6x4x4color1] (1,-3) rectangle (2,-4);
\node at (1.5,-3.5) {1};
\draw[fill=6x4x4color10] (2,-3) rectangle (3,-4);
\node at (2.5,-3.5) {10};
\draw[fill=6x4x4color13] (3,-3) rectangle (4,-4);
\node at (3.5,-3.5) {13};
\draw[fill=6x4x4color14] (4,-3) rectangle (5,-4);
\node at (4.5,-3.5) {14};
\draw[fill=6x4x4color15] (5,-3) rectangle (6,-4);
\node at (5.5,-3.5) {15};
\end{scope}

\begin{scope}[xshift=14cm,yshift=0cm]
\draw[very thick] (0,0) rectangle (6,-4);
\foreach \b in {0,...,3}
 \foreach \c in {0,...,5}
  \draw (\c,\b*-1) rectangle (\c+1,(\b*-1-1);

\draw[fill=red] (0,-0) rectangle (1,-1);
\draw[fill=6x4x4color1] (1,-0) rectangle (2,-1);
\node at (1.5,-0.5) {1};
\draw[fill=red] (2,-0) rectangle (3,-1);
\draw[fill=6x4x4color1] (3,-0) rectangle (4,-1);
\node at (3.5,-0.5) {1};
\draw[fill=red] (4,-0) rectangle (5,-1);
\draw[fill=6x4x4color1] (5,-0) rectangle (6,-1);
\node at (5.5,-0.5) {1};
\draw[fill=6x4x4color1] (0,-1) rectangle (1,-2);
\node at (0.5,-1.5) {1};
\draw[fill=6x4x4color2] (1,-1) rectangle (2,-2);
\node at (1.5,-1.5) {2};
\draw[fill=6x4x4color3] (2,-1) rectangle (3,-2);
\node at (2.5,-1.5) {3};
\draw[fill=red] (3,-1) rectangle (4,-2);
\draw[fill=6x4x4color5] (4,-1) rectangle (5,-2);
\node at (4.5,-1.5) {5};
\draw[fill=6x4x4color6] (5,-1) rectangle (6,-2);
\node at (5.5,-1.5) {6};
\draw[fill=red] (0,-2) rectangle (1,-3);
\draw[fill=6x4x4color1] (1,-2) rectangle (2,-3);
\node at (1.5,-2.5) {1};
\draw[fill=6x4x4color8] (2,-2) rectangle (3,-3);
\node at (2.5,-2.5) {8};
\draw[fill=6x4x4color9] (3,-2) rectangle (4,-3);
\node at (3.5,-2.5) {9};
\draw[fill=6x4x4color10] (4,-2) rectangle (5,-3);
\node at (4.5,-2.5) {10};
\draw[fill=6x4x4color11] (5,-2) rectangle (6,-3);
\node at (5.5,-2.5) {11};
\draw[fill=6x4x4color1] (0,-3) rectangle (1,-4);
\node at (0.5,-3.5) {1};
\draw[fill=red] (1,-3) rectangle (2,-4);
\draw[fill=6x4x4color9] (2,-3) rectangle (3,-4);
\node at (2.5,-3.5) {9};
\draw[fill=6x4x4color12] (3,-3) rectangle (4,-4);
\node at (3.5,-3.5) {12};
\draw[fill=6x4x4color11] (4,-3) rectangle (5,-4);
\node at (4.5,-3.5) {11};
\draw[fill=red] (5,-3) rectangle (6,-4);
\end{scope}

\begin{scope}[xshift=21cm,yshift=0cm]
\draw[very thick] (0,0) rectangle (6,-4);
\foreach \b in {0,...,3}
 \foreach \c in {0,...,5}
  \draw (\c,\b*-1) rectangle (\c+1,(\b*-1-1);

\draw[fill=6x4x4color21] (0,-0) rectangle (1,-1);
\node at (0.5,-0.5) {21};
\draw[fill=6x4x4color20] (1,-0) rectangle (2,-1);
\node at (1.5,-0.5) {20};
\draw[fill=6x4x4color19] (2,-0) rectangle (3,-1);
\node at (2.5,-0.5) {19};
\draw[fill=6x4x4color18] (3,-0) rectangle (4,-1);
\node at (3.5,-0.5) {18};
\draw[fill=6x4x4color17] (4,-0) rectangle (5,-1);
\node at (4.5,-0.5) {17};
\draw[fill=red] (5,-0) rectangle (6,-1);
\draw[fill=red] (0,-1) rectangle (1,-2);
\draw[fill=6x4x4color3] (1,-1) rectangle (2,-2);
\node at (1.5,-1.5) {3};
\draw[fill=6x4x4color10] (2,-1) rectangle (3,-2);
\node at (2.5,-1.5) {10};
\draw[fill=6x4x4color15] (3,-1) rectangle (4,-2);
\node at (3.5,-1.5) {15};
\draw[fill=6x4x4color16] (4,-1) rectangle (5,-2);
\node at (4.5,-1.5) {16};
\draw[fill=6x4x4color17] (5,-1) rectangle (6,-2);
\node at (5.5,-1.5) {17};
\draw[fill=6x4x4color1] (0,-2) rectangle (1,-3);
\node at (0.5,-2.5) {1};
\draw[fill=red] (1,-2) rectangle (2,-3);
\draw[fill=6x4x4color9] (2,-2) rectangle (3,-3);
\node at (2.5,-2.5) {9};
\draw[fill=6x4x4color14] (3,-2) rectangle (4,-3);
\node at (3.5,-2.5) {14};
\draw[fill=6x4x4color15] (4,-2) rectangle (5,-3);
\node at (4.5,-2.5) {15};
\draw[fill=6x4x4color18] (5,-2) rectangle (6,-3);
\node at (5.5,-2.5) {18};
\draw[fill=6x4x4color2] (0,-3) rectangle (1,-4);
\node at (0.5,-3.5) {2};
\draw[fill=6x4x4color1] (1,-3) rectangle (2,-4);
\node at (1.5,-3.5) {1};
\draw[fill=red] (2,-3) rectangle (3,-4);
\draw[fill=6x4x4color13] (3,-3) rectangle (4,-4);
\node at (3.5,-3.5) {13};
\draw[fill=red] (4,-3) rectangle (5,-4);
\draw[fill=6x4x4color19] (5,-3) rectangle (6,-4);
\node at (5.5,-3.5) {19};
\end{scope}

\end{tikzpicture}
\end{center}

%% file: 4-5-5-tikz.txt
\definecolor{5x5x4color1}{RGB}{255, 128, 128}
\definecolor{5x5x4color2}{RGB}{255, 134, 134}
\definecolor{5x5x4color3}{RGB}{255, 141, 141}
\definecolor{5x5x4color4}{RGB}{255, 148, 148}
\definecolor{5x5x4color5}{RGB}{255, 154, 154}
\definecolor{5x5x4color6}{RGB}{255, 161, 161}
\definecolor{5x5x4color7}{RGB}{255, 168, 168}
\definecolor{5x5x4color8}{RGB}{255, 174, 174}
\definecolor{5x5x4color9}{RGB}{255, 181, 181}
\definecolor{5x5x4color10}{RGB}{255, 188, 188}
\definecolor{5x5x4color11}{RGB}{255, 194, 194}
\definecolor{5x5x4color12}{RGB}{255, 201, 201}
\definecolor{5x5x4color13}{RGB}{255, 208, 208}
\definecolor{5x5x4color14}{RGB}{255, 214, 214}
\definecolor{5x5x4color15}{RGB}{255, 221, 221}
\definecolor{5x5x4color16}{RGB}{255, 228, 228}
\definecolor{5x5x4color17}{RGB}{255, 234, 234}
\definecolor{5x5x4color18}{RGB}{255, 241, 241}
\definecolor{5x5x4color19}{RGB}{255, 255, 255}

\begin{center}
\begin{tikzpicture}[scale=0.5]
\begin{scope}[xshift=0cm,yshift=0cm]
\draw[very thick] (0,0) rectangle (5,-5);
\foreach \b in {0,...,4}
 \foreach \c in {0,...,4}
  \draw (\c,\b*-1) rectangle (\c+1,(\b*-1-1);

\draw[fill=red] (0,-0) rectangle (1,-1);
\draw[fill=5x5x4color1] (1,-0) rectangle (2,-1);
\node at (1.5,-0.5) {1};
\draw[fill=red] (2,-0) rectangle (3,-1);
\draw[fill=5x5x4color13] (3,-0) rectangle (4,-1);
\node at (3.5,-0.5) {13};
\draw[fill=red] (4,-0) rectangle (5,-1);
\draw[fill=5x5x4color11] (0,-1) rectangle (1,-2);
\node at (0.5,-1.5) {11};
\draw[fill=red] (1,-1) rectangle (2,-2);
\draw[fill=5x5x4color9] (2,-1) rectangle (3,-2);
\node at (2.5,-1.5) {9};
\draw[fill=5x5x4color14] (3,-1) rectangle (4,-2);
\node at (3.5,-1.5) {14};
\draw[fill=5x5x4color15] (4,-1) rectangle (5,-2);
\node at (4.5,-1.5) {15};
\draw[fill=5x5x4color12] (0,-2) rectangle (1,-3);
\node at (0.5,-2.5) {12};
\draw[fill=5x5x4color1] (1,-2) rectangle (2,-3);
\node at (1.5,-2.5) {1};
\draw[fill=5x5x4color10] (2,-2) rectangle (3,-3);
\node at (2.5,-2.5) {10};
\draw[fill=5x5x4color15] (3,-2) rectangle (4,-3);
\node at (3.5,-2.5) {15};
\draw[fill=5x5x4color16] (4,-2) rectangle (5,-3);
\node at (4.5,-2.5) {16};
\draw[fill=5x5x4color13] (0,-3) rectangle (1,-4);
\node at (0.5,-3.5) {13};
\draw[fill=red] (1,-3) rectangle (2,-4);
\draw[fill=5x5x4color11] (2,-3) rectangle (3,-4);
\node at (2.5,-3.5) {11};
\draw[fill=5x5x4color16] (3,-3) rectangle (4,-4);
\node at (3.5,-3.5) {16};
\draw[fill=5x5x4color17] (4,-3) rectangle (5,-4);
\node at (4.5,-3.5) {17};
\draw[fill=red] (0,-4) rectangle (1,-5);
\draw[fill=5x5x4color9] (1,-4) rectangle (2,-5);
\node at (1.5,-4.5) {9};
\draw[fill=5x5x4color12] (2,-4) rectangle (3,-5);
\node at (2.5,-4.5) {12};
\draw[fill=5x5x4color17] (3,-4) rectangle (4,-5);
\node at (3.5,-4.5) {17};
\draw[fill=5x5x4color18] (4,-4) rectangle (5,-5);
\node at (4.5,-4.5) {18};
\end{scope}

\begin{scope}[xshift=6cm,yshift=0cm]
\draw[very thick] (0,0) rectangle (5,-5);
\foreach \b in {0,...,4}
 \foreach \c in {0,...,4}
  \draw (\c,\b*-1) rectangle (\c+1,(\b*-1-1);

\draw[fill=5x5x4color9] (0,-0) rectangle (1,-1);
\node at (0.5,-0.5) {9};
\draw[fill=5x5x4color8] (1,-0) rectangle (2,-1);
\node at (1.5,-0.5) {8};
\draw[fill=5x5x4color9] (2,-0) rectangle (3,-1);
\node at (2.5,-0.5) {9};
\draw[fill=5x5x4color12] (3,-0) rectangle (4,-1);
\node at (3.5,-0.5) {12};
\draw[fill=5x5x4color13] (4,-0) rectangle (5,-1);
\node at (4.5,-0.5) {13};
\draw[fill=5x5x4color10] (0,-1) rectangle (1,-2);
\node at (0.5,-1.5) {10};
\draw[fill=5x5x4color7] (1,-1) rectangle (2,-2);
\node at (1.5,-1.5) {7};
\draw[fill=5x5x4color8] (2,-1) rectangle (3,-2);
\node at (2.5,-1.5) {8};
\draw[fill=5x5x4color9] (3,-1) rectangle (4,-2);
\node at (3.5,-1.5) {9};
\draw[fill=5x5x4color14] (4,-1) rectangle (5,-2);
\node at (4.5,-1.5) {14};
\draw[fill=5x5x4color11] (0,-2) rectangle (1,-3);
\node at (0.5,-2.5) {11};
\draw[fill=red] (1,-2) rectangle (2,-3);
\draw[fill=5x5x4color1] (2,-2) rectangle (3,-3);
\node at (2.5,-2.5) {1};
\draw[fill=red] (3,-2) rectangle (4,-3);
\draw[fill=5x5x4color15] (4,-2) rectangle (5,-3);
\node at (4.5,-2.5) {15};
\draw[fill=5x5x4color14] (0,-3) rectangle (1,-4);
\node at (0.5,-3.5) {14};
\draw[fill=5x5x4color1] (1,-3) rectangle (2,-4);
\node at (1.5,-3.5) {1};
\draw[fill=red] (2,-3) rectangle (3,-4);
\draw[fill=5x5x4color1] (3,-3) rectangle (4,-4);
\node at (3.5,-3.5) {1};
\draw[fill=5x5x4color16] (4,-3) rectangle (5,-4);
\node at (4.5,-3.5) {16};
\draw[fill=5x5x4color15] (0,-4) rectangle (1,-5);
\node at (0.5,-4.5) {15};
\draw[fill=5x5x4color8] (1,-4) rectangle (2,-5);
\node at (1.5,-4.5) {8};
\draw[fill=5x5x4color1] (2,-4) rectangle (3,-5);
\node at (2.5,-4.5) {1};
\draw[fill=red] (3,-4) rectangle (4,-5);
\draw[fill=5x5x4color17] (4,-4) rectangle (5,-5);
\node at (4.5,-4.5) {17};
\end{scope}

\begin{scope}[xshift=12cm,yshift=0cm]
\draw[very thick] (0,0) rectangle (5,-5);
\foreach \b in {0,...,4}
 \foreach \c in {0,...,4}
  \draw (\c,\b*-1) rectangle (\c+1,(\b*-1-1);

\draw[fill=red] (0,-0) rectangle (1,-1);
\draw[fill=5x5x4color7] (1,-0) rectangle (2,-1);
\node at (1.5,-0.5) {7};
\draw[fill=5x5x4color10] (2,-0) rectangle (3,-1);
\node at (2.5,-0.5) {10};
\draw[fill=5x5x4color11] (3,-0) rectangle (4,-1);
\node at (3.5,-0.5) {11};
\draw[fill=red] (4,-0) rectangle (5,-1);
\draw[fill=5x5x4color1] (0,-1) rectangle (1,-2);
\node at (0.5,-1.5) {1};
\draw[fill=5x5x4color6] (1,-1) rectangle (2,-2);
\node at (1.5,-1.5) {6};
\draw[fill=5x5x4color7] (2,-1) rectangle (3,-2);
\node at (2.5,-1.5) {7};
\draw[fill=5x5x4color8] (3,-1) rectangle (4,-2);
\node at (3.5,-1.5) {8};
\draw[fill=5x5x4color1] (4,-1) rectangle (5,-2);
\node at (4.5,-1.5) {1};
\draw[fill=red] (0,-2) rectangle (1,-3);
\draw[fill=5x5x4color5] (1,-2) rectangle (2,-3);
\node at (1.5,-2.5) {5};
\draw[fill=5x5x4color4] (2,-2) rectangle (3,-3);
\node at (2.5,-2.5) {4};
\draw[fill=5x5x4color3] (3,-2) rectangle (4,-3);
\node at (3.5,-2.5) {3};
\draw[fill=red] (4,-2) rectangle (5,-3);
\draw[fill=5x5x4color15] (0,-3) rectangle (1,-4);
\node at (0.5,-3.5) {15};
\draw[fill=5x5x4color6] (1,-3) rectangle (2,-4);
\node at (1.5,-3.5) {6};
\draw[fill=5x5x4color3] (2,-3) rectangle (3,-4);
\node at (2.5,-3.5) {3};
\draw[fill=5x5x4color2] (3,-3) rectangle (4,-4);
\node at (3.5,-3.5) {2};
\draw[fill=5x5x4color1] (4,-3) rectangle (5,-4);
\node at (4.5,-3.5) {1};
\draw[fill=5x5x4color16] (0,-4) rectangle (1,-5);
\node at (0.5,-4.5) {16};
\draw[fill=5x5x4color7] (1,-4) rectangle (2,-5);
\node at (1.5,-4.5) {7};
\draw[fill=red] (2,-4) rectangle (3,-5);
\draw[fill=5x5x4color1] (3,-4) rectangle (4,-5);
\node at (3.5,-4.5) {1};
\draw[fill=red] (4,-4) rectangle (5,-5);
\end{scope}

\begin{scope}[xshift=18cm,yshift=0cm]
\draw[very thick] (0,0) rectangle (5,-5);
\foreach \b in {0,...,4}
 \foreach \c in {0,...,4}
  \draw (\c,\b*-1) rectangle (\c+1,(\b*-1-1);

\draw[fill=5x5x4color1] (0,-0) rectangle (1,-1);
\node at (0.5,-0.5) {1};
\draw[fill=red] (1,-0) rectangle (2,-1);
\draw[fill=5x5x4color11] (2,-0) rectangle (3,-1);
\node at (2.5,-0.5) {11};
\draw[fill=5x5x4color12] (3,-0) rectangle (4,-1);
\node at (3.5,-0.5) {12};
\draw[fill=5x5x4color13] (4,-0) rectangle (5,-1);
\node at (4.5,-0.5) {13};
\draw[fill=red] (0,-1) rectangle (1,-2);
\draw[fill=5x5x4color1] (1,-1) rectangle (2,-2);
\node at (1.5,-1.5) {1};
\draw[fill=red] (2,-1) rectangle (3,-2);
\draw[fill=5x5x4color9] (3,-1) rectangle (4,-2);
\node at (3.5,-1.5) {9};
\draw[fill=red] (4,-1) rectangle (5,-2);
\draw[fill=5x5x4color13] (0,-2) rectangle (1,-3);
\node at (0.5,-2.5) {13};
\draw[fill=5x5x4color12] (1,-2) rectangle (2,-3);
\node at (1.5,-2.5) {12};
\draw[fill=5x5x4color11] (2,-2) rectangle (3,-3);
\node at (2.5,-2.5) {11};
\draw[fill=5x5x4color10] (3,-2) rectangle (4,-3);
\node at (3.5,-2.5) {10};
\draw[fill=5x5x4color1] (4,-2) rectangle (5,-3);
\node at (4.5,-2.5) {1};
\draw[fill=5x5x4color16] (0,-3) rectangle (1,-4);
\node at (0.5,-3.5) {16};
\draw[fill=5x5x4color13] (1,-3) rectangle (2,-4);
\node at (1.5,-3.5) {13};
\draw[fill=5x5x4color12] (2,-3) rectangle (3,-4);
\node at (2.5,-3.5) {12};
\draw[fill=5x5x4color11] (3,-3) rectangle (4,-4);
\node at (3.5,-3.5) {11};
\draw[fill=red] (4,-3) rectangle (5,-4);
\draw[fill=5x5x4color17] (0,-4) rectangle (1,-5);
\node at (0.5,-4.5) {17};
\draw[fill=red] (1,-4) rectangle (2,-5);
\draw[fill=5x5x4color13] (2,-4) rectangle (3,-5);
\node at (2.5,-4.5) {13};
\draw[fill=5x5x4color14] (3,-4) rectangle (4,-5);
\node at (3.5,-4.5) {14};
\draw[fill=5x5x4color15] (4,-4) rectangle (5,-5);
\node at (4.5,-4.5) {15};
\end{scope}

\end{tikzpicture}
\end{center}

%% file: 4-5-6-tikz.txt
\definecolor{6x5x4color1}{RGB}{255, 128, 128}
\definecolor{6x5x4color2}{RGB}{255, 135, 135}
\definecolor{6x5x4color3}{RGB}{255, 142, 142}
\definecolor{6x5x4color4}{RGB}{255, 149, 149}
\definecolor{6x5x4color5}{RGB}{255, 156, 156}
\definecolor{6x5x4color6}{RGB}{255, 163, 163}
\definecolor{6x5x4color7}{RGB}{255, 170, 170}
\definecolor{6x5x4color8}{RGB}{255, 177, 177}
\definecolor{6x5x4color9}{RGB}{255, 184, 184}
\definecolor{6x5x4color10}{RGB}{255, 191, 191}
\definecolor{6x5x4color11}{RGB}{255, 198, 198}
\definecolor{6x5x4color12}{RGB}{255, 205, 205}
\definecolor{6x5x4color13}{RGB}{255, 212, 212}
\definecolor{6x5x4color14}{RGB}{255, 219, 219}
\definecolor{6x5x4color15}{RGB}{255, 226, 226}
\definecolor{6x5x4color16}{RGB}{255, 233, 233}
\definecolor{6x5x4color17}{RGB}{255, 240, 240}
\definecolor{6x5x4color18}{RGB}{255, 255, 255}

\begin{center}
\begin{tikzpicture}[scale=0.5]
\begin{scope}[xshift=0cm,yshift=0cm]
\draw[very thick] (0,0) rectangle (6,-5);
\foreach \b in {0,...,4}
 \foreach \c in {0,...,5}
  \draw (\c,\b*-1) rectangle (\c+1,(\b*-1-1);

\draw[fill=6x5x4color16] (0,-0) rectangle (1,-1);
\node at (0.5,-0.5) {16};
\draw[fill=6x5x4color15] (1,-0) rectangle (2,-1);
\node at (1.5,-0.5) {15};
\draw[fill=red] (2,-0) rectangle (3,-1);
\draw[fill=6x5x4color1] (3,-0) rectangle (4,-1);
\node at (3.5,-0.5) {1};
\draw[fill=red] (4,-0) rectangle (5,-1);
\draw[fill=6x5x4color13] (5,-0) rectangle (6,-1);
\node at (5.5,-0.5) {13};
\draw[fill=6x5x4color15] (0,-1) rectangle (1,-2);
\node at (0.5,-1.5) {15};
\draw[fill=red] (1,-1) rectangle (2,-2);
\draw[fill=6x5x4color1] (2,-1) rectangle (3,-2);
\node at (2.5,-1.5) {1};
\draw[fill=red] (3,-1) rectangle (4,-2);
\draw[fill=6x5x4color1] (4,-1) rectangle (5,-2);
\node at (4.5,-1.5) {1};
\draw[fill=red] (5,-1) rectangle (6,-2);
\draw[fill=6x5x4color14] (0,-2) rectangle (1,-3);
\node at (0.5,-2.5) {14};
\draw[fill=6x5x4color11] (1,-2) rectangle (2,-3);
\node at (1.5,-2.5) {11};
\draw[fill=6x5x4color8] (2,-2) rectangle (3,-3);
\node at (2.5,-2.5) {8};
\draw[fill=6x5x4color1] (3,-2) rectangle (4,-3);
\node at (3.5,-2.5) {1};
\draw[fill=6x5x4color2] (4,-2) rectangle (5,-3);
\node at (4.5,-2.5) {2};
\draw[fill=6x5x4color3] (5,-2) rectangle (6,-3);
\node at (5.5,-2.5) {3};
\draw[fill=6x5x4color13] (0,-3) rectangle (1,-4);
\node at (0.5,-3.5) {13};
\draw[fill=6x5x4color12] (1,-3) rectangle (2,-4);
\node at (1.5,-3.5) {12};
\draw[fill=6x5x4color9] (2,-3) rectangle (3,-4);
\node at (2.5,-3.5) {9};
\draw[fill=red] (3,-3) rectangle (4,-4);
\draw[fill=6x5x4color3] (4,-3) rectangle (5,-4);
\node at (4.5,-3.5) {3};
\draw[fill=6x5x4color16] (5,-3) rectangle (6,-4);
\node at (5.5,-3.5) {16};
\draw[fill=red] (0,-4) rectangle (1,-5);
\draw[fill=6x5x4color13] (1,-4) rectangle (2,-5);
\node at (1.5,-4.5) {13};
\draw[fill=6x5x4color14] (2,-4) rectangle (3,-5);
\node at (2.5,-4.5) {14};
\draw[fill=6x5x4color15] (3,-4) rectangle (4,-5);
\node at (3.5,-4.5) {15};
\draw[fill=6x5x4color16] (4,-4) rectangle (5,-5);
\node at (4.5,-4.5) {16};
\draw[fill=6x5x4color17] (5,-4) rectangle (6,-5);
\node at (5.5,-4.5) {17};
\end{scope}

\begin{scope}[xshift=7cm,yshift=0cm]
\draw[very thick] (0,0) rectangle (6,-5);
\foreach \b in {0,...,4}
 \foreach \c in {0,...,5}
  \draw (\c,\b*-1) rectangle (\c+1,(\b*-1-1);

\draw[fill=6x5x4color15] (0,-0) rectangle (1,-1);
\node at (0.5,-0.5) {15};
\draw[fill=6x5x4color14] (1,-0) rectangle (2,-1);
\node at (1.5,-0.5) {14};
\draw[fill=6x5x4color7] (2,-0) rectangle (3,-1);
\node at (2.5,-0.5) {7};
\draw[fill=6x5x4color8] (3,-0) rectangle (4,-1);
\node at (3.5,-0.5) {8};
\draw[fill=6x5x4color9] (4,-0) rectangle (5,-1);
\node at (4.5,-0.5) {9};
\draw[fill=6x5x4color12] (5,-0) rectangle (6,-1);
\node at (5.5,-0.5) {12};
\draw[fill=6x5x4color14] (0,-1) rectangle (1,-2);
\node at (0.5,-1.5) {14};
\draw[fill=6x5x4color11] (1,-1) rectangle (2,-2);
\node at (1.5,-1.5) {11};
\draw[fill=6x5x4color6] (2,-1) rectangle (3,-2);
\node at (2.5,-1.5) {6};
\draw[fill=6x5x4color3] (3,-1) rectangle (4,-2);
\node at (3.5,-1.5) {3};
\draw[fill=6x5x4color2] (4,-1) rectangle (5,-2);
\node at (4.5,-1.5) {2};
\draw[fill=6x5x4color1] (5,-1) rectangle (6,-2);
\node at (5.5,-1.5) {1};
\draw[fill=6x5x4color11] (0,-2) rectangle (1,-3);
\node at (0.5,-2.5) {11};
\draw[fill=6x5x4color10] (1,-2) rectangle (2,-3);
\node at (1.5,-2.5) {10};
\draw[fill=6x5x4color7] (2,-2) rectangle (3,-3);
\node at (2.5,-2.5) {7};
\draw[fill=red] (3,-2) rectangle (4,-3);
\draw[fill=6x5x4color1] (4,-2) rectangle (5,-3);
\node at (4.5,-2.5) {1};
\draw[fill=red] (5,-2) rectangle (6,-3);
\draw[fill=6x5x4color10] (0,-3) rectangle (1,-4);
\node at (0.5,-3.5) {10};
\draw[fill=6x5x4color9] (1,-3) rectangle (2,-4);
\node at (1.5,-3.5) {9};
\draw[fill=6x5x4color8] (2,-3) rectangle (3,-4);
\node at (2.5,-3.5) {8};
\draw[fill=6x5x4color1] (3,-3) rectangle (4,-4);
\node at (3.5,-3.5) {1};
\draw[fill=6x5x4color2] (4,-3) rectangle (5,-4);
\node at (4.5,-3.5) {2};
\draw[fill=6x5x4color15] (5,-3) rectangle (6,-4);
\node at (5.5,-3.5) {15};
\draw[fill=6x5x4color1] (0,-4) rectangle (1,-5);
\node at (0.5,-4.5) {1};
\draw[fill=red] (1,-4) rectangle (2,-5);
\draw[fill=6x5x4color9] (2,-4) rectangle (3,-5);
\node at (2.5,-4.5) {9};
\draw[fill=6x5x4color10] (3,-4) rectangle (4,-5);
\node at (3.5,-4.5) {10};
\draw[fill=6x5x4color11] (4,-4) rectangle (5,-5);
\node at (4.5,-4.5) {11};
\draw[fill=6x5x4color16] (5,-4) rectangle (6,-5);
\node at (5.5,-4.5) {16};
\end{scope}

\begin{scope}[xshift=14cm,yshift=0cm]
\draw[very thick] (0,0) rectangle (6,-5);
\foreach \b in {0,...,4}
 \foreach \c in {0,...,5}
  \draw (\c,\b*-1) rectangle (\c+1,(\b*-1-1);

\draw[fill=red] (0,-0) rectangle (1,-1);
\draw[fill=6x5x4color13] (1,-0) rectangle (2,-1);
\node at (1.5,-0.5) {13};
\draw[fill=red] (2,-0) rectangle (3,-1);
\draw[fill=6x5x4color9] (3,-0) rectangle (4,-1);
\node at (3.5,-0.5) {9};
\draw[fill=6x5x4color10] (4,-0) rectangle (5,-1);
\node at (4.5,-0.5) {10};
\draw[fill=6x5x4color11] (5,-0) rectangle (6,-1);
\node at (5.5,-0.5) {11};
\draw[fill=6x5x4color13] (0,-1) rectangle (1,-2);
\node at (0.5,-1.5) {13};
\draw[fill=6x5x4color12] (1,-1) rectangle (2,-2);
\node at (1.5,-1.5) {12};
\draw[fill=6x5x4color5] (2,-1) rectangle (3,-2);
\node at (2.5,-1.5) {5};
\draw[fill=6x5x4color4] (3,-1) rectangle (4,-2);
\node at (3.5,-1.5) {4};
\draw[fill=6x5x4color3] (4,-1) rectangle (5,-2);
\node at (4.5,-1.5) {3};
\draw[fill=red] (5,-1) rectangle (6,-2);
\draw[fill=red] (0,-2) rectangle (1,-3);
\draw[fill=6x5x4color7] (1,-2) rectangle (2,-3);
\node at (1.5,-2.5) {7};
\draw[fill=6x5x4color6] (2,-2) rectangle (3,-3);
\node at (2.5,-2.5) {6};
\draw[fill=6x5x4color1] (3,-2) rectangle (4,-3);
\node at (3.5,-2.5) {1};
\draw[fill=red] (4,-2) rectangle (5,-3);
\draw[fill=6x5x4color1] (5,-2) rectangle (6,-3);
\node at (5.5,-2.5) {1};
\draw[fill=6x5x4color1] (0,-3) rectangle (1,-4);
\node at (0.5,-3.5) {1};
\draw[fill=6x5x4color2] (1,-3) rectangle (2,-4);
\node at (1.5,-3.5) {2};
\draw[fill=6x5x4color1] (2,-3) rectangle (3,-4);
\node at (2.5,-3.5) {1};
\draw[fill=red] (3,-3) rectangle (4,-4);
\draw[fill=6x5x4color1] (4,-3) rectangle (5,-4);
\node at (4.5,-3.5) {1};
\draw[fill=6x5x4color14] (5,-3) rectangle (6,-4);
\node at (5.5,-3.5) {14};
\draw[fill=red] (0,-4) rectangle (1,-5);
\draw[fill=6x5x4color1] (1,-4) rectangle (2,-5);
\node at (1.5,-4.5) {1};
\draw[fill=red] (2,-4) rectangle (3,-5);
\draw[fill=6x5x4color1] (3,-4) rectangle (4,-5);
\node at (3.5,-4.5) {1};
\draw[fill=6x5x4color2] (4,-4) rectangle (5,-5);
\node at (4.5,-4.5) {2};
\draw[fill=6x5x4color15] (5,-4) rectangle (6,-5);
\node at (5.5,-4.5) {15};
\end{scope}

\begin{scope}[xshift=21cm,yshift=0cm]
\draw[very thick] (0,0) rectangle (6,-5);
\foreach \b in {0,...,4}
 \foreach \c in {0,...,5}
  \draw (\c,\b*-1) rectangle (\c+1,(\b*-1-1);

\draw[fill=6x5x4color15] (0,-0) rectangle (1,-1);
\node at (0.5,-0.5) {15};
\draw[fill=6x5x4color14] (1,-0) rectangle (2,-1);
\node at (1.5,-0.5) {14};
\draw[fill=6x5x4color13] (2,-0) rectangle (3,-1);
\node at (2.5,-0.5) {13};
\draw[fill=6x5x4color12] (3,-0) rectangle (4,-1);
\node at (3.5,-0.5) {12};
\draw[fill=6x5x4color11] (4,-0) rectangle (5,-1);
\node at (4.5,-0.5) {11};
\draw[fill=red] (5,-0) rectangle (6,-1);
\draw[fill=6x5x4color14] (0,-1) rectangle (1,-2);
\node at (0.5,-1.5) {14};
\draw[fill=6x5x4color13] (1,-1) rectangle (2,-2);
\node at (1.5,-1.5) {13};
\draw[fill=red] (2,-1) rectangle (3,-2);
\draw[fill=6x5x4color9] (3,-1) rectangle (4,-2);
\node at (3.5,-1.5) {9};
\draw[fill=6x5x4color10] (4,-1) rectangle (5,-2);
\node at (4.5,-1.5) {10};
\draw[fill=6x5x4color11] (5,-1) rectangle (6,-2);
\node at (5.5,-1.5) {11};
\draw[fill=6x5x4color9] (0,-2) rectangle (1,-3);
\node at (0.5,-2.5) {9};
\draw[fill=6x5x4color8] (1,-2) rectangle (2,-3);
\node at (1.5,-2.5) {8};
\draw[fill=6x5x4color7] (2,-2) rectangle (3,-3);
\node at (2.5,-2.5) {7};
\draw[fill=6x5x4color8] (3,-2) rectangle (4,-3);
\node at (3.5,-2.5) {8};
\draw[fill=6x5x4color9] (4,-2) rectangle (5,-3);
\node at (4.5,-2.5) {9};
\draw[fill=6x5x4color12] (5,-2) rectangle (6,-3);
\node at (5.5,-2.5) {12};
\draw[fill=red] (0,-3) rectangle (1,-4);
\draw[fill=6x5x4color3] (1,-3) rectangle (2,-4);
\node at (1.5,-3.5) {3};
\draw[fill=red] (2,-3) rectangle (3,-4);
\draw[fill=6x5x4color1] (3,-3) rectangle (4,-4);
\node at (3.5,-3.5) {1};
\draw[fill=red] (4,-3) rectangle (5,-4);
\draw[fill=6x5x4color13] (5,-3) rectangle (6,-4);
\node at (5.5,-3.5) {13};
\draw[fill=6x5x4color5] (0,-4) rectangle (1,-5);
\node at (0.5,-4.5) {5};
\draw[fill=6x5x4color4] (1,-4) rectangle (2,-5);
\node at (1.5,-4.5) {4};
\draw[fill=6x5x4color1] (2,-4) rectangle (3,-5);
\node at (2.5,-4.5) {1};
\draw[fill=red] (3,-4) rectangle (4,-5);
\draw[fill=6x5x4color1] (4,-4) rectangle (5,-5);
\node at (4.5,-4.5) {1};
\draw[fill=red] (5,-4) rectangle (6,-5);
\end{scope}

\end{tikzpicture}
\end{center}

%% file: 4-5-7-tikz.txt
\definecolor{7x5x4color1}{RGB}{255, 128, 128}
\definecolor{7x5x4color2}{RGB}{255, 135, 135}
\definecolor{7x5x4color3}{RGB}{255, 142, 142}
\definecolor{7x5x4color4}{RGB}{255, 150, 150}
\definecolor{7x5x4color5}{RGB}{255, 157, 157}
\definecolor{7x5x4color6}{RGB}{255, 165, 165}
\definecolor{7x5x4color7}{RGB}{255, 172, 172}
\definecolor{7x5x4color8}{RGB}{255, 180, 180}
\definecolor{7x5x4color9}{RGB}{255, 187, 187}
\definecolor{7x5x4color10}{RGB}{255, 195, 195}
\definecolor{7x5x4color11}{RGB}{255, 202, 202}
\definecolor{7x5x4color12}{RGB}{255, 210, 210}
\definecolor{7x5x4color13}{RGB}{255, 217, 217}
\definecolor{7x5x4color14}{RGB}{255, 225, 225}
\definecolor{7x5x4color15}{RGB}{255, 232, 232}
\definecolor{7x5x4color16}{RGB}{255, 240, 240}
\definecolor{7x5x4color17}{RGB}{255, 255, 255}

\begin{center}
\begin{tikzpicture}[scale=0.5]
\begin{scope}[xshift=0cm,yshift=0cm]
\draw[very thick] (0,0) rectangle (7,-5);
\foreach \b in {0,...,4}
 \foreach \c in {0,...,6}
  \draw (\c,\b*-1) rectangle (\c+1,(\b*-1-1);

\draw[fill=red] (0,-0) rectangle (1,-1);
\draw[fill=7x5x4color1] (1,-0) rectangle (2,-1);
\node at (1.5,-0.5) {1};
\draw[fill=red] (2,-0) rectangle (3,-1);
\draw[fill=7x5x4color1] (3,-0) rectangle (4,-1);
\node at (3.5,-0.5) {1};
\draw[fill=red] (4,-0) rectangle (5,-1);
\draw[fill=7x5x4color1] (5,-0) rectangle (6,-1);
\node at (5.5,-0.5) {1};
\draw[fill=red] (6,-0) rectangle (7,-1);
\draw[fill=7x5x4color8] (0,-1) rectangle (1,-2);
\node at (0.5,-1.5) {8};
\draw[fill=7x5x4color7] (1,-1) rectangle (2,-2);
\node at (1.5,-1.5) {7};
\draw[fill=7x5x4color6] (2,-1) rectangle (3,-2);
\node at (2.5,-1.5) {6};
\draw[fill=red] (3,-1) rectangle (4,-2);
\draw[fill=7x5x4color1] (4,-1) rectangle (5,-2);
\node at (4.5,-1.5) {1};
\draw[fill=red] (5,-1) rectangle (6,-2);
\draw[fill=7x5x4color1] (6,-1) rectangle (7,-2);
\node at (6.5,-1.5) {1};
\draw[fill=red] (0,-2) rectangle (1,-3);
\draw[fill=7x5x4color8] (1,-2) rectangle (2,-3);
\node at (1.5,-2.5) {8};
\draw[fill=7x5x4color9] (2,-2) rectangle (3,-3);
\node at (2.5,-2.5) {9};
\draw[fill=7x5x4color10] (3,-2) rectangle (4,-3);
\node at (3.5,-2.5) {10};
\draw[fill=7x5x4color11] (4,-2) rectangle (5,-3);
\node at (4.5,-2.5) {11};
\draw[fill=7x5x4color12] (5,-2) rectangle (6,-3);
\node at (5.5,-2.5) {12};
\draw[fill=7x5x4color13] (6,-2) rectangle (7,-3);
\node at (6.5,-2.5) {13};
\draw[fill=7x5x4color9] (0,-3) rectangle (1,-4);
\node at (0.5,-3.5) {9};
\draw[fill=7x5x4color10] (1,-3) rectangle (2,-4);
\node at (1.5,-3.5) {10};
\draw[fill=7x5x4color11] (2,-3) rectangle (3,-4);
\node at (2.5,-3.5) {11};
\draw[fill=7x5x4color12] (3,-3) rectangle (4,-4);
\node at (3.5,-3.5) {12};
\draw[fill=7x5x4color13] (4,-3) rectangle (5,-4);
\node at (4.5,-3.5) {13};
\draw[fill=7x5x4color14] (5,-3) rectangle (6,-4);
\node at (5.5,-3.5) {14};
\draw[fill=7x5x4color15] (6,-3) rectangle (7,-4);
\node at (6.5,-3.5) {15};
\draw[fill=red] (0,-4) rectangle (1,-5);
\draw[fill=7x5x4color11] (1,-4) rectangle (2,-5);
\node at (1.5,-4.5) {11};
\draw[fill=7x5x4color12] (2,-4) rectangle (3,-5);
\node at (2.5,-4.5) {12};
\draw[fill=7x5x4color13] (3,-4) rectangle (4,-5);
\node at (3.5,-4.5) {13};
\draw[fill=7x5x4color14] (4,-4) rectangle (5,-5);
\node at (4.5,-4.5) {14};
\draw[fill=7x5x4color15] (5,-4) rectangle (6,-5);
\node at (5.5,-4.5) {15};
\draw[fill=7x5x4color16] (6,-4) rectangle (7,-5);
\node at (6.5,-4.5) {16};
\end{scope}

\begin{scope}[xshift=8cm,yshift=0cm]
\draw[very thick] (0,0) rectangle (7,-5);
\foreach \b in {0,...,4}
 \foreach \c in {0,...,6}
  \draw (\c,\b*-1) rectangle (\c+1,(\b*-1-1);

\draw[fill=7x5x4color11] (0,-0) rectangle (1,-1);
\node at (0.5,-0.5) {11};
\draw[fill=red] (1,-0) rectangle (2,-1);
\draw[fill=7x5x4color4] (2,-0) rectangle (3,-1);
\node at (2.5,-0.5) {4};
\draw[fill=7x5x4color3] (3,-0) rectangle (4,-1);
\node at (3.5,-0.5) {3};
\draw[fill=7x5x4color4] (4,-0) rectangle (5,-1);
\node at (4.5,-0.5) {4};
\draw[fill=7x5x4color5] (5,-0) rectangle (6,-1);
\node at (5.5,-0.5) {5};
\draw[fill=7x5x4color6] (6,-0) rectangle (7,-1);
\node at (6.5,-0.5) {6};
\draw[fill=7x5x4color10] (0,-1) rectangle (1,-2);
\node at (0.5,-1.5) {10};
\draw[fill=7x5x4color6] (1,-1) rectangle (2,-2);
\node at (1.5,-1.5) {6};
\draw[fill=7x5x4color5] (2,-1) rectangle (3,-2);
\node at (2.5,-1.5) {5};
\draw[fill=7x5x4color2] (3,-1) rectangle (4,-2);
\node at (3.5,-1.5) {2};
\draw[fill=7x5x4color3] (4,-1) rectangle (5,-2);
\node at (4.5,-1.5) {3};
\draw[fill=7x5x4color4] (5,-1) rectangle (6,-2);
\node at (5.5,-1.5) {4};
\draw[fill=red] (6,-1) rectangle (7,-2);
\draw[fill=7x5x4color9] (0,-2) rectangle (1,-3);
\node at (0.5,-2.5) {9};
\draw[fill=red] (1,-2) rectangle (2,-3);
\draw[fill=7x5x4color1] (2,-2) rectangle (3,-3);
\node at (2.5,-2.5) {1};
\draw[fill=red] (3,-2) rectangle (4,-3);
\draw[fill=7x5x4color4] (4,-2) rectangle (5,-3);
\node at (4.5,-2.5) {4};
\draw[fill=7x5x4color5] (5,-2) rectangle (6,-3);
\node at (5.5,-2.5) {5};
\draw[fill=7x5x4color10] (6,-2) rectangle (7,-3);
\node at (6.5,-2.5) {10};
\draw[fill=7x5x4color8] (0,-3) rectangle (1,-4);
\node at (0.5,-3.5) {8};
\draw[fill=7x5x4color5] (1,-3) rectangle (2,-4);
\node at (1.5,-3.5) {5};
\draw[fill=red] (2,-3) rectangle (3,-4);
\draw[fill=7x5x4color1] (3,-3) rectangle (4,-4);
\node at (3.5,-3.5) {1};
\draw[fill=7x5x4color6] (4,-3) rectangle (5,-4);
\node at (4.5,-3.5) {6};
\draw[fill=7x5x4color7] (5,-3) rectangle (6,-4);
\node at (5.5,-3.5) {7};
\draw[fill=7x5x4color11] (6,-3) rectangle (7,-4);
\node at (6.5,-3.5) {11};
\draw[fill=7x5x4color7] (0,-4) rectangle (1,-5);
\node at (0.5,-4.5) {7};
\draw[fill=7x5x4color6] (1,-4) rectangle (2,-5);
\node at (1.5,-4.5) {6};
\draw[fill=7x5x4color5] (2,-4) rectangle (3,-5);
\node at (2.5,-4.5) {5};
\draw[fill=red] (3,-4) rectangle (4,-5);
\draw[fill=7x5x4color7] (4,-4) rectangle (5,-5);
\node at (4.5,-4.5) {7};
\draw[fill=7x5x4color8] (5,-4) rectangle (6,-5);
\node at (5.5,-4.5) {8};
\draw[fill=7x5x4color12] (6,-4) rectangle (7,-5);
\node at (6.5,-4.5) {12};
\end{scope}

\begin{scope}[xshift=16cm,yshift=0cm]
\draw[very thick] (0,0) rectangle (7,-5);
\foreach \b in {0,...,4}
 \foreach \c in {0,...,6}
  \draw (\c,\b*-1) rectangle (\c+1,(\b*-1-1);

\draw[fill=7x5x4color12] (0,-0) rectangle (1,-1);
\node at (0.5,-0.5) {12};
\draw[fill=7x5x4color8] (1,-0) rectangle (2,-1);
\node at (1.5,-0.5) {8};
\draw[fill=7x5x4color7] (2,-0) rectangle (3,-1);
\node at (2.5,-0.5) {7};
\draw[fill=red] (3,-0) rectangle (4,-1);
\draw[fill=7x5x4color5] (4,-0) rectangle (5,-1);
\node at (4.5,-0.5) {5};
\draw[fill=7x5x4color6] (5,-0) rectangle (6,-1);
\node at (5.5,-0.5) {6};
\draw[fill=7x5x4color7] (6,-0) rectangle (7,-1);
\node at (6.5,-0.5) {7};
\draw[fill=7x5x4color11] (0,-1) rectangle (1,-2);
\node at (0.5,-1.5) {11};
\draw[fill=7x5x4color7] (1,-1) rectangle (2,-2);
\node at (1.5,-1.5) {7};
\draw[fill=7x5x4color6] (2,-1) rectangle (3,-2);
\node at (2.5,-1.5) {6};
\draw[fill=7x5x4color1] (3,-1) rectangle (4,-2);
\node at (3.5,-1.5) {1};
\draw[fill=red] (4,-1) rectangle (5,-2);
\draw[fill=7x5x4color5] (5,-1) rectangle (6,-2);
\node at (5.5,-1.5) {5};
\draw[fill=7x5x4color8] (6,-1) rectangle (7,-2);
\node at (6.5,-1.5) {8};
\draw[fill=7x5x4color10] (0,-2) rectangle (1,-3);
\node at (0.5,-2.5) {10};
\draw[fill=7x5x4color5] (1,-2) rectangle (2,-3);
\node at (1.5,-2.5) {5};
\draw[fill=7x5x4color4] (2,-2) rectangle (3,-3);
\node at (2.5,-2.5) {4};
\draw[fill=red] (3,-2) rectangle (4,-3);
\draw[fill=7x5x4color1] (4,-2) rectangle (5,-3);
\node at (4.5,-2.5) {1};
\draw[fill=red] (5,-2) rectangle (6,-3);
\draw[fill=7x5x4color9] (6,-2) rectangle (7,-3);
\node at (6.5,-2.5) {9};
\draw[fill=red] (0,-3) rectangle (1,-4);
\draw[fill=7x5x4color4] (1,-3) rectangle (2,-4);
\node at (1.5,-3.5) {4};
\draw[fill=7x5x4color3] (2,-3) rectangle (3,-4);
\node at (2.5,-3.5) {3};
\draw[fill=7x5x4color2] (3,-3) rectangle (4,-4);
\node at (3.5,-3.5) {2};
\draw[fill=7x5x4color5] (4,-3) rectangle (5,-4);
\node at (4.5,-3.5) {5};
\draw[fill=7x5x4color6] (5,-3) rectangle (6,-4);
\node at (5.5,-3.5) {6};
\draw[fill=7x5x4color10] (6,-3) rectangle (7,-4);
\node at (6.5,-3.5) {10};
\draw[fill=7x5x4color6] (0,-4) rectangle (1,-5);
\node at (0.5,-4.5) {6};
\draw[fill=7x5x4color5] (1,-4) rectangle (2,-5);
\node at (1.5,-4.5) {5};
\draw[fill=7x5x4color4] (2,-4) rectangle (3,-5);
\node at (2.5,-4.5) {4};
\draw[fill=7x5x4color3] (3,-4) rectangle (4,-5);
\node at (3.5,-4.5) {3};
\draw[fill=7x5x4color4] (4,-4) rectangle (5,-5);
\node at (4.5,-4.5) {4};
\draw[fill=red] (5,-4) rectangle (6,-5);
\draw[fill=7x5x4color11] (6,-4) rectangle (7,-5);
\node at (6.5,-4.5) {11};
\end{scope}

\begin{scope}[xshift=24cm,yshift=0cm]
\draw[very thick] (0,0) rectangle (7,-5);
\foreach \b in {0,...,4}
 \foreach \c in {0,...,6}
  \draw (\c,\b*-1) rectangle (\c+1,(\b*-1-1);

\draw[fill=7x5x4color16] (0,-0) rectangle (1,-1);
\node at (0.5,-0.5) {16};
\draw[fill=7x5x4color15] (1,-0) rectangle (2,-1);
\node at (1.5,-0.5) {15};
\draw[fill=7x5x4color14] (2,-0) rectangle (3,-1);
\node at (2.5,-0.5) {14};
\draw[fill=7x5x4color13] (3,-0) rectangle (4,-1);
\node at (3.5,-0.5) {13};
\draw[fill=7x5x4color12] (4,-0) rectangle (5,-1);
\node at (4.5,-0.5) {12};
\draw[fill=7x5x4color11] (5,-0) rectangle (6,-1);
\node at (5.5,-0.5) {11};
\draw[fill=red] (6,-0) rectangle (7,-1);
\draw[fill=7x5x4color15] (0,-1) rectangle (1,-2);
\node at (0.5,-1.5) {15};
\draw[fill=7x5x4color14] (1,-1) rectangle (2,-2);
\node at (1.5,-1.5) {14};
\draw[fill=7x5x4color13] (2,-1) rectangle (3,-2);
\node at (2.5,-1.5) {13};
\draw[fill=7x5x4color12] (3,-1) rectangle (4,-2);
\node at (3.5,-1.5) {12};
\draw[fill=7x5x4color11] (4,-1) rectangle (5,-2);
\node at (4.5,-1.5) {11};
\draw[fill=7x5x4color10] (5,-1) rectangle (6,-2);
\node at (5.5,-1.5) {10};
\draw[fill=7x5x4color9] (6,-1) rectangle (7,-2);
\node at (6.5,-1.5) {9};
\draw[fill=7x5x4color13] (0,-2) rectangle (1,-3);
\node at (0.5,-2.5) {13};
\draw[fill=7x5x4color12] (1,-2) rectangle (2,-3);
\node at (1.5,-2.5) {12};
\draw[fill=7x5x4color11] (2,-2) rectangle (3,-3);
\node at (2.5,-2.5) {11};
\draw[fill=7x5x4color10] (3,-2) rectangle (4,-3);
\node at (3.5,-2.5) {10};
\draw[fill=7x5x4color9] (4,-2) rectangle (5,-3);
\node at (4.5,-2.5) {9};
\draw[fill=7x5x4color8] (5,-2) rectangle (6,-3);
\node at (5.5,-2.5) {8};
\draw[fill=red] (6,-2) rectangle (7,-3);
\draw[fill=7x5x4color1] (0,-3) rectangle (1,-4);
\node at (0.5,-3.5) {1};
\draw[fill=red] (1,-3) rectangle (2,-4);
\draw[fill=7x5x4color1] (2,-3) rectangle (3,-4);
\node at (2.5,-3.5) {1};
\draw[fill=red] (3,-3) rectangle (4,-4);
\draw[fill=7x5x4color6] (4,-3) rectangle (5,-4);
\node at (4.5,-3.5) {6};
\draw[fill=7x5x4color7] (5,-3) rectangle (6,-4);
\node at (5.5,-3.5) {7};
\draw[fill=7x5x4color8] (6,-3) rectangle (7,-4);
\node at (6.5,-3.5) {8};
\draw[fill=red] (0,-4) rectangle (1,-5);
\draw[fill=7x5x4color1] (1,-4) rectangle (2,-5);
\node at (1.5,-4.5) {1};
\draw[fill=red] (2,-4) rectangle (3,-5);
\draw[fill=7x5x4color1] (3,-4) rectangle (4,-5);
\node at (3.5,-4.5) {1};
\draw[fill=red] (4,-4) rectangle (5,-5);
\draw[fill=7x5x4color1] (5,-4) rectangle (6,-5);
\node at (5.5,-4.5) {1};
\draw[fill=red] (6,-4) rectangle (7,-5);
\end{scope}

\end{tikzpicture}
\end{center}

%% file: 4-6-7-tikz.txt
\definecolor{7x6x4color1}{RGB}{255, 128, 128}
\definecolor{7x6x4color2}{RGB}{255, 133, 133}
\definecolor{7x6x4color3}{RGB}{255, 138, 138}
\definecolor{7x6x4color4}{RGB}{255, 143, 143}
\definecolor{7x6x4color5}{RGB}{255, 148, 148}
\definecolor{7x6x4color6}{RGB}{255, 153, 153}
\definecolor{7x6x4color7}{RGB}{255, 158, 158}
\definecolor{7x6x4color8}{RGB}{255, 163, 163}
\definecolor{7x6x4color9}{RGB}{255, 168, 168}
\definecolor{7x6x4color10}{RGB}{255, 173, 173}
\definecolor{7x6x4color11}{RGB}{255, 178, 178}
\definecolor{7x6x4color12}{RGB}{255, 183, 183}
\definecolor{7x6x4color13}{RGB}{255, 188, 188}
\definecolor{7x6x4color14}{RGB}{255, 194, 194}
\definecolor{7x6x4color15}{RGB}{255, 199, 199}
\definecolor{7x6x4color16}{RGB}{255, 204, 204}
\definecolor{7x6x4color17}{RGB}{255, 209, 209}
\definecolor{7x6x4color18}{RGB}{255, 214, 214}
\definecolor{7x6x4color19}{RGB}{255, 219, 219}
\definecolor{7x6x4color20}{RGB}{255, 224, 224}
\definecolor{7x6x4color21}{RGB}{255, 229, 229}
\definecolor{7x6x4color22}{RGB}{255, 234, 234}
\definecolor{7x6x4color23}{RGB}{255, 239, 239}
\definecolor{7x6x4color24}{RGB}{255, 244, 244}
\definecolor{7x6x4color25}{RGB}{255, 255, 255}

\begin{center}
\begin{tikzpicture}[scale=0.5]
\begin{scope}[xshift=0cm,yshift=0cm]
\draw[very thick] (0,0) rectangle (7,-6);
\foreach \b in {0,...,5}
 \foreach \c in {0,...,6}
  \draw (\c,\b*-1) rectangle (\c+1,(\b*-1-1);

\draw[fill=red] (0,-0) rectangle (1,-1);
\draw[fill=7x6x4color13] (1,-0) rectangle (2,-1);
\node at (1.5,-0.5) {13};
\draw[fill=7x6x4color14] (2,-0) rectangle (3,-1);
\node at (2.5,-0.5) {14};
\draw[fill=7x6x4color15] (3,-0) rectangle (4,-1);
\node at (3.5,-0.5) {15};
\draw[fill=red] (4,-0) rectangle (5,-1);
\draw[fill=7x6x4color23] (5,-0) rectangle (6,-1);
\node at (5.5,-0.5) {23};
\draw[fill=7x6x4color24] (6,-0) rectangle (7,-1);
\node at (6.5,-0.5) {24};
\draw[fill=7x6x4color1] (0,-1) rectangle (1,-2);
\node at (0.5,-1.5) {1};
\draw[fill=7x6x4color12] (1,-1) rectangle (2,-2);
\node at (1.5,-1.5) {12};
\draw[fill=7x6x4color13] (2,-1) rectangle (3,-2);
\node at (2.5,-1.5) {13};
\draw[fill=7x6x4color14] (3,-1) rectangle (4,-2);
\node at (3.5,-1.5) {14};
\draw[fill=7x6x4color15] (4,-1) rectangle (5,-2);
\node at (4.5,-1.5) {15};
\draw[fill=7x6x4color22] (5,-1) rectangle (6,-2);
\node at (5.5,-1.5) {22};
\draw[fill=7x6x4color23] (6,-1) rectangle (7,-2);
\node at (6.5,-1.5) {23};
\draw[fill=red] (0,-2) rectangle (1,-3);
\draw[fill=7x6x4color11] (1,-2) rectangle (2,-3);
\node at (1.5,-2.5) {11};
\draw[fill=red] (2,-2) rectangle (3,-3);
\draw[fill=7x6x4color9] (3,-2) rectangle (4,-3);
\node at (3.5,-2.5) {9};
\draw[fill=7x6x4color16] (4,-2) rectangle (5,-3);
\node at (4.5,-2.5) {16};
\draw[fill=7x6x4color21] (5,-2) rectangle (6,-3);
\node at (5.5,-2.5) {21};
\draw[fill=7x6x4color22] (6,-2) rectangle (7,-3);
\node at (6.5,-2.5) {22};
\draw[fill=7x6x4color13] (0,-3) rectangle (1,-4);
\node at (0.5,-3.5) {13};
\draw[fill=7x6x4color12] (1,-3) rectangle (2,-4);
\node at (1.5,-3.5) {12};
\draw[fill=7x6x4color11] (2,-3) rectangle (3,-4);
\node at (2.5,-3.5) {11};
\draw[fill=red] (3,-3) rectangle (4,-4);
\draw[fill=7x6x4color17] (4,-3) rectangle (5,-4);
\node at (4.5,-3.5) {17};
\draw[fill=7x6x4color20] (5,-3) rectangle (6,-4);
\node at (5.5,-3.5) {20};
\draw[fill=7x6x4color21] (6,-3) rectangle (7,-4);
\node at (6.5,-3.5) {21};
\draw[fill=7x6x4color16] (0,-4) rectangle (1,-5);
\node at (0.5,-4.5) {16};
\draw[fill=7x6x4color15] (1,-4) rectangle (2,-5);
\node at (1.5,-4.5) {15};
\draw[fill=7x6x4color16] (2,-4) rectangle (3,-5);
\node at (2.5,-4.5) {16};
\draw[fill=7x6x4color17] (3,-4) rectangle (4,-5);
\node at (3.5,-4.5) {17};
\draw[fill=7x6x4color18] (4,-4) rectangle (5,-5);
\node at (4.5,-4.5) {18};
\draw[fill=7x6x4color19] (5,-4) rectangle (6,-5);
\node at (5.5,-4.5) {19};
\draw[fill=red] (6,-4) rectangle (7,-5);
\draw[fill=7x6x4color17] (0,-5) rectangle (1,-6);
\node at (0.5,-5.5) {17};
\draw[fill=red] (1,-5) rectangle (2,-6);
\draw[fill=7x6x4color17] (2,-5) rectangle (3,-6);
\node at (2.5,-5.5) {17};
\draw[fill=7x6x4color18] (3,-5) rectangle (4,-6);
\node at (3.5,-5.5) {18};
\draw[fill=7x6x4color19] (4,-5) rectangle (5,-6);
\node at (4.5,-5.5) {19};
\draw[fill=7x6x4color20] (5,-5) rectangle (6,-6);
\node at (5.5,-5.5) {20};
\draw[fill=7x6x4color21] (6,-5) rectangle (7,-6);
\node at (6.5,-5.5) {21};
\end{scope}

\begin{scope}[xshift=8cm,yshift=0cm]
\draw[very thick] (0,0) rectangle (7,-6);
\foreach \b in {0,...,5}
 \foreach \c in {0,...,6}
  \draw (\c,\b*-1) rectangle (\c+1,(\b*-1-1);

\draw[fill=7x6x4color1] (0,-0) rectangle (1,-1);
\node at (0.5,-0.5) {1};
\draw[fill=red] (1,-0) rectangle (2,-1);
\draw[fill=7x6x4color13] (2,-0) rectangle (3,-1);
\node at (2.5,-0.5) {13};
\draw[fill=7x6x4color16] (3,-0) rectangle (4,-1);
\node at (3.5,-0.5) {16};
\draw[fill=7x6x4color17] (4,-0) rectangle (5,-1);
\node at (4.5,-0.5) {17};
\draw[fill=7x6x4color18] (5,-0) rectangle (6,-1);
\node at (5.5,-0.5) {18};
\draw[fill=7x6x4color19] (6,-0) rectangle (7,-1);
\node at (6.5,-0.5) {19};
\draw[fill=red] (0,-1) rectangle (1,-2);
\draw[fill=7x6x4color11] (1,-1) rectangle (2,-2);
\node at (1.5,-1.5) {11};
\draw[fill=7x6x4color12] (2,-1) rectangle (3,-2);
\node at (2.5,-1.5) {12};
\draw[fill=7x6x4color13] (3,-1) rectangle (4,-2);
\node at (3.5,-1.5) {13};
\draw[fill=7x6x4color14] (4,-1) rectangle (5,-2);
\node at (4.5,-1.5) {14};
\draw[fill=7x6x4color15] (5,-1) rectangle (6,-2);
\node at (5.5,-1.5) {15};
\draw[fill=7x6x4color16] (6,-1) rectangle (7,-2);
\node at (6.5,-1.5) {16};
\draw[fill=7x6x4color1] (0,-2) rectangle (1,-3);
\node at (0.5,-2.5) {1};
\draw[fill=7x6x4color10] (1,-2) rectangle (2,-3);
\node at (1.5,-2.5) {10};
\draw[fill=7x6x4color9] (2,-2) rectangle (3,-3);
\node at (2.5,-2.5) {9};
\draw[fill=7x6x4color8] (3,-2) rectangle (4,-3);
\node at (3.5,-2.5) {8};
\draw[fill=7x6x4color7] (4,-2) rectangle (5,-3);
\node at (4.5,-2.5) {7};
\draw[fill=red] (5,-2) rectangle (6,-3);
\draw[fill=7x6x4color1] (6,-2) rectangle (7,-3);
\node at (6.5,-2.5) {1};
\draw[fill=7x6x4color12] (0,-3) rectangle (1,-4);
\node at (0.5,-3.5) {12};
\draw[fill=7x6x4color11] (1,-3) rectangle (2,-4);
\node at (1.5,-3.5) {11};
\draw[fill=7x6x4color10] (2,-3) rectangle (3,-4);
\node at (2.5,-3.5) {10};
\draw[fill=7x6x4color1] (3,-3) rectangle (4,-4);
\node at (3.5,-3.5) {1};
\draw[fill=7x6x4color2] (4,-3) rectangle (5,-4);
\node at (4.5,-3.5) {2};
\draw[fill=7x6x4color1] (5,-3) rectangle (6,-4);
\node at (5.5,-3.5) {1};
\draw[fill=red] (6,-3) rectangle (7,-4);
\draw[fill=7x6x4color15] (0,-4) rectangle (1,-5);
\node at (0.5,-4.5) {15};
\draw[fill=7x6x4color14] (1,-4) rectangle (2,-5);
\node at (1.5,-4.5) {14};
\draw[fill=7x6x4color13] (2,-4) rectangle (3,-5);
\node at (2.5,-4.5) {13};
\draw[fill=red] (3,-4) rectangle (4,-5);
\draw[fill=7x6x4color1] (4,-4) rectangle (5,-5);
\node at (4.5,-4.5) {1};
\draw[fill=red] (5,-4) rectangle (6,-5);
\draw[fill=7x6x4color1] (6,-4) rectangle (7,-5);
\node at (6.5,-4.5) {1};
\draw[fill=7x6x4color16] (0,-5) rectangle (1,-6);
\node at (0.5,-5.5) {16};
\draw[fill=7x6x4color15] (1,-5) rectangle (2,-6);
\node at (1.5,-5.5) {15};
\draw[fill=7x6x4color16] (2,-5) rectangle (3,-6);
\node at (2.5,-5.5) {16};
\draw[fill=7x6x4color17] (3,-5) rectangle (4,-6);
\node at (3.5,-5.5) {17};
\draw[fill=7x6x4color18] (4,-5) rectangle (5,-6);
\node at (4.5,-5.5) {18};
\draw[fill=7x6x4color19] (5,-5) rectangle (6,-6);
\node at (5.5,-5.5) {19};
\draw[fill=7x6x4color20] (6,-5) rectangle (7,-6);
\node at (6.5,-5.5) {20};
\end{scope}

\begin{scope}[xshift=16cm,yshift=0cm]
\draw[very thick] (0,0) rectangle (7,-6);
\foreach \b in {0,...,5}
 \foreach \c in {0,...,6}
  \draw (\c,\b*-1) rectangle (\c+1,(\b*-1-1);

\draw[fill=7x6x4color14] (0,-0) rectangle (1,-1);
\node at (0.5,-0.5) {14};
\draw[fill=7x6x4color13] (1,-0) rectangle (2,-1);
\node at (1.5,-0.5) {13};
\draw[fill=red] (2,-0) rectangle (3,-1);
\draw[fill=7x6x4color17] (3,-0) rectangle (4,-1);
\node at (3.5,-0.5) {17};
\draw[fill=7x6x4color18] (4,-0) rectangle (5,-1);
\node at (4.5,-0.5) {18};
\draw[fill=7x6x4color13] (5,-0) rectangle (6,-1);
\node at (5.5,-0.5) {13};
\draw[fill=red] (6,-0) rectangle (7,-1);
\draw[fill=7x6x4color13] (0,-1) rectangle (1,-2);
\node at (0.5,-1.5) {13};
\draw[fill=7x6x4color12] (1,-1) rectangle (2,-2);
\node at (1.5,-1.5) {12};
\draw[fill=7x6x4color1] (2,-1) rectangle (3,-2);
\node at (2.5,-1.5) {1};
\draw[fill=7x6x4color2] (3,-1) rectangle (4,-2);
\node at (3.5,-1.5) {2};
\draw[fill=7x6x4color7] (4,-1) rectangle (5,-2);
\node at (4.5,-1.5) {7};
\draw[fill=7x6x4color12] (5,-1) rectangle (6,-2);
\node at (5.5,-1.5) {12};
\draw[fill=7x6x4color13] (6,-1) rectangle (7,-2);
\node at (6.5,-1.5) {13};
\draw[fill=red] (0,-2) rectangle (1,-3);
\draw[fill=7x6x4color1] (1,-2) rectangle (2,-3);
\node at (1.5,-2.5) {1};
\draw[fill=red] (2,-2) rectangle (3,-3);
\draw[fill=7x6x4color1] (3,-2) rectangle (4,-3);
\node at (3.5,-2.5) {1};
\draw[fill=7x6x4color6] (4,-2) rectangle (5,-3);
\node at (4.5,-2.5) {6};
\draw[fill=7x6x4color5] (5,-2) rectangle (6,-3);
\node at (5.5,-2.5) {5};
\draw[fill=red] (6,-2) rectangle (7,-3);
\draw[fill=7x6x4color1] (0,-3) rectangle (1,-4);
\node at (0.5,-3.5) {1};
\draw[fill=red] (1,-3) rectangle (2,-4);
\draw[fill=7x6x4color1] (2,-3) rectangle (3,-4);
\node at (2.5,-3.5) {1};
\draw[fill=red] (3,-3) rectangle (4,-4);
\draw[fill=7x6x4color3] (4,-3) rectangle (5,-4);
\node at (4.5,-3.5) {3};
\draw[fill=7x6x4color4] (5,-3) rectangle (6,-4);
\node at (5.5,-3.5) {4};
\draw[fill=7x6x4color5] (6,-3) rectangle (7,-4);
\node at (6.5,-3.5) {5};
\draw[fill=7x6x4color14] (0,-4) rectangle (1,-5);
\node at (0.5,-4.5) {14};
\draw[fill=7x6x4color13] (1,-4) rectangle (2,-5);
\node at (1.5,-4.5) {13};
\draw[fill=7x6x4color12] (2,-4) rectangle (3,-5);
\node at (2.5,-4.5) {12};
\draw[fill=7x6x4color1] (3,-4) rectangle (4,-5);
\node at (3.5,-4.5) {1};
\draw[fill=red] (4,-4) rectangle (5,-5);
\draw[fill=7x6x4color1] (5,-4) rectangle (6,-5);
\node at (5.5,-4.5) {1};
\draw[fill=7x6x4color6] (6,-4) rectangle (7,-5);
\node at (6.5,-4.5) {6};
\draw[fill=7x6x4color15] (0,-5) rectangle (1,-6);
\node at (0.5,-5.5) {15};
\draw[fill=red] (1,-5) rectangle (2,-6);
\draw[fill=7x6x4color13] (2,-5) rectangle (3,-6);
\node at (2.5,-5.5) {13};
\draw[fill=7x6x4color14] (3,-5) rectangle (4,-6);
\node at (3.5,-5.5) {14};
\draw[fill=7x6x4color15] (4,-5) rectangle (5,-6);
\node at (4.5,-5.5) {15};
\draw[fill=7x6x4color16] (5,-5) rectangle (6,-6);
\node at (5.5,-5.5) {16};
\draw[fill=7x6x4color17] (6,-5) rectangle (7,-6);
\node at (6.5,-5.5) {17};
\end{scope}

\begin{scope}[xshift=24cm,yshift=0cm]
\draw[very thick] (0,0) rectangle (7,-6);
\foreach \b in {0,...,5}
 \foreach \c in {0,...,6}
  \draw (\c,\b*-1) rectangle (\c+1,(\b*-1-1);

\draw[fill=7x6x4color23] (0,-0) rectangle (1,-1);
\node at (0.5,-0.5) {23};
\draw[fill=7x6x4color22] (1,-0) rectangle (2,-1);
\node at (1.5,-0.5) {22};
\draw[fill=7x6x4color21] (2,-0) rectangle (3,-1);
\node at (2.5,-0.5) {21};
\draw[fill=7x6x4color20] (3,-0) rectangle (4,-1);
\node at (3.5,-0.5) {20};
\draw[fill=7x6x4color19] (4,-0) rectangle (5,-1);
\node at (4.5,-0.5) {19};
\draw[fill=red] (5,-0) rectangle (6,-1);
\draw[fill=7x6x4color15] (6,-0) rectangle (7,-1);
\node at (6.5,-0.5) {15};
\draw[fill=7x6x4color14] (0,-1) rectangle (1,-2);
\node at (0.5,-1.5) {14};
\draw[fill=7x6x4color13] (1,-1) rectangle (2,-2);
\node at (1.5,-1.5) {13};
\draw[fill=red] (2,-1) rectangle (3,-2);
\draw[fill=7x6x4color1] (3,-1) rectangle (4,-2);
\node at (3.5,-1.5) {1};
\draw[fill=red] (4,-1) rectangle (5,-2);
\draw[fill=7x6x4color11] (5,-1) rectangle (6,-2);
\node at (5.5,-1.5) {11};
\draw[fill=7x6x4color14] (6,-1) rectangle (7,-2);
\node at (6.5,-1.5) {14};
\draw[fill=7x6x4color13] (0,-2) rectangle (1,-3);
\node at (0.5,-2.5) {13};
\draw[fill=7x6x4color12] (1,-2) rectangle (2,-3);
\node at (1.5,-2.5) {12};
\draw[fill=7x6x4color1] (2,-2) rectangle (3,-3);
\node at (2.5,-2.5) {1};
\draw[fill=red] (3,-2) rectangle (4,-3);
\draw[fill=7x6x4color7] (4,-2) rectangle (5,-3);
\node at (4.5,-2.5) {7};
\draw[fill=7x6x4color10] (5,-2) rectangle (6,-3);
\node at (5.5,-2.5) {10};
\draw[fill=7x6x4color11] (6,-2) rectangle (7,-3);
\node at (6.5,-2.5) {11};
\draw[fill=red] (0,-3) rectangle (1,-4);
\draw[fill=7x6x4color11] (1,-3) rectangle (2,-4);
\node at (1.5,-3.5) {11};
\draw[fill=7x6x4color10] (2,-3) rectangle (3,-4);
\node at (2.5,-3.5) {10};
\draw[fill=7x6x4color9] (3,-3) rectangle (4,-4);
\node at (3.5,-3.5) {9};
\draw[fill=7x6x4color8] (4,-3) rectangle (5,-4);
\node at (4.5,-3.5) {8};
\draw[fill=7x6x4color9] (5,-3) rectangle (6,-4);
\node at (5.5,-3.5) {9};
\draw[fill=7x6x4color10] (6,-3) rectangle (7,-4);
\node at (6.5,-3.5) {10};
\draw[fill=7x6x4color13] (0,-4) rectangle (1,-5);
\node at (0.5,-4.5) {13};
\draw[fill=7x6x4color12] (1,-4) rectangle (2,-5);
\node at (1.5,-4.5) {12};
\draw[fill=7x6x4color11] (2,-4) rectangle (3,-5);
\node at (2.5,-4.5) {11};
\draw[fill=7x6x4color10] (3,-4) rectangle (4,-5);
\node at (3.5,-4.5) {10};
\draw[fill=7x6x4color1] (4,-4) rectangle (5,-5);
\node at (4.5,-4.5) {1};
\draw[fill=red] (5,-4) rectangle (6,-5);
\draw[fill=7x6x4color7] (6,-4) rectangle (7,-5);
\node at (6.5,-4.5) {7};
\draw[fill=red] (0,-5) rectangle (1,-6);
\draw[fill=7x6x4color1] (1,-5) rectangle (2,-6);
\node at (1.5,-5.5) {1};
\draw[fill=red] (2,-5) rectangle (3,-6);
\draw[fill=7x6x4color11] (3,-5) rectangle (4,-6);
\node at (3.5,-5.5) {11};
\draw[fill=red] (4,-5) rectangle (5,-6);
\draw[fill=7x6x4color1] (5,-5) rectangle (6,-6);
\node at (5.5,-5.5) {1};
\draw[fill=red] (6,-5) rectangle (7,-6);
\end{scope}

\end{tikzpicture}
\end{center}

%% file: 4-4-8-tikz.txt
\definecolor{8x4x4color1}{RGB}{255, 128, 128}
\definecolor{8x4x4color2}{RGB}{255, 132, 132}
\definecolor{8x4x4color3}{RGB}{255, 137, 137}
\definecolor{8x4x4color4}{RGB}{255, 142, 142}
\definecolor{8x4x4color5}{RGB}{255, 146, 146}
\definecolor{8x4x4color6}{RGB}{255, 151, 151}
\definecolor{8x4x4color7}{RGB}{255, 156, 156}
\definecolor{8x4x4color8}{RGB}{255, 160, 160}
\definecolor{8x4x4color9}{RGB}{255, 165, 165}
\definecolor{8x4x4color10}{RGB}{255, 170, 170}
\definecolor{8x4x4color11}{RGB}{255, 175, 175}
\definecolor{8x4x4color12}{RGB}{255, 179, 179}
\definecolor{8x4x4color13}{RGB}{255, 184, 184}
\definecolor{8x4x4color14}{RGB}{255, 189, 189}
\definecolor{8x4x4color15}{RGB}{255, 193, 193}
\definecolor{8x4x4color16}{RGB}{255, 198, 198}
\definecolor{8x4x4color17}{RGB}{255, 203, 203}
\definecolor{8x4x4color18}{RGB}{255, 207, 207}
\definecolor{8x4x4color19}{RGB}{255, 212, 212}
\definecolor{8x4x4color20}{RGB}{255, 217, 217}
\definecolor{8x4x4color21}{RGB}{255, 222, 222}
\definecolor{8x4x4color22}{RGB}{255, 226, 226}
\definecolor{8x4x4color23}{RGB}{255, 231, 231}
\definecolor{8x4x4color24}{RGB}{255, 236, 236}
\definecolor{8x4x4color25}{RGB}{255, 240, 240}
\definecolor{8x4x4color26}{RGB}{255, 245, 245}
\definecolor{8x4x4color27}{RGB}{255, 255, 255}

\begin{center}
\begin{tikzpicture}[scale=0.5]
\begin{scope}
\draw[very thick] (0,-0) rectangle (8,-4);
\foreach \b in {0,...,3}
 \foreach \c in {0,...,7}
  \draw (\c,\b*-1-0) rectangle (\c+1,(\b*-1-1);
\draw[very thick] (0,-5) rectangle (8,-9);
\foreach \b in {0,...,3}
 \foreach \c in {0,...,7}
  \draw (\c,\b*-1-5) rectangle (\c+1,(\b*-1-6);
\draw[very thick] (0,-10) rectangle (8,-14);
\foreach \b in {0,...,3}
 \foreach \c in {0,...,7}
  \draw (\c,\b*-1-10) rectangle (\c+1,(\b*-1-11);
\draw[very thick] (0,-15) rectangle (8,-19);
\foreach \b in {0,...,3}
 \foreach \c in {0,...,7}
  \draw (\c,\b*-1-15) rectangle (\c+1,(\b*-1-16);

\draw[fill=red] (0,-0) rectangle (1,-1);
\draw[fill=8x4x4color9] (1,-0) rectangle (2,-1);
\node at (1.5,-0.5) {9};
\draw[fill=8x4x4color10] (2,-0) rectangle (3,-1);
\node at (2.5,-0.5) {10};
\draw[fill=8x4x4color11] (3,-0) rectangle (4,-1);
\node at (3.5,-0.5) {11};
\draw[fill=8x4x4color12] (4,-0) rectangle (5,-1);
\node at (4.5,-0.5) {12};
\draw[fill=8x4x4color13] (5,-0) rectangle (6,-1);
\node at (5.5,-0.5) {13};
\draw[fill=8x4x4color24] (6,-0) rectangle (7,-1);
\node at (6.5,-0.5) {24};
\draw[fill=8x4x4color25] (7,-0) rectangle (8,-1);
\node at (7.5,-0.5) {25};
\draw[fill=8x4x4color1] (0,-1) rectangle (1,-2);
\node at (0.5,-1.5) {1};
\draw[fill=red] (1,-1) rectangle (2,-2);
\draw[fill=8x4x4color1] (2,-1) rectangle (3,-2);
\node at (2.5,-1.5) {1};
\draw[fill=red] (3,-1) rectangle (4,-2);
\draw[fill=8x4x4color3] (4,-1) rectangle (5,-2);
\node at (4.5,-1.5) {3};
\draw[fill=red] (5,-1) rectangle (6,-2);
\draw[fill=8x4x4color23] (6,-1) rectangle (7,-2);
\node at (6.5,-1.5) {23};
\draw[fill=red] (7,-1) rectangle (8,-2);
\draw[fill=8x4x4color6] (0,-2) rectangle (1,-3);
\node at (0.5,-2.5) {6};
\draw[fill=8x4x4color5] (1,-2) rectangle (2,-3);
\node at (1.5,-2.5) {5};
\draw[fill=red] (2,-2) rectangle (3,-3);
\draw[fill=8x4x4color5] (3,-2) rectangle (4,-3);
\node at (3.5,-2.5) {5};
\draw[fill=8x4x4color6] (4,-2) rectangle (5,-3);
\node at (4.5,-2.5) {6};
\draw[fill=8x4x4color7] (5,-2) rectangle (6,-3);
\node at (5.5,-2.5) {7};
\draw[fill=8x4x4color24] (6,-2) rectangle (7,-3);
\node at (6.5,-2.5) {24};
\draw[fill=8x4x4color25] (7,-2) rectangle (8,-3);
\node at (7.5,-2.5) {25};
\draw[fill=8x4x4color19] (0,-3) rectangle (1,-4);
\node at (0.5,-3.5) {19};
\draw[fill=8x4x4color18] (1,-3) rectangle (2,-4);
\node at (1.5,-3.5) {18};
\draw[fill=8x4x4color17] (2,-3) rectangle (3,-4);
\node at (2.5,-3.5) {17};
\draw[fill=8x4x4color16] (3,-3) rectangle (4,-4);
\node at (3.5,-3.5) {16};
\draw[fill=8x4x4color15] (4,-3) rectangle (5,-4);
\node at (4.5,-3.5) {15};
\draw[fill=red] (5,-3) rectangle (6,-4);
\draw[fill=8x4x4color25] (6,-3) rectangle (7,-4);
\node at (6.5,-3.5) {25};
\draw[fill=8x4x4color26] (7,-3) rectangle (8,-4);
\node at (7.5,-3.5) {26};
\draw[fill=8x4x4color9] (0,-5) rectangle (1,-6);
\node at (0.5,-5.5) {9};
\draw[fill=8x4x4color8] (1,-5) rectangle (2,-6);
\node at (1.5,-5.5) {8};
\draw[fill=8x4x4color3] (2,-5) rectangle (3,-6);
\node at (2.5,-5.5) {3};
\draw[fill=red] (3,-5) rectangle (4,-6);
\draw[fill=8x4x4color3] (4,-5) rectangle (5,-6);
\node at (4.5,-5.5) {3};
\draw[fill=red] (5,-5) rectangle (6,-6);
\draw[fill=8x4x4color1] (6,-5) rectangle (7,-6);
\node at (6.5,-5.5) {1};
\draw[fill=red] (7,-5) rectangle (8,-6);
\draw[fill=red] (0,-6) rectangle (1,-7);
\draw[fill=8x4x4color3] (1,-6) rectangle (2,-7);
\node at (1.5,-6.5) {3};
\draw[fill=8x4x4color2] (2,-6) rectangle (3,-7);
\node at (2.5,-6.5) {2};
\draw[fill=8x4x4color1] (3,-6) rectangle (4,-7);
\node at (3.5,-6.5) {1};
\draw[fill=8x4x4color2] (4,-6) rectangle (5,-7);
\node at (4.5,-6.5) {2};
\draw[fill=8x4x4color1] (5,-6) rectangle (6,-7);
\node at (5.5,-6.5) {1};
\draw[fill=8x4x4color22] (6,-6) rectangle (7,-7);
\node at (6.5,-6.5) {22};
\draw[fill=8x4x4color23] (7,-6) rectangle (8,-7);
\node at (7.5,-6.5) {23};
\draw[fill=8x4x4color1] (0,-7) rectangle (1,-8);
\node at (0.5,-7.5) {1};
\draw[fill=8x4x4color4] (1,-7) rectangle (2,-8);
\node at (1.5,-7.5) {4};
\draw[fill=8x4x4color3] (2,-7) rectangle (3,-8);
\node at (2.5,-7.5) {3};
\draw[fill=8x4x4color4] (3,-7) rectangle (4,-8);
\node at (3.5,-7.5) {4};
\draw[fill=8x4x4color5] (4,-7) rectangle (5,-8);
\node at (4.5,-7.5) {5};
\draw[fill=8x4x4color8] (5,-7) rectangle (6,-8);
\node at (5.5,-7.5) {8};
\draw[fill=8x4x4color23] (6,-7) rectangle (7,-8);
\node at (6.5,-7.5) {23};
\draw[fill=8x4x4color24] (7,-7) rectangle (8,-8);
\node at (7.5,-7.5) {24};
\draw[fill=red] (0,-8) rectangle (1,-9);
\draw[fill=8x4x4color11] (1,-8) rectangle (2,-9);
\node at (1.5,-8.5) {11};
\draw[fill=8x4x4color12] (2,-8) rectangle (3,-9);
\node at (2.5,-8.5) {12};
\draw[fill=8x4x4color13] (3,-8) rectangle (4,-9);
\node at (3.5,-8.5) {13};
\draw[fill=8x4x4color14] (4,-8) rectangle (5,-9);
\node at (4.5,-8.5) {14};
\draw[fill=8x4x4color15] (5,-8) rectangle (6,-9);
\node at (5.5,-8.5) {15};
\draw[fill=8x4x4color24] (6,-8) rectangle (7,-9);
\node at (6.5,-8.5) {24};
\draw[fill=8x4x4color25] (7,-8) rectangle (8,-9);
\node at (7.5,-8.5) {25};
\draw[fill=8x4x4color10] (0,-10) rectangle (1,-11);
\node at (0.5,-10.5) {10};
\draw[fill=8x4x4color7] (1,-10) rectangle (2,-11);
\node at (1.5,-10.5) {7};
\draw[fill=red] (2,-10) rectangle (3,-11);
\draw[fill=8x4x4color1] (3,-10) rectangle (4,-11);
\node at (3.5,-10.5) {1};
\draw[fill=8x4x4color2] (4,-10) rectangle (5,-11);
\node at (4.5,-10.5) {2};
\draw[fill=8x4x4color1] (5,-10) rectangle (6,-11);
\node at (5.5,-10.5) {1};
\draw[fill=red] (6,-10) rectangle (7,-11);
\draw[fill=8x4x4color25] (7,-10) rectangle (8,-11);
\node at (7.5,-10.5) {25};
\draw[fill=8x4x4color7] (0,-11) rectangle (1,-12);
\node at (0.5,-11.5) {7};
\draw[fill=8x4x4color6] (1,-11) rectangle (2,-12);
\node at (1.5,-11.5) {6};
\draw[fill=8x4x4color1] (2,-11) rectangle (3,-12);
\node at (2.5,-11.5) {1};
\draw[fill=red] (3,-11) rectangle (4,-12);
\draw[fill=8x4x4color1] (4,-11) rectangle (5,-12);
\node at (4.5,-11.5) {1};
\draw[fill=red] (5,-11) rectangle (6,-12);
\draw[fill=8x4x4color21] (6,-11) rectangle (7,-12);
\node at (6.5,-11.5) {21};
\draw[fill=8x4x4color24] (7,-11) rectangle (8,-12);
\node at (7.5,-11.5) {24};
\draw[fill=red] (0,-12) rectangle (1,-13);
\draw[fill=8x4x4color5] (1,-12) rectangle (2,-13);
\node at (1.5,-12.5) {5};
\draw[fill=red] (2,-12) rectangle (3,-13);
\draw[fill=8x4x4color1] (3,-12) rectangle (4,-13);
\node at (3.5,-12.5) {1};
\draw[fill=8x4x4color2] (4,-12) rectangle (5,-13);
\node at (4.5,-12.5) {2};
\draw[fill=8x4x4color9] (5,-12) rectangle (6,-13);
\node at (5.5,-12.5) {9};
\draw[fill=8x4x4color20] (6,-12) rectangle (7,-13);
\node at (6.5,-12.5) {20};
\draw[fill=8x4x4color21] (7,-12) rectangle (8,-13);
\node at (7.5,-12.5) {21};
\draw[fill=8x4x4color1] (0,-13) rectangle (1,-14);
\node at (0.5,-13.5) {1};
\draw[fill=8x4x4color10] (1,-13) rectangle (2,-14);
\node at (1.5,-13.5) {10};
\draw[fill=8x4x4color11] (2,-13) rectangle (3,-14);
\node at (2.5,-13.5) {11};
\draw[fill=8x4x4color12] (3,-13) rectangle (4,-14);
\node at (3.5,-13.5) {12};
\draw[fill=8x4x4color13] (4,-13) rectangle (5,-14);
\node at (4.5,-13.5) {13};
\draw[fill=8x4x4color16] (5,-13) rectangle (6,-14);
\node at (5.5,-13.5) {16};
\draw[fill=8x4x4color17] (6,-13) rectangle (7,-14);
\node at (6.5,-13.5) {17};
\draw[fill=red] (7,-13) rectangle (8,-14);
\draw[fill=8x4x4color11] (0,-15) rectangle (1,-16);
\node at (0.5,-15.5) {11};
\draw[fill=red] (1,-15) rectangle (2,-16);
\draw[fill=8x4x4color3] (2,-15) rectangle (3,-16);
\node at (2.5,-15.5) {3};
\draw[fill=8x4x4color4] (3,-15) rectangle (4,-16);
\node at (3.5,-15.5) {4};
\draw[fill=8x4x4color5] (4,-15) rectangle (5,-16);
\node at (4.5,-15.5) {5};
\draw[fill=8x4x4color20] (5,-15) rectangle (6,-16);
\node at (5.5,-15.5) {20};
\draw[fill=8x4x4color23] (6,-15) rectangle (7,-16);
\node at (6.5,-15.5) {23};
\draw[fill=8x4x4color26] (7,-15) rectangle (8,-16);
\node at (7.5,-15.5) {26};
\draw[fill=8x4x4color10] (0,-16) rectangle (1,-17);
\node at (0.5,-16.5) {10};
\draw[fill=8x4x4color7] (1,-16) rectangle (2,-17);
\node at (1.5,-16.5) {7};
\draw[fill=8x4x4color2] (2,-16) rectangle (3,-17);
\node at (2.5,-16.5) {2};
\draw[fill=8x4x4color1] (3,-16) rectangle (4,-17);
\node at (3.5,-16.5) {1};
\draw[fill=red] (4,-16) rectangle (5,-17);
\draw[fill=8x4x4color19] (5,-16) rectangle (6,-17);
\node at (5.5,-16.5) {19};
\draw[fill=8x4x4color22] (6,-16) rectangle (7,-17);
\node at (6.5,-16.5) {22};
\draw[fill=8x4x4color25] (7,-16) rectangle (8,-17);
\node at (7.5,-16.5) {25};
\draw[fill=8x4x4color9] (0,-17) rectangle (1,-18);
\node at (0.5,-17.5) {9};
\draw[fill=8x4x4color8] (1,-17) rectangle (2,-18);
\node at (1.5,-17.5) {8};
\draw[fill=8x4x4color1] (2,-17) rectangle (3,-18);
\node at (2.5,-17.5) {1};
\draw[fill=red] (3,-17) rectangle (4,-18);
\draw[fill=8x4x4color1] (4,-17) rectangle (5,-18);
\node at (4.5,-17.5) {1};
\draw[fill=8x4x4color18] (5,-17) rectangle (6,-18);
\node at (5.5,-17.5) {18};
\draw[fill=8x4x4color19] (6,-17) rectangle (7,-18);
\node at (6.5,-17.5) {19};
\draw[fill=red] (7,-17) rectangle (8,-18);
\draw[fill=red] (0,-18) rectangle (1,-19);
\draw[fill=8x4x4color9] (1,-18) rectangle (2,-19);
\node at (1.5,-18.5) {9};
\draw[fill=red] (2,-18) rectangle (3,-19);
\draw[fill=8x4x4color1] (3,-18) rectangle (4,-19);
\node at (3.5,-18.5) {1};
\draw[fill=red] (4,-18) rectangle (5,-19);
\draw[fill=8x4x4color17] (5,-18) rectangle (6,-19);
\node at (5.5,-18.5) {17};
\draw[fill=red] (6,-18) rectangle (7,-19);
\draw[fill=8x4x4color1] (7,-18) rectangle (8,-19);
\node at (7.5,-18.5) {1};

\end{scope}
\end{tikzpicture}
\end{center}

%% file: perc-thin.bib
@article{Allikvere26,
author = {Allikvere, J.},
title = {private communication},
year = {2026}
}

@article{BushawClifton26,
author = {Bushaw, N. and Clifton, A.},
title = {3-Neighbor bootstrap percolation on two-dimensional grids},
journal = {https://arxiv.org/abs/2608.06133},
volume = {},
number = {},
pages = {},
year = {2026},
doi = {},

URL = {https://arxiv.org/abs/2608.06133} 
}

@article{DukesNoelRomer23,
author = {Dukes, P.J. and Noel, J. and Romer, A.E.},
title = {Extremal Bounds for 3-Neighbor Bootstrap Percolation in Dimensions Two and Three},
journal = {SIAM Journal on Discrete Mathematics},
volume = {37},
number = {3},
pages = {2088-2125},
year = {2023},
doi = {10.1137/22M1534195},

URL = {     
        https://doi.org/10.1137/22M1534195
},
eprint = { 
        https://doi.org/10.1137/22M1534195
}
}

@article{HedzetHenning25,
author = {Hed\v{z}et, J and Henning, M.A.},
title = {3-neighbor bootstrap percolation on grids},
journal = {Discussiones Mathematicae Graph Theory},
volume = {45},
number = {1},
pages = {283-310},
year = {2025},
doi = {10.7151/dmgt.2531},
URL = {     
        https://doi.org/10.7151/dmgt.2531
},
eprint = { 
        https://doi.org/10.7151/dmgt.2531
}
}

@article {ChalupaLeathReich79,
    AUTHOR = {Chalupa, J. and Leath, P. L. and Reich, G. R},
     TITLE = {Bootstrap percolation on a {B}ethe lattice},
   JOURNAL = {J. Phys. C},
  FJOURNAL = {Journal of Physics C: Solid State Physics},
    VOLUME = {12},
      YEAR = {1979},
     PAGES = {L31--L35}
}

@article {PrzykuckiShelton20,
    AUTHOR = {Przykucki, M.  and Shelton, T.},
     TITLE = {Smallest percolating sets in bootstrap percolation on grids},
   JOURNAL = {Electron. J. Combin.},
  FJOURNAL = {Electronic Journal of Combinatorics},
    VOLUME = {27},
      YEAR = {2020},
    NUMBER = {4},
     PAGES = {Paper No. 4.34, 11}
}

@techreport{Benevides+21+,
  TITLE = {Minimum lethal sets in grids and tori under 3-neighbour bootstrap percolation},
  AUTHOR = {Benevides, F. and Bermond, J.-C. and Lesfari, H. and Nisse, N.},
  NOTE = {{HAL Research Report 03161419v4}},
  YEAR = {2021}
}

@article {GravnerHolroydSivakoff21,
    AUTHOR = {Gravner, J. and Holroyd, A. E. and Sivakoff, D.},
     TITLE = {Polluted bootstrap percolation in three dimensions},
   JOURNAL = {Ann. Appl. Probab.},
  FJOURNAL = {The Annals of Applied Probability},
    VOLUME = {31},
      YEAR = {2021},
    NUMBER = {1},
     PAGES = {218--246}
}
